\documentclass[11pt]{article}
\usepackage{bbm}
 \usepackage{amssymb}
\usepackage{amssymb, amsthm, amsmath, amscd}
\usepackage[usenames,dvipsnames]{color}

\newtheorem{thm}{Theorem}[section]
\newtheorem{lem}[thm]{Lemma}
\newtheorem{prop}[thm]{Proposition}
\newtheorem{cor}[thm]{Corollary}
\theoremstyle{definition}
\newtheorem{NN}[thm]{}
\theoremstyle{definition}
\newtheorem{df}[thm]{Definition}
\theoremstyle{definition}
\newtheorem{rem}[thm]{Remark}

\theoremstyle{definition}

\renewcommand{\phi}{\varphi}

\newcommand{\N}{\mathbb{N}}
\newcommand{\Z}{\mathbb{Z}}
\newcommand{\Q}{\mathbb{Q}}
\newcommand{\R}{\mathbb{R}}
\newcommand{\C}{\mathbb{C}}
\newcommand{\T}{\mathbb{T}}

\numberwithin{equation}{section}

\newcommand{\Aff}{\operatorname{Aff}}

\newcommand{\id}{\operatorname{id}}

\newcommand{\aff}{\rm aff}

\newcommand{\cpc}{completely positive contractive linear map}
\newcommand{\morp}{contractive completely positive linear map}

\newcommand{\hm}{homomorphism}
\newcommand{\dt}{\delta}
\newcommand{\ep}{\varepsilon}

\newcommand{\ld}{\lambda}
\newcommand{\cd}{\cdots}

\newcommand{\qq}{{\quad \quad}}

\newcommand{\gm}{\gamma}
\newcommand{\sm}{\sigma}

\newcommand{\la}{\langle}
\newcommand{\ra}{\rangle}
\newcommand{\andeqn}{\,\,\,{\rm and}\,\,\,}
\newcommand{\rforal}{\,\,\,{\rm for\,\,\,all}\,\,\,}
\newcommand{\CA}{$C^*$-algebra}
\newcommand{\SCA}{$C^*$-subalgebra}

\newcommand{\af}{{\alpha}}
\newcommand{\bt}{{\beta}}
\newcommand{\dist}{{\rm dist}}

\newcommand{\one}{{\bf 1}}

\newcommand{\diag}{{\rm diag}}

\newcommand{\wilog}{without loss of generality}
\newcommand{\Wlog}{Without loss of generality}

\newcommand{\beq}{\begin{eqnarray}}
\newcommand{\eneq}{\end{eqnarray}}
\newcommand{\tforal}{\,\,\,\text{for\,\,\,all}\,\,\,}
\newcommand{\tand}{\,\,\,\text{and}\,\,\,}
\newcommand{\zo}{{\cal Z}_0}

\newcommand{\p}{\mathfrak{p}}
\newcommand{\q}{\mathfrak{q}}

\usepackage{amsfonts}
\usepackage{mathrsfs}
\usepackage{textcomp}
\usepackage[all]{xy}

\title{On classification  of  non-unital  amenable  simple  C*-algebras, II}
\author{Guihua Gong and  Huaxin Lin
 }
\date{
}

\begin{document}

\maketitle

\begin{abstract}
We present a classification theorem for separable amenable simple stably projectionless \CA s
with generalized tracial rank one whose $K_0$ vanish on traces which satisfy the Universal Coefficient Theorem.
One of \CA s in the class
is denoted by ${\cal Z}_0$ which has  a unique tracial state,
$K_0(\zo)=\Z$   and $K_1(\zo)=\{0\}.$
 Let $A$ and $B$ be two separable simple $C^*$-algebras satisfying the UCT and have finite nuclear dimension.
We show that $A\otimes \zo\cong B\otimes \zo$ if and only if ${\rm Ell}({{A}}\otimes \zo)={\rm Ell}(B\otimes \zo).$
A class of simple separable $C^*$-algebras which are  approximately sub-homogeneous whose spectra having bounded
dimension is shown to exhaust all possible Elliott invariant for $C^*$-algebras of the form
$A\otimes \zo,$ where $A$ is any finite separable simple amenable $C^*$-algebras.
Suppose that  $A$ and $B$ are two finite separable simple $C^*$-algebras with finite nuclear dimension satisfying the UCT
such that
traces vanishe on  $K_0(A)$ and $K_0(B)$
 (but arbitrary $K_1$).  One consequence of the main results
in this situation is that $A\cong B$
if and only if $A$ and $B$ have the isomorphic Elliott invariant.
\end{abstract}

\section{Introduction}
Recently some  sweeping progresses have been made in the Elliott  program (\cite{Ellicm}),
the program of classification of  separable amenable \CA s by the Elliott invariant
(a $K$-theoretical set of invariant) (see \cite{GLN}, \cite{TWW} and \cite{EGLN}).
These are the results of decades of work by many  mathematicians (see also \cite{GLN},
\cite{TWW} and \cite{EGLN} for the historical discussion there).
These progresses   could be summarized  briefly as the following:
Two unital  finite separable simple \CA s $A$ and $B$ with finite nuclear dimension which satisfy
the UCT are isomorphic if and only if  their Elliott invariant ${\rm Ell}(A)$ and ${\rm Ell}(B)$ are
isomorphic.  Moreover, all  weakly unperforated Elliott invariant
can be achieved by a finite separable simple \CA s in UCT class with finite nuclear dimension
(In fact these can be constructed as so-called ASH-algebras--see \cite{GLN}).
Combining with the previous classification of purely infinite simple \CA s, results
of Kirchberg and Phillips (\cite{Pclass} and \cite{KP}), now all unital separable simple \CA s in the UCT class with
finite nuclear dimension are classified by the Elliott invariant.

This research  studies the non-unital cases.

Suppose that $A$ is a separable simple \CA.
In the case that $K_0(A)_+\not=\{0\},$ then $A\otimes {\cal K}$ has a non-zero projection, say
$p.$ Then $p(A\otimes K)p$ is unital.  Therefore if $A$ is in the UCT class and has finite
nuclear dimension, then $p(A\otimes {\cal K})p$ falls into the class of \CA s which has been classified.
Therefore isomorphism theorem  for  these \CA s  is an immediate
consequence of that in \cite{GLN} (see section 8.4 of \cite{Lncbms}) using
the stable isomorphism theorem of \cite{Br1}.

Therefore this paper considers the case that $K_0(A)_+=\{0\}.$
Simple \CA s with $K_0(A)_+=\{0\}$ are stably projectionless in the sense
that not only $A$ has no non-zero projections but $M_n(A)$ also has no non-zero
projections for every integer $n\ge 1.$   However, as one may see in this paper,
$K_0(A)$ could still exhaust any countable abelian groups as well as any
possible $K_1(A).$ In particular, the results in \cite{GLN} cannot be applied
in the stably projectionless case. It is entirely new situation.   If one views \CA\, theory
 as the study of non-commutative topological spaces,  then unital \CA s correspond to the compact
  spaces and non-unital ones correspond
to non-compact spaces. However, stably projectionless simple \CA s may be viewed as
non-commutative topological spaces which are not even locally compact. This causes
great difficulties. Different methods have to be developed.  In fact, the current paper 
is mostly independent of \cite{GLN}.

{{In \cite{eglnp1},}}
we introduce  a  class  of stably projectionless simple \CA s
${\cal D}$
(see \ref{DD0} below).
We also introduced
the notion of generalized tracial rank one for stably projectionless simple
\CA s. These are separable stably projectionless simple \CA s which are stably isomorphic to
\CA s in ${\cal D}$ (see \ref{DD0} below).  If $A$ is stably isomorphic to one in ${\cal D},$
we will write $gTR(A)\le 1.$
Some study of the structure of these \CA s were also presented in \cite{eglnp1}.
For example, among other things, we show that \CA s have stable rank one.
Let $A$ and $B$ be two stably projectionless simple amenable \CA s satisfy the UCT. Suppose that
$K_0(A)=K_1(A)=K_0(B)=K_1(B)=\{0\}.$  In the first part of this research {{(see \cite{eglnkk0}),}} we show
that $A\cong B$ if and only if  ${\rm Ell}(A)\cong  {\rm Ell}(B)$
(see \cite{eglnkk0}).
In {{this}} case  the Elliott invariant is reduced to ${\rm Ell}(A)=({\tilde T}(A), \Sigma_A)$ (see \ref{DEll} below).
Combining the above mentioned results,
this also gives a classification for separable stably finite projectionless simple \CA s with finite
nuclear dimension in the UCT class with trivial $K_i$-theory.

In  {{the current paper,}}
 we study the general case that $K$-theory of \CA s are non-trivial.
We give the following theorem:

\begin{thm}\label{TTT1}{\rm (see \ref{T1main})}
Let $A$ and $B$ be two separable simple amenable \CA s which satisfy the UCT.
Suppose that $gTR(A)\le 1$ and $gTR(B)\le 1$ and
$K_0(A)={\rm ker}\rho_A$ and $K_0(B)={\rm ker}\rho_B.$
Then $A\cong B$ if and only if
\beq
{\rm Ell}(A)\cong {\rm Ell}(B).
\eneq
\end{thm}
%
Among all stably projectionless separable simple \CA s,   one particularly interesting one
is ${\cal W},$ a separable \CA\, with only one tracial state such that
$K_0({\cal W})=K_1({\cal W})=\{0\}.$ ${\cal W}$ is also an inductive limit of sub-homogeneous \CA s (see \cite{Raz}).
It was shown in the first part {{(\cite{eglnp1}
and \cite{eglnkk0})}} of this research that if $A$ is a separable simple \CA\, in the UCT class,
with finite nuclear dimension, with a unique tracial state and zero $K_i(A),$ then $A\cong {\cal W}.$

In this part of the research, another stably projectionless simple \CA\, $\zo$ with a unique tracial state
plays a prominent role. This \CA\, has the property that $K_0(\zo)=\Z$
 and $K_1(\zo)=\{0\}.$ As abelian groups,
$K_i(\zo)=K_i(\C),$ $i=0,1.$  Therefore, by {{the}} K\"unneth Formula, for any separable \CA\, $A,$
$K_i(A\otimes \zo)=K_i(A),$ as abelian group, $i=0,1.$ Moreover,  if the tracial state
space of $A$ is not empty, then
$T(A\otimes \zo)=T(A),$ since $\zo$ has only one tracial state.
As consequence of our main
results, $\zo\otimes \zo\cong \zo.$  Moreover, we show that
if $A$ is a separable simple \CA\, in the UCT class, with finite nuclear dimension,  unique tracial state,
$K_1(A)=\{0\}$ and $K_0(A)={\rm ker}\rho_A\cong \Z,$ then $A\cong \zo$ (see \ref{Czo}
below).
Therefore we are particularly interested in  $\zo$-stable \CA s, i.e., those \CA s with
the property that $A\otimes \zo\cong A.$

We prove the following theorem:

\begin{thm}\label{TTT3}{\rm (see \ref{Mclass3})}
Let $A$ and $B$ be two   separable simple \CA s with finite nuclear dimension which
satisfies the UCT.
Then $A\otimes \zo\cong B\otimes \zo$ if and only if
\beq
{\rm Ell}(A\otimes \zo)\cong {\rm Ell}(B\otimes \zo).
\eneq
\end{thm}

When $A$ and $B$ are infinite,  then  both $A\otimes \zo$ and
$B\otimes \zo$ are purely infinite simple. This case is covered by Kirchberg-Phillips
classification theorem (see \cite{KP} and \cite{Pclass}).

We also present models for \CA s stably isomorphic to \CA s in ${\cal D}.$
These model  \CA s are  locally approximated by sub-homogeneous \CA s whose spectra
have dimension no more than 3.
We show that these \CA s exhaust all possible  Elliott invariant for separable $\zo$-stable \CA s  as
stated as follows (see also \ref{Ctsang} below):

\begin{thm}\label{TTT4} {\rm (see \ref{Mainrange})}
Let $A$ be a finite separable simple amenable \CA. Then
there exists a stably projectionless simple \CA\, $B$ which is locally approximated
by sub-homogeneous \CA s  and which is stably isomorphic to a \CA\, in ${\cal D}$ such that
\beq
{\rm Ell}(A\otimes \zo)={\rm Ell}(B).
\eneq

\end{thm}

Finally, we show that the condition that $A$ and $B$ have generalized tracial rank at most one  in Theorem \ref{TTT1} can be replaced by
finite nuclear dimension when traces vanishe on $K_0(A)$ and $K_0(B).$  In fact, we have the following:

\begin{thm}[see \ref{TMT} below]\label{TTT6}
Let $A$ and $B$ be two finite separable simple \CA s with finite nuclear dimension which satisfy the UCT.
Suppose that
%
$K_0(A)={\rm ker}\rho_A$ and $K_0(B)={\rm ker}\rho_B.$
Then $A\cong B$ if and only if ${\rm Ell}(A)\cong {\rm Ell}(B).$
\end{thm}

The paper also includes an appendix which shows every separable and amenable \CA\, in ${\cal D}$ is ${\cal Z}$-stable
which is based on \cite{MS}.
This research is also  benefitted from  previous results related to the classification of simple projectionless \CA s (such as   \cite{SL}, \cite{Rl}, \cite{Raz}, \cite{Rz}, and \cite{aTz}, as well as many others).





{\small {\bf Acknowledgement}:    This research
began when both authors stayed
in the Research Center for Operator Algebras in East China Normal University
in the summer of 2016
and December 2016.    Both authors acknowledge the support by the Center
 which is in part supported  by NNSF of China (11531003)  and Shanghai Science and Technology
 Commission (13dz2260400)
 and  Shanghai Key Laboratory of PMMP.
The second named author was also supported by NSF grants (DMS 1361431 and DMS 1665183).
}

\section{Preliminaries}

\begin{df}\label{Dceil}
Let $A$ be a {{unital}}
 \CA\, and let  $x\in A.$ Suppose
that
$\|xx^*-1\|<1$ and $\|x^*x-1\|<1.$ Then $x|x|^{-1}$ is a unitary.
Let us use $\lceil x \rceil $ to denote $x|x|^{-1}.$

Denote by $U(A)$  the unitary group of $A$ and denote by $U_0(A)$ the  normal subgroup
of $U(A)$ consisting of those unitaries which are path connected with $1_A.$  Denote by $CU(A)$
the closure of the commutator subgroup of $U(A).$

If $u\in A$ is a unitary, then ${\bar u}$ is the image of $u$ in $U(A)/CU(A),$ and if ${\cal U}\subset U(A)$
is a subset, then $\overline{{\cal U}}=\{{\bar u}: u\in {\cal U}\}.$

\end{df}

\begin{df}\label{DtensorMr}
Let $A$ be a \CA. Denote by $A^{\bf 1}$ the unit ball of $A.$

Let $B$ be another \CA\, and let $\phi: A\to B$ be a {{completely}} positive linear map.
Suppose that $r\ge 1$ be an integer.
This map induces a {{completely}} positive linear map  $\phi\otimes {\rm id}_{M_r}: A\otimes M_r \to B\otimes M_r.$
{\it Throughout this paper, we will use notation $\phi$ instead of  $\phi\otimes {\rm id}_{M_r}$
whenever it is convenient.}

Let $A$ be a non-unital \CA\, and let $\phi: A\to B$ (for some \CA\, $B$) be a linear map.
{\it   {{Sometime  in}}
the paper, we will continue to use $\phi$ for the {{unital}} extension from ${\tilde A}$ to ${\tilde B},$
whenever it is convenient.}
\end{df}

\begin{df}\label{DTtilde}
Let $A$ be a \CA.  Denote by $T(A)$ the tracial state of $A$ {{(which could be an empty set).}}
Let $\Aff(T(A))$ be the space of all real valued affine continuous functions on $T(A)$.
Let ${\tilde{T}}(A)$ be the cone of densely defined,
positive lower semi-continuous traces on $A$ equipped with the topology
of point-wise convergence on elements of the Pedersen ideal  ${\rm Ped}(A)$ of $A.$
Let $B$ be another \CA\, with $T(B)\not=\emptyset$ and let $\phi: A\to B$ be a \hm.
We will use then  $\phi_T: T(B)\to T(A)$ for the induced continuous affine map.

Let $r\ge 1$ be an integer and $\tau\in {\tilde T}(A).$
We will continue to use $\tau$  on $A\otimes M_r$ for $\tau\otimes {\rm Tr},$ where ${\rm Tr}$ is the standard
trace on $M_r.$
Let  $S\subset {\tilde T}(A)$  be a convex subset. Define (see \cite{Rl})
\beq
\Aff(S)_+&=&\{f: C(S, \R)_+: f \,\, 
{{\rm affine}}, f(\tau)\ge 0\},\\
\Aff_+(S)&=&\{f: C(S, \R)_+: f \,\, 
{{\rm affine}}, f(\tau)>0\,\,{\rm for}\,\,\tau\not=0\}\cup \{0\},\\
{\rm LAff}_f(S)_+&=&\{f:S\to [0,\infty): \exists \{f_n\}, f_n\nearrow f,\,\,
 f_n\in \Aff(S)_+\},\\
{\rm LAff}_{f,+}(S)&=&\{f:S\to [0,\infty): \exists \{f_n\}, f_n\nearrow f,\,\,
 f_n\in \Aff_+(S)\},\\
{\rm LAff}(S)_+&=&\{f:S\to [0,\infty]: \exists \{f_n\}, f_n\nearrow f,\,\,
 f_n\in \Aff(S)_+\},\\
{\rm LAff}_+(S)&=&\{f:S\to [0,\infty]: \exists \{f_n\}, f_n\nearrow f,\,\,
 f_n\in \Aff_+(S)\}\andeqn\\
 {\rm LAff}^{\sim}(S) &=&\{f_1-f_2: f_1\in {\rm LAff}_+(S)\andeqn f_2\in
 {\rm Aff}_+(S)\}.
 \eneq
 For most part of this paper, $S={\tilde T}(A)$ {{or}} $S=T(A)$
in the above definition will be used.
Moreover, ${\rm LAff}_{b,+}({{\tilde{T}}}(A))$ is the subset of those bounded functions
 in ${\rm LAff}_{f,+}({\tilde{T}}(A)).$ 
\end{df}

\begin{df}\label{DJc}
Let $A$ be a  \CA\, with $T(A)\not=\emptyset.$
Let $\pi_A: {\tilde A}\to \C$ be the quotient map and $s: \C\to {\tilde A}$ be
the \hm\, such that $\pi\circ s={\rm id}_{\C}.$
Recall that we also use $\pi_A$ for the induced \hm\, $\pi_A\otimes {\rm id}_{M_r}: M_r({\tilde A})\to M_r$
and use $s$ for the induced \hm\, $s\otimes {\rm id}_{M_r}: M_r\to M_r({\tilde A})$
for all integer $r\ge 1.$
Let $\rho_A: K_0(A)\to \Aff(T(A))$ be the order preserving
\hm\, defined by $\rho([p]-[s\circ \pi_A(p)])(\tau)=\tau(p-s\circ \pi_A(p))$ for any projections
in $M_r({\tilde A})$ for all integer $r\ge 1.$

Suppose that $A$ is non-unital and  separable, and  {{${\tilde T}(A)\not=\{0\}.$}}
Suppose that there exists $a\in {\rm Ped}(A)_+$  which is full.
Let $A_a=\overline{aAa}.$ Then $T(A_a)\not=\emptyset.$
We define
\beq
{\rm ker}\rho_A=\{x\in K_0(A_a): \rho_A(x)=0\}
\eneq
Here we also identify $K_0(A_a)$ with $K_0(A)$ using the Brown's stable isomorphism
theorem (\cite{Br1}).

Suppose that $A$ is unital and has stable rank one. Then we have  (by \cite{Thomsen} and \cite{GLX})
the following splitting short exact sequence (we will fix one such $J_c$)
\beq\label{CUsplit}
0\longrightarrow \Aff(T(A))/{\overline{\rho_A(K_0(A))}} \longrightarrow U(A)/CU(A)\rightleftarrows^{\kappa_1^A}_{J_c} K_1(A)\longrightarrow 0.
\eneq

If $u\in U_0(A)$ and $\{u(t): t\in [0,1]\}$ is a piece-wise smooth and continuous path of unitaries in
$A$ such that $u(0)=u$ and $u(1)=1.$
Then, for each $\tau\in T(A),$
\beq
D_A(u)(\tau)={1\over{2\pi i}}\int_0^1 \tau({du(t)\over{dt}}u(t)^*) dt
\eneq
modulo $\overline{\rho_A(K_0(A))}$  induces  (independent of  the path)  an isomorphism (denote by  ${\bar D}_A$)
from $U_0(A)/CU(A)$ onto ${\rm Aff}(T(A))/{\overline{\rho_A(K_0(A))}}$ as mentioned
above (see also 2.15 of \cite{GLN} ).

Now suppose that $A$ is a non-unital  separable \CA\, and  ${\rm Ped}(A)=A$  with $T(A)\not=\emptyset.$
Suppose that ${\rm ker}\rho_A=K_0(A).$
Then
\beq
\Aff(T({\tilde A}))/\overline{\rho_A(K_0({\tilde A}))}=\Aff(T({\tilde A}))/\Z.
\eneq
\end{df}

\begin{df}\label{Dalsr1}
Let $A$ be a non-unital \CA. We say that $A$ has almost stable rank one (see \cite{Rz} and  \cite{eglnp1})
if, for each $n,$ the invertible elements in  any nonzero hereditary \SCA\, ${\tilde B}$ of $M_n({\tilde A})$ is dense 
in $B,$ i.e., for any $b\in B$ and any $\ep>0,$ there exists an invertible element $x\in {\tilde B}$ such that $\|b-x\|<\ep.$
\end{df}

\begin{prop}[cf. Theorem 3 of \cite{CEI}; see also 1.5 of \cite{LnHilbert}]\label{Srk1}
{{Let $A$ be a $\sigma$-unital \CA\, which has almost stable rank one and let $a, b\in A_+\setminus\{0\}$ such that
$a\sim b$ in Cuntz semigroup. Then there is a partial isometry $w\in A$ such that
$w^*x, xw\in A$ and $ww^*x=xww^*=x$ for all $x\in \overline{aAa},$ $wy, yw^*\in {{A}}$ {{for all}} ${{y\in}} \overline{bAb}$ and
$w^*aw$ is a strictly positive element of $\overline{bAb}.$}}
\end{prop}

\begin{proof}
Let $H_1=\overline{aA}$ and $H_2=\overline{bA}$ be Hilbert $A$-modules.
By 3.3 of \cite{Rz}, there is a  Hilbert $A$-module isomorphism $\phi:H_1\to H_2.$
Since $a^{1/2}\in \overline{aA},$ $\phi(a^{1/2})\phi(a^{1/2})^*$ is a strictly positive element of $\overline{bAb}$
and $\phi(a^{1/2})^*\phi(a^{1/2}){{=\la \phi(a^{1/2}), \phi(a^{1/2})\ra_{H_2}=\la a^{1/2}, a^{1/2}\ra_{H_1}= a.}}$
Consider $H=H_1\oplus H_2,$ $a_1=\diag(a, 0)$ and $b_1=\diag(0,b)\in M_2(A).$
Let $\{e_{i,j}\}_{1\le i,j\le 2}$ be a matrix unit for $M_2.$
Set $B=\overline{(a_1+b_1)M_2(A)(a_1+b_1)}.$  There is $T_1\in LM(K(H))=LM(B)$ (see  Theorem 1.5 of \cite{Linbmp}) such
that $T_1(x_1\oplus x_2)=0\oplus \phi(x_1)$ for all $(x_1, x_2)\in H.$
Put $T=e_{1,2}T.$ Then $Tx=\phi(x)$ for all $x\in H_1$ and
 $a^{1/2}T^*Ta^{1/2}=\phi(a^{1/2})^*\phi(a^{1/2})=a.$
Moreover $TaT^*=\phi(a^{1/2})\phi(a^{1/2})^*$ is a strictly positive element in $\overline{bAb}$ {{and $Ta^{1/2}\in A.$}}
{{Write}}  $Ta^{1/2}=v|a^{1/2}T^Ta^{1/2}|=v|a|$
as polar decomposition in {{$A^{**}.$}}
One then checks that $w:=v^*$ satisfies the requirement.
\end{proof}

\begin{df}\label{DLddag}
Let $A$ be a unital separable amenable \CA.
For any finite subset ${\cal U}\subset U(A),$ there exists $\dt>0$ and  a finite subset ${\cal G}\subset  A$
satisfying the following:
If $B$ is another unital \CA\, and if $L: A\to B$ is {{a ${\cal G}$-$\dt$}}-
multiplicative \cpc,
then
$\overline{\lceil L(u)\rceil}$ is a well defined element in $U(B)/CU(B)$  for all $u\in {\cal U}.$
{{We may assume that $[L]|_{{\cal S}}$ is well defined, where ${\cal S}$ is the image of ${\cal U}$
in $K_1(A)$ (see, for example, 2.12 of \cite{GLN}).}}
Let $G({\cal U})$ be the subgroup generated by ${\cal U}.$
{{Suppose that $1/2>\ep>0$ is given.  By  Appendix in \cite{Lnloc},  we may assume that
there is a \hm\, $L^\dag: G({\cal U})\to U(B)/CU(B)$ such that
\beq\label{Ddag-n1}
{\rm dist}(L^\dag(\bar u),\overline{\lceil L(u)\rceil})<\ep\rforal u\in {\cal U}.
\eneq}}
{{Moreover, as in  Definition 2.17 of  \cite{GLN},  we may also assume that
\beq\label{Ddag-n2}
L^{\ddag}((G({\cal U})\cap U_0(A))/CU(A))\subset U_0(B)/CU(B).
\eneq}}
{{It follows that
$\kappa_1^B\circ L^{\ddag}(\overline{u})=[L]\circ \kappa_1^A([u])\rforal u\in G({\cal U}),$
where $\kappa_1^A$ and $\kappa_1^B$ are defined as in \eqref{CUsplit}
(see Definition 2.17 of \cite{GLN}).}}
In what follows,  {{when $1/2>\ep>0$ is given, whenever we write
 $L^{\dag},$ we mean that $\dt$ is small enough and ${\cal G}$ is large enough so that $L^{\dag}$ is
 defined,  \eqref{Ddag-n1} and \eqref{Ddag-n2} holds (see 2.17 of \cite{GLN}).}}
 Moreover, for an integer $k\ge 1,$ we will also use $L^{\dag}$
 {{for the map on  some given subgroup of}} $U(M_k(A))/CU(M_k(A))$ induced by $L\otimes {\rm id}_{M_k}.$  In particular, when $L$ is a unital \hm, the map $L^{\dag}$ is well defined
 on $U(M_k(A))/CU(M_k(A)).$

{{If $A$ is not unital, $L^\dag$ is defined to be ${\tilde L}^\dag,$ where ${\tilde L}: {\tilde A}\to {\tilde B}$
 is the unital extension of $L.$}}

\end{df}

\begin{df}\label{Dfep}
{{Let $1>\ep>0.$  Define
\beq
f_{\ep}(t)=\begin{cases} 0, &\text{if}\, t\in [0,\ep/2];\\
                                      {t-\ep/2\over{\ep/2}}, &\text{if}\, t\in (\ep/2, \ep];\\
                                       1 & \text{if}\, t\in (\ep, \infty).\end{cases}
                                       \eneq}
                                       }
\end{df}

\begin{df}\label{Ddimf}
{{Let $A$ be a \CA\, and let $a\in A_+.$  Suppose that {{${\tilde T}(A)\not=\{0\}.$}}
Recall that
$$
d_\tau(a)=\lim_{\ep\to 0} \tau(f_\ep(a))
$$
with possible infinite value.  Note that $f_\ep(a)\in {\rm Ped}(A)_+$ {{for any $\ep>0.$}} Therefore
$\tau\mapsto d_\tau(a)$ is a lower semi-continuous affine function  on ${\tilde T}(A)$
(to $[0,\infty]$).
Suppose that $A$ is non-unital. Let $a\in A_+$ be a strictly positive element.
Define
$$
\Sigma_A(\tau)=d_\tau(a)\rforal \tau\in {\tilde T}(A).
$$
It is standard and routine to check that $\Sigma_A$ is independent of the choice
of $a.$
The lower semi-continuous affine function $\Sigma_A$ is called the scale function of $A.$
(see
2.3 of \cite{eglnp1}). }}
\end{df}

\begin{df}\label{DEll}
{{
Let $C_1$ and $C_2$ be two cones. A cone map $\gamma: C_1\to C_2$
is an additive  map such that $\gamma(0)=0,$ $\gamma(rc)=r\gamma(c)$ for all $r\in \R_+.$
}}

{{Let $A$ be a stably projectionless simple \CA s such that $K_0(A)={\rm ker}\rho_A.$
Then the Elliott invariant
is defined as follows:
$$
{\rm Ell}(A)=(K_0(A), K_1(A), {\tilde T}(A), \Sigma_A).
$$
Suppose that $B$ is another stably projectionless simple \CA s such that $K_0(B)={\rm ker}\rho_B.$
Then we write
$$
{\rm Ell}(A)\cong {\rm Ell}(B),
$$
if there are group isomorphisms $\kappa_i: K_i(A)\to K_i(B),$ $i=0,1,$
a  cone homeomorphism $\kappa_T: {\tilde T}(A)\to {\tilde T}(B),$ i.e.,
$\kappa_T$ is $1-1$ and onto, $\kappa_T$ and $\kappa^{-1}$ are both cone maps which
are continuous (regarding topology  of point-wise convergence on elements in ${\rm Ped}(A)$),
and $\Sigma_A(\tau)=\Sigma_B(\kappa_T(\tau))$ for all $\tau\in {\tilde T}(A).$
}}
{{
{{In the case}} that $A$ has continuous scale, then one can simplify {{${\rm Ell}(A)$ to}}
$$
{\rm Ell}(A)=(K_0(A), K_1(A), T(A)).
$$
}}
\end{df}

\begin{df}\label{Dcomparible}
{{Let $A$ and $B$ be  \CA s with $T(A)\not=\emptyset$  and
$T(B)\not=\emptyset$ and both have stable rank one. Let $\kappa\in KL(A,B),$ $\kappa_T: T(B)\to T(A)$
be an affine continuous map and $\kappa_u: U({\tilde A})/CU({\tilde A})\to  U({\tilde B})/CU({\tilde B})$ {{be
a continuous \hm.}}
We say $(\kappa, \kappa_T, \kappa_u)$ {{is}} compatible, if
$\rho_B(\kappa(x))(t)=\rho_A(x)(\kappa_T(t))$ for all $x\in K_0(A)$ and $t\in T(B),$
$\kappa(\kappa_1^A({\bar w}))={{\kappa_1^B(\kappa_u({\bar w}))}}$ for all ${\bar w}\in U({{\tilde{A}}})/CU({{\tilde{A}}})
$
 and
$D_{{\tilde B}}(z)(t)=D_{{\tilde A}}(w)(\kappa_T(t))$ for all $t\in T(B),$ where
$w\in U_0(A),$ $z\in U_0(B)$ such that ${\bar z}=\kappa_u({\bar w})$ for all $w\in U_0(A),$}}  {{where $\kappa_1^A$ (and
$\kappa_1^B$) are  as in \eqref{CUsplit}.}}
\end{df}

\begin{df}\label{Dappmul}
Let $A$ and $B$ be two separable \CA s and let $\phi_n: A\to B$ be a sequence
of linear maps. We say that $\{\phi_n\}$ is approximately multiplicative, if
\beq
\lim_{n\to\infty}\|\phi_n(a)\phi_n(b)-\phi_n(ab)\|=0\rforal a, b\in A.
\eneq

{{Recall that $\tau$ is said to be a ${\cal W}$-trace in \cite{eglnp1}
if there exists a sequence of approximately multiplicative
\cpc s $\{\phi_n\}$ from $A$ into ${\cal W}$ such that
\beq\nonumber
&&\tau(a)=\lim_{n\to\infty}{ \tau_{\mathcal W}}(\phi_n(a))\rforal a\in A,
\eneq
where ${ \tau_{\mathcal W}}$ is the unique tracial state on ${\cal W}.$}}

\end{df}

\begin{df}\label{DQ}
Throughout this paper, $Q$ will be the universal UHF-algebra
with $K_0(Q)=\Q$ {{and}} $[1_Q]=1.$
\end{df}

\begin{df}\label{Dlocapp}
Let ${\cal B}$ be a class of \CA s and let $A$ be a separable \CA.
We say $A$ is {\it locally approximated by} \CA s in ${\cal B},$
if, for $\ep>0$ and any finite subset ${\cal F}\subset A,$ there exists a {{\SCA}}\, $B\in {\cal B}$
such that
${\rm dist}(a, B)<\ep$ for all $a\in {\cal F}.$
\end{df}

\begin{df}\label{Dlambdas}
Let $A$ be a \CA\, with $T(A)\not=\emptyset.$
Suppose that $A$ has a strictly positive element $e_A\in {\rm Ped}(A)_+$ with $\|e_A\|=1.$
Then $0\not\in \overline{T(A)}^w,$ the closure of $T(A)$ in ${\tilde T}(A)$
(see section 5 of \cite{eglnp1}).
Define
$$
\lambda_s(A)=\inf\{d_\tau(e_A): \tau\in A\}.
$$
Let $A$ be a \CA\, with $T(A)\not=\{0\}$ such that $0\not\in \overline{T(A)}^w.$
There is an affine  map
$r_{\aff}: A_{s.a.}\to \Aff(\overline{T(A)}^w)$ defined by
$$
r_{\aff}(a)(\tau)=\hat{a}(\tau)=\tau(a)\tforal \tau\in \overline{T(A)}^w
$$
and for all $a\in A_{s.a.}.$ Denote by $A_{s.a.}^q$ the space  $r_{\aff}(A_{s.a.})$ and
$A_+^q=r_{\aff}(A_+).$  
\end{df}

\begin{df}\label{Dconsc}(see 2.5 of \cite{Lncs1})
Let $A$ be a $\sigma$-unital, nonunital,  non-elementary, simple \CA\, and
$\{e_n\}$ be an approximate identity such that $e_{n+1}e_n=e_n$ for all $n.$
We say $A$ has continuous scale if, for any $a\in A_+\setminus \{0\},$ there exists
$n_0\ge 1$ such that $e_m-e_n\lesssim a$ for all $m\ge n\ge n_0.$
\end{df}

\begin{df}[5.5 of \cite{eglnp1}]\label{Dfulln}
{{Let $A$ be a separable \CA, let $B$ be {{a}} non-unital \CA\, and let $L: A\to B$ be a positive linear map.
 Let $F: A_+\setminus \{0\}\to \N\times \R_+\setminus\{0\}.$
Suppose that ${\cal H}\subset A_+\setminus \{0\}$ is  a subset.
We shall say that $L$ is $F$-${\cal H}$-full, if, for any $a\in {\cal H},$  for any $b\in B_+$ with $\|b\|\le 1,$ any $\ep>0,$
there are $x_1, x_2,...,x_m\in B$ such that
$m\le N(a)$ and $\|x_i\|\le M(a),$ where
$(N(a), M(a))=F(a),$  and
\beq\label{localfull-1}
\|\sum_{i=1}^mx_i^*L(a)x_i-b\|{{< }}\ep.
\eneq
This term is consistent with the uniformly $F$-${\cal H}$-fullness (3.11 of \cite{eglnkk0}) since $F$ does not depends on $\ep.$}}

\end{df}

\section{Non-commutative 1-dimensional complices, revisited}

 \begin{df}[See \cite{ET-PL} and  \cite{point-line}]\label{DfC1}
{\rm
Let $F_1$ and $F_2$ be two finite dimensional \CA s.
Suppose that there are two (not necessary unital)  \hm s
$\phi_0, \phi_1: F_1\to F_2.$
Denote the mapping torus $M_{\phi_1, \phi_2}$ by
$$
A=A(F_1, F_2,\phi_0, \phi_1)
=\{(f,g)\in  C([0,1], F_2)\oplus F_1: f(0)=\phi_0(g)\andeqn f(1)=\phi_1(g)\}.
$$


Denote by ${\cal C}$ the class of all  \CA s of the form $A=A(F_1, F_2, \phi_0, \phi_1)$ and  all finite dimensional \CA s.
These \CA s {{are called  Elliott-Thomsen building blocks as well as
one  dimensional non-commutative CW complexes.}}


{{Recall that ${{\cal C}_0}$ is the class of all $A\in {\cal C}$
with $K_0(A)_+=\{0\}$ such that $K_1(A)=0$,  and ${\cal C}_0^{(0)}$ the class of all $A\in {\cal C}_0$ such that $K_0(A)=0.$  Denote by {{${\cal C}'$,}} ${\cal C}_0'$  and ${\cal C}_0^{0'}$ the class of all full hereditary \SCA s of \CA s in {{$\cal C$,}} ${\cal C}_0$ and
{{${\cal C}_0^{0},$}} respectively.}}
}

Recall that ${\cal R}$ denotes the class of finite direct sums of Razak algebras and
${\cal M}_0$ denotes the class of all simple inductive limits of \CA s in ${\cal R}$ (with injective connecting maps)
{{(see 6.1 and 9.5 of \cite{eglnp1} and also
10.1, 16.2  and  16.5 of \cite{GLp1}).}}
\end{df}

\begin{NN}\label{2Rg11}
Let $F_1=M_{R_1 }(\C)\oplus M_{R_2 }(\C)\oplus \cdots \oplus M_{R_l}(\C),$  let
 $F_2=M_{r_1}(\C)\oplus M_{r_2 }(\C)\oplus \cdots \oplus M_{r_k }(\C)$  {{and}}  let $\phi_0,\,\,
  \phi_1:~ F_1\to F_2$ be (not necessary unital) homomorphisms, where $R_j$ and $r_i$ are positive integers.  Then $\phi_0$ and $\phi_1$ induce homomorphisms $$\phi_{0*}, \phi_{1*}: ~ K_0(F_1)=\Z^l \longrightarrow K_0(F_2)=\Z^k$$
 by matrices $(a_{ij})_{k\times l}$ and $(b_{ij})_{k\times l}$, respectively, { and $\sum_{j=1}^l a_{ij}R_j\leq r_i$ for $i=1,2,...,k.$}
{{ We may write $C([0,1], F_2)=\oplus_{j=1}^k C([0,1]_j, M_{r_j}{{)}},$ where $[0,1]_j$ denotes
 the $j$-th interval.}}
\end{NN}

\begin{thm}\label{Ldert}
Let $A$ {{be a full hereditary \SCA\, of a \CA\, in ${\cal C}.$}}
Then ${\rm cer}(u)\le 2+\ep$ if $u\in U_0({{\tilde A}}).$
Moreover, if $u\in CU({\tilde{A}})$ then, for any $\ep>0,$   there exists a
continuous path $\{u(t): t\in [0,1]\}\subset CU({\tilde A})$ with
$u(0)=u,$ $u(1)=1_{\tilde A}$ and ${\rm length}(\{u(t): t\in [0,1]\})\le 4\pi+\ep.$
In particular, ${\rm cel}(u)\le 4\pi.$
\end{thm}

\begin{proof}
Let $e\in B:=A(F_1, F_2, \phi_0, \phi_1)$ with $\|e\|=1$ and $A=\overline{eBe}.$
Let $u\in U_0({\tilde A})$ and let $\ep>0.$
\Wlog, we may assume
that $\ep<{{\frac{1}{4\max\{R(i)r_j: 1\le i \le l ,1\le j\le k\}}}}.$

It follows  8.8 of \cite{eglnp1}
that
$e$ is approximately unitarily equivalent  (in ${\tilde B}$) to
another positive element $e'$ which has the following form $e'=(g,a)\in B$ such that
\beq
g_j:=g|_{[0,1]_j}=\sum_{i=1}^{r_j}\lambda_{i,j}p_{i,j}, \,\,\,j=1,2,...,k,
\eneq
where $\lambda_{1,j}, \lambda_{2,j},...,\lambda_{r_j,j}\in C([0,1])$ and
{{$p_{1, j},p_{2, j},...,p_{r_j, j}\in C([0,1], M_{r_j})$ are mutually orthogonal
rank one projections.}}

It follows that $\la e'\ra =\la e\ra$ in the Cuntz semi-group.  Since $B$ has stable rank one,
by \cite{CEs},
$A$ is isomorphic  to $C:=\overline{e'Be'}.$  Therefore, \wilog,
we may assume that $u\in {{\tilde C}}.$
Note that, for any $f\in {{C([0,1])_+}},$
\beq
f(e')|_{[0,1]_j}=\sum_{i=1}^{r_j} f(\lambda_{i,j})p_{i,j},\,\,\,j=1,2,...,k.
\eneq

Write
$u=\prod_{i=1}^m\exp(\sqrt{-1}a_i),$ where each $a_i=\af_i\cdot 1_{\tilde A}+x_i$ with $\af_i\in \R$
{{and $x_i\in C_{s.a.},$}}
$i=1,2,...,m.$
Let $\dt>0.$
There is $1/2>\eta>0$  such
that $\|f_{\eta}(e')x_if_{\eta}(e')-x_i\|<\dt,$
 $i=1,2,...,m.$
By choosing $\dt$ small enough, we have that
\beq
\|u-\prod_{i=1}^m\exp(\sqrt{-1}\af_i\cdot 1_{\tilde C}+ f_{\eta}(e')x_if_{\eta}(e'))\|<\ep/4.
\eneq
To simplify notation, \wilog,
we may  further assume that
$f_{\eta}(e')x_if_{\eta}(e')=x_i,$ $i=1,2,...,m.$
Let $\dt_1>0.$
It follows from
8.9 of \cite{eglnp1} that  {{there is }}   $e''\le f_{\eta}(e')$ such that
\beq
\|e''-f_{\eta}(e')\|<\dt_1
\eneq
and $\overline{e''Ce''}\in {\cal C}.$
With sufficiently small $\dt_1,$  we may assume that
\beq
\|u-\prod_{j=1}^m \exp(i \af_j\cdot 1_{\tilde C} +e''x_je'')\|<\ep/3.
\eneq
Put $v=\prod_{j=1}^m \exp(
{{\sqrt{-1}}}\af_j\cdot 1_{\tilde C}+e''x_je'').$
We may now view $v\in {\tilde D},$ where $D=\overline{e''Ce''}.$
Since ${\tilde D}\in {\cal C}$
(see 6.2 of \cite{eglnp1}),  it follows from  5.19 of  \cite{Lncbms} that
there are $b_1, b_2\in {\tilde D}_{s.a.}$
such that
$\|v-\exp(i b_1)\exp(i b_2)\|<\ep/3.$
Note that, if we view $v\in U_0({\tilde A}),$ $b_1, b_2$ may be viewed as elements in ${\tilde C}_{s.a.}$
since $e''\le f_{\eta}(e').$
This follows that ${\rm cer}(A)\le 2+\ep.$

Now suppose that $u\in CU({\tilde A}).$
There exists
{{$v\in CU({\tilde A})$}}
such that $\|u-v\|<\ep/4,$
$v=\prod_{s=1}^{m_1} v_s,$ and $v_s=v_{s,1}v_{s,2}\cdots v_{s,r(s)}v_{s,1}^*v_{s,2}^*\cdots v_{s,r(s)}^*,$
where each $v_{s,i}\in U({\tilde A}),$ $s=1,2,...,m_1.$
Write $v_{s,i}=\bt_{s,i}\cdot 1_{\tilde A}+z_{s,i},$ where $\bt_{s,i}\in \C$ with
$|\bt_{s,i}|=1$ and $z_{s,i}\in A.$
For any $\dt_2>0,$
 with sufficiently small $\eta>0,$
 we may assume that
 \beq
 \|z_{s,i}-f_{\eta}(e')z_{s,i}f_{\eta}(e')\|<\dt_2/16m_1(\sum_{i=1}^{m_1}r(s)),\,\,\, 1\le i\le r(s), \,\,1\le s\le m_1.
 \eneq
 So we may assume that
 \hspace{-0.13in}\beq
 \|z_{s,i}-e''z_{s,i}e''\|<\dt_2/8m_1(\sum_{i=1}^{m_1}r(s)), \,\,\, 1\le i\le r(s), \,\,1\le s\le m_1.
 \eneq
It follows that there is a unitary in $w_{s,i}\in \C\cdot 1_{\tilde A}+\overline{e''Ae''}$  such that
\beq
\|v_{s,i}-w_{s,i}\|<\dt_2/4m_1(\sum_{i=1}^{m_1}r(s)), \,\,\, 1\le i\le r(s), \,\,1\le s\le m_1.
\eneq
Put $w_s=w_{s,1}w_{s,2}\cdots w_{s,r(s)}w_{s,1}^*w_{s,2}^*\cdots w_{s,r(s)}^*$
and $w=\prod_{s=1}^{m_1}w_s.$
With sufficiently small $\dt_2,$ we may assume that
\beq
\|w-v\|<\ep/4.
\eneq
Now $v\in CU(\C\cdot 1_{\tilde A}+\overline{e''Ae''}).$  As mentioned above,
$\C\cdot 1_{\tilde A}+\overline{e''Ae''}\in {\cal C}.$  By 3.16 of \cite{GLN},  in $\C\cdot 1_{\tilde A}+\overline{e''Ae''},$
there is a continuous path $\{u(t): t\in [1/2,1]\}\subset  CU(\C\cdot 1_{\tilde A}+\overline{e''Ae''})$
such that $u(1/2)=w$ and $u(1)=1_{\tilde A}$ which has the length no more than $4\pi+\ep/16\pi.$
Note
$v\in CU({\tilde A})$ and
\beq
\|w-u\|<\ep/2,\,\,\, {\rm or}\,\,\, \|uw^*-1\|<\ep/2.
\eneq
Write $uw^*=\exp(\sqrt{-1}d)$ for some $d\in {\tilde A}_{s.a.}.$  Then $\|d\|<2\arcsin(\ep/4).$
{{Note that $uw^*\in CU({\tilde A}).$
Therefore, for each irreducible representation $\pi$ of ${\tilde A}_{s.a.},$
$
{\rm Tr}_\pi(d)=2m'\pi
$
for some integer $m',$ where ${\rm Tr}_\pi$ is the standard trace on $\pi({\tilde A}).$}}
Since we choose $\ep<{{\frac{1}{4\max\{R(i)r_j: i,j\}}}},$
$
{{{\rm Tr}_\pi(d)=0.}}
$
{{It follows that $\tau(d)=0$ for all $\tau\in T({\tilde A}).$}}
Define $u(t)=\exp(\sqrt{-1} (1-2t)d )w$ for $t\in [0,1/2].$
Note that $u(t)$ is in $CU({\tilde A})$ for all $t\in [0,1]$ with $u(0)=u,$ $u(1)=1$
and total length no more than $4\pi+\ep.$
\end{proof}

\begin{NN}\label{2Rg15}
Let $A=A(F_1, F_2, \phi_0, \phi_1)\in {\cal C},$
where  $F_1=M_{R_1 }(\C)\oplus M_{R_2 }(\C)\oplus \cdots \oplus M_{R_l}(\C), ~~
 F_2=M_{r_1}(\C)\oplus M_{r_2 }(\C)\oplus \cdots \oplus M_{r_k }(\C)$.
Recall  {{that}} the irreducible representations of $A$,  are given by
\vspace{-0.1in}$$\coprod_{i=1}^k (0,1)_i \cup \{ \rho_1, \rho_2, ..., \rho_l\} = \mathrm{Irr}(A),$$
where $(0,1)_i$ is the same open interval $(0,1)$.
Any trace $\tau\in T(A) $ is corresponding to\\
 $(\mu_1, \mu_2,\cdots, \mu_k, s_1,s_2,\cdots, s_l)$, where $\mu_i$ are {{nonnegative}}  measures on $(0,1)_i$ and $s_j\in \R_+$ and  we have
$$\|\tau\|=\sum_{i=1}^k\int_0^1  {{d\mu_i}}~+~\sum_{j=1}^l s_j.$$
Let $t \in(0, 1)_i$ and $\dt_t$ be the canonical point measure at point $t$ with measure $1$,
then
$$\lim_{t\to 0} \dt_t=(\mu_1, \mu_2,\cdots, \mu_k, s_1,s_2,\cdots, s_l)~~~~\mbox{and}~~~ \lim_{t\to 1} \dt_t=(\mu_1, \mu_2,\cdots, \mu_k, s'_1,s'_2,\cdots, s'_l)$$
with $\mu_j=0$,  $s_j=a_{ij}\cdot\frac{R_j}{r_i}$ and $s'_j=b_{ij}\cdot\frac{R_j}{r_i}$, where $(a_{ij})_{k\times l}=\phi_{0*}$ and $(b_{ij})_{k\times l}=\phi_{1*}$ as in \ref{2Rg11}. Let
$$\lambda=\min_i \{\frac{\sum_{j=1}^l a_{ij}R_j} {r_i}, \frac{\sum_{j=1}^l b_{ij}R_j} {r_i} \}.$$
A direct calculation shows that if $\tau_n\in T(A)$ converge to $\tau$ in weak* topology, then $\|\tau\|\geq \lambda\cdot\limsup \|\tau_n\|.$
{{ In notation of \ref{Dlambdas}, we
have
\hspace{-0.14in}\beq
\lambda_s(A)=\lambda.
\eneq}}
Evidently, the number $\lambda$ above is the largest positive number satisfying the following conditions
$$\phi_{0*}([\one_{F_1}])\geq \lambda\cdot [\one_{F_2}], ~~~\phi_{1*}([\one_{F_1}])\geq \lambda\cdot [\one_{F_2}]~~\mbox{in} ~~K_0(F_2).$$

{\ {In the notation  {{of}} \ref{DTtilde},  {{both  affine  spaces}} $\mathrm{Aff}({\tilde T}(A))$}} and $\mathrm{Aff}(\mathrm{T}(A))$ can be identified with the subset of
$$\bigoplus_{j=1}^k C([0,1]_j,\R)\oplus \R^l=\bigoplus_{j=1}^k C([0,1]_j,\R)\oplus \underbrace{(\R\oplus\R\oplus\cd\oplus \R)}_{l~~ copies}$$
consisting of $(f_1,f_2,..., f_k, g_1, g_2,...,g_l)$ satisfying the condition
$$f_i(0)=\frac1{r_i}\sum_{j=1}^l a_{ij}g_j\cdot R_j \qq\qq \mbox{and} \qq\qq
f_i(1)=\frac1{r_i}\sum_{j=1}^l b_{ij}g_j\cdot R_j.$$
{{The positive cone $\Aff({\tilde T}(A))_+$ is the subset of $\mathrm{Aff}({\tilde T}(A))$ consisting all
elements of those elements $(f_1,f_2,..., f_k, g_1, g_2,...,g_l)$ with $f_i(t)\geq 0$ and $g_j\geq 0$ for all $i,j, t$.}}
{{Set $\R^{\sim}=\R\cup\{\infty\}$, $\R_+^{\sim}=\R_+\cup\{\infty\}$.}}
Then ${\rm LAff}({\tilde{T}}(A))_+$ (${\rm LAff}^{\sim}({\tilde T}(A))$, respectively)
is identified with the subset of
\beq\label{FaffC=LSC}
\bigoplus_{j=1}^k LSC([0,1]_j,\R_+^{\sim})\oplus (\R_+^{\sim})^l
\,\,{\rm (or}\,\,\,\bigoplus_{j=1}^k LSC([0,1]_j,\R^{\sim})\oplus (\R^{\sim})^l{\rm)}
\eneq
\vspace{-0.1in}consisting of $(f_1,f_2,..., f_k, g_1, g_2,...,g_l)$ satisfying the same condition
$$f_i(0)=\frac1{r_i}\sum_{j=1}^l a_{ij}g_j\cdot R_j \qq\qq \mbox{and} \qq\qq
f_i(1)=\frac1{r_i}\sum_{j=1}^l b_{ij}g_j\cdot R_j.$$

\end{NN}







\begin{NN}\label{DCusim}  Suppose that $A=A(F_1, F_2, \phi_0, \phi_1)$ is not unital.
Let $e_{i,F_2}=({{e_{i,1}}}, e_{i,2},...,e_{i, k})\in F_2$ be a projection such
that ${\bf 1}_{F_2}-\phi_i({\bf 1}_{F_1})=e_{i, F_2},$ $i=0,1.$
Put $F_{2,i}=e_{i, F_2}F_2e_{i, F_2},$ $i=0,1.$
Define $\phi_i': \C\to F_{2,i}$ by
$\phi_i'(\lambda)=\lambda e_{i, F_2},$ $i=1,2.$
Define  $F_1^{\sim}=F_1\oplus \C$ and $\phi_i^{\sim}: F_1^{\sim}\to F_2$
by $\phi_i^{\sim}(a\oplus \lambda)=\phi_i(a)\oplus \lambda e_{i,F_2},$ $i=0,1.$
Then ${\tilde A}=A(F_1^{\sim}, F_2, \phi_0^{\sim}, \phi_1^{\sim}).$

{{In what follows, we will use notations $\Z^{\sim}=\Z\cup\{\infty\}$, and $\Z_+^{\sim}=\Z_+\cup\{\infty\}$.}}
Let $B=A(F_1, F_2, \phi_0, \phi_1).$
Let $a\in B_+$, {{define}}  ${{r_a}}\in {\rm LAff}({\tilde{T}}(A))_+$ by $r_a(\tau)=d_\tau(a)=\lim\limits_{n\to \infty}\tau(a^{1/n})$. When one identifies ${\rm LAff}({\tilde{T}}(A))_+$ with the subspace of
$\bigoplus_{j=1}^k LSC([0,1]_j,{{\R_+^{\sim})\oplus (\R_+^{\sim}}})^l$ as in \ref{2Rg15},
$r_a \in \bigoplus_{j=1}^k LSC([0,1]_j,\frac{1}{r_j}\Z_+^{\sim})\oplus \bigoplus_{i=1}^l(\frac{1}{R_i}\Z_+^{\sim})$. (Recall that  map  $\phi_{i,*}: K_0(F_1)=\Z^l \to K_0(F_2)=\Z^k$  ($i=0,1$), induced by $\phi_i: F_1\to F_2$ is given by the matrix ${{(a_{ij})_{k\times l}}}$ and  $(b_{{ij}})_{k\times l}$ with nonnegative integer entries, which can be extended to maps (still denoted by $\phi_{i,*}$) from $(\Z^{\sim})^l$ to $(\Z^{\sim})^k$.) If we identify each $\frac{1}{r_j}\Z$ (or  $\frac{1}{R_i}\Z$ respectively) with $\Z$ by identifying $\frac{1}{r_j}$ with $1\in \Z$ (or by  identifying $\frac{1}{R_i}$ with $1\in \Z$), $r_a$ is identified with $$\big((f_1,f_2,\cdots, f_k), (j_1, j_2, \cdots, j_l)\big)\in \bigoplus_{j=1}^k LSC([0,1]_j,\Z_+^{\sim})\oplus (\Z_+^{\sim})^l $$

\vspace{-0.1in}
\noindent
which satisfy
$$(f_1(0), f_2(0), \cdots, f_k(0))=\phi_{0,*}(j_1, j_2, \cdots, j_l)~~\mbox{and}~~(f_1(1), f_2(1), \cdots, f_k(1))=\phi_{1,*}(j_1, j_2, \cdots, j_l).$$


Let $LSC([0,1], \R^{\sim})$ be the set of lower-semicontinuous function{{s}} from $[0,1]$ to $\R^{\sim}$. We will use the notation $LSC([0,1],(\R^{\sim})^k)\bigoplus_{(\phi_{0,*}, \phi_{1,*})} (\R^{\sim})^l)$ to denote the subset of $LSC([0,1],(\R^{\sim})^k)\bigoplus(\R^{\sim})^l)$ consisting of elements $\big((f_1,f_2,\cdots, f_k), (j_1, j_2, \cdots, j_l)\big)\in LSC([0,1],(\R^{\sim})^k)\bigoplus(\R^{\sim})^l$ satisfying
$$(f_1(0), f_2(0), \cdots, f_k(0))=\phi_{0,*}(j_1, j_2, \cdots, j_l)~~\mbox{and}~~(f_1(1), f_2(1), \cdots, f_k(1))=\phi_{1,*}(j_1, j_2, \cdots, j_l).$$
Let $LSC([0,1],(\R_+^{\sim})^k)\bigoplus_{(\phi_{0,*}, \phi_{1,*})} (\R_+^{\sim})^l$ ( $LSC([0,1],(\Z^{\sim})^k)\bigoplus_{(\phi_{0,*}, \phi_{1,*})} (\Z^{\sim})^l$, or \\ $LSC([0,1],(\Z_+^{\sim})^k)\bigoplus_{(\phi_{0,*}, \phi_{1,*})} (\Z_+^{\sim})^l$ respectively) be the subset of $LSC([0,1],(\R^{\sim})^k)\bigoplus_{(\phi_{0,*}, \phi_{1,*})} (\R^{\sim})^l)$ consisting of the above elements with $f_i(t)\,\, \mbox{and}\,\, j_i \in  \R_+^{\sim}$ ( $\in \Z^{\sim}$ or $\in Z_+^{\sim}  $ respectively).
If we insist not take the value $+\infty$, then we will use the notation $LSC_f$ instead of $LSC$.  {{So}} the sets $LSC_f([0,1],(\R_+)^k)\bigoplus_{(\phi_{0,*}, \phi_{1,*})}(\R_+)^l)$ and $LSC_f([0,1],(\Z_+)^k)\bigoplus_{(\phi_{0,*}, \phi_{1,*})}(\Z_+)^l)$ can  {{also}}  be defined similarly.

Now let $B\in {\cal C}_0.$ Let {{$C$}}
be a full hereditary subalgebra of {{$B.$}}  Using the rank function in 3.17  of \cite{GLN} and applying 3.18 of \cite{GLN},
{{The map $r:\la a\ra \mapsto r_a$  gives an injective semi-group \hm\,}}
from $W(C)$ to $LSC_f([0,1],(\Z_+)^k)\bigoplus_{(\phi_{0,*}, \phi_{1,*})}(\Z_+)^l)$
(see also 3.18 of \cite{GLN}) {{which extends to an order injective semi-group \hm\,}}
from $Cu(C)$ to \\
{{$LSC([0,1],(\Z_+^{\sim})^k)\bigoplus_{(\phi_{0,*}, \phi_{1,*})} (\Z_+^{\sim})^l$. }}
{{Note  ${\tilde C}\in {\cal C}.$  Also note that  $Cu^{\sim}(C)$ (see \cite{Rl}) is the semigroup of the  formal differences
$f-n[1_{\tilde C}],$ with $n\in \Z_+$ and $f\in Cu({\tilde C})$ such that $Cu(\pi_C)(f)=[n],$
where $Cu(\pi_C)$ is the map induced by the quotient map $\pi_C: {\tilde C}\to \C.$}}
With the help of discussion of  {{8.8 of \cite{eglnp1},}}
it is straight forward   to check the  following:
\end{NN}

\begin{prop}\label{PCuR}
Let $C\in C_0'.$
{{Then}}
\beq
W(C)=LSC_f([0,1],(\Z_+)^k)\bigoplus_{(\phi_{0,*}, \phi_{1,*})}(\Z_+)^l)\andeqn \\
Cu(C)=LSC([0,1],(\Z_+^{\sim})^k)\bigoplus_{(\phi_{0,*}, \phi_{1,*})} (\Z_+^{\sim})^l.
\eneq
Moreover (see {\rm \cite{Rl}} for the definition of $Cu^{\sim}$)
\beq
Cu^{\sim}(C)=LSC([0,1],(\Z^{\sim})^k)\bigoplus_{(\phi_{0,*}, \phi_{1,*})} (\Z^{\sim})^l.
\eneq
\end{prop}
 Since $C$ is stably projectionless, it follows
that  the order $Cu^{\sim}(C)$ is determined by $Cu(C).$


\begin{df}\label{DRaf1}
{{Fix an integer $a_1\ge 1.$ Let $\af={a_1\over{a_1+1}}.$
For each $r\in \Q_+\setminus \{0\},$ let $e_r\in Q$ {{(see \ref{DQ})}} be a projection with ${\rm tr}(e_r)=r.$
Let ${\bar Q}_r:=(1\otimes e_r)(Q\otimes Q)(1\otimes e_r).$
Define $q_r: Q\to {\bar Q}_r$ by $a\mapsto a\otimes e_r$ for $a\in Q.$ We will also use $q_r$ to denote any homomorphism from $B$ to $B\otimes e_rQe_r$ (or to $B\otimes Q$) defined by sending $b\in B$ to $b\otimes e_r\in B\otimes e_rQe_r \subset B\otimes Q$.}}

{{For $r=\af={a_1\over{a_1+1}}$, one can identify $Q$ with $Q\otimes M_{a_1+1}$, then the projection $e_\af$ is identified with
$\one_Q\otimes \diag(\underbrace{1,\cdots, 1}_{a_1},0)$. }}

Let
\vspace{-0.1in}$$
R(\af,1)=\{(f,a)\in C([0,1], Q\otimes Q)\oplus Q: f(0)=q_\af(a)\andeqn f(1)=a\otimes 1_Q\}.
$$
{{Note that an element $(f,a)$ is full in $R(\af, 1)$ if and only
if $a\not=0$ and $f(t)\not={{0}}$ for all $t\in (0,1).$
}}
Let $a_{\af}=(f, 1)$ be defined as follows.
Let
\beq\label{Dconsteaf}
f(t)=(1-t)(1\otimes e_\af)+t(1\otimes 1)\rforal t\in (0,1).
\eneq
{{Note that $a_\af$ is a  {\em strictly positive element} of $R(\af,1),$ {{morover,}} for any $1/2>\eta>0,$ $f_\eta(a_\af)$ is full.}}
{\em \CA\, $R(\af,1)$ and $a_\af$ will appear frequently in this paper.}
\vspace{0.1in}

Let $LSC([0,1], \R^{\sim})\oplus_\af \R^{\sim}$ (or $LSC_f([0,1], \R_+)\oplus_\af \R_+$ respectively) be the subset of \\$LSC([0,1],  \R^{\sim})\oplus \R^{\sim}$ (or $LSC_f([0,1],  \R_+)\oplus\R_+$ respectively) consisting of elements $(f, x)$ such that $f(0)=\af x$ and $f(1)=x$.  {{The rank function ${{r}}: \la a\ra \mapsto
r(a)=d_\tau(a) $ gives maps from}} $W(R(\af,1))$ to ${{LSC_f([0,1], \R_+)}}\oplus_\af \R_+$ and from $Cu(R(\af,1))$ to $LSC([0,1],  \R^{\sim}_+)\oplus_\af \R^{\sim}_+$
{{which are order semi-group  \hm s.}} But these {{maps}} are only surjective {{ not}} injective.

Recall that $W(Q)$ and $Cu(Q)$ can be identified with the semi-groups
${{\R_+\setminus \{0\}\sqcup \Q_+}}$ and
${{\R_+^{\sim}\setminus \{0\}\sqcup \Q_+}}$,  where the second copy of $\Q$ is identified with
$K_0(Q)$ and $\R_+^{\sim}\setminus \{0\}$ identified with the rank functions of non-projection  and non-zero positive elements.  If ${{s}}\in \Q\subset \R,$ we will use $[{{s}}]$ for the corresponding element in $K_0(Q).$
With the order in $Cu(Q),$ in  ${{\R^{\sim}}}\sqcup \Q,$ $t<[t]$ for $t\in \Q\subset \R$ and $[t]\in K_0(Q)=\Q.$
But $s>[t]$ if $s>t$ as in $\R^{\sim}.$  The addition on $\R\sqcup \Q$ is defined by $s+[r]=s+r$ and $[s]+[r]=[s+r]$.

A function $f: [0,1] \to {{\R^{\sim}\sqcup \Q}}$ is called lower-semicontinuous if, for each $t_0\in [0,1],$
and if ${{ f(t_0)=[r]\in K_0(Q)}},$
there exists $\dt>0$ such
that $f(t)\ge f(t_0)$ for all $t\in (t_0-\dt, t_0+\dt)\cap [0,1],$
or, if ${{f(t_0)=r\in \R^{\sim}}},$
for any non zero $\ep\in   \R_+^{\sim}\setminus\{0\}\sqcup \Q_+,$ there exists $\dt>0$ such that
$$f(t)+\ep\ge f(t_0) \rforal \in [0,1]\cap (t_0-\dt, t_0+\dt)\setminus \{t_0\},$$
 where the order is in  $\R^{\sim}\sqcup \Q$ mentioned above.

  Let $LSC([0,1],  {{\R^{\sim}\sqcup \Q}})$ be the set of all  lower-semicontinuous functions.
Let $LSC([0,1],  {{\R^{\sim}}}\sqcup \Q)\oplus_\af {{\R^{\sim}}}\sqcup \Q$ be the subset of
$LSC([0,1],  {{\R^{\sim}}}\sqcup \Q)\oplus {{\R^{\sim}}}\sqcup \Q$ consisting of elements $(f, x)$ such that $f(0)=\af x$ and $f(1)=x$. (Here we define $\af [r]= [\af r].$ {{Note that $\af$ is rational.}}).
 The sets
 $LSC([0,1],
 (\R^{\sim}\setminus\{0\}\sqcup \Q)_+)\oplus_\af (\R^{\sim}\setminus \{0\}\sqcup \Q)_+$ and
 $LSC_f([0,1],  (\R\setminus \{0\}\sqcup \Q)_+)\oplus_\af (\R\setminus \{0\}\sqcup \Q)_+$ can be defined similarly. Then we have the following fact.
\end{df}

\begin{cor}\label{CcuR}
{{
Let $A=R(\af,1)$ for some $1>\af>0.$
Then}}
{{\beq
W(A)=LSC_f([0,1],  (\R\setminus \{0\}\sqcup \Q)_+)\oplus_\af (\R\setminus \{0\}\sqcup \Q)_+,\\
Cu(A)=LSC([0,1],
 (\R^{\sim}\setminus\{0\}\sqcup \Q)_+)\oplus_\af (\R^{\sim}\setminus \{0\}\sqcup \Q)_+\andeqn\\\label{Fcusim1}
Cu^{\sim}(A)=LSC([0,1],  {{\R^{\sim}}}\sqcup \Q)\oplus_\af {{\R^{\sim}}}\sqcup \Q.
\eneq
}}
\end{cor}

{{Note, with \eqref{Fcusim1}, map $r$ can be extended to an order semi-group
\hm\, from $Cu^{\sim}(A)$ to $LSC([0,1],  \R^{\sim})\oplus_\af \R^{\sim}$
defined by $r(f(s), a)=(r(f(s)),r(a)),$   where $r(t)=t$ for all $t\in \R^{\sim}$ and $r([t])=t$ for all
$t\in \Q.$}}

 \begin{df}[cf. 8.1 and 8.2 of \cite{eglnp1}]\label{DD0}
{{Recall}} the definition of class ${\cal D}$ and ${\cal D}_0.$

 Let $A$ be a non-unital simple \CA\,  with a strictly positive element $a\in A$
with $\|a\|=1.$   Suppose that there exists
$1> \mathfrak{f}_a>0,$ for any $\ep>0,$  any
finite subset ${\cal F}\subset A$ and any $b\in A_+\setminus \{0\},$  there are ${\cal F}$-$\ep$-multiplicative \cpc s $\phi: A\to A$ and  $\psi: A\to D$  for some
\SCA\, $D\subset A$ with $D\in {\cal C}_0^{0'}$ (or ${\cal C}_0'$),  $D\perp \phi(A),$ and
\beq\label{DNtr1div-1++}
&&\|x-(\phi(x)+\psi(x))\|<\ep\rforal x\in {\cal F}\cup \{a\},\\\label{DNtrdiv-2}
&&c\lesssim b,\\\label{DNtrdiv-4}
&&t(f_{1/4}(\psi(a)))\ge \mathfrak{f}_a\rforal t\in T(D),
\eneq
where $c$ is a strictly positive element of $\overline{\phi(A)A\phi(A)}.$
  Then
 we say $A\in {\cal D}_{0}$ (or ${\cal D}$).

 {{ Note, by  Remark 8.11 of \cite{eglnp1},
 $D$ can {\it always} be chosen to be in ${\cal C}_0$ (or ${\cal C}_0^0$).}}

{{When $A\in {\cal D}$ and is separable, {{then}}  $A=\mathrm{Ped}(A)$ (see 11.3  of \cite{eglnp1}).
Let $a\in A_+$ with $\|a\|=1$ be a strict positive element.
Put
\beq\label{1225dd}
d=\inf\{\tau(f_{1/4}(a)): \tau\in T(A)\}.
\eneq
Then, for any $0<\eta<d,$ $\mathfrak{f}_a$ can be chosen to  be
$d-\eta$ (see  Remark 9.8 of \cite{eglnp1}).
One may also assume that $f_{1/4}(\psi(a))$ is
{{full}} in $D.$
Furthermore,  there exists a map:  $T: A_+\setminus \{0\}\to \N\times \R$
($a\mapsto (N(a), M(a))\rforal a\in A_+\setminus \{0\}$) which is independent of
${\cal F}$ and $\ep$ such that, for any  finite subset ${\cal H}\subset A_+\setminus \{0\},$ we can further require that $\psi$ is $T$-${\cal H}$-full (see 8.3
and 9.2
of \cite{eglnp1}).
}}
{{
For any $n\ge 1,$ one can choose a strictly positive element $b\in A$ with $\|b\|=1$ such that
$f_{1/4}(b)\ge f_{1/n}(a).$ Therefore, if $A$ has continuous scale, $d$ can be chosen to be $1,$
if the strictly positive element is chosen accordingly.
}}


Let $A$ be a separable stably projectionless simple \CA.  Recall that  $A$ has generalized tracial rank at most
one and write $gTR(A)\le 1,$ if there exists $e\in {\rm Ped}(A)_+$ with $\|e\|=1$ such that
$\overline{eAe}\in {\cal D}$ (see  11.6 of \cite{eglnp1}).

 \end{df}

 \begin{df}\label{DD1}
 Let $A\in {\cal D}$ as defined \ref{DD0}.
 If, in  addition,
 for any integer $n,$
 $D=M_n(D_1)$ for some $D_1\in {\cal C}_0$ such that
\vspace{-0.14in} \beq\label{DD1-1}
 \psi(x)=\diag(\overbrace{\psi_1(x), \psi_1(x),...,\psi_1(x)}^n)\rforal x\in {\cal F},
 \eneq
where $\psi_1: A\to D_1$ is an ${\cal F}$-$\ep$-multiplicative \cpc, then we say $A\in {\cal D}^d.$

Note that here, as in  8.3 and 9.2 of \cite{eglnp1}, the map  $T$ mentioned in \ref {DD0}  is also assumed
{{to exist}} and
$\mathfrak{f}_a$ can be also
chosen as $d-\eta$ for any $\eta>0$ with $d$ as in \eqref{1225dd} for a certain strictly positive element $a.$

\end{df}

\begin{rem}\label{RDd}
  It follows  from
  {{10.4 and 10.7 of \cite{eglnp1}}} that, if $A\in {\cal D}_0,$  then $A\in {\cal D}^d.$ Moreover,
  $D_1$ can be chosen in {{${\cal C}_0^{(0)},$}}
  and if $A\in {\cal D},$ then $D_1$ can be chosen in ${\cal C}_0.$ 
  If $A$ is a separable simple \CA\, in ${\cal D}$ and $A$ is tracially approximate divisible  {{(in the sense of
  10.1 of \cite{eglnp1}),}}
  then $A\in {\cal D}^d.$
\end{rem}

\begin{prop}\label{Pdivisiblehere}
Let $A$ be a non-unital     simple \CA\, which is tracially approximate divisible.
Then every hereditary \SCA\, is also tracially approximate divisible.
Consuquently, if $A\in {\cal D}^d,$ then every hereditary \SCA\, is in ${\cal D}^d.$
\end{prop}

\begin{proof}
Let $B\subset A$ be a hereditary \SCA.  Fix $\ep>0,$ a finite subset
${\cal F}\subset B,$ a nonzero element $b\in B_+$ and an integer $n\ge 1.$
By choosing a member $b_e$ in an approximate identity of $B,$ \wilog
(with an error within, say $\ep/2$),
we may assume that $xb_e=b_ex{{=x}}$ for all $x\in {\cal F}.$

Since $A$ is tracially approximate divisible,  there {{are \SCA s}} $A_0$ and $A_1$ of $A$
such that
\beq
{\rm dist}(x, C_d)<\ep\rforal x\in {\cal F},
\eneq
where $C_d\subset C\subset A,$  $C=A_0\oplus M_n(A_1),$
$$
C_d=\{(y_0,\diag(\overbrace{y_1, y_1,...,y_1}^n)): y_0\in A_0, y_1\in A_1\},
$$
and where $a_0\lesssim b,$ where $a_0$ is a strictly positive element of $A_0.$

Let $B_0$ be the \SCA\, generated by $b_eab_e$ for all $a\in A_0$ and let
$B_1$ be the \SCA\, generated by $b_ecb_e$ for all $c\in A_1.$
Then $B_0$ and $B_1$ are \SCA s of $B.$  {{Since}} $B_0\subset \overline{b_ea_0b_eAb_ea_0b_e},$
$b_ea_0b_e$ is a strictly positive element of $B_0.$ Moreover, $b_ea_0b_e\lesssim a_0\lesssim b.$
Put
\beq
B_1^d=\{(x_0, \overbrace{x_1,x_1,...,x_1}^n): x_0\in B_0, x_1\in B_1\},
\eneq
$B_1^d\subset B_3,$ where $B_3=B_0\oplus M_n(B_1).$
For each $x\in {\cal F},$ let $y_x=(y_{0,x}, y_{1,x},...,y_{1,x})\in C_d$ such that
$
{{\|x-y_x\|<\ep/2.}}
$
Then
\beq
\|x-b_ey_xb_e\|<\ep\rforal x\in {\cal F}.
\eneq
Note that $b_ey_xb_e\in B_1^d.$   This proves the first part of the statement.
If $A\in {\cal D}^d,$ then, $B\in {\cal D}$ for any hereditary \SCA\, $B,$ by
{{8.6 of \cite{eglnp1}.}}  By the first part of the statement,
$B$ is tracially approximately divisible. {{Therefore $B\in {\cal D}^d.$}}
\end{proof}

\begin{prop}\label{Pwwdivisible}
Let $A\in {\cal D}$ be with continuous scale and let $e\in A_+$ with $\|e\|=1$ be a strictly positive element,
and $1>{\mathfrak{f}}_e>0$ be as in \ref{DD0}. Then, for any finite subset ${\cal F}\subset A,$ any $\ep>0,$
 any $b\in A_+\setminus \{0\}$ and any integer $n\ge 1,$  there are ${\cal F}$-$\ep$-multiplicative \cpc s $\phi: A\to A$ and  $\psi: A\to M_n(D)$  for some
\SCA\, $D\in {\cal C}_0$  with $M_n(D)\subset A$ and $\phi(A)\perp M_n(D)$ such that
\beq\label{Pwwdiv-1}
&&\|x-(\phi(x)\oplus \psi(x))\|<\ep\tforal x\in {\cal F}\cup \{e\},\\\label{}
&&\phi(e)\lesssim b,\\\label{Pwwdiv4}
&&t(f_{1/4}(\psi(e)))\ge \mathfrak{f}_e/2\tforal t\in T(D).
\eneq

 \end{prop}

\begin{proof}
Fix $\ep>0,$ $b$ and ${\cal F}$ as described in the statement.
Let $\eta=\inf\{\tau(b): \tau\in \overline{T(A)}^w\}>0.$
Choose $e_0\in A_+$ with $\|e_0\|=1$ such
that $\|e_0ee_0-e\|<\ep/16.$ \Wlog, we may also assume
that $e_0f=fe_0=f$ for all $f\in {\cal F}.$
It follows from
{{11.8 of \cite{eglnp1}}}  that  the map from $Cu(A)$ to ${\rm LAff}_{b+}(\overline{T(A)}^w)$
{{is an isomorphism.}}  {{Therefore there is $e_{0,1}\in A_+$ such that $n\la e_{0,1}\ra =\la e_0\ra$ and $\la e_0\ra =\la b\ra,$
where $b=\diag(e_{0,1}, e_{0,1},...,e_{0,1})$ ($e_{0,1}$ repeated $n$ times) in  $M_n(A)_+.$
By 11.5 of \cite{eglnp1}, $A$ has stable rank one. It follows that $\overline{e_0Ae_0}$ and $\overline{bM_n(A)b}$ are isomorphic.
In particular, $\overline{e_0Ae_0}\cong M_n(\overline{e_{0,1}Ae_{0,1})}.$}}
Therefore, \wilog, {{(replacing $e_0$ by another strictly positive element in $\overline{e_0Ae_0}$),}} we may also write that
$e_0=\sum_{i=1}^ne_{0,i},$ where
$\{e_{0,1}, e_{0,2},...,e_{0,n}\}$ are mutually orthogonal and
there exists $w_i\in A$ such that $w_i^*w_i=e_{0,1}$ and
$w_iw_i^*=e_{0,i},$ $i=1,2,,,,n.$

Since $A$ is stably projectionless, \wilog, we may assume that
$sp(e_0)=[0,1].$
Then elements $e_{0,i}$ and $w_i$ generate a \SCA\, $C$ which
is isomorphic to $C_0((0,1])\otimes M_n$ which is semi-projective.
Let ${\cal G}_1=\{
{{e_{0,i},}} w_i: 1\le i\le n\}.$

Put $\dt_0=\min\{\ep/16(n+1), \eta/2(n+1), {\mathfrak{f}}_e/(4(n+1)\}.$
Choose $\dt_1>0$ such that
for any ${\cal G}_1$-$\dt_1$-multiplicative \cpc\,  $L$ from $C$ to a \CA\,
$B,$
there is a \hm\, $\phi': C\to B$ such that
\beq
\|\phi'(g)-L(g)\|<{{\dt_0/4}}\rforal g\in {\cal G}_1.
\eneq
Put ${\cal F}_1={\cal F}\cup {\cal G}_1\cup \{ab: a, b\in {\cal F}\cup {\cal G}_1\}.$

Fix
a positive number $\ep_1<\min\{\dt_0, \dt_1/2\}/(4(n+1)).$
Since $A\in {\cal D},$ there  are ${\cal F}_2$-$\ep_1$-multiplicative \cpc s $\phi: A\to A$ and  $\psi_0: A\to B$  for some
\SCA\, $B\subset A$ with $B\in {\cal C}_0$ such that $\phi(e)\lesssim b,$ $\phi(A)\perp B,$
\beq\label{Pwwdiv-1}
&&\|x-(\phi(x)\oplus \psi_0(x))\|<\ep_1\rforal x\in {\cal F}_1\cup \{e,e_0\},\\\label{}
&&t(f_{1/4}(\psi(e)))\ge \mathfrak{f}_e \rforal t\in T(B).
\eneq
By the choice of ${\cal G}_1$ and $\dt_1,$
we obtain a \hm\, $h: C\to B$ such that
\beq
\|h(g)-\psi_0(g)\|<\dt_0/4\rforal g\in {\cal G}_1.
\eneq
Let $e'_i=h(e_i)$ and $v_i=h(w_i),$ $i=1,2,...,n.$
Let $B'=h(e_0)Bh(e_0).$  Since $h$ is a \hm\, and $e', v_i\in B',$
$B'\cong M_n(\overline{e_1'Be_1'}).$ Set $D=\overline{e_1'Be_1'}.$
Define  $\psi: A\to B'$ by $\psi(a)=h(e_0)\psi(a)h(e_0).$
One checks
\beq
\tau(\psi(e))\ge {\mathfrak{f}}_a/2\rforal \tau\in T(B')
\eneq
and $\psi$ is ${\cal F}$-$\ep$-multiplicative. Moreover,
\beq
\|x-(\phi(x)\oplus \psi(x))\|<\ep\rforal x\in {\cal F}.
\eneq

\end{proof}








\section{The unitary group}

\begin{lem}\label{L2matrix}
Let $A$ be a non-unital \CA\, and let $e_1, e_2\in A_+$ with $\|e_i\|=1$
($i=1,2 $) such that
$$
e_1e_2=e_2e_1=0
$$
and there is a unitary $u\in {\tilde A}$ such that
$u^*e_1u=e_2.$
Suppose that $w=1_{{\tilde A}_0}+x_0 \in {\tilde A_0}$ is a unitary
with $x_0\in A_0,$ where $A_0=\overline{e_1Ae_1}.$
Then
$w_1=1+x_0+u^*x_0^*u\in CU({\tilde A}),$
${\rm cel}(w_1)\le \pi$ and ${\rm cer}(w_1)\le 1+\ep.$
\end{lem}

\begin{proof}
Let $B$ be the \SCA\, of $A$ generated by
$A_0$ {{and}} ${{ue_2}}.$ Note that $u^*A_0u=\overline{e_2Ae_2}.$  {{One can define a map from $M_2(\overline{e_1Ae_1})= M_2(\overline{e_1^2Ae_1^2})$ to $B$ by
$$M_2(\overline{e_1Ae_1})\ni\left(
\begin{array}{cc}
e_1^2a_{11}e_1^2 & e_1^2a_{12}e_1^2 \\
e_1^2a_{21}e_1^2 & e_1^2a_{22}e_1^2
\end{array}\right)\mapsto e_1^2a_{11}e_1^2+e_1^2a_{12}e_1ue_2+e_2u^*e_1a_{21}e_1^2+
e_2u^*e_1a_{22}e_1ue_2.$$
It is easy to verify that this is an isomorphism by using $e_1e_2=e_2e_1{{=0}}$ and {{$u^*e_1u=e_2$.}}}} Therefore $B\cong M_2(A_0)\in B.$
Consider $M_2({\tilde A_0}).$
Put $p_{1,1}=1_{{\tilde A_0}}.$ We view $p_{1,1}$ as the open projection
associated to $A_0.$
Let  $p_{2,2}=u^*p_{1,1}u.$ Since $1_{{\tilde A}_0}+x_0 $ is a unitary, we have
$$(p_{11}+x_0^*)(p_{11}+x_0)=(p_{11}+x_0)(p_{11}+x_0^*)=p_{11}.$$
Define, for $t\in [0, 1],$
$$
X(t)=((\cos(t\pi/2))p_{1,1}+(\sin(t\pi/2))p_{1,1}u+(\sin(t\pi/2))u^*p_{1,1}+(\cos(t\pi/2))p_{2,2})+((1_{\tilde A})-p_{1,1}-p_{2,2})).
$$
Define
$$
W(t)=(1+x_0)X(t)(1+x_0^*)X(t)^*\rforal t\in [0,1].
$$
Let $X'(t)=X(t)-((1_{\tilde A})-p_{1,1}-p_{2,2}))\in M_2({\tilde A_0})$ {{(by identifying $p_{11}u$ with $\left(
\begin{array}{cc}
0& 1\\
0 & 0
\end{array}\right)$, $u^*p_{11}$ with $\left(
\begin{array}{cc}
0& 0\\
1 & 0
\end{array}\right)$, and $p_{22}$ with $\left(
\begin{array}{cc}
0& 0\\
0 & 1
\end{array}\right)$). Set}}
$$
W'(t)=(p_{1,1}+p_{2,2}+x_0)X'(t)(p_{1,1}+p_{2,2}+x_0^*)X'(t)^*\in M_2({\tilde A_0}).
$$
We have
$$
X'(0)=p_{1,1}+p_{2,2}\andeqn X'(1)=p_{1,1}u+u^*p_{1,1}.
$$
Then
$$
W'(0)=p_{1,1}+p_{2,2}\andeqn W'(1)=(p_{1,1}+x_0)+(p_{2,2}+u^*x_0^*u).
$$

Let $\pi: M_2({\tilde A_0})\to M_2$ be the quotient map.
Then
$\pi(W'(t))=1_{M_2}$ for all $t\in [0,1].$ This implies that $W'(t)\in {\widetilde{M_2(A_0)}}$ for all
$t\in [0,1].$ It follows that
$W(t)\in U({\tilde A})$ for all $t\in [0,1].$
Note that
$W(0)=1_{\tilde A}$ and $W(1)=1+x_0+u^*x_0^*u.$
Moreover, one computes that (since each $W(t)\in U_0({\tilde A})$),
$$
{\rm cel}(\{W(t)\})\le \pi.
$$
It follows that ${\rm cer}(W(1))\le 1+\ep.$
Moreover
$$
1+x_0+u^*x_0^*u=(1+x_0)u^*(1+x_0^*)u.
$$
It follows that $1+x_0+u^*x_0u\in CU({\tilde A}).$
\end{proof}

The following is a variation of a lemma of N. C. Phillips
\begin{lem}[Lemma 3.1 of \cite{Ln-hmtp}]\label{ph}
Let $H>0$ be a positive number and let $N\ge 2$ be an integer.
Then, for any non-unital \CA\, which has almost stable rank one,  any
positive element $e_0\in A_+$ with $\|e_0\|=1,$ and $u=\lambda\cdot 1_{{\tilde A}_0} +
{{x_0'}}\in {\tilde A_0}$
(where $
{{x_0'}}\in A_0$ and $|\lambda|=1$)  such
that ${\rm cel}_{{\tilde A_0}}(u)\le H,$ where $A_0=\overline{e_0Ae_0}.$
Suppose that there are mutually orthogonal positive elements
$e_1, e_2,...,e_{2N}\in A_0^{\perp}$ such that
$e_0\sim e_i,$ $i=1,2,...,2N.$
Then  there exists $z\in CU({\tilde A})$ with
${\rm cel}(z)\le2 \pi$ and ${\rm cer}(z)\le 2+\ep$ such that
$$
\|u'-\lambda \cdot z\|<2H/N,
$$
where $u'=\lambda\cdot 1_{\tilde A}+
{{x_0'}}.$
\end{lem}

\begin{proof}
Since ${\rm cel}_{{\tilde A_0}}(u)\le H,$ there are $u_0, u_1,...,u_N\in {\tilde A_0}$
such that
\beq\label{ph-1}
u_0=u,\,\,\,u_N=1_{\tilde A_0}\andeqn \|u_i-u_{i-1}\|<H/N,\,\,\, i=1,2,...,N.
\eneq
Write $u_i=\lambda_i\cdot 1_{{\tilde A_0}}+x_i',$
where $x_i'\in A_0,$ $i=1,2,...,N.$ In particular, $x_N'=0$.
It follows from \eqref{ph-1} that {{($\lambda_0=\lambda$)}}
$$
|\lambda_i-\lambda_{i-1}|<H/N,\,\,\, i=1,2,...,N.
$$
Let $v=v_0={\bar \lambda}u=1_{\tilde A_0}+{\bar \lambda}x_0'$ and
$v_i={\bar \lambda}_iu_i=1_{\tilde A_0}+{\bar \lambda}_ix_i',$ $i=1,2,...,N.$
Put $x_i={\bar \lambda}_ix_i',$ $i=0,1,...,N.$  We have $x_N=0$ and $v_N=1_{\tilde A_0}$.
Now
\beq\label{ph-4}
\|v_i-v_{i-1}\|=\|{\bar \lambda}_iu_i-{\bar \lambda}_{i-1}u_{i-1}\|<2H/N,\,\,\,i=1,2,...,N.
\eneq

Let
$$
\ep_0=2H/N-\sup\{\|v_i-v_{i-1}\|: , i=1,2,...,N\}.
$$
Choose $1>\dt>0$ such that
$$
\|x_i-f_\dt(e_0)x_if_\dt(e_0)\|<\ep_0/16N,\,\,\,i=0,1,2,....,N.
$$
Put $B_0=\overline{f_\dt(e_0)Af_\dt(e_0)}.$
There is a unitary  $w_i\in 1_{\tilde A_0}+B_0$ such that
$$
\|v_i-w_i\|<
{{\ep_0/4N}},\,\,\,i=0,1,...,N.
$$
Write $w_i=1_{\tilde A_0}+y_i,$ where $y_i\in B_0$ and $y_N=0.$
Since $A$ has almost stable rank one, there are unitaries $U_i\in {\tilde A}$ such that
$$
U_i^*f_{\dt/2}(e_0)U_i\in \overline{e_iAe_i} ,\,\,\,i=1,2,...,2N.
$$

Let
\beq\label{ph-5}
X_1&=& 1_{\tilde A}+y_0+\sum_{i=1}^NU_{2i-1}^*y_i^*U_{2i-1}+\sum_{i=1}^N U_{2i}^*y_iU_{2i}\\
X_2&=&1_{\tilde A}+y_0+\sum_{i=1}^N
{{U_{2i-1}^*}}y_{i-1}^*U_{2i-1}+\sum_{i=1}^N U_{2i}^*y_{i}U_{2i}\andeqn\\
X_3&=& 1_{\tilde A}+
\sum_{i=1}^NU_{2i-1}^*y_iU_{2i-1}+\sum_{i=1}^N U_{2i}^*y_i^*U_{2i}.
\eneq
Note that $X_1\in U({\tilde A}).$ Since $y_N=0$, as in \ref{L2matrix}, for $i=2,3$, we have
\beq\label{ph-5+}
X_i\in CU({\tilde A}),\,\,\, {\rm cel}(X_i)\le \pi\andeqn {\rm cer}(X_i)\le 1+\ep.
\eneq
Moreover
\beq\label{ph-6}
\|X_1-X_2\|{{\leq\sup\{\|y_i^*-y_{i-1}^*\|: , i=1,2,...,N\}}}\\{{<\ep_0/4N}}
+\sup\{\|v_i-v_{i-1}\|: , i=1,2,...,N\}
\eneq
Furthermore,
\vspace{-0.14in}\beq\label{ph-7}
1_{\tilde A}+y_0=X_1X_3.
\eneq
Put $z=X_2X_3.$    Then,  by \eqref{ph-5+},
$$
z\in CU({\tilde A}),\,\,\, {\rm cel}(z)\le 2\pi \andeqn {\rm cer}(z)\le 2+\ep.
$$
Moreover,
\beq
\|{\bar \lambda}\cdot u'-z\| &\le & \|(1_{\tilde A}-1_{\tilde{A_0}})+v_0-(1_{\tilde A}+y_0)\|+\|(1_{\tilde A}+y_0)-z\|\\
&<&
{{\ep_0/8N+\ep_0/4N}}+\sup\{\|v_i-v_{i-1}\|: , i=1,2,...,N\}<2H/N.
\eneq
\end{proof}

\begin{thm} {\rm (cf. Theorem 6.5 of \cite{LnTAI})}\label{UL1}
Let $A$ be a non-unital separable  simple \CA\,  in ${\cal D}$ and
let $u\in U_0({\tilde A})$  with $u=\lambda\cdot 1+x_0,$ where
$\lambda\in \C$ with $|\lambda|=1$ and $x_0\in A.$
Then, for any $\ep>0,$  there exists a unitary
$u_1, u_2\in {\tilde A}$ such that $u_1$ has exponential length no more than $2\pi,$ $u_2$ has exponential rank $3$ and
$$
\|u-u_1u_2\|<\ep.
$$
Moreover, ${\rm cer}(A)\le 5+\ep.$
\end{thm}

\begin{proof}
Let $1/2>\ep>0.$  Let $u'={\bar \lambda} \cdot u.$
Let
$v_0, v_1,...,v_n\in U_0({\tilde A})$ such that
$$
v_0=u', \,\,\, v_n=1\andeqn \|v_i-v_{i-1}\|<\ep/32,\,\,\,i=0,1,...,n-1.
$$
Write
$v_i=\lambda_i\cdot 1+x_i,$ where $|\lambda_i|=1$ and $x_i\in A,$ 
{{$i=1,...,n-1,$ and $v_0=1+\tilde{x}_0$, where $\tilde{x}_0=\bar{\ld}x_0$.}}
Note that $x_n=0.$

As demonstrated in the proof of \ref{ph}, we may assume
that there is a strictly positive element
$e\in A_+$ such that $\|e\|=1$ such that
\beq\label{UL1-5}
f_\eta(e)x_i=x_if_\eta(e)=x_i,\,\,\,i=0,1,2,...,n,
\eneq
for some $\eta>0.$ Let
$$
{\cal G}_1=\{e, f_\eta(e), f_{\eta/2}(e), {{\tilde{x}_0,}}\, x_i,\, 0\le i\le n\}.
$$
Put
$$
d=\inf\{d_\tau(e): \tau\in \overline{T(A)}^w\}>0.
$$
\Wlog, we may assume that $\tau(f_{1/2}(e))\ge d/2$ for all $\tau\in \overline{T(A)}^w.$

 Note that we may assume that $A$ is infinite dimensional.
Hence we may choose mutually orthogonal positive non-zero elements  $c_0, c_1,...,c_{n+1}$
such that $c_0\sim c_i$ ($1\le i\le n+1$) and
\beq\label{UL-n100}
d_\tau(c_0)<d/5(n+1)
\rforal \tau\in T(A).
\eneq
Let $\dt>0$
and let 
{{${\cal G}\supset {\cal G}_1$ be a finite subset of $A$}}.
Since $A\in {\cal D},$ there  are $A_0$ and $D\subset A$ with
$D\in {\cal C}_0'$ and $A_0\perp D,$ ${\cal G}$-$\dt$-multiplicative \cpc s
$\phi_0: A\to A_0$ and $\phi_1: A\to D,$
such that
\beq\label{UL1-6}
&&\|x-(\phi_0(x)\oplus  \phi_1(x))\|<\dt \rforal x\in {\cal G}\\\label{UL1-6+}
&&\phi_0(e)\lesssim  c_0,\\\label{UL1-6++}
&&\tau(f_{1/4}(\phi_1(e)))\ge  d/4\rforal \tau\in \overline{T(A)}^w.
\eneq
By choosing 
{{smaller $\dt$ and larger ${\cal G},$}} we may assume the following:
there are $y_i\in \overline{\phi_0(f_{\eta/2}(e))A\phi_0(f_{\eta/2}(e))}$
such that {{$1+y_0$,}} $\lambda_i\cdot 1+y_i$ 
{{are unitaries}} with $y_n=0$ such that
$\|\phi_0(x_i)-y_i\|<\ep/32,$ 
{{$i=1,2...,n,$ and $\|\phi_0(\tilde{x}_0)-y_0\|<\ep/32$.}} Consequently,
\beq\label{UL1-7}
\|y_i-y_{i+1}\|<\ep/16,\,\,\, \|(\lambda_i\cdot 1+y_i)-(\lambda_i\cdot 1+y_{i-1})\|<\ep/16,
\eneq
$i=0,1,...,n.$
Moreover,
there is $z_1\in  \overline{f_{\eta/2}(\phi_1(e))Df_{\eta/2}(\phi_1(e))}$ such that
$1_{\tilde A}+z_1$ is a unitary and
\beq\label{UL1-8}
\|v_0-(1_{\tilde A}+y_0+z_1)\|<\ep/16.
\eneq
Put $u_1'=1+ y_0,$ $u_2'=1+z_1$ and $u_2=\lambda u_2'.$
Then
$$
\|u-u_1'\cdot u_2\|<{{\ep/4}}.
$$
Put  $B_0=\overline{\phi_0(f_{\eta/2}(e))A\phi_0(f_{\eta/2}(e))}.$
Let $w_i=\lambda_i\cdot 1_{{\tilde B_0}}+y_i,$
$i=0,1,...,n.$
Then $w_n=1_{\tilde B_0},$ $w_0=1\cdot 1_{\tilde B_0}+ y_0$ and
$$
\|w_i-w_{i-1}\|<\ep/16,\,\,\,i=1,2,...,n.
$$
This implies that $w_0\in U_0({\tilde B_0})$ and {{$H:={\rm cel}(w_0)\le n\pi \ep/8.$}}
By \eqref{UL-n100} and \eqref{UL1-6++}, there are mutually orthogonal elements $c_i'\in A_0^{\perp},$
with $c_i'\sim c_0,$ $i=0,1,...,n+1.$
Then, by \eqref{UL1-6+}
and by  Lemma \ref{ph}, ${\rm cel}(u_1')\le {{2\pi+2H/n}}<2\pi+\pi\ep/8.$
On the other hand, by \ref{Ldert}, ${\rm cer}(u_2)\le 2+\ep.$
{{Lemma then follows.}}

\end{proof}


\begin{thm}\label{Ulength}
Let $A$ be a separable  simple \CA\, in ${\cal D}$  and let $u\in CU({\tilde A}).$
Then $u\in U_0({\tilde A})$ and ${\rm cel}(u)\le {{6\pi}}.$
\end{thm}


\begin{proof}
Let $\pi: {\tilde A}\to \C$ be the quotient map. Since $u\in CU({\tilde A}),$
$\pi(u)=1.$  So we write $u=1+x_0,$ where  $x_0\in A.$

Let $1/2>\ep>0.$ There are $v_1, v_2,...,v_k\in U({\tilde A})$ such that
$$
\|u-v_1v_2\cdots v_k\|<\ep/32,
$$
and $v_i=
{{a_ib_ia_i^*b_i^*}},$ $a_i, b_i\in U({\tilde A}).$
It is standard that  $v_1v_2\cdots v_k\oplus 1_{M_{4k}}\in U_0(M_{4k+1}({\tilde A})).$
Since ${\tilde A}$ has stable rank one (see  11.5 of \cite{eglnp1} and 15.5 of \cite{GLp1}),  by \cite{Ref}, $v_1v_2\cdots v_k\in U_0({\tilde A}).$
It follows that $u\in U_0({\tilde A}).$
Put $u_0=v_1v_2\cdots v_k.$
Let $H={\rm cel}(u_0).$

Write $a_i=\lambda_i+x_i$ and $b_i=\mu_i+y_i,$
where $|\lambda_i|=|\mu_i|=1$ and $x_i, y_i\in A,$
$i=1,2,...,k.$

The rest of the  proof  is similar to that of \ref{UL1}.  We will repeat some of the argument.
we may assume
that there is a strictly positive element
$e\in A_+$ such that $\|e\|=1$  and
\beq\label{UL2-5}
f_\eta(e)x_i=x_if_\eta(e)=x_i, f_\eta(e)y_i=y_if_\eta(e)=y_i,\,\,\,i=0,1,2,...,k,
\eneq
for some $\eta>0.$
Let
$$
{\cal G}_1=\{e, f_\eta(e), f_{\eta/2}(e), x_i, 
{{y_i,}}\, 0\le i\le k\}.
$$
Put
$$
d=\inf\{d_\tau(e): \tau\in \overline{T(A)}^w\}>0.
$$
\Wlog, we may assume that $\tau(f_{1/2}(e))\ge d/2$ for all $\tau\in \overline{T(A)}^w.$

Choose $n\ge 1$ such that
\vspace{-0.13in}$$
4H/n<\ep/64k.
$$
There are mutually orthogonal elements $c_0, c_1,...,c_{n+1}$ in $A$ such that
$c_0\sim c_i$ and $d_\tau(c_0)<\ep d/n64k$ for all $\tau\in T(A).$
Let $\dt>0$
and let 
{{${\cal G}\supset {\cal G}_1$ be a finite subset of $A$}}.

Since $A\in {\cal D},$  there are $A_0$ and
$D\subset A$ with
$D\in {\cal C}_0'$ and $A_0\perp D,$  ${\cal G}$-$\dt$-multiplicative \cpc s
$\phi_0: A\to A_0$ and $\phi_1: A\to D,$
such that
\beq\label{UL2-6}
&&\|x-(\phi_0(x)\oplus \phi_1(x))\|<\dt \rforal x\in {\cal G}\\\label{UL2-6+}
&&\phi_0(e)\lesssim  c_0,\\
&&\tau(f_{1/4}(\phi_1(e)))\ge  d/4\rforal \tau\in \overline{T(A)}^w.
\eneq

By choosing 
{{smaller $\dt$ and larger ${\cal G},$}} we may assume the following:
there is $x_0'\in \overline{\phi_0(f_{\eta/2}(e))A\phi_0(f_{\eta/2}(e))}$
such that $1+x_0'$ is a unitary,
${\rm cel}(p+x_0')\le {{2H}},$ where $p$ is the unit of unitization of ${\tilde B},$
where $B=\overline{\phi_0(f_{\eta/2}(e))A\phi_0(f_{\eta/2}(e))},$
and
there are $z,\, z_i,\,x_i',\, y_i' \in  \overline{f_{\eta/2}(\phi_1(e))Df_{\eta/2}(\phi_1(e))}$ such that
$\lambda_i+a_i'$ and $\mu_i+b_i'$ are unitaries,  {{and}}
\beq\label{UL1-8+}
\|(1+z)-(1+z_1)(1+z_2)\cdots (1+z_k)\|<\ep/16\andeqn \|u_0-(1+x_0'+z)\|<\ep/16,
\eneq
where
$$
1+z_i=(\lambda_i\cdot 1+x_i')(\mu_i\cdot 1+y_i')(
{{\lambda_i\cdot 1}}+x_i')^*(\mu_i+y_i')^*,\,\,\,i=1,2,...,k.
$$
In particular, $(1+z_1)(1+z_2)\cdots (1+z_k)\in CU({\tilde C}),$
where $C=\overline{f_{\eta/2}(\phi_1(e))Df_{\eta/2}(\phi_1(e))}.$
It follows from \ref{Ldert} that
$$
{\rm cel}((1+z_1)(1+z_2)\cdots (1+z_k))\le 4\pi.
$$
As in the proof of \ref{UL1}, using \eqref{UL2-6+},  by applying \ref{ph},
we have
$$
{\rm cel}(1+x_0')\le 4H/n+2\pi +\ep<2\pi +2\ep.
$$
It follows that
\vspace{-0.13in}$$
{\rm cel}(u)<6\pi.
$$

\end{proof}



\begin{prop}{\rm (cf. Theorem 4.6 of \cite{GLX})}\label{PCue}
{{Let $A$ be a separable simple \CA\,  with continuous scale and let
$e\in A_+\setminus \{0\}.$ Then the map
${\imath_e}:U_0({\widetilde{\overline{eAe}}})/CU({\widetilde{\overline{eAe}}})\to U_0({\tilde A})/CU({\tilde A})$
is surjective. If, in addition, $A$ has stable rank one,  then
the map is also injective.}}
\end{prop}

\begin{proof}
The proof is almost identical to  that of the unital case (see Theorem 4.6 of \cite{GLX}).

First, we claim that, for any $h\in A_{s.a.},$ there exists
$h'\in (\overline{eAe})_{s.a.}$ such that
$\tau(h')=\tau(h)$ for all $\tau\in T(A).$
{{Let $h=h_+-h_-.$}}

Put $A_0=\overline{eAe}.$ By Proposition {{5.6 of \cite{eglnp1},}}
there are $x_i, y_j\in A$ ($1\le i\le n$ and $1\le j\le m$) such that
\vspace{-0.12in}\beq
\sum_{i=1}^n x_i^*ex_i=h_+\andeqn\sum_{j=1}^my_j^*ey_j=h_-.
\eneq
\vspace{-0.12in}Then
\beq
h':=\sum_{i=1}^n{{e}}^{1/2}x_i^*x_ie^{1/2}-\sum_{j=1}^me^{1/2}y_j^*y_je^{1/2}\in A_0.
\eneq
Moreover, $\tau(h')=\tau(h)$ for all $\tau\in T(A).$  This proves the claim.

To show $\imath_e$ is surjective, let $u\in U_0({\tilde A})$ with $u=\prod_{j=1}^l\exp(i 2\pi h_j)$
with $h_j\in {\tilde A}_{s.a.}.$ Write $h_j=\af_i\cdot 1_{\tilde A}+h_j',$
where $\af_j\in \R$ with $|\af_j|=1$ and $h_j'\in A_{s.a.}.$ By the claim
that there {{exists}} $h_{0,j}'\in (A_0)_{s.a.}$ such that
$\tau(h_{0,j}')=\tau(h_j')$ for all $\tau\in T(A).$
Let $h_{0,j}=\af_j\cdot 1_{{\tilde A_0}}+h_{0,j}',$ $j=1,2,...,l.$
Put $w=\prod_{j=1}^l \exp(i h_{0,j}).$ Then $w\in U_0({\tilde A_0}).$
Put $v=\prod_{j=1}^l \exp(i {\tilde h}_{0,j}),$
where ${\tilde h}_{0,j}=\af_j\cdot 1_{{\tilde A}}+h_{0,j}',$ $j=1,2,...,l.$
Then $v\in U_0({\tilde A}).$
Moreover, $\imath_e({\bar w})={\bar v}.$
Since
\beq
D_{{\tilde A}}(v)(\tau)=\sum_{j=1}^l \tau({\tilde{h}}_{0,j})=\sum_{j=1}^l\tau({\tilde{h}_j})=D_{\tilde A}(u)(\tau)
\eneq
for all $\tau\in T({\tilde A}),$  {{by Lemma 3.1 of \cite{Thomsen},}}  $\imath_e({\bar w})=\bar u.$
This proves that $\imath_e$ is surjective.

To see it is injective,  let $e_A\in A$ be a strictly positive element
of $A$ with $\|e_A\|=1.$
Since $A$ has continuous scale,  by (the proof of) Proposition {{5.6 of \cite{eglnp1},}}
 there exists an integer $K\ge 1$
such that
\beq
K\la a_0 \ra > \la e_A \ra
\eneq
(in Cuntz semi-group).
Since $A$ has stable rank one,
\wilog, we may write $A\subset M_K(A_0).$
Put $E_0=1_{\tilde A_0}.$
Let $u\in {\tilde A}_0$ with $u=\lambda\cdot E_0+x$
for some $\lambda\in \C$ with $|\lambda|=1$ and
$x\in (A_0)_{s.a.}.$
Write $w=\lambda \cdot 1_{\tilde A}+x.$
Then $\imath_e(\bar u)={\bar w}.$
Suppose that $w\in CU({\tilde A}).$
Write $E=1_{M_K({\tilde A}_0)}$  {{and}}
$w'=\lambda \cdot E+x.$ Then $w'\in CU(M_K({\tilde A}_0)).$
However, since ${\tilde A}_0$ has stable rank one,
it follows from Theorem 4.6 of \cite{GLX} that ${\bar u}\in CU({\tilde A}_0).$
This shows that $\imath_e$ is injective.

\end{proof}

\begin{lem}\label{LKlength}
Let $A$ be a non-unital and $\sigma$-unital simple \CA\, of stable rank one with continuous scale.
Suppose that there is $H>0$ such that,  for any hereditary \SCA\, $B$ of $A,$ ${\rm cel}(z)\le H$ for
any  $z\in
CU({\tilde B}).$
Suppose that there are two mutually orthogonal $\sigma$-unital hereditary \SCA s
$A_0$ and $A_1$ (of $A$) with strictly positive elements $a_0$ and $a_1$ with
$\|a_0\|=1$ and $\|a_1\|=1,$ respectively.
Suppose  that $x\in A_0$ and suppose
that for some $\lambda\in \C$ with $|\lambda|=1,$  $w=\lambda+x\in U_0({\tilde A}).$
Suppose also that there  is an integer $K\ge 1$
%
such that
\beq\label{Llngth-n1}
Kd_\tau(a_0) \ge 1\tforal \tau\in T(A).
\eneq
 Let $u=\lambda\cdot 1_{{\tilde A}_0}+x.$
Suppose   that, for some $\eta\in (0,2],$
$$
{\rm dist}({\bar w},{\bar 1})\leq \eta\,\,\, {{{\rm in}\,\,\, U_0({\tilde A})/CU({\tilde A}).}}
$$
Then, if $\eta<2,$ one has
\vspace{-0.12in}$$
{\rm cel}_{{\tilde A}_0}(u)<({K\pi\over{2}}+1/16)\eta+H
\andeqn {\rm
dist}({\bar u}, {\bar 1_{{\tilde A_0}}})<(K+1/8)\eta,
$$
and if $\eta=2,$ one has
$$
{\rm cel}_{{\tilde A}_0}(u)<{K\pi\over{2}}{\rm cel}(w)+1/16+H.
$$
\end{lem}

\begin{proof}
Let $L={\rm cel}(w).$
{{Since $A$ is simple and has stable rank one,}}
$u\in U_0({\tilde A_0}).$

First consider the case that $\eta<2.$
Let $c\in CU({\tilde A})$ such that
\vspace{-0.12in}$$
\|c-w\|\le \eta.
$$
Choose ${\eta\over{32K(K+1)\pi}}>\ep>0$ such that $\ep+\eta<2.$
Choose $h\in {{\tilde A}}_{s.a.}$ such that
with $\|h\|\le
2\arcsin({\ep+\eta\over{2}})$ such that
\vspace{-0.12in}\beq\label{LKl-2}
w\exp(ih)=c.
\eneq
Thus
\vspace{-0.12in}\beq\label{LKl-3}
\overline{D_{\tilde A}}(w\exp(ih))={\bar 0} \,\,\,{\rm ( in} \,\,\, \Aff(T({\tilde A}))/\overline{\rho_{\tilde A}(K_0({\tilde A}))}{\rm )}.
\eneq
It follows that
\vspace{-0.12in}\beq\label{LKl-4}
|\overline{D_{\tilde A}}(w)(\tau)|\le 2\arcsin({\ep+\eta\over{2}}).
\eneq
Put $h=\af\cdot 1_{\tilde A}+h_0,$ where
$\af\in \R$ with $|\af|\le 2\arcsin({\ep+\eta\over{2}})$ and $h_0\in A_{s.a.}.$
As in the proof of surjectivity of $\imath_e$ in \ref{PCue}, there is $h_0'\in (A_0)_{s.a}$
such that $\tau(h_0')=\tau(h_0)$ for all $\tau\in T(A).$
Put $h_0''=\af\cdot 1_{{\tilde A}}+h_0'.$
Moreover, $\tau(h_0'')=\tau(h)$ for all $\tau\in T({\tilde A}).$
Therefore
\beq
\overline{D_{\tilde A}}(w\exp(i
{{h''_0}}))(\tau)={\bar 0}.
\eneq
It follows from  \ref{PCue} that
\beq
D_{\tilde A_0}(u\exp(ih_{00}))={\bar 0} \,\,\, {\rm ( in} \,\,\Aff(T({\tilde A}_0))/
\overline{\rho_{\tilde A_0}(K_0({\tilde A}_0))}{\rm )},
\eneq
where $h_{00}=\af\cdot 1_{\tilde A_0}+h_0'.$
By \eqref{Llngth-n1}, {{$\|\tau|_{A_0}\|\ge 1/K.$ Then,  by \eqref{LKl-4},}}  in ${\tilde A}_{0},$ {{one computes}}
\beq\label{LKl-5}
|\overline{D_{\tilde A_{0}}}(u)|\le K2\arcsin({\ep+\eta\over{2}}).
\eneq
Thus there is $v\in CU({\tilde A_0})$  and $h_1\in {\tilde A}_{s,a.}$ such that
\beq\label{LKl-6}
u=v\exp(2\pi ih_1)\andeqn \|h_1\|\le K2\arcsin({\ep+\eta\over{2}}).
\eneq
Therefore
\vspace{-0.14in}\beq\label{LKl-7}
{\rm cel}(u)&\le & H+K2\arcsin({\ep+\eta\over{2}})\le H+K(\ep+\eta){\pi\over{2}}\\
&\le&  H+(K{\pi\over{2}}+{1\over{64(K+1)}})\eta.
\eneq
One can also compute that
$$
{\rm dist}({\bar u}, {\bar 1_{\tilde A_0}})\le K(\ep+\eta)\le K\eta+{\eta\over{32(K+1)\pi}}.
$$
This proves the case that $\eta<2.$

Now suppose that $\eta=2.$
Define $R=[{\rm cel}(w)+1].$ Note that ${{\rm cel}(w)\over{R}}<1.$
Put $w'=\lambda \cdot 1_{M_{R+1}}+x.$
It follows from \ref{ph} that
\beq\label{PH-111}
{\rm dist}(\overline{w'},{\overline{1_{M_{R+1}}}})<{{\rm
cel}(w)\over{R+1}}
\eneq
Put $K_1=K(R+1).$ To simplify notation,  replacing $A$ by $M_{R+1}(A),$ without loss of
generality, we may now consider that
\beq\label{PH-102}
K_1d_\tau(a_0)\ge 1
\andeqn {\rm dist}({\bar w},{\bar 1})<{{\rm
cel}(w)\over{R+1}}.
\eneq
Then we can apply the case that $\eta<2$ with $\eta={{\rm
cel}(w)\over{R+1}}.$

\end{proof}

\section{A Uniqueness theorem for \CA s in ${\cal D}$}

\begin{prop}\label{Porth}
Let $A$ be a separable amenable \CA.
Let $\ep>0$ and ${\cal F}\subset A$ be a finite subset.
Then there {{exist}} $\dt>0$ and  {{a finite subset}} ${\cal G}\subset A$ satisfy the following:
Suppose that there are {{two}} mutually
orthogonal \SCA s $A_0$ and $A_1$ and two ${\cal F}$-$\ep/2$-multiplicative \cpc s
$\phi_0: A\to A_0$ and $\phi_1: A\to A_1$
such that
$$
\|x-(\phi_0(x)\oplus \phi_1(x))\|<\ep/2 \rforal x\in {\cal F}
$$
and suppose that there
is
$\psi: A\to B$ (for any \CA\, $B$) which  is a ${\cal G}$-$\dt$-multiplicative \cpc.
Then
there exist a pair of mutually orthogonal \SCA s $B_0$ and $B_1$ of $B$
and ${\cal F}$-$\ep$-multiplicative \cpc s
$\psi_0: A\to B_0\subset B$ and $\psi_1: A\to B_1\subset B$
such that
\beq
&&\|\psi_0(x)-\psi\circ \phi_0(x)\|<\ep\andeqn \\
&&\|\psi_1(x)-\psi\circ \phi_1(x)\|<\ep\rforal x\in {\cal F}.
\eneq
\end{prop}

\begin{proof}
Fix $1/2>\ep>0$ and a finite subset ${\cal F}\subset A.$
Let $\{B_n\}$ be any sequence of \CA s and let $\phi_n: A\to B_n$ be any sequence
of \cpc s such that
\beq
\lim_{n\to\infty}\|\phi_n(a)\phi_n(b)-\phi_n(ab)\|=0\rforal a, b\in A.
\eneq
Let $B_{\infty}=\prod_{n=1}^{\infty}B_n,$ $B_q=B_{\infty}/\oplus_{n=1}^{\infty}B_n$ and
$\Pi: B_{\infty}\to B_q$ {{be}} the quotient map.
Define $\Phi: A\to B_{\infty}$ by $\Phi(a)=\{\phi_n(a)\}$ for all $a\in A.$
Then $\Pi\circ \Phi: A\to B_q$ is a \hm.
Suppose $A_0$ and $A_1$ are in the statement of the {{proposition}}.
Let $a_0\in (A_0)_+$ with $\|a_0\|=1$ and $a_1\in (A_1)_+$ with $\|a_1\|=1$ be
strictly positive elements of $A_0$ and $A_1,$ respectively.
Then
$a_0a_1=a_1a_0=0.$
Therefore there are $b^{(0)}, b^{(1)}\in B_{\infty}$ such that
$b^{(0)}b^{(1)}=b^{(1)}b^{(0)}=0$  and such that $\Pi(b^{(i)})={{\Pi\circ \Phi(a_i)}},$ $i=0,1$ {{(see, for example, 10.1.10 of \cite{Lor}).}}
Write $b^{(i)}=\{b_n^{(i)}\}.$
Let $B_{n,i}=\overline{b_n^{(i)}B_nb_n^{(i)}},$ $i=0,1.$ Then
$B_{n,0}$ and $B_{n,1}$ are mutually orthogonal.
Since $A$ is amenable, there is a \cpc\, $\Psi: A\to B_\infty$ such that
$\Psi=\Pi\circ \Phi.$
Define $\psi_n': A\to B_n$ by
\beq
\psi_n'(a)=b_n^{(i)}\phi_n(a) b_n^{(i)}\rforal a\in A.
\eneq
Let $\psi_0=\psi_n'\circ \phi_0$ and $\psi_1=\psi_n'\circ \phi_0.$
If $n$ is sufficiently large, then $\psi_0$ and $\psi_1$ can be ${\cal F}$-$\ep$-multiplicative.
Moreover, if $n$ sufficiently large,
\beq
&&\|\psi_0(a)-\phi_n\circ \phi_0(a)\|<\ep\rforal a\in {\cal F}\andeqn\\
&&\|\psi_1(a)-\phi_n\circ \phi_1(a)\|<\ep \rforal a\in {\cal F}.
\eneq

If the proposition fails,  then such  $\{\phi_n\}$ could not exists for some choice of $\{B_n\}$,
$\ep$ and ${\cal F}.$
This proves the proposition.
\end{proof}

\begin{NN}\label{RbTuniq}
Fix a map ${\bf T}(n,k): \N\times \N\to \N.$ Let $A\in {\cal D}.$
Denote by ${\cal D}_{{\bf T}(n,k)}$ the class of \CA s in
${\cal D}\cap {\bf C}_{(r_0, r_1, T,  s, R)}$
with $r_0=0, $ $r_1=0,$  ${\bf T}={\bf T}(n,k),$ $s=1$ and $R=7,$ as defined in 3.13 of \cite{eglnkk0}.

Note if $A\in {\cal D},$ then $A$ has stable rank one {{(see 11.5 of \cite{eglnp1})}}
(so $r_0=0$ and $r_1=0$ in 3.14 of \cite{eglnp1}) and  by \ref{Ulength}, ${\rm cer}(M_n({\tilde A}))\le 6+\ep$
($R\le 7$) for all $n.$
If  $A$ is  also ${\cal Z}$-stable, then
$K_0({\tilde A})$ is weakly unperforated. Thus  $A\in {\cal D}_{{\bf T}(n,k)}$
for ${\bf T}(n,k)=n$ for all $(n,k)\in \N\times \N$ (see \ref{Tweakunp} below).

In the appendix to this paper, we show that every  {{amenable}} \CA\, in ${\cal D}$ are ${\cal Z}$-stable.
{{In fact, it is shown that $K_0({ A})$ is always weakly unperforated in   the appendix of \cite{eglnkk0}
for all $A\in {\cal D}.$  Therefore $A\in {\cal D}_{{\bf T}(n,k)}$ for the above ${\bf T}.$}}

\end{NN}

\begin{thm}\label{TUNIq}
Fix ${\bf T}(n,k).$ Let $A$ be a non-unital separable simple \CA\, in ${\cal D}^d$ with continuous scale
which satisfies the UCT.
Let $T: A_+\setminus \{0\}\to \N\times (\R_+\setminus \{0\})$ be a map.
For any $\ep>0$ and any finite subset ${\cal F}\subset A,$  there exists
$\dt>0,$ $\gamma>0,$
$\eta>0,$
a finite subset ${\cal G}\subset A,$ a finite subset ${\cal H}_1\subset A_+\setminus \{0\},$
a finite subset ${\cal P}\subset \underline{K}(A),$ a finite subset ${\cal U}=\{v_1, v_2,...,v_{m_0}\}\subset U({\tilde A})$
such that $\{[v_1], [v_2],...,[v_{m_0}]\}={\cal P}\cap K_1(A),$
and a finite subset ${\cal H}_2\subset A_{s.a.}$ satisfy the following:
Suppose that $\phi_1, \phi_2: A\to B$ are two ${\cal G}$-$\dt$-multiplicative \cpc s
which are $T$-${\cal H}_1$-full {{(see \ref{Dfulln}),}}  where
$B\in {\cal D}_{{\bf T}(n,k)}$  with continuous scale such that
\beq\label{TUNIq-1}
[\phi_1]|_{\cal P}&=&[\phi_2]|_{\cal P},\\\label{TUNIq-1+}
|\tau\circ \phi_1(h)-\tau\circ \phi_2(h)|&<&\gamma\tforal h\in {\cal H}_2\tand \tau\in T(B)\tand\\\label{TUNIq-1++}
{\rm dist}(\overline{\lceil \phi_1(v_i)\rceil},  \overline{\lceil \phi_2(v_i)\rceil} ) &<&\eta \tforal v_i\in {\cal U}\,\,\,{{\rm{(recall\,\,  \ref{Dceil}
\,\,for\,\, \lceil  -  \rceil)}.}}
\eneq
Then there exists a unitary $w\in {\tilde B}$ such that
\beq\label{TUNIq-2}
\|{\rm Ad}\, w\circ \phi_1(a)-\phi_2(a)\|<\ep\rforal a\in {\cal F}.
\eneq

\end{thm}

\begin{proof}
Fix a finite subset ${\cal F}$ and $1/4>\ep>0.$
As pointed out in \ref{RbTuniq}, $B\in {\bf C}_{(0,0, {\bf T}(n,k), 1,7)}$ for all $B\in {\cal D}={\cal D}_{{\bf T}(n,k)},$
{{where ${\bf T}(n,k)=n$ for all $(k,n).$}}
\Wlog, we may assume that ${\cal F}\subset A^{\bf 1}.$

 Since $A$ has the continuous scale,  $T(A)$ is compact {{(see 5.3 of \cite{eglnp1}).}}
Fix a strictly positive element $a_0\in A_+$ with $\|a_0\|=1.$
We may assume, \wilog,
that
\beq\label{Tuniq2-1}
a_0y=ya_0=y,\,\, a_0\ge y^*y\andeqn  a_0\ge yy^* \rforal y\in {\cal F}\andeqn\\\label{Tuniq2-1+}
\tau(f_{1/4}(a_0))\ge 1-\ep/2^{12}\rforal \tau\in T(A).
\eneq
Let $T_1: A_+\setminus \{0\}\to \N\times (\R_+\setminus \{0\})$
 with $T_1(a)=(N(a), M(a))$ ($a\in A_+\setminus \{0\}$) be the map described
 {{after }} \eqref{1225dd} in  \ref{DD1} and \ref{DD0}
 (see also
 8.3 and 10.8 of \cite{eglnp1})
 (in place of $T$).
 Suppose that $T(a)=(N_T(a), M_T(a))$ for $a\in A_+\setminus \{0\}.$


Define $T_2, T_3: A_+\setminus \{0\}\to \N\times (\R_+\setminus \{0\})$
 by $T_2(a)=(N(a), (4/3)M(a))$   and
 $T_3(a)=(N_T(a)N(a), (8/6) (M_T(a)+1)M(a))$ for all $a\in  A_+\setminus \{0\}.$
Define ${\bf L}(u)=8\pi$ for all $u\in U({\tilde A}).$

 Let $\dt_1>0$ (in place of  $\dt$), let
 ${\cal G}_1\subset A$ (in place of ${\cal G}$) be a finite subset,
 let ${\cal H}_{1,0}\subset A_+\setminus \{0\}$ (in place of ${\cal H}$)
 be a finite subset, ${\cal P}_1\subset \underline{K}(A)$ (in place of ${\cal P}$)
 be a finite subset,   let ${\cal U}_1\subset U({\tilde A})$ (in place of ${\cal U}$) be a finite subset
 and  let $K_1\ge 1$ (in place of $K$) be an integer
 given by  {{3.14 and 3.15 of \cite{eglnkk0}, or  by
 7.9 (together with 7.13 of \cite{GLp1})}}
 for the above $T_3$ (in place of $F$), $\ep/16$ (in place of $\ep$) and ${\cal F}$ and ${\bf L}$  with
 ${\bf L}(u)=8\pi.$
 We assume that $a_0, f_{1/16}(a_0),$ $f_{1/8}(a_0)$ and $f_{1/4}(a_0)\in {\cal F}\cup{\cal H}_{1,0}$
{{ (with $r_0=r_1=0,$  $T=T(k,n)$ above,  $s=1$ and $R=7$).}}


We may also assume  {{that}} $\dt_1$ is sufficiently small and ${\cal G}_1$ is sufficiently
 large that
 $
 [L_i]|_{\cal P}
 $
 is well-defined,  and
 $$
 [L_1]|_{\cal P}=[L_2]|_{\cal P},
 $$
 provided that $L_i$ is ${\cal G}_1$-$2\dt_1$-multiplicative and
 $$
 \|L_1(x)-L_2(x)\|<\dt_1\rforal x\in {\cal G}_1.
 $$
 \Wlog, we may also assume that
 $$
 {\cal F}\cup {\cal H}_{1,0}\cup \{xy: x, y\in {\cal F}\}\subset {\cal G}_1\subset A^{\bf 1}.
 $$
 Choose $b_0\in A_+\setminus \{0\}$ with
 \beq\label{513-n-1}
 d_\tau(b_0)<1/2^{10}(2K_1+1)\rforal \tau\in T(A).
 \eneq



{{Choose also}}  a larger finite subset ${\cal G}_1'$ of $A$  and a smaller  $\dt_1'$ so  that
\beq\label{Tuniq2-2+1}
&&\|\lceil L(u)\rceil -L(u)\|<\min\{1/4, \ep\cdot\dt_1/2^{10}\}/ 8\pi
{{\rforal u\in {\cal U}_1\andeqn}}\\
&&\label{Tuniq2-2+n}
\|f_{1/8}(L(a_0))-L(f_{1/8}(a_0))\|<1/2^{10}(K_1+1)
\eneq
%
provided that
$L$ is a ${\cal G}_1'$-$\dt_1'$-multiplicative \cpc\, (to any other \CA).

We may assume that $0<\dt_1'\le {\ep\cdot \dt_1\over{2^{12}(K_1+1)}}.$
For each $v\in {\cal U}_1,$ there is $\af(v)\in \C$ and $a(v)\in A$
such that
\beq\label{Tuniq2-2+2}
v=\af(v)\cdot 1_{{\tilde A}}+a(v),\,\,\, |\af(v)|=1\andeqn \|a(v)\|\le 2.
\eneq

Let $\Omega=\{a(v): v\in {\cal U}_1\}.$
We may also assume that
${\cal G}_1'\supset {\cal G}_1\cup {\cal F}\cup {\cal H}_{1,0}\cup \{xy: x, y\in {\cal G}_1\}\cup \Omega.$
 It follows from \ref{DD1} and \ref{DD0} (see also 8.3 and 10.8 of \cite{eglnp1})
 that there are ${\cal G}_1'$-$\dt_1'/64$-multiplicative \cpc s
 $\phi_0: A\to A$ and $\psi_0: A\to D$ for some $M_{2K_1+1}(D)\subset A$ with
 $D\in {\cal C}_0'$   and $A\perp M_{2K_1+1}(D)$ such that
 \beq\label{Tuniq2-3}
&&\hspace{-0.7in} \|x-(\phi_0(x),\diag(\overbrace{\psi_0(x), \psi_0(x),...,\psi_0(x)}^{2K_1+1})\|<\min\{\ep/K_12^{12}, \dt_1'/128K_1\}\rforal x\in {\cal G}_1',\\\label{Tuniq517n}
&&a_{00}'\lesssim b_0\andeqn
\tau(f_{1/4}(\psi_0(a_0)))\ge 1-\ep/2^{10}\rforal \tau\in T(D),
\eneq
and  $\psi_0(a_0)$ is strictly positive,
 where $a_{00}'$ is a strictly positive element of $\overline{\phi_0(a_0)A\phi_0(a_0)}.$
  Moreover,
 $\psi_0$ is $T_1$- ${\cal H}_{1,0}$-full
 in $\overline{DAD}.$

 We compute that, by \eqref{Tuniq2-1+}, \eqref{Tuniq2-3} and \eqref{Tuniq2-2+n},
 \beq\label{Tuniq2-3+n1}
 2\tau(f_{1/8}(\psi_0(a_0)))\ge 3/4K_1\rforal \tau\in T(A).
 \eneq
We also compute  that  (see \eqref{Tuniq2-1+}, \eqref{513-n-1}, \eqref{Tuniq2-3} and \eqref{Tuniq517n}), for all $\tau\in T(A),$
\vspace{-0.1in}\beq\label{Tuniq2-3+n2}
 \tau((f_{1/4}(\phi_0(a_0)), \diag(\overbrace{f_{1/4}(\psi_0(a_0)), f_{1/4}(\psi_0(a_0)),...,f_{1/4}(\psi_0(a_0))}^{2K_1+1}))>1-\ep/2^9.
 \eneq

 Let $A_{00}=\overline{(\phi_0(a_0), \psi_0(a_0))A(\phi_0(a_0), \psi_0(a_0))}$
 and let $\phi_{00}: A\to A_{00}$ be defined by
 $$
 \phi_{00}(x)=\phi_0(x)\oplus  \psi_0(x)\rforal x\in A.
 $$
 Let $a_{00}=a_{00}' \oplus \psi_0(a_0)\in A_{00}$ be a strictly positive element of $A_{00}.$

By choosing even possibly smaller $\dt_1'$ and larger ${\cal G}_1',$ if necessary,
we may assume that
$[\phi_{00}]|_{{\cal P}_1}$ is well defined and
denote ${\cal P}_2=[\phi_{00}]({\cal P}_1).$
Moreover, we may also assume, \wilog, that
\vspace{-0.1in}\beq\label{Tunq2-5+n1}
[L']|_{{\cal P}_2}=[L'']|_{{\cal P}_2},
\eneq
if
\vspace{-0.13in}$$
\|L'(x)-L''(x)\|<\dt_1'\rforal x\in {\cal G}_1'
$$
and $L'$ and $L''$ are ${\cal G}_1'$-$\dt_1'$-multiplicative \cpc s.
We may also assume that
 \beq\label{Tuniq2-3+}
 &&\|f_{\dt'}(a_{00})\phi_{00}(x)-\phi_{00}(x)\|<\dt_1'/2^{10}\andeqn\\
 &&\|f_{\dt'}(a_{00})\phi_{00}(x)f_{\dt'}(a_{00})-\phi_{00}(x)\|<\dt_1'/2^{10}\rforal x\in {\cal G}_1'
 \eneq
 for some $1/64>\dt'>0.$
Furthermore,
  \beq\label{Tuniq2-3+nn}
 &&\|f_{\dt'}(\psi_0(a_{0}))\psi_0(x)-\psi_{0}(x)\|<\dt_1'/2^{10}\andeqn\\
 &&\|f_{\dt'}(\psi_0(a_{0}))\psi_{0}(x)f_{\dt'}(\psi_0(a_{0}))-\psi_{0}(x)\|<\dt_1'/2^{10}\rforal x\in {\cal G}_1'.
 \eneq
 It follows from \eqref{Tuniq2-3+n1} that  $a_{00}'\lesssim b_0\lesssim f_{1/8}(\psi_0(a_0))$ and,  {{by 3.1 of \cite{eglnp1},}}
 there exists $x_0\in A$ {{such that}}
\vspace{-0.05in} \beq\label{Tuniq2-3+n1+1}
f_{\dt'/256}(a_{00}')(x_0^*f_{1/8}(\psi_0(a_0))x_0)=f_{\dt'/256}(a_{00}').
 \eneq
Let $g\in C_0((0,1])_+$ be such that $\|g\|=1,$
 $g(t)=0$ for all $t\in (0,\dt'/64)$ and $t\in (\dt'/8, 1].$

 Put (keep in mind that $A$ is projectionless and simple)
 \beq\label{Tuniq2-3+1}
 \sigma_0=\inf\{\tau(g(a_{00})): \tau\in T(A)\}>0.
 \eneq
Let
 ${\bar D}=M_{2K_1}(D).$
 Let $j_1: D\to M_{2K_1}(D)={\bar D}$ be defined by
 $$
 j_1(d)=\diag(d,d,...,d)\rforal d\in D.
 $$


 \noindent
 Let $\imath_1: {\bar D}\to A$ be the embedding.
Let $\ep_1=\min\{\ep/2^{10}, \dt_1/2^{10}, \dt_1'/2^{10}\}.$
Choose  a  finite subset ${\cal G}_2'\subset {\bar D}$
 which contains $\bigoplus_{i=1}^{2K_1}\pi_i\circ \psi_0({\cal G}_1'),$
 where $\pi_i: \bigoplus_{i=1}^{2K_1} D\to D$ is the projection to the $i$-th summand.
 Let $e_d\in { D}_+$ with $\|e_d\|=1$
 such that
 \beq\label{Tuniq2-nn-1}
 \|f_{1/4}(e_d)y-y\|< \ep_1/16\ \andeqn \|yf_{1/4}(e_d)-y\|<\ep_1/16
 \eneq
 for all $y\in \psi_0({\cal G}_1').$
 Let ${\bar e_d}=\diag(\overbrace{e_d, e_d,...,e_d}^{2K_1}).$
 \Wlog, we may assume that ${\bar e_d}, f_{1/4}({\bar e}_d)\in {\cal G}_2'.$

 Define $\Delta: D^{q, {\bf 1}}_+\setminus \{0\}\to (0,1)$ by, \, for $d\in D^{\bf 1}_+\setminus \{0\},$
 \beq\label{Tuniq2-5}
 \Delta(\hat{d})=\min\{\inf\{\tau\circ \imath_1\circ j_1(d): \tau\in T(A)\}, \min\{{1\over{2^{10}M(d)^2N(d)}}:d\in \hat{d}\}\}.
 \eneq

For $\ep_1,$  choose $\ep_2>0$  (in place of $\dt$) associated with $\ep_1/16$ (in place of $\ep$)
 and $1/16$ (in place of $\sigma$)  required by {{Lemma 3.3 of \cite{eglnp1}.}}
 \Wlog, we may assume that $\ep_2<\ep_1.$

 Let ${\cal G}_d$ (in place of ${\cal G}$) be a finite subset, ${\cal P}_d\subset K_0({\bar D})$ (in place of ${\cal P}$)
 be a finite subset,
 ${\cal H}_{1,d}\subset  ({\bar D})_+^{\bf 1}\setminus \{0\}$ (in place of ${\cal H}_1$) be a finite subset,
 ${\cal H}_{2,d}\subset ({\bar D})_{s.a.}$ (in place of ${\cal H}_2$) be a finite subset,  $\dt_d>0$ (in place of $\dt$),
 $\gamma_d>0$ (in place of $\gamma$) required by  {{Theorem 7.8 of \cite{eglnp1}}}
 for $C={\bar D},$
 $\ep_2/4$ (in place of $\ep$), ${\cal G}_2'$ (in place of ${\cal G}$) and $\Delta$ above.

 By \eqref{Tuniq2-2+1}, there is a finite subset ${\cal U}_2\subset  {{U({\widetilde{A}_{00}})}}$
 such that, for any $w\in {\cal U}_1,$ there is $w'\in {\cal U}_2$  {{with}}
 \beq\label{Tunq2-5+}
 \|\phi_{00}(w)-w'\|<\min\{1/4, \ep_1/2^{10}\}/8\pi.
 \eneq
 %
 For each $w'\in {\cal U}_2,$ there is $\af(w')\in \C$ with $|\af(w')|=1$ and
 $a(w')\in A_{00}$ with $\|a(w')\|\le 2$ such that
 $$
 w'=\af(w')\cdot 1_{{\tilde A}_{00}}+a(w').
 $$
Define
 $$
 \Omega_0=\{a(w'): w'\in {\cal U}_2\}.
 $$
 Note that by viewing ${\widetilde{A}_{00}}$ as
  a \SCA\, of ${\tilde A},$ we may also view ${\cal U}_2$ as a subset of ${\tilde A}.$

 Let
 $$
 {\cal G}_2=\{a_{00}, f_{\dt'/4}(a_{00}), g(a_{00}), x_0, x_0^*\}\cup {\cal G}_1'\cup \phi_0({\cal G}_1')\cup
{{ \psi_0({\cal G}_1')\cup}} {\cal G}_2'\cup {\cal G}_d\cup {\cal H}_{1,d}\cup {\cal H}_{2,d}\cup \Omega_0\subset A_{00},
 $$
 $${\cal H}_1=\{a_{00}, f_{\dt'/4}(a_{00}), f_{1/4}(a_0), f_{1/4}(\psi_0(a_0)),  g(a_{00})\}\cup {\cal H}_{1,0}\cup {{\psi_0({\cal H}_{1,0})\cup}} {\cal H}_{1,d},$$
   $${\cal H}_2={\cal H}_1\cup {\cal H}_{2,d},\,\,\,
   K_2=2^8\max\{M(a)^2N(a)^2: a\in {\cal H}_1\},$$
    $$\sigma_{00}={\sigma_0\over{K_2}},\,\,\,\dt_2={\min\{\dt_1/16, \dt_d/4, \gamma/2, \eta/2, \dt'/256, \sigma_{00}/4\}\over{4(K_1+1)}},$$
  $${\cal P}={\cal P}_1\cup {\cal P}_2\cup (j_1)_{*0}({\cal P}_d)\cup\{[w']: w'\in {\cal U}_2\},$$
  $$\gamma={\gamma_d\cdot \dt'\cdot \sigma_{00}\over{128(K_1+1)}},$$
  $\eta=1/2^{10}(K_1+1)K_2,$
  and ${\cal U}={\cal U}_1\cup {\cal U}_3$

  Now let ${\cal G}_0$ (in place of ${\cal G}$) and $\dt_0$ (in place of $\dt$)
be as required by \ref{Porth} for ${\cal G}_2$ (in place of ${\cal F}$) and $\dt_2$ (in place of $\ep$).
Since ${\bar D}$ is  weakly semi-projective, we may choose even large ${\cal G}_0$ and
smaller $\dt_0$ such that there is a \hm\, $\Phi$ from ${\bar D}$ such that
$$
\|L(x)-\Phi(x)\|<\dt_2/2\rforal x\in {\cal G}_2\cap {\bar D}
$$
for any ${\cal G}_0$-$\dt_0$-multiplicative \cpc\,  {{$L$}} from ${\bar D}.$
We also assume that
\beq\label{Tunq2-5+1}
\|L(f_{\dt'/4}(a_{00}))-f_{\dt'/4}(L(a_{00}))\|&<&\min\{\dt_2/2, \dt'/32\},\\
\|L(g(a_{00}))-g(L(a_{00}))\|&<&\min\{\dt_2/2, \dt'/32\},\\\label{Tuniq2-5+3}
\tau(g(L(a_{00})))&>&(1/2)\sigma_{00}\rforal  \tau\in T(C)\andeqn\\\label{Tuniq2-5+4}
\tau(f_{\dt'/128}(L(a_{00}')))&<&1/16(2K_1+1)\rforal \tau \in T(C)\\\label{Tuniq2-5+5}
&&\hspace{-1.8in}\tau(f_{1/8}(L(\psi_0(a_{0}))))\ge 1/K_2\rforal \tau\in T(C)
\hspace{0.1in}{\rm (since} \,f_{1/4}(\psi_0(a_0))\in {\cal H}_1{\rm )}
\eneq
for any ${\cal G}_0$-$\dt_0$-multiplicative \cpc\, $L$  from $A$ to $C$
which is also $T$-${\cal H}_1$-full (used for \eqref{Tuniq2-5+3} and \eqref{Tuniq2-5+5}),
where $C$ is any \CA\, with $T(C)\not=\emptyset.$


Let ${\cal G}={\cal G}_2\cup {\cal G}_0$ and $\dt=\min\{\dt_0/2, \dt_2/2\}.$

  Now suppose that $\phi_1, \phi_2: A\to B$ satisfy the assumption of the theorem for the
  above chosen ${\cal G},$ $\dt,$ $\gamma,$ ${\cal P},$ $\eta,$ ${\cal H}_1,$ ${\cal H}_2$ and ${\cal U}$
  (for $T$).

  Let $\phi_{i,0}=\phi_i\circ \phi_{00},$
  $i=1,2.$
  Let $\psi_{i,1}': {\bar D}\to B$ be defined by $(\phi_i)|_{{\bar D}}.$
  By applying \ref{Porth}
  \wilog, we may assume
  that  there are two pairs of  hereditary \SCA s $B_0, B_1$ and $B_0'$ and $B_1',$
  with $B_0\perp B_1$ and $B_0'\perp B_1'$ such that
  $\phi_1(A_{00})\subset B_0$ and $\psi_{1,1}({\bar D})\subset B_1,$
  $\phi_2(A_{00})\subset B_0'$ and $\psi_{2,1}({\bar D})\subset B_1',$
  and $\phi_i|_{A_{00}}$ is ${\cal G}_2$-$\dt_2$-multiplicative, $\psi_{1,1}: {\bar D}\to B_1$
  and $\psi_{2,1}: {\bar D}\to B_1'$  are \hm s such that
  \beq\label{Uni205-2}
  \|\psi_{i,1}'(x)-\psi_{i,1}(x)\|<\dt_2/2\rforal x\in {\cal G}_2\cap {\bar D},\,\,i=1,2.
  \eneq
We may further assume, by \eqref{Tuniq2-3+n1+1} (and $x_0\in {\cal G}_2$),
  \vspace{-0.1in}\beq\label{Uni204}
  f_{\dt'/128}(\phi_1(a_{00}'))\lesssim \psi_{1,1}(f_{1/16}(\psi_0(a_{0}))).
  \eneq

  Choose $b_{00}\in B_+$ such that $\tau(b_{00})\ge 1/2$ for all $\tau\in T(B).$
  Since both $\phi_1, \phi_2$ are $T$-${\cal H}_{1}$-full,
  $\psi_{1,0}$ and $\psi_{2,0}$ are $(4/3)T$-$(\psi_0({\cal H}_{1,0})\cup {\cal H}_{1,d})$-full.
  We then  compute that
  \beq\label{Tuniq2-15}
  \tau(\psi_{i,0}(x))\ge \Delta(\hat{x})\rforal x\in {\cal H}_{1,d} \andeqn \rforal \tau\in T(B).
  \eneq
  Then, by the choice of ${\cal P},$ ${\cal H}_{2,d}$ and $\gamma,$
  by applying 7.8 of \cite{eglnp1}
 we obtain a  unitary
  $U_1'\in {\tilde B}$ such that
  \beq\label{Tuniq2-16}
  \|{\rm Ad}\, U_1'\circ\psi_{2,1}(x)-\psi_{1,1}(x)\|<\ep_2/4\rforal x\in {\cal G}_2'.
  \eneq
  In particular,
 \vspace{-0.1in} \beq\label{Tuniq2-16+}
  \|{\rm Ad}\, U_1'\circ\psi_{2,1}({\bar e_d})-\psi_{1,1}({\bar e_d})\|<\ep_2/4.
  \eneq
  By applying {{Lemma 3.3 of \cite{eglnp1},}}
 there is a unitary $U_1''\in {\tilde B}$ such that
  \beq\label{Tuniq2-16+2}
  {\rm Ad}\, U_1''\circ {\rm Ad}U_1'\circ \psi_{2,1}(x)\in  \overline{ \psi_{1,1}({\bar e_d})B\psi_{1,1}({\bar e_d})}\rforal
  x\in \overline{\psi_{2,1}({\bar e_d})B\psi_{2,1}({\bar e_d})}\andeqn\\
  \|(U_1'')^*cU_1''-c\|<(\ep_1/16)\|c\| \rforal c\in \overline{\psi_{2,1}({\bar e_d})B\psi_{2,1}({\bar e_d})}.
  \eneq
  Put $U_1=U_1'U_1''.$ Then
  we have
  \beq\label{Tuniq2-16+3}
  {\rm Ad}\, U_1\circ \psi_{2,1}(f_{1/4}({\bar e_d})xf_{1/4}({\bar e_d}))\in   \overline{ \psi_{1,1}({\bar e_d})B\psi_{1,1}({\bar e_d})}\rforal x\in A\andeqn\\
  \|{\rm Ad}\, U_1\circ \psi_{2,1}(x)-\psi_{1,1}(x)\|<\ep_1/4\rforal x\in j_1\circ \psi_0({\cal G}_1').
  \eneq

  Let  $B'=\overline{({\rm Ad}\, U_1\circ \psi_{2,1}(f_{1/4}({\bar e_d}))B({\rm Ad}\,U_1\circ \psi_{2,1}(f_{1/4}({\bar e_d}))}$
 and let
  \beq\label{Tuniq2-16+4}
\hspace{-0.3in}  B_p=\{b\in B: bx=xb=0\rforal x\in  B'\}.
  \eneq
  By the choice of ${\cal H}_2$ and the assumption \eqref{TUNIq-1+},  for all $\tau\in T(B),$
  \beq\label{Tuniq2-17}
 |\tau(\phi_{1,0}(f_{\dt'/4}(a_{00})))-\tau(\phi_{2,0}(f_{\dt'/4}(a_{00})))|<\min\{\gamma/2, \dt_2/2\},
  \eneq
 With \eqref{Tunq2-5+1}  in mind, by the assumption,  we have that
 \beq\label{Tuniq2-18}
\|f_{\dt'/4}(\phi_i(a_{00}))-\phi_i(f_{\dt'/4}(a_{00}))\|&<&\min\{\dt_2/2, \dt'/32\}\andeqn\\
\|g(\phi_i(a_{00}))-\phi_i(g(a_{00}))\|&<&\min\{\dt_2/2, \dt'/32\},
 \eneq
  $i=1,2.$ We then compute that,  by \eqref{Tunq2-5+1},
  by the choice of ${\cal H}_2$ and $\gamma,$ and by \eqref{Tuniq2-5+3},
  \beq\label{Tuniq2-19}
  \hspace{-0.6in}\tau(f_{\dt'/4}(\phi_2(a_{00})))&\le &
  \min\{\dt_2/2, \dt'/32\}+\tau(\phi_{2}(f_{\dt'/4}(a_{00}))\\
%
&\le & \min\{\dt_2/2, \dt'/32\}+
\gamma +\tau(\phi_1(f_{\dt'/4}(a_{00}))\\
 &< &  \min\{\dt_2/2, \dt'/32\}+\gamma +\min\{\dt_2/2, \dt'/32\}\\
  &&\hspace{0.6in}+\tau(f_{\dt'/4}(\phi_1((a_{00})))\\
  &<& \tau(g(\phi_1(a_{00})))+\tau(f_{\dt'/4}(\phi_1((a_{00})))\\
  &\le &  \tau(f_{\dt'/64}(\phi_1(a_{00})))
  \eneq
  for all $\tau\in T(B).$
It is important to note that
  $$
  U_1^*f_{\dt'/2}(\phi_2(a_{00}))U_1,\,\,f_{\dt'/64}(\phi_1(a_{00}))\in B_p.
  $$
  Also note that $B_p$ is a hereditary \SCA\, of $B.$
  Since $B$ has
  {{strictly comparison}} for positive elements and $B$ has stable rank one,
  by
  {{3.2 of \cite{eglnp1},}}
 there is a unitary $U_2'\in {\tilde B_p}$ such that
  \beq\label{Tuniq2-20}
  (U_2')^*U_1^*f_{\dt'/2}(\phi_2(a_{00}))U_1(U_2')\in \overline{f_{\dt'/128}(\phi_1(a_{00}))Bf_{\dt'/128}(\phi_1(a_{00}))}:=
 B_{00}.
  \eneq
  Write $U_2'=\af\cdot 1_{{\tilde B_p}}+z$ with $z\in B_p$ and
  $\af\in \C$ with $|\af|=1.$
  Put $U_2=\af\cdot 1_{{\tilde B}}+z.$ Then \eqref{Tuniq2-20} still holds by replacing $U_2'$ by $U_2.$
  Moreover,
\beq\label{Tuniq2-20+}
  U_2^*xU_2=x
  \eneq
  for any  $x\in B'.$
In particular,
  \beq\label{Tuniq2-20+1}
   \|U_2^*({\rm Ad}\, U_1\circ \psi_{2,1}(x))U_2-\psi_{1,1}(x)\|<\ep_1/4+\ep_1/16=5\ep_1/16
   \eneq
   for all $x\in j_1\circ \psi_0({\cal G}_1').$

  Put  $\phi_{2,0}'={\rm Ad}\, U_2\circ {\rm Ad}\, U_1\circ \phi_{2,0}$
   and define
   $\phi_{2,0}'': A\to B_{00}$ by
   \beq\label{Tuniq2-20+1+1}
   \phi_{2,0}''(x)=U_2^*U_1^*f_{\dt'/2}(\phi_2(a_{00}))\phi_{2,0}(x)f_{\dt'/2}(\phi_2(a_{00}))U_1U_2\rforal x\in A.
   \eneq
  By \eqref{Tuniq2-3+},
  $\psi_{2,1}''$ is ${\cal G}_1'$-$\dt_1'/2^4$-multiplicative \cpc.
  Define $\phi_{1,0}': A\to B_{00}$ by
  \beq\label{Tuniq2-20+n1}
  \phi_{1,0}'(x)=f_{\dt'/2}(\phi_1(a_{00}))\phi_{1,0}(x)f_{\dt'/2}(\phi_1(a_{00}))\rforal x\in A
  \eneq
  which is also ${\cal G}_1'$-$\dt_1'/2^4$-multiplicative \cpc.
  Now both $\phi_{1,0}'$ and $\phi_{2,0}''$ are \cpc s from $A$ into
  $B_{00}.$
  Note that $B$ is separable and simple and has stable rank one.
  From the assumption, \eqref{Tuniq2-3+}  and \eqref{Tunq2-5+n1},
  we have
 \vspace{-0.14in} \beq\label{Tuniq2-21}
  [\phi_{2,0}'']|_{\cal P}=[\phi_{2,0}]|_{\cal P}=[\phi_{1,0}]|_{\cal P}=[\phi_{1,0}']|_{\cal P}.
 \eneq
  It follows from  the choice of ${\cal U}_2$ and
  assumption \eqref{TUNIq-1++} (as well as \eqref{Tuniq2-3+} and \eqref{Tuniq2-17} among
  others) that
  \beq\label{Tuniq2-22}
  {\rm dist}(\overline{\lceil \phi_{2,0}''(v)\rceil}, \overline{\lceil \phi_{1,0}'(v)\rceil})<\eta+\dt_1'/2^4
  \rforal v\in {\cal U}_1
  \eneq
  as elements in $U({\tilde B})/CU({\tilde B}).$
  It follows from \eqref{Tuniq2-5+5}
that
\beq\label{Tuniq2-23-1}
\tau(f_{\dt'/128}(\phi_1(a_{00})))>\tau(f_{\dt'/128}(\phi_1(\psi_0(a_{0}))))\ge 1/K_2\rforal \tau\in T(B).
\eneq
   It follows from \ref{LKlength} that, in $U({\tilde B_{00}}),$ for all $v\in {\cal U}_1,$
  \beq\label{Tuniq2-23}
  {\rm cel}_{{\tilde B}_3}(\overline{\lceil \phi_{2,0}''(v)\rceil}\overline{\lceil \phi_{1,0}'(v)\rceil ^*})&<&
  ({K_2\pi\over{2}} +{1\over{16}})(\eta+\dt_2) +6\pi\\
  &\le & 7\pi<{\bf L}(v).
\eneq

  Now let $\tilde{\psi}_d=\psi_{1,1}\circ \diag(\psi_0, \psi_0)$ and $B_2=\overline{{\tilde \psi}_d(A)B{\tilde \psi}_d(A)}.$
  Let $b_{00}'\in B_{00}$ be a strictly positive element with $\|b_{00'}\|=1$
  and let $b_2\in B_2$ be a strictly positive element with $\|b_2\|=1.$
  It follows from \eqref{Uni204} that
  \beq\label{Tuniq2-24}
  b_{00}'\lesssim b_2.
  \eneq
  Recall that $\psi_{1,1}$ and $\psi_{2,1}$ are assumed to be \hm s which
  are $T$-$\psi_0({\cal H}_{1,0})$-full.
  Since $\psi_0$ is $T_1$-${\cal H}_{1,0}$-full in $D,$  ${\tilde \psi}_d$ is also $T_3$-${\cal H}_{1,0}$-full in
  $B_2.$
  Recall that
 \vspace{-0.14in} \beq\label{Tuniq2-25}
  \psi_{1,1}(j_1\circ \psi_0(x))=\diag(\overbrace{{\tilde \psi}_d(x), {\tilde \psi}_d(x),...,{\tilde \psi}_d(x)}^{K_1})\rforal x\in A.
  \eneq
  Now we are ready to apply the stable uniqueness theorem {{3.14 of \cite{eglnkk0}.}}
  By that theorem,  viewing $B$ as a hereditary \SCA\, of $M_{K_1+1}(B_2),$
  there exists a unitary $U_3\in {\widetilde{M_{K_1+1}(B_2)}}$ such that
  \beq\label{Tuniq2-26}
  \|U_3^*(\phi_{2,0}''(x)\oplus \psi_{1,1}(j_1\circ \psi_0(x)))U_3-(\phi_{1,0}'(x)\oplus \psi_{1,1}(j_1\circ \psi_0(x))))\|<\ep/16 \eneq
  for all $x\in {\cal F}.$
  It follows from  \eqref{Tuniq2-nn-1} ,\eqref{Tuniq2-20+n1},\eqref{Tuniq2-3+}  and \eqref{Tuniq2-20+1+1} that
  \beq\label{Tuniq2-26+1}
    \|U_3^*(\phi_{2,0}'(x)\oplus  \psi_{1,1}(j_1\circ \psi_0(x)))U_3-(\phi_{1,0}(x)\oplus \psi_{1,1}(j_1\circ \psi_0(x))))\|\\
    <\ep/16 +\ep_1/4+\dt_1/2^8<\ep/8
    \eneq
  for all $x\in {\cal F}.$
  Since $B$ has stable rank one, one easily find a unitary $U_3'\in {\tilde B}$
  such that the above hold with $\ep/7$ instead of $\ep/8.$

Put $U_4=U_1U_2U_3'.$
  It follows   from \eqref{Tuniq2-3}, \eqref{Tuniq2-20+1}, \eqref{Uni205-2} and {{the above}} that,  for all $x\in {\cal F},$
  \beq\nonumber
  &&\hspace{-0.3in}\|{\rm Ad}\, U_4\circ \phi_2(x)-\phi_1(x)\|\\\nonumber
  &&\le \|U_4^*\phi_2(\phi_{00}(x)\oplus  j_1\circ \psi_0(x))U_4-
  \phi_1(\phi_{00}(x)\oplus  j_1\circ \psi_0(x))\| +2\min\{\ep/128, \dt_1'/128\}\\\nonumber
  &&<\|(U_3')^*( \phi_{2,0}'(x)\oplus  \psi_{1,1}(j_1\circ \psi_0(x)))U_3'-
  (\phi_{1,0}(x)\oplus \psi_{1,1}(j_1\circ \psi_0(x)))\|\\\nonumber
 &&\hspace{2in} +5\ep_1/16 +\dt_2/2+\ep/64\\
 &&<\ep/7++5\ep_1/16 +\dt_2/2+ \ep/64<\ep.
  \eneq

\end{proof}

\begin{rem}\label{Rsec51}
It is easy to see that, with \eqref{TUNIq-1+},
we may assume that $[v_i]\not=\{0\}$ (see \eqref{CUsplit}).
\end{rem}

%
The following  follows from A6 and A7 of the appendix of \cite{eglnkk0}.
We keep here since the proof is much simpler.

\begin{prop}\label{Tweakunp}
Let $A$ be a non-unital separable stably projectionless   exact simple \CA\,
with continuous scale  which is ${\cal Z}$-stable and $T(A)\not=\emptyset.$
Then $K_0({\tilde A})$ is weakly unperforated, i.e.,
if $x\in K_0({\tilde A})$ with $kx\in K_0({\tilde A})_+\setminus \{0\}$ for
some integer $k\ge 1,$ then $x\in K_0({\tilde A})_+.$
Furthermore, if $p, q\in M_s({\tilde A})$ (for some $s\ge 1$) are two projections
such that $\tau(q)<\tau(p)$ for all $\tau\in T({\tilde A}),$ then
$q\lesssim p.$
\end{prop}

\begin{proof}
Put $A_1={\widetilde{A\otimes {\cal Z}}}.$ Note that, since $A$ is ${\cal Z}$-stable,
$A_1={\tilde A}.$
Let $B={\tilde A}\otimes {\cal Z}$ and let $\imath: A_1\to B$ be the embedding.
Then $\imath_{*0}: K_0(A_1)\to K_0(B)$ is an isomorphism.
Let $\pi_A: A_1\to \C$  and $\pi_z: B\to {\cal Z}$ be the quotient maps.
Note that $\pi_z\circ \imath=
{{\imath_{\C, {\cal Z}}\circ \pi_A,}}$ {{where $\imath_{\C, {\cal Z}}$ is the embedding from $\C$ to ${\cal Z}$.}}
Let $t_0$ be the tracial state of $A_1$ which vanishes on $A\otimes {\cal Z}=A$
and $t_z$ be the tracial state of ${\cal Z}.$ Note $T(A_1)=T(A)\cup {{\{t_0\}}}$
and $T(B)=T(A)\cup \{t_z\circ \pi_z\}.$

Let $x\in K_0(A_1)$ such that $kx>0$ in $K_0(A_1)$ for some integer $k\ge 1.$
Suppose that $p, \, q\in M_s(A_1)$ are two projections
such that $[p]-[q]=x$ in $K_0(A_1).$ {{Suppose that $k[p]-k[q]$ is realized by a projection $r\in M_n({\tilde A})$. If $\pi_A(r)=0,$ then $r\in M_n(A)$ which is contradicted with $A$ being stable projectionless. That is $r$ is a full projection in $A_1$. Hence, for all $\tau\in T(A_1), \tau (r)>0$. That is}} 
$\tau(p)>\tau(q)$ for all $\tau\in T(A_1).$ It follows that
$\tau(\imath(p))>\tau(\imath(q))$ for all $\tau\in T(B).$
Also $pM_s(A\otimes {\cal Z})p\not=\{0\}.$
Note $p\not\in M_s(A\otimes {\cal Z})$ since $A$ is stably projectionless. Therefore
the ideal generated by $p$ in $M_s(B)$ contains $q.$
Since $B$ is ${\cal Z}$-stable, by 4.10 of \cite{Rrzstable},
 $q\lesssim p$ in $M_s(B).$ Therefore there is a projection $p_1\le p$  in $M_s(B)$
such that {{$[p_1]=\imath_*(x).$}} 
There is a unitary $w\in {\cal Z}$ such that
$w^*\pi_z(p_1)w=1_{M_k},$ where $1_{M_k}\in M_s(\C)$ is a scalar matrix of rank $k\le s.$
Since $K_1({\cal Z})=\{0\},$ there exists a unitary $W\in M_s(B)$ such
that $\pi_z(W)=w.$ Then $W^*p_1W-1_{M_k}\in {\rm ker}\pi_z=M_s(A\otimes {\cal Z}).$
Let $e=W^*p_1W.$ Then $e\in M_s(A_1).$ We compute that $[e]=x$ in $K_0(A_1).$
This implies that $x>0$ and $K_0(A_1)$ is weakly unperforated.

\end{proof}


\begin{rem}\label{Rsec53}
In Theorem \ref{TUNIq}, if both $\phi_1$ and $\phi_2$ map strictly positive elements
to strictly positive elements,
then, by the virtue of
{{5.7 of \cite{eglnp1},}}
 the fullness condition can be replaced
by
 $\tau(f_{1/2}\phi_1(e))),\,\tau(f_{1/2}(\phi_2(e)))\ge d$ for some given $1>d>0$
and a strictly positive element $e\in A.$
for all $\tau\in T(B).$ If furthermore, $\phi_1$ and $\phi_2$ are assumed to be
\hm s, then, $\tau\circ \phi_i$ are tracial states of $T(A)$ for all $\tau\in T(B).$
Therefore,  the  fullness condition can be dropped.
\end{rem}

\section{\CA s of the form $B\otimes {\cal W}$}
The main purpose of this section is to prove Proposition \ref{CLuniq}.

The following is known {{(in particular, the case that $n=1$).}}
\begin{lem}\label{Thelemma}
Let $B$ be a \CA\, and {{$n\ge 1$ be an integer. Let {{$u\in 1_{M_n({\tilde B})}+M_n(B).$}}
Suppose that
$u\in U_0(M_n ({\tilde B})).$ Then there exists a continuous path  $\{u(t): t\in [0,1]\}\subset U_0(M_n({\tilde B}))$
such that $u(0)=u,$ $u(1)=1_{M_n(\tilde B)}$ and $\phi(u(t))=1_{M_n({\tilde B})}$ for all $t\in [0,1],$ where
$\phi: M_n({\tilde B})\to M_n$ is the quotient map.}}

Moreover, one may write $u=\prod_{k=1}^m \exp( i h_j)$ for some $h_j\in M_n({\tilde B})_{s.a.}$ with
$\phi(h_j)=0$ {\rm{(}}and $\phi(\exp(i h_j))=1_{M_n({\tilde B})}$\rm{)}, $j=1,2,...,m$ (for some $m\ge 1$).
\end{lem}
\begin{proof}
Let $\{w(t): t\in [0,1]\}$ be a continuous path of unitaries such that $w(0)=u$ and $w(1)=1_{M_n({\tilde B})}.$
Let ${\bar w}(t)=\phi(w(t))\in M_n.$  Let $w'(t)\in M_n\subset M_n({\tilde B})$ be the same scalar matrix as ${\bar w}(t)$ with
$\phi(w'(t))={\bar w}(t).$ Note that $w'(0)={{1_{M_n({\tilde B})}}}=w'(1).$
 Define $u(t)=w(t)(w'(t))^*.$  Then $u(0)=u$ and $u(1)=1_{M_n({\tilde B})}.$
 Moreover $\phi(u(t))=1_{M_n({\tilde B})}$ for all $t\in [0,1].$

{{To see the last part, one chooses a partition $0=t_0<t_1<\cdots t_m=1$ of $[0,1]$ such that $\|u(t_{j-1})u(t_j)^*-1\|<1.$
Define $h_j={1\over{2\pi i}}\log(u(t_{j-1})u(t_j)^*),$ $j=1,2,...,m.$  Then $u=\prod_{j=1}^m \exp(ih_j).$ Note
that $\phi(u(t_j))=1_{M_n({\tilde B})},$ whence $\phi(h_j)=0,$ $j=1,2,...,m.$}}

\end{proof}

\begin{lem}\label{Lreduction}
Let $B$ be a \CA\, and $u=1_{\tilde B}+x\in {\tilde B}$ {{be a unitary,}} where $x\in B$.
~ Suppose
that $\diag(u, 1_{M_m({\tilde B})})\in U_0(M_{m+1}({\tilde B})).$   Let $v=1_{C}+x\otimes 1_{Q}\in
C,$ where $C={\widetilde{B\otimes Q}}.$ Then $v\in U_0(C).$ Moreover,
there exists a continuous path of unitaries $\{v(t): t\in [0,1]\}\subset C$ such that
$v(0)=v,$ $v(1)=1_C$ and $\pi(v(t))=1_C$ for all $t\in [0,1],$  where $\pi: C\to \C$ is the quotient map.

\end{lem}

\begin{proof}

By \ref{Thelemma}, there exists a continuous path of unitaries $\{u(t): t\in [0,1]\}\subset U_0({\widetilde{M_{m+1}(B)}})$ such that
\beq
u(0)= \diag(u, 1_{M_m({\tilde B})}),\,\,\, u(1)={{1_{M_{m+1}(\tilde B)}}}\andeqn \pi(u(t))={{1_{M_{m+1}}}},
\eneq
where $\pi: {\tilde B}\to \C$ is the quotient map.
Write  $u(t)=1_{m+1}({\tilde B})+x(t),$ where $\{x(t): t\in [0,1]\}\subset M_{m+1}(B)$ is a
continuous path such that $x(t)+x(t)^*+x(t)^*x(t)=0,$  $x(t)+x(t)^*+x(t)x^*(t)=0$ and
$x(1)=0.$

Let $e_1, e_2, ...,e_{m+1}\in Q$ be mutually orthogonal and mutually equivalent projections
such that $\sum_{i=1}^{{m+1}}e_i=1_Q.$
Put
\beq
E_i=1_{m+1}\otimes e_i=\diag(\overbrace{e_i, e_i, ...,e_i}^{m+1}) \in M_{m+1}({{{\tilde B}\otimes Q}}),\,\,\,i=1,2,...,m+1.
\eneq
Define
\beq
w_i=1_C+x(0)\otimes e_i=1_C+x\otimes e_i,\,\,\,i=1,2,...,m.
\eneq
Then
\beq
u=w_1w_2\cdots w_{m+1}.
\eneq
Let  $X_i\in
{{M_{m+1}({\tilde B}\otimes Q)}}$ be  a unitary such that
\beq
X_iE_iX_i^*={{1_{{\tilde B}\otimes Q}=1_C.}}
\eneq
{{Note that $X_i(x(t)\otimes e_i)X_i^*\in B\otimes Q$.}} Define
$w_i(t)=1_C+X_i(x(t)\otimes e_i)X_i^*\in C$
 for $t\in [0,1],$ $i=1,2,...,m+1.$
Then
\beq
\hspace{-0.2in}w_i(t)^*w_i(t)&=&1_C+X_i(x(t)\otimes e_i)X_i^*+X_i(x^*(t)\otimes e_i)X_i^*+X_i(x(t)x^*(t)\otimes e_i)X_i^*\\
                     &=&1_C+X_i(x(t)+x^*(t)+x(t)x^*(t))X_i^*=1_C.
\eneq
Similarly,
\beq
w_i^*(t)w_i(t)=1_C,\,\,\,i=1,2,...,m+1.
\eneq
So $\{w_i(t): t\in [0,1]\}\subset U_0(C)$ with
$w_i(0)=w_i$ and $w_i(1)=1_C.$ Moreover,
\beq
\pi(w_i(t))=1\rforal t\in [0,1],\,\,\, i=1,2,...,m+1.
\eneq
Define $v(t)=w_1(t)w_2(t)\cdots w_{m+1}(t)$ for $t\in [0,1].$
Then
\beq
v(t)\in U(C),\,\,\,v(0)=w_1(0)w_2(0)\cdots  w_{m+1}(0)=1_C+x=u\andeqn v(1)=1_C.
\eneq
\end{proof}

\begin{thm}\label{T1228T1}
Let $B$ be a \CA\, and  $C={\widetilde{B\otimes {\cal W}}}.$
Then $U(M_m(C))=U_0(M_m(C))$ for all integer $m\ge 1.$
Moreover, if $u{=}1_C+x$  {{is a unitary in $C$}} for some $x\in B\otimes {\cal W},$
then there exists a continuous path of unitaries $\{u(t): t\in [0,1]\}\subset C$
such that $\pi(u(t))=1$ for all $t\in [0,1],$ where $\pi: C\to\C$ is the quotient map.
\end{thm}

\begin{proof}
Note $K_1({\cal W})=\{0\}.$ Therefore $K_1(B\otimes {\cal W})=\{0\}.$

Let $u\in U(M_m(C)).$  Write $u=w+x,$ where $w\in M_m$ is a scalar unitary and $x\in M_m(C).$
By considering $w^*u,$
\wilog, we may {{assume}} $u=1_{M_m}+x$ for some $x\in M_m(C).$
{{Hence (by \ref{Thelemma})}} there exists a continuous path $\{w(t): t\in [0,1]\}\subset 1_{M_{n+m}(C)}+
M_{n+m}{{(B\otimes {\cal W})}}$
such that $w(0)=\diag(u, 1_n)$ and $w(1)=1_{n+m}$ for some integer $n\ge 1.$
Since ${\cal W}\cong {\cal W}\otimes Q$ and $Q$ is strong self absorbing,
\wilog, we may assume that $u=1_{M_m(C)}+x\otimes 1_Q.$
Thus \ref{Lreduction} applies.   This  proves the second part of the lemma.
To see the first part, we let $m=1.$

\end{proof}

\begin{lem}\label{1228proj}
Let $B$ be a \CA\, and
let $q\in M_m({\tilde C})$ be a projection, where $C=B\otimes {\cal W},$
such that $\pi_C(q)=p~ { {\in M_m(\C)}},$ a projection matrix, where $\pi_C: {\tilde C}\to \C$ is the quotient map. Then
there exists an integer $r_1\ge r\ge 0$ and   unitary  $w\in M_{m+r_1}({\tilde C})$ such that
$w(q\oplus 1_r)w^*=P\oplus 1_r,$ where $P$ is the matrix in $M_m(\C\cdot  1_{\tilde C})$ which is the same matrix as $p.$
 Moreover, $\pi_C(w)=1_{m+{r_1}}.$
\end{lem}

\begin{proof}
Since ${\cal W}$ is $KK$-contractible, $K_0(B\otimes {\cal W})=\{0\}.$
Therefore, for some large $r_1\ge r\ge 0,$ there is $w_1\in M_{m+r_1}({\tilde C})$
such that
\beq
w_1(q\oplus 1_r)w_1^*=P\oplus 1_r.
\eneq
Note $\pi_C(w_1)(P\oplus 1_r)\pi_C(w_1)^*=P\oplus 1_r,$ where we identify {{these elements with}} matrices with scalar entries.
Write $W=\pi_C(w_1)$ as a unitary matrix with scalar entries.
Note that $W(P\oplus 1_r)W^*=P\oplus 1_r.$
Let  $w=W^*w_1.$ Then, $\pi_C(w)=1_{m+r_1}$ and
\beq
w(1\oplus 1_r)w^*=P\oplus 1_r.
\eneq

\end{proof}

\begin{lem}\label{L1228L1}
Let $B$ be a \CA\, and $u=1_{\tilde B}+x\in U_0({\tilde B}),$ where $x\in B.$
Then, for any $\ep>0,$ there exists $m\ge 1$ such that
there exists a continuous path of unitaries $\{y(t): t\in [0,1]\}\subset 1_{M_{m+1}}+M_{m+1}(B)$
such that $y(0)=\diag(u, 1_{M_{m}}),$ $y(1)=1_{M_{m+1}}$ and
${\rm length}(\{y(t):t\in [0,1]\})\le 4\pi+\ep.$

Moreover,
 let $C=\widetilde{B\otimes {\cal K}}$ and $v=1_C+z\in U_0(C)$ for some $z\in C.$
Then, for any $\ep>0,$ there exists a continuous path of unitaries $\{v(t): t\in [0,1]\}$ such that
$v(0)=v,$ $v(1)=1_C, $ $\Pi(v(t))=1$ for all $t\in [0,1]$ and ${\rm cel}(v(t))\le 4\pi +\ep,$ where
$\Pi: C\to \C$ is the quotient map.  Consequently,
$v\in U_0(C)$ and ${\rm cel}(v)\le 4\pi+\ep.$
\end{lem}

\begin{proof}
By \ref{Thelemma},
we may assume that there exists a continuous path of unitaries $\{u(t): t\in [0,1]\}\subset {\tilde B}$ such that
 $u(t)=1_{\tilde B}+x(t),$ where $x(t)\in B$ such that $u(0)=u$ and $u(1)=1_{\tilde B}.$
 Note that  $\{x(t): t\in [0,1]\}$ is continuous,
$x(0)=x$ and $x(1)=0.$ Moreover, for all $t\in [0,1],$
\beq\label{L11228n-4}
x(t)+x(t)^*+x(t)x^*(t)=0\andeqn x(t)+x(t)^*+x(t)^*x(t)=0.
\eneq
Fix $1/4>\ep>0.$ There is a partition $0=t_0<t_1<\cdots t_n=1$ such that
\beq\label{L11228n-4+}
\|x(t)-x(t_i)\|<\ep/2\rforal t\in [t_i, t_{i+1}],\,\,\, i=0,1,2...,n-1.
\eneq

Put  $D=M_{2n+1}(B).$ Then
\beq\label{L11228n-5}
Z_0=1_{\tilde D}+\diag(x(0), x^*(0), x(t_1), x^*(t_1), ..., x(t_{n-1}), x^*(t_{n-1}),0 )\in {\tilde D}.
\eneq
Put $y_i=x(t_i)+x^*(t_i),$  $i=0,1,...,n-1.$
Then,  by \eqref{L11228n-4},
\beq\nonumber
Z_0Z_0^*
&=&1_{\tilde D}+\diag(y_0+x(0)x^*(0), y_0+x(0)^*x(0),  y_1+x(t_1)x^*(t_1), y_1+x(t_1)^*x(t_1), ..., 0)\\\nonumber
&=&1_{\tilde D}.
\eneq
Similarly,
$
Z_0^*Z_0=1_{\tilde D}.
$
In other words, $Z_0$ is a unitary in ${\tilde D}.$
Put, for $t\in [0,1],$
\beq
V(t)=\begin{pmatrix} \cos(t\pi/2) & -\sin(t\pi/2)\\
                                    \sin(t\pi/2) & \cos(t\pi/2)\end{pmatrix}
\eneq
Note that $V(t)$ is a unitary in $M_2({\tilde B}),$
$V(0)=1_{M_2({\tilde B})}$
and $V(t)^*V(t)=1_{M_2({\tilde B})}.$
Put $w_i(t)=V(t)\diag(1_{\tilde B}+x(t_i), 1_{\tilde B}) V^*(t) \diag(1_{\tilde B}, 1_{\tilde B}+x^*(t_i)),$ $t\in [0,1]$ and
$i=0,1,...,n-1.$
Note that
\beq
w_i(0)=\diag(1_{\tilde B}+x(t_i), 1_{\tilde B}+x^*(t_i)),\andeqn
w_i(1)=\diag(1_{\tilde B}, 1_{\tilde B}),\,\,\,i=0,1,...,n-1.
\eneq
Let $\phi: M_2({\tilde B})\to M_2$ be the quotient map.
We also have
\beq\label{L1228n-8}
\phi(w_i(t))=\phi(V(t))\diag(1, 1)\phi(V^*(t))\diag(1,1)=\diag(1, 1).
\eneq
Define
\beq
Z(t)=\diag(w_0(t), w_1(t),...,w_{n-1}(t), 1_{\tilde B})
\eneq
Then $Z(0)=Z_0,$ $Z(1)=1_{\tilde D}.$ Moreover,  by \eqref{L1228n-8},
$$
\pi_D(Z(t))=1_{\pi_D({\tilde D})},
$$
where $\pi_D: {\tilde D}\to \C$ is the quotient map.  Note that $\{Z(t): t\in [0,1]\}$ is a continuous path of unitaries
in ${\tilde D}.$  It is standard to compute that ${\rm length}(\{Z(t):t\in [0,1]\})=2\pi.$
Put
\beq
Z_{-1}=
{{1_{\tilde D}}}+\diag(x(0), x^*(t_1), x(t_1), x^*(t_2), x(t_2), ..., x^*(t_n), x(t_n)).
\eneq
By \eqref{L11228n-4+},
\beq
\|Z_0-Z_{-1}\|<\ep.
\eneq
There exists a continuous path of unitaries $\{Z_{-1}(t): t\in [0, 1]\}\subset {\tilde D}$ such that
$Z_{-1}(0)=Z_{-1},$ $Z_{-1}(1)=Z_0$ and ${\rm length}(\{Z_{-1}(t): t\in [0,1]\})\le 
{{2\arcsin(\ep/2)}}.$
Define
$$\omega_i(t)=V(t)\diag(1_{\tilde B}+x^*(t_i), 1_{\tilde B})V^*(t) \diag(1_{\tilde B}+x(t_i), 1_{\tilde B})$$
for $t\in [0,1]$ and $i=1,2,...,n.$
Then, similar to some computation above,
$\{\omega_i(t): t\in [0,1]\}\subset M_2({\tilde B})$ is a continuous path of unitaries such that
$\omega_i(0)=1_{M_2({\tilde B})},$ $\omega_i(1)=\diag(1_{\tilde B}+x^*(t_i), 1_{\tilde B}+x(t_i))$ and
\beq
\phi(w_i(t))=\phi(V(t))\diag(1,1)\phi(V^*(t))\diag(1,1)=\diag(1,1)\rforal t\in [0,1],
\eneq
$i=1,2,...,n.$ Define
\beq
Z_{-2}(t)=\diag(1_{\tilde B}+x(0), \omega_1(t), \omega_2(t),...,\omega_n(t))\rforal t\in [0,1].
\eneq
Then $\{Z_{-2}(t): t\in [0,1]\}\subset U_0({\tilde D})$ is a continuous path such that
\beq
Z_{-2}(0)&=&1_{\tilde D}+x(0)=1_{\tilde D}+x=\diag(u,1_{2n})\\
 Z_{-2}(1)&=&1_{\tilde D}+\diag(x(0), x^*(t_1), x(t_1),..., x_n^*(t_n), x_n(t_n))=Z_{-1}.
\eneq
Moreover, ${\rm length}(\{Z_{-2}(t)\}: t\in [0,1]\})\le 2\pi.$
Now define
\beq
y(t)=\begin{cases} Z_{-2}(3t) & \text{if $t\in [0,1/3]$ }\\
                              Z_{-1}(3( t-1/3)) & \text{if  $t\in [1/3, 2/3]$}\\
                              Z (3(t-2/3)) &\text{if $t\in [2/3, 1].$}\end{cases}
\eneq
Now $\{y(t): t\in [0,1]\}\subset U_0({\tilde D}),$ $y(0)=Z_{-2}(0)=\diag(u,1_{2n})$  and $y(1)=Z(1)=1_{\tilde D}.$
Moreover, for any $1/4>\ep>0,$
\beq
{{{\rm length}(\{y(t): t\in [0,1]\})}}\le 4\pi+ 2\arcsin(\ep/2).
\eneq
This proves the first part of the statement.
The second part follows the same way.
\end{proof}

\begin{thm}\label{TUnituniform}
Let $\{B_n\}$ be a sequence of \CA s
and $C=\prod_{n=1}^{\infty}(B_n\otimes {\cal W}\otimes {\cal K}).$
 Then $K_1(C)=\{0\}.$
Moreover $K_1(E)=\{0\},$ where $E=\prod_{n=1}^{\infty} (B_n\otimes {\cal W}\otimes {\cal K})^\sim.$

\end{thm}

\begin{proof}
Let $u\in M_m({\tilde C}).$
Let $\pi_C:  {\tilde C}\to \C$ be the quotient map and let $z=\pi_C(u)\in M_m(\C)$ be a unitary matrix.
{{Denote}} by $Z$ the same scalar unitary matrix (as $z$) in $M_m(\C 1_{\tilde C}).$
{{Replacing}} $u$ by $Z^*u,$ \wilog, we may assume that $u\in 1_{M_m(\tilde C)}+M_m(C).$
To simplify notation, \wilog, we may further assume that $u\in 1_{\tilde C}+C.$
 We write $u=\{u_n\},$ where
$u_n=1_{(B_n\otimes {\cal W}\otimes {\cal K})^\sim_n}+x_n$ {{and  where}} $x_n\in B_n\otimes {\cal W}\otimes {\cal K},$ $n=1,2,....$

By applying \ref{L1228L1}, {{we obtain,}}  for each $n,$  a continuous path of unitaries
$v_n(t)\in 1_{(B_n\otimes {\cal W}\otimes {\cal K})^\sim}+B_n\otimes {\cal W}\otimes {\cal K}$ with $v_n(0)=u_n$ and $
{{v_n(1)}}=1_{(B_n\otimes {\cal W}\otimes {\cal K})^\sim}$ and
${\rm cel}(v_{{n}}(t))< 5\pi.$ Thus, we {{obtain (see Lemma 1.1 of \cite{GL1})}} a sequence of equi-continuous paths of unitaries $\{w_n(t): t\in [0,1]\}\subset
1_{(B_n\otimes {\cal W}\otimes {\cal K})^\sim}+B_n\otimes {\cal W}\otimes {\cal K}$ with $w_n(0)=u_n$ and $w_n(1)=1_{(B_n\otimes {\cal W}\otimes {\cal K})^\sim},$ $n=1,2,....$
Define
\beq
u(t)=\{w_n(t)\}\subset 1_{\tilde C}+C.
\eneq
Then $\{u(t):t\in [0,1]\}$ is a continuous path of unitaries in $1_{\tilde C}+C$ such
that $u(0)=u$ and $u(1)=1_{\tilde C}.$ Thus $K_1(C)=\{0\}.$

If $u\in U(E),$ we write $u=\{u_n\},$ where $u_n\in U((B_n\otimes {\cal W}\otimes {\cal K})^\sim).$
Denote by $\pi_n:  (B_n\otimes {\cal W}\otimes {\cal K})^\sim \to \C$ the quotient map.
Let $\lambda_n=\pi_n(u_n).$ Then $\lambda_n\in \T.$ Consider the unitary
$Z=\{\lambda_n\}\in U(E).$ Then $Z\in U_0(E).$ Now consider $v=uZ^*.$ Then $v\in {\tilde C}.$
Thus the above shows
 that $K_1(E)=\{0\}.$

\end{proof}

\begin{lem}\label{1229C1}
Let $B_n$ be a sequence of \CA s and let $C=\prod_{n=1}^{\infty} B_n\otimes {\cal W}.$
Then
$K_i(C)$ is  divisible and the map from $K_i(C)$  to
$K_i(C, \Z/k\Z)$  is zero, $i=0,1$ and $k=2,3,....$
\end{lem}

\begin{proof}
Let $\psi: Q\otimes Q\to Q$ be an isomorphism.
Since ${\cal W}\otimes Q\cong {\cal W},$ we may write $C=\prod_{n=1}^{\infty} B_n\otimes {\cal W}\otimes Q.$
Define $\Phi: C\to C\otimes Q$ by $\Phi(\{a_n\})=\{a_n\}\otimes 1_Q$
for all $\{a_n\}\in C.$  It is an injective \hm.
Let $\Psi: C\otimes Q\to C$  be a \hm\,
defined by $\Psi(\{(b_n\otimes r_n)\otimes r\})=\{b_n\otimes \psi(r_n\otimes r)\}$ for all
$\{b_n\otimes r_n\}\in C,$ where $b_n\in B_n\otimes {\cal W}$ and $r_n\in Q.$
{{Denote by ${\tilde \Phi}: {\tilde C}\to {\widetilde{C\otimes Q}}$ and
${\tilde \Psi}: {\widetilde{C\otimes Q}}\to {\tilde C}$ the extensions.}}

Fix $z\in K_0(C).$ We will show that, for any integer $k\ge 2,$ there exists $y\in K_0(C)$
such that $ky=z.$  \Wlog, we may assume that $z=[p]-[q],$ where
$p\in M_r(\C)$ is a matrix with scaler entries and $q=p+x,$ where $x\in M_r(C)_{s.a},$ and
both $p$ and $q$ are projections.  Let $D$ be a separable \SCA\, of $C$ such that
$M_r(D)$ contains $x.$  Let $\iota: D\to C$ be the embedding.
Consider $\Phi(D).$ Then, $\Phi(D)\subset D_1:=\{\iota(d)\otimes r: d\in D, r\in Q\} \cong D\otimes Q.$
Note, for each $\ep>0$ and   each finite subset ${\cal F}\subset Q,$ there exists a unitary $u\in Q$ such that
\beq
u^*\psi(y\otimes 1_Q)u\approx_{\ep} y\rforal y\in {\cal F}.
\eneq
Now write $x=\{(x_{i,j}^{(n)})_{r\times r}\}.$
For each $n, i, j,$ there are $a_{i,j,k}^{(n)}\in B_n\otimes {\cal W}$ and $r_{i,j,k}^{(n)}\in Q,$
$k=1,2,..., N(i,j, n),$ such that
\beq
\|x_{i,j}^{(n)}-\sum_k a_{i,j,k}^{(n)}\otimes r_{i,j,k}^{(n)}\|<1/(4^{n+1}(\|x\|+1) r^2),\,\,\, 1\le i,j\le r.
\eneq
Let ${{M_n=\max\{\|a_{i,j,k}^n\|: 1\le k\le N(i,j,n)\}}}.$ Therefore there is a unitary $u_n\in Q$ such that
\beq
\|u_n^*\psi(r_{i,j,k}^{(n)}\otimes 1_Q)u_n-r_{i,j,k}^{(n)}\|<1/(4^n (\|x\|+1)N(i,j,n)M_n r^2),\,\,\, 1\le i, j\le r, \,\,n=1,2,....
\eneq
%
Let $u=\{1_{{\tilde B}_n}\otimes u_n\}$ and $U=\diag(\overbrace{u, u, ...,u}^r).$ Then
$u^*\{b_n\}u\in C\rforal \{b_n\}\in C.$
In other words, ${\rm Ad}\, u: C\to C$ is an automorphism.
Moreover
\beq\label{18-616}
\|U^*(\Psi\circ \Phi(x))U-x\|<1/4.
\eneq
Let $H_1={\rm Ad}\, u\circ \Psi: C\otimes Q\to C$ and $H_2=H_1\circ \Phi: C\to C.$
Put $\imath_1=\imath\otimes {\rm id}_Q:  D_1:=D\otimes Q\to C\otimes Q.$
 Recall that $p,q \in M_r({\tilde D})$ and $p-q=x\in M_r(D).$
Thus  $[p]-[q]$ also defines an element $z'\in K_0(D)$ and $z=\iota_{*0}(z')$ in $K_0(C)$. Since  $\Phi (D)\subset D_1\cong D\otimes Q$, We {{may}} regard $\Phi|_D$ as a map from $D$ to $D_1$---
let us {{denote it by}} $\Phi',$  that is $\Phi':D \to D_1$ is the same map as  $\Phi|_D=\Phi\circ \iota$ but with different codomain algebra $D_1$ (instead of $C\otimes Q$). Formally, we have $\Phi|_D=\iota_1\circ \Phi'$. Then
 $$\Phi_{*0}(z)=\Phi_{*0}(\iota_{*0}(z'))=(\Phi|_D)_{*0}(z')=(\iota_1\circ\Phi')_{*0}(z')=(\iota_1)_{*0}(\Phi'_{*0}(z')). $$

{{By K\"unneth formula,}}  $K_0(D_1)$ is divisible.
{{Therefore}} there is $y'\in K_0(D_1)$ such that $\Phi'_{*0}(z')=k y'\in K_0(D_1)$. That is $\Phi_{*0}(z)=(\iota_1)_{*0}(ky')= k(\iota_1)_{*0}(y')$. Hence $(H_2)_{{*0}}(z)=(H_1)_{{*0}}(\Phi_{{*0}}(z))=ky$, where $y=(H_1)_{{*0}}((\iota_1)_{*0}(y'))
{{\in K_0(C)}}$. That is, $(H_2)_{{*0}}(z)$  {{divisible}} by $k$.

Since $p=(\lambda_{i,j})_{r\times r},$ where $\lambda_{i,j}\in \C.$ {{(by \eqref{18-616})}},
\beq
{{\tilde{H_2}}}(p)=U^*pU=p\andeqn
\|({{{\tilde H_2}}})(q)-q\|=\|H_2(x)-x\|<1/4,
\eneq
{{where we use $\tilde{H}_2:{\tilde C}\to {\tilde C}$
{{for}} the {{unitization}}  of $H_2: C\to C$ and its induced map $\tilde{H}_2: M_r({\tilde C}) \to M_r({\tilde C})$.}} {{It follows that}}
\beq
(H_2)_{*0}(z)=z\,\,\, {\rm in}\,\,\, K_0(C).
\eneq
Hence $z$ is divisible by $k.$
This shows that $K_0(C)$ is divisible. It follows that $K_0(C)/kK_0(C)=\{0\}$ for
all $k\ge 2.$

A similar argument shows that $K_1(C)$ is also divisible.

\end{proof}

\begin{lem}\label{LpartK}
Let $A$ be a separable \CA\,  and let $\phi_n
: A\to B_n\otimes {\cal W}$ be a
sequence of  approximately multiplicative \cpc s.
Then there exists a sequence  of integers $\{m(n)\}$ {{satisfying}} the following condition:
Put $C=\prod_n(M_{m(n)}(B_n\otimes {\cal W}))$ and $C_0=\bigoplus_{n=1}^{\infty} M_{m(n)}(B_n\otimes {\cal W}).$
Then $[\pi\circ (\{\phi_n\})]=0$
in ${\rm Hom}_{\Lambda}(\underline{K}(A), \underline{K}(C/C_0)),$
 where $\pi: C\to C/C_0$  is the quotient map and where we view $\phi_n$ maps $A$ into $M_{m(n)}(B_n\otimes {\cal W})$
 (as $\phi_n(a)=\diag(\phi_n(a),0,...,0)$).

\end{lem}

\begin{proof}
Let $K_0(A)=\{x_1,x_2,...,x_n,...\}.$
Put $D_n= (B_n\otimes {\cal W})^\sim.$
Suppose that, \wilog, that $[\phi_n(x_i)]$ is well defined{{, for all $i\leq n$}}.
We may write $x_i=[p_{i}]-[q_{i}]\in 
{{K_0(A)}},$ where $p_{i}\in M_{r(i)}(\C)$ is a projection and
$q_{i}=p_{i}+b_{i}$ and $b_{i}\in M_{r(i)}(A)_{s.a.}.$
By \ref{1228proj}, {{for any $i\leq n$,}} there are integers $r(i,n)'$ and $m(n)'= r(i)+r(i,n)',$ and
 a unitary $u_{i,n}\in M_{m(n)'}(D_{n})$
with $\pi_{d,n}(u_{i,n})=1_{m(n)'}
,$ where $\pi_{d,n}: M_{m(n)'}(D_{n})\to M_{m(n)'}$ is
the quotient map, such that
\beq
&&\|u_{i, n}^*(\phi_n(p_{i})\oplus 1_{r(i,n)'})u_{i,n}-(\phi_n(q_{i})\oplus 1_{r(i,n)'})|<1/2^{n+1}
\eneq
for all large $n.$
Note that $\{u_{i,n}\}
\in (\prod_n M_{m(n)'}(B_n\otimes {\cal W}))^\sim.$
It follows that, for any $i\ge 1,$ there exists $k(i)\ge i$  {{such that}}
\beq
[\{\phi_n(x_i)\}_{n\ge k(i)}]=0
\,\,\,{\rm in}
\,\,\, K_0(\prod_{n\ge m(k)} M_{m(n)'}(B_n\otimes {\cal W})).
\eneq
Thus, for any $x\in K_0(A),$ $[\pi'(\{\phi_n(x))]=0,$ where $\pi': \prod_n M_{m(n)'}(B_n\otimes {\cal W})
\to \prod_n M_{m(n)'}(B_n\otimes {\cal W})/\oplus_n (M_{m(n)'}(B_n\otimes {\cal W})$ is the quotient map.

Let $K_1(A)=\{g_1, g_2,...,g_n,...\}$ and $z_i\in M_{d(i)}$ be a unitary so that $[z_i]=g_i,$ $i=1,2,....$
Let $G_m=\{g_1, g_2,...,g_m\}.$
By the first part of \ref{L1228L1},  there exist $l(i)\ge 1,$ $k_1(i)\ge i$ and $m(n)''=d(i)+l(i,n)$ such that
\beq
[\{\la \phi_n(z_i)\ra\}_{n\ge k_1(i)} ]=0
\,\,\, {\rm in}\,\,\,
 K_1(\prod_{n\ge k_1(i)}M_{m(n)''}(B_n\otimes {\cal W})),
\eneq
by viewing $\phi_n$ as a map from $A$ into $M_{m(n)''}(B_n\otimes {\cal W}).$
Let $m(n)=\max\{m(n)', m(n)''\},$ $n=1,2,....$
Put $C=\prod_{n=1}^{\infty}M_{m(n)}(B_n\otimes {\cal W}).$
Then $(\pi\circ \{\phi_n\})_{*j}=0$  ($j=0,1$) as we view $\{\phi_n\}$ as {{a}} map from $A$ to $C,$
where $\pi: C\to C/C_0$ is the quotient map.
Fix  an integer $k\ge 2$ and a finite  subset  $F'\subset K_0(A, \Z/k\Z),$ we may assume
that the image of $F'$ is in $G_i=\{g_1,g_2,...,g_i\}.$ Then, by the following commutative diagram
\beq
\xymatrix{
K_0(A)\ar[r]\ar[d]_{[\{\phi_n\}_{n\ge k_1(i)}]}&
K_0(A,\Z/k\Z)\ar[r]\ar[d]_{[\{\phi_n\}_{n\ge k_1(i)}]} & K_1(A)\ar[d]^{[\{\phi_n\}_{n\ge k_1(i)}]}\\
K_0(C')\ar[r] & K_0(C',\Z/k\Z)
\ar[r]& K_1(C')),
}
\eneq
where $C'=\prod_{n\ge k_1(i) }  M_{m(n)}(B_n\otimes {\cal W}),$
since $[\{\phi_n\}_{n\ge k_1(i)}]|_{G_i}=0,$
$[\{\phi_n\}_{n\ge k_1(i)}]|_{F'}\subset K_0(C')/kK_0(C').$ However, by
\ref{1229C1}, $K_0(C')/kK_0(C')=0.$
It follows that  $[\pi{{(\{\phi_n\})}}]|_{F'}=0.$   It follows
that
\beq
[\pi(\{\phi_n\})]|_{K_0(A, \Z/k\Z)}=0,\,\,\, k=2,3,....
\eneq
Exactly the same argument shows that
\beq
[\pi(\{\phi_n\})]|_{K_1(A, \Z/k\Z)}=0,\,\,\, k=2,3,....
\eneq
This implies that
\beq
[\pi\circ (\{\phi_n\})]=0\,\,\,{\rm in}\,\,\, {\rm Hom}_{\Lambda}(\underline{K}(A), \underline{K}(C/C_0)),
\eneq
where
$C_0=\bigoplus_{n=1}^{\infty}M_{m(n)}(B_n\otimes {\cal W}).$

\end{proof}

{{ We would like to recall
Definition  \ref{Dfulln}
for definition of $T$-${\cal H}$-fullness (see also
5.5 of \cite{eglnp1}).}}

\begin{thm}\label{CLuniq}
Let $A$ be a non-unital separable amenable  \CA\, which satisfies the UCT
and let $T: A_+\setminus \{0\}\to \N\times \R_+\setminus \{0\}$
be a  map.
  For any $\ep>0$ and any finite subset ${\cal F}\subset A,$ there exists
$\dt>0,$ a finite subset ${\cal G}\subset A,$
a finite subset ${\cal H}\subset A_+\setminus \{0\}$
satisfying the following {{condition}}:
For any two ${\cal G}$-$\dt$-multiplicative \morp s $\phi, \psi: A\to B\otimes {{{\cal W}}},$  where
$B$ is any  $\sigma$-unital \CA,
and any
${\cal G}$-$\dt$-multiplicative \morp\, $\sigma: A\to M_l(B\otimes {\cal W})$ (for any integer $l\ge 1$) which is also
$T$-${\cal H}$-full, 
there {{exist  integers}} $N_1, N_2\ge 1$ and
a unitary $U\in M_{1+N_1l+N_2}(B\otimes {\cal W})^\sim $ such that
\beq\label{CLauct-2}
\|{\rm Ad}\, U\circ ((\phi\oplus S_{N_1})(a)\oplus 0_{N_2})-((\psi\oplus S_{N_1})(a)\oplus 0_{N_2})\|<\ep\tforal a\in {\cal F},
\eneq
where
\vspace{-0.1in} $$
S_K(f)={\rm diag}(\overbrace{\sigma(a), \sigma(a),...,\sigma(a)}^K)\tforal a\in A
$$
for integer $K\ge 1$ and where $0_{N_2}=\diag(\overbrace{0,0,...,0}^{N_2}).$
\end{thm}

\begin{proof}
This follows from the proof of  {{3.14 of \cite{eglnp1}.}}
Suppose that the conclusion of the theorem is false, then there exist  $\ep_0>0$ and a finite subset ${\cal F}\subseteq A$
such that there are a sequence of positive numbers $(\dt_n)$ with $\dt_n\searrow 0,$
an increasing sequence
$({\cal G}_n)$ of finite subsets of $A$ such that $\bigcup_n {\cal G}_n$ is dense in $A,$
an increasing sequence $({\cal H}_n)$ of finite subsets of
$A_+^{\boldsymbol{1}}\setminus \{0\}$ such that, if $a\in {\cal H}_n$
and $f_{1/2}(a)\not=0,$ then
$f_{1/2}(a)\in {\cal H}_{n+1}$, and $\bigcup_{{n}}{\cal H}_n$ is dense in
$A^{\boldsymbol  1}$, and has dense intersection with the unital ball of each closed two-sided ideal of $A$, 
two sequences of ${\cal G}_n$-$\dt_n$-multiplicative completely positive contractive maps $\phi_{n}, \psi_{n}: A\to B_n$
a sequence of unital ${\cal G}_n$-$\dt_n$-multiplicative completely positive contractive linear maps $\sigma_n: A\to
\mathrm{M}_{l(n)}(B_n)$ which are
$F$-${\cal H}_n$-full and satisfy, for each $n=1, 2,...$,
\beq\label{stableun2-2}
&&\hspace{-0.7in} \inf\{\sup\|v_n^*( \phi_{n}(a)\oplus S_{k_1(n)}(a)\oplus 0_{k_2(n)})v_n-(\psi_n(a)\oplus S_{k_1(n)}(a)\oplus 0_{k_2(n)})\|: a\in {\cal F}\}\ge \ep_0,
\eneq
where the infimum is taken among all
$k_1(n), k_2(n)\to\infty,$ and all  unitaries $v_n\in \mathrm{M}_{k_1(n)l(n)+1+k_2(n)}(B_n),$
and $S_{k_1(n)}: A \to \mathrm{M}_{k_1(n)l(n)}(B_n)$
 is as above.

Let $\{m(n)\}$ be as in \ref{LpartK}.
 Set $\mathrm{M}_{m(n)l(n)}(B_n)=B_n',$ $\bigoplus_{n=1}^{\infty}B_n'=C_0,$ $\prod_{n=1}^{\infty}B_n'=C,$
and $C/C_0=Q(C),$ and denote by
$\pi: C\to Q(C)$ the quotient map. Consider the maps $\Phi, \Psi, S: A\to C$ defined by
$\Phi(a)=(\phi_n(a))_{n\ge 1},$ $\Psi(a)=(\psi_n(a))_{n\ge 1}$, and $S(a)=({\bar \sigma}_n(a))_{n\ge 1}$ for all $a\in A,$
where
\beq
{\bar \sigma_n}(a)=\diag(\overbrace{\sigma_n(a), \sigma_n(a),...,\sigma_n(a)}^{m(n)})\rforal a\in A.
\eneq

Note that $\pi\circ \Phi,$
$\pi\circ \Psi$ and $\pi\circ S$ are \hm s.  
Consider also the truncations
$\Phi^{(m)}, \Psi^{(m)}, S^{(m)}: A\to \prod_{{n\ge m}} B_n'$ defined by
$\Phi^{(m)}(a)=(\phi_n(a))_{n\ge m},$ $\Psi^{(m)}(a)=(\psi_n(a))_{n\ge m},$
and $S^{(m)}(a)=({\bar \sigma}_n(a))_{n\ge m}.$

 It follows from \ref{LpartK}  that
\beq
[\pi\circ \Phi]=[\pi\circ \Psi]\,\,\,{\rm in}\,\,\, {\rm Hom}_{\Lambda}( \underline{K}(A), \underline{K}(C/C_0)).
\eneq


We will show  that ${\bar \sigma_n}$ is
$T$-${\cal H}_n$-full in $M_{m(n)l(n)}(B_n\otimes {\cal W}).$  {{Let $T(a)=(N(a), M(a))\in \N\times \R\setminus \{0\}$
for all $a\in A_+\setminus \{0\}.$}}
Fixed any  nonzero element $0\le a\le 1$ in  ${\cal H}_n.$
Let $b\in  M_{m(n)l(n)}(B_n\otimes {\cal W})_+$ with $\|b\|\le 1,$
and $\ep_1>0.$
{{Since $B_n$ is $\sigma$-unital, there exists $0\le e\le 1$ in $B_n\otimes {\cal W}$ such that}}
\beq\label{Aeabsorb-2}
\|b-b^{1/2}(1_{m(n)l(n)}\otimes {{e}})b^{1/2}\|<\ep_1/2.
\eneq
Choose $\ep_1/2>\eta>0$ such that
\beq\label{Aeabsorb-2+}
\|b-b^{1/2}(1_{m(n)l(n)}\otimes {{(e-\eta)_+}})b^{1/2}\|<\ep_1.
\eneq
{{Since $\sigma_n$ is
$T$-${\cal H}_n$-full, by also applying 3.1 of \cite{eglnp1},}}
there are $x_1, x_2,...,x_{N(a)}
\in M_{l(n)}(B_n\otimes {\cal W})$
with $\|x_i\|\le M(a),$ $1\le i\le N(a)$ such that
$(e-\eta)_+\otimes 1_{l(n)}=\sum_{i=1}^{N(a)}x_i^*\sigma_n(a)x_i.$
Therefore (identifying ${\bar \sigma}_n(a)$ with $1_{m(n)}\otimes {\sigma}_n(a)$)
$$
\|\sum_{i=1}^{N(a)}b^{1/2}(1_{m(n)}\otimes x_i)^*{\bar \sigma}_n(a)(1_{m(n)}\otimes x_i)b^{1/2}-b\|<\ep_1.
$$
This shows that ${\bar \sigma}_n$ is
$T$-${\cal H}_n$-full in
$M_{m(n)l(n)}((B\otimes {\cal W})).$
As in the proof of 3.14 of \cite{eglnp1}, this implies
$\pi\circ \{{\bar \sigma_n}\}$ is full  in $C/C_0.$

Then, by  the proof 3.14 of \cite{eglnp1}, there exists an integer $K\ge 1$ and
there exists a unitary $v\in M_{Km(n)l(n)+m(n)l(n)}(C/C_0)$
such that
$$
\|v^*\diag(\pi\circ \Phi(a), \Sigma_n(a))v-\diag(\pi\circ \Phi_n(a), \Sigma_n(a))\|<\ep_0/4\rforal a\in {\cal F},
$$
where
$$
\Sigma_n(a)=\diag(\overbrace{\pi\circ {\bar \sigma_n}(a),\pi\circ {\bar \sigma_n}(a),...,\pi\circ {\bar \sigma_n}(a)}^K).
$$
Lifting this to $C,$ one obtains, for all sufficiently large $n\ge 1,$ a
unitary\\ $u_n\in M_{Km(n)l(n)+m(n)l(n)}((B_n\otimes {\cal W})^\sim)$
such that
$$
\|u_n^*\diag(\phi_n(a)\oplus 0_{m(n)l(n)-1}, {\bar \sigma}_n(a))u_n-\diag(\psi_n(a)\oplus 0_{m(n)l(n)-1}, {\bar \sigma}_n(a))\|<\ep_0/2\rforal a\in {\cal F}.
$$
By replacing  $u_n$ by another unitary $w_n,$ if necessary, we have, for all sufficiently large $n\ge1,$
\beq
\|u_n^*\diag(\phi_n(a), {\bar \sigma}_n(a)\oplus 0_{m(n)l(n)-1})u_n-\diag(\psi_n(a), {\bar \sigma}_n(a)\oplus 0_{m(n)l(n)-1})\|<\ep_0/2,
\eneq
for all $ a\in {\cal F}.$

This contradicts \eqref{stableun2-2}. Lemma follows.

\end{proof}


\section{Models and range of invariant}

\begin{lem}\label{Dcc0} Let $A$ be an AF algebra and $\phi_1,\phi_2: A\to Q$ be two unital homomorphisms with
${{(\phi_1)_{*0}=(\phi_2)_{*0}}}$. Let $n$ be a positive integer. Define $B_i$ ($i=1,2$) to be the {{\SCA}}\,of $C([0,1], Q\otimes M_{n+1})\oplus A$ given by
\vspace{-0.1in}\beq
B_i=\{(f, a)\in C([0,1], Q\otimes M_{n+1})\oplus A: \begin{matrix} {\small  f(0)=\phi_i(a)\otimes \diag(\overbrace{1,\cdots,1}^n,0)}\\
                             {\small \hspace{-0.1in}f(1){{=}}\phi_i(a)\otimes \diag(\underbrace{1,\cdots,1, 1}_{n+1})}\end{matrix}\}
                            \eneq
for $i=1,2$. Then $B_1\cong B_2$.

\end{lem}

\begin{proof}

Since both $A$ and $Q$ are AF algebras and ${{(\phi_1)_{*0}=(\phi_2)_{*0}}}$, there is a unitary path $\{u(t)\}_{0\leq t <1}$
such that $\phi_2 (a)=\lim\limits_{t\to 1} u(t) \phi_1(a) u(t)^*$ ({{see}} \cite{Lnamj}). Define the
{ isomorphism} $\psi: B_1\to B_2$ by sending
$(f,a)\in B_1$ to $(g,a)\in B_2$, where $g$ is given by
$$
g(t)=\left\{\begin{array}{lll}
             (u(|2t-1|)\otimes \one_{n+1}) f(t) (u(|2t-1|)\otimes \one_{n+1})^* &  & ~~\mbox{if}~~t\in (0,1),\\
               \phi_2(a)\otimes \diag(\overbrace{1,\cdots,1}^n,0)&   & ~~\mbox{if} ~~t=0, \\
         \phi_2(a)\otimes \diag(\underbrace{1,\cdots 1}_{n},1)&   & ~~\mbox{if} ~~t=1.
                    \end{array}
           \right.
   $$
It is straight forward to verify that $g$ is continuous, that $(g,a)\in B_2$, and that $\psi$ defines a desired isomorphism.

\end{proof}


\begin{df}\label{Dcc1}

Let $G_0$ and $G_1$ be two countable abelian groups.
Let $A$ be a unital  AH-algebra with $TR(A)=0,$ unique tracial state,
$K_1(A)=G_1$ and $K_0(A)=\Q\oplus G_0$ with ${\rm ker}\rho_A=G_0$
{{and $[1_A]=(1,0).$}}

There is a unital  \hm\,  $s: A\to Q$ such that $s_{*0}(r,g)=r$
for $(r,g)\in \Q\oplus G_0.$
Fix a  unital embedding $j: Q\to A$ with $j_{*0}(r)=(r,0)$ for $r\in \Q.$
(Note that both $j\circ s$ and $s\circ j$ induce the identity maps on ${{T(A)}}$ and ${{T(Q)}}$ respectively. Furthermore the {{homomorphism}} $j$ and $s$ identify the spaces $T(A)$ and ${{T(Q)}}$)

Fix an integer $a_1\ge 1.$ Let $\af={a_1\over{a_1+1}}.$
For each $r\in \Q_+\setminus \{0\},$ let $e_r\in Q$ be a projection with ${\rm tr}(e_r)=r.$
Let ${\bar Q}_r:=(1\otimes e_r)(Q\otimes Q)(1\otimes e_r).$
Define $q_r: Q\to {\bar Q}_r$ by $a\mapsto a\otimes e_r$ for $a\in Q.$ {{We will also use $q_r$ to denote
{{a}} homomorphism from $B$ to $B\otimes e_rQe_r$ (or to $B\otimes Q$) defined by sending $b\in B$ to $b\otimes e_r\in B\otimes e_rQe_r \subset B\otimes Q$.}}







We fix an isomorphism $Q\otimes Q\to Q$ which will be denoted by
$\iota^Q.$ Moreover the composition of the maps which  first maps $a$ to $a\otimes 1_Q$ and then
to $Q$ via $\iota^Q$ is approximately inner.   In fact every unital endomorphism on $Q$ is
approximately inner.
If we identify $Q$ with $Q\otimes 1_Q$ in $Q\otimes Q$ then $\iota^Q$ is an approximately inner
endomorphism.

For each $1>r>r'>0,$  we assume that $e_r\ge e_{r'}.$
Fix $1>r>0,$ define $\iota_r^Q: {\bar Q}_r\to Q_r:=e_rQe_r$ by $\iota_r^Q={\rm Ad}\, v_r\circ \iota^Q|_{{\bar Q}_r},$
where $v_r^*(\iota^Q(1\otimes e_r))v_r=e_r.$

Let
$$
R(\af, r)=\{(f,a)\in C([0,1], Q\otimes Q_r)\oplus Q: f(0)=a\otimes e_{r\af}\andeqn f(1)=a\otimes e_r\}.
$$
(Recall that $R({{\af}},1)$ has been defined in \ref{DRaf1}.)

Let
$$
A(W,\af)=\{(f, a)\in C([0,1], Q\otimes Q)\oplus A: f(0)=q_\af\circ s(a)\andeqn f(1)=s(a)\otimes 1_Q\}.
$$
We also note that $(f, a)$ is full in $A(W, \af)$ if and only if $a\not=0$ and $f(t)\not=0$
for all $t\in (0,1).$

Let ${\mathcal M}_+$ denote the set of nonnegative regular measure{{s}} on $(0,1)$.
As in {{\ref{2Rg15}, trace spaces}} $  \tilde{T}{{(A(W,\af))}}$ and $  \tilde{T}{{(R(\af,1))}}$ are isomorphic, and each $\tau \in \tilde{T}{{(R(\af,1))}}\cong \tilde{T}{{(A(W,\af))}}$ corresponds to $(\mu, s)\in {{\mathcal M}_+}(0,1)\times \R_+$.
{{ Furthermore  we have
$$\|\tau\|={ {\|\mu\|+s=}}\int_0^1d \mu~+~   s.$$}}
 Note that in the weak topology of $  \tilde{T}{{(A(W,\af))}}$ (or $  \tilde{T}{{(R(\af,1))}}$), under {{the}} above identification, one has that $$\lim\limits_{t\to 0}(\dt_t,0)=(0,\af)\in {{\mathcal M}_+}(0,1)\times \R_+~~~~~\mbox{ and }~~~~\lim\limits_{t\to 1}(\dt_t,0)=(0,1)\in {{\mathcal M}_+}(0,1)\times \R_+,$$ where $\dt_t$ is the unit measure of the point mass at $t$.

The affine space $\mathrm{Aff}(\mathrm{\tilde{T}}(A(W,\af)))$ and $\mathrm{Aff}(\mathrm{\tilde{T}}(R(\af,1)))$ can be identified with
\beq
{{\{(f,x)\in C([0,1],\R)\oplus \R: f(0)=\af\cdot x\andeqn
f(1)=x\},}}
\eneq
{{ a   subspace of $C([0,1],\R)\oplus \R$.}}

Let
$$
A(W,\af,r)=\{(f, a)\in C([0,1], Q\otimes Q_r)\oplus A: f(0)=q_{r\af}\circ s(a)\andeqn f(1)=q_r\circ s(a)\}.
$$

Define $\phi_{A,R,\af}: A(W,\af)\to R(\af,1)$ by
$$
\phi_{A, R,\af}((f,a))=(f, s(a))\rforal (f,a)\in A(W, \af).
$$

Define ${\tilde{sj}}: C([0,1], Q\otimes Q)\to C([0,1],Q\otimes Q)$
by
$$
{\tilde{ sj}}(f)(t)=((s\circ j)\otimes {\rm id}_Q)(f(t)).
$$

Define $\phi_{R, A, \af}: R(\af, 1)\to A(W,\af,1)$ by
$$
\phi_{R, A, \af}((f,a))=({\tilde{sj}}(f), j(a)) \rforal (f,a)\in R(\af, 1).
$$
Note that
\beq
{\tilde{sj}}(f)(0)=((s\circ j)\otimes {\rm id}_Q)(a\otimes e_\af)=s\circ j(a)\otimes e_\af\andeqn\\
{\tilde{sj}}(f)(1)=((s\circ j)\otimes {\rm id}_Q)(a\otimes 1)=s\circ j(a)\otimes 1.
\eneq
Also
$$
q_\af\circ s\circ j(a)=s\circ j(a)\otimes e_\af.
$$
In particular, $\phi_{R, A, \af}$ does map $R(\af, 1)$ into $A(W, \af,1).$
Moreover $\phi_{R, A, \af}$ is injective {{and map the strictly positive element
$a_\af$ to a strictly positive element (with the same form{{--see \ref{DRaf1}}}).}}

With the identification of both $\mathrm{Aff}(\mathrm{\tilde{T}}(A(W,\af)))$ and $\mathrm{Aff}(\mathrm{\tilde{T}}(R(\af,1)))$  with the same subspace  of
$ C([0,1],\R)\oplus \R$, the homomorphism $\phi_{A,R,\af}$ and $\phi_{R,A,\af}$ induce the identity map on that subspace at the level of ${{\mathrm{Aff}(\tilde{T}(-))}}$ maps. They also induce the identity maps at level of trace {{spaces}}, when we identify the corresponding trace spaces.  In particular,  $\phi_{A,R,\af}^*: \mathrm{\tilde{T}}(R(\af,1)) \to \mathrm{\tilde{T}}(A(W,\af))$ (or  $\phi_{R,A,\af}^*:   \mathrm{\tilde{T}}(A(W,\af)) \to \mathrm{\tilde{T}}(R(\af,1))$, respectively) takes the subset
$\mathrm{{T}}(R(\af,1))$ to the subset $\mathrm{{T}}(A(W,\af))$ (or takes $\mathrm{{T}}(A(W,\af))$ to $\mathrm{{T}}(R(\af,1))$, respectively)

Fix $\af, r\in \Q_+\setminus \{0\}.$
There are  unitaries  $u_{\af, r},\, u_{1,r}\in {\bar Q}_r$ such that
$$
u_{\af,r}^*(e_\af\otimes e_r)u_{\af,r}=(\iota_r^Q)^{-1}(e_{r\af})\andeqn u^*_{1,r}(1\otimes e_r)u_{1,r}=
(\iota_r^Q)^{-1}(e_r){{=1\otimes e_r}}.
$$
{{(Note that $u_{1,r}$ can be chosen to  be $1_{\bar Q_r}$.) }}
There is a continuous path of unitaries $\{u(t): t\in [0,1]\}$ in ${\bar Q}_r$
such that $u(0)=u_{\af,r}$ and $u(1)=u_{1,r}.$

Let $v(t)=1\otimes u(t)\in Q\otimes  {\bar Q}_r$ for $t\in [0,1].$
Note if $f(t)\in Q\otimes Q,$
then
$$
v(t)^*(f(t)\otimes e_r)v(t)\in Q\otimes {\bar Q}_r\rforal t\in (0,1).
$$


Let $\phi_{R,r}: R(\af, 1)\to R(\af, r)$ {{be defined}} by
$$
\phi_{R, r}((f,a))=({\rm id}_Q\otimes \iota_r^Q)\circ {\rm Ad}\, v\circ q_r(f), a).
$$
Note that, for $t\in (0,1),$
\beq
({\rm id}_Q\otimes \iota_r^Q)\circ {\rm Ad}\, v(t)\circ q_r(f)(t)&=&({\rm id}_Q\otimes \iota_r^Q)\circ {\rm Ad}\, v(t)(f(t)\otimes e_r)\\
&=&({\rm id}_Q\otimes \iota_r^Q)(v(t)^*f(t)\otimes e_r)v(t))\in Q\otimes Q_r,
\eneq
\beq
({\rm id}_Q\otimes \iota_r^Q)\circ {\rm Ad}\, v(0)\circ q_r(f)(0)&=&({\rm id}_Q\otimes \iota_r^Q)\circ {\rm Ad}\, v(0)(a\otimes e_\af\otimes e_r)\\
&=&({\rm id}_Q\otimes \iota_r^Q)(a\otimes (\iota_r^{Q})^{-1}(e_{\af r}))\\
&=& a\otimes e_{\af r}\andeqn
\eneq
\beq
({\rm id}_Q\otimes \iota_r^Q)\circ {\rm Ad}\, v(1)\circ q_r(f)(1)&=&({\rm id}_Q\otimes \iota_r^Q)\circ {\rm Ad}\, v(1)(a\otimes
1\otimes e_r)\\
&=&({\rm id}_Q\otimes \iota_r^Q)(a\otimes (\iota_r^{Q})^{-1}(e_{r}))\\
&=& a\otimes e_{r}.
\eneq
 Evidently, when we identify $\mathrm{\tilde{T}}(R(\af,r))$ and $\mathrm{\tilde{T}}(R(\af,1))$ with ${\mathcal M}_+(0,1)\times \R_+$, the map $\phi_{R,r}^*$ is the identity map and takes the subset $\mathrm{{T}}(R(\af,r))$ to the subset $\mathrm{{T}}(R(\af,1))$.




Define $s^{(2,3)}: Q\otimes Q\otimes Q\to Q\otimes Q\otimes Q$
by
$$
s^{(2,3)}(x\otimes y\otimes z)=
{{x\otimes z\otimes y}}$$
for all $x, y, z\in Q.$ {{To make the notation {{clearer}}, we will often write the above $x\otimes z\otimes y$ as $(x\otimes z)\otimes y$, later.}}
Define  a \hm\, ${\widetilde{\iota^{Q}}}: R(\af, 1)\otimes Q\to R(\af,1)$
 by
$$
{{{\widetilde{\iota^{Q}}}}}(f\otimes b, a\otimes b)=((\iota^{Q})\otimes {\rm id}_Q)\circ s^{(2,3)}(f\otimes b), \iota^{Q}( a\otimes b))
$$
for $(f,a)\in R(\af, 1)$ and $b\in Q.$

Note, at $t=0,$
\beq
(\iota^{{Q}}\otimes {\rm id}_Q)\circ s^{(2,3)}(f\otimes b)(0))&=&(\iota^{{Q}}\otimes {\rm id}_Q)\circ s^{(2,3)}(a\otimes e_\af\otimes b)\\
&=&(\iota^{Q}\otimes {\rm id}_Q)((a\otimes b)\otimes e_\af)\\
&=& \iota^{Q}(a\otimes b)\otimes e_\af;
\eneq
 and, at $t=1,$
 \beq
(\iota^{Q}\otimes {\rm id}_Q)\circ s^{(2,3)}(f\otimes b)(1))&=&(\iota^{Q}\otimes {\rm id}_Q)\circ s^{(2,3)}(a\otimes 1\otimes b)\\
&=&(\iota^{Q}\otimes {\rm id}_Q)((a\otimes b)\otimes 1)\\
&=& \iota^{Q}(a\otimes b)\otimes 1.
\eneq

Let $m\ge 2$ be an integer.
Viewing $M_m$ as a unital \SCA\, of $Q,$
 Put $\iota^{M_m}=\iota^Q|_{Q\otimes M_m}.$
Define  ${\widetilde{\iota^{M_m}}}: R(\af,1)\otimes M_m\to R(\af,1)$
by ${\widetilde{\iota^{M_m}}}={\widetilde{\iota^{Q}}}|_{R(\af,1)\otimes M_m}.$
Note also that (recall \eqref{Dconsteaf})
\beq
{\widetilde{\iota^{Q}}}(a_\af\otimes 1_Q)=a_\af \andeqn
{\widetilde{\iota^{M_m}}}(a_\af\otimes 1_{M_m})=a_\af.
\eneq
{{We need one}}  more map.
Let $\psi_{A_w}: A(W,\af)\to C([0,1], Q)\oplus A$ {{be defined}}
by
$$
\psi_{A_w}(f,a)=(g,a){{,}}
$$
where $g(t)=s(a)$ for all $t\in [0,1].$
{{Define}} $\psi_{A_w,r}: A(W,\af)\to C([0,1], Q\otimes Q_r)\oplus A$ by
$$
\psi_{A_w,  r}{{((f,a))}}=(q_r(g),a)
$$
with  $g(t)=s(a)$ (and $q_r(g)=q_r\circ s(a)$).
{{Note that $\psi_{A_w,r}(a_\af)=(1\otimes e_r,1)$ is the unit of $C([0,1], Q\otimes Q_r)\oplus A.$
It follows that $\psi_{A_w,r}$ maps strictly positive elements to strictly positive elements.}}

 {{ When we identify $\mathrm{\tilde{T}}(A(W,\af))$ with ${{\mathcal M}_+}(0,1)\times \R_+$, and $\mathrm{\tilde{T}}(C([0,1], Q\otimes Q_{{r}})\oplus A)$ with ${{\mathcal M}_+}[0,1]\times \R_+$, the map $\psi_{A_w,r}^*$ is given by
$$\psi_{A_w,r}^* (\mu, s)= (0, s+\int_0^1 d\mu),$$
which takes $\mathrm{{T}}(C([0,1], Q\otimes Q_{{r}})$ to $\mathrm{{T}}(A(W,\af))$.}}

 Warning: $C([0,1], Q\otimes Q_r)\oplus A\not=A(W, \af).$


One more notation:  {{define}} $P_f: (f,a)\to f$ and $P_a: (f,a)=a.$

Now let $\af<\bt <1.$
Let us choose $x$ such that
$
\bt(1/2+x)= (\af/2+x).
$
So
$$
x={(1/2)(\bt-\af)\over{1-\bt}}>0.
$$

Let
$$
y=1/2+x={1\over{2}}+{(1/2)(\bt-\af)\over{(1-\bt)}}={(1-\af)\over{2(1-\bt)}}.
$$

Let $r_1=(1/2)(1/y)={(1-\bt)\over{(1-\af)}}$ and  $r_2=x(1/y)={(\bt-\af)\over{(1-\af)}}.$
Then
$$
\af r_1+r_2=(1/y)(1/2+x)=\bt\andeqn r_1+r_2=(1/y)(1/2+x)=1.
$$


%
 Define
 $\Phi_{A_w, \af, \bt}: A(W, \af)\to A(W,\bt)$ by
 \beq\nonumber
&&P_a( \Phi_{A_w, \af, \bt}((f,a)))=a\andeqn\\\nonumber
&&P_f( \Phi_{A_w, \af, \bt}((f,a)))=\diag( P_f\circ \phi_{R,r_{{1}}}\circ \phi_{A,R, \af}((f,a)), P_f\circ  \psi_{A_w, r_2}((f,a))).
\eneq

One computes that, for $t\in (0,1),$
\beq
P_f(\phi_{R,r_1}\circ \phi_{A,R, \af}((f,a))(t)&=&({\rm id}_Q\otimes \iota_{r_1}^Q)\circ {\rm Ad}\, v(t)\circ q_{r_1}(f)(t)\\
&=&({\rm id}_Q\otimes \iota_{r_1}^Q)(v(t)^*f(t)\otimes e_{r_1})v(t))\\
&&\in Q\otimes Q_{r_{{1}}}\subset Q\otimes Q\andeqn\\
P_f(\psi_{A_w, r_2}((f,a)))(t)&=&q_{r_2}(s(a))=s(a)\otimes e_{r_2}\in Q\otimes Q.
\eneq
At $t=0,$
\beq
P_f(\phi_{R,r_1}\circ \phi_{A,R, \af}((f,a))(0)&=&({\rm id}_Q\otimes \iota_{r_1}^Q)\circ {\rm Ad}\, v(0)\circ q_{r_1}(f)(0)\\
&=&({\rm id}_Q\otimes \iota_{r_1}^Q)\circ {\rm Ad}\, v(0)(s(a)\otimes e_\af\otimes e_{r_1})\\
&=&({\rm id}_Q\otimes \iota_{r_1}^Q)(a\otimes (\iota_{r_1}^{Q})^{-1}(e_{\af r_1}))\\
&=& s(a)\otimes e_{\af r_1}.
\eneq
Hence
\beq
P_f( \Phi_{A_w, \af, \bt}((f,a)))(0)&=&\diag(s(a)\otimes e_{\af r_1}, s(a)\otimes e_{r_2})\\
&=&s(a)\otimes e_{\af r_1+r_2}=s(a)\otimes e_\bt.
\eneq
At $t=1,$
\beq
P_f(\phi_{R,R,r_2}\circ \phi_{A,R, \af}((f,a))(1)&=&({\rm id}_Q\otimes \iota_{r_1}^Q)\circ {\rm Ad}\, v(1)\circ q_{r_1}(f)(1)\\
&=&({\rm id}_Q\otimes \iota_{r_1}^Q)\circ {\rm Ad}\, v(1)(s(a)\otimes
1\otimes e_{r_1})\\
&=&({\rm id}_Q\otimes \iota_{r_1}^Q)(s(a)\otimes (\iota_{r_1}^{Q})^{-1}(e_{r_1}))\\
&=& s(a)\otimes e_{r_1}.
\eneq
Hence
\beq
P_f( \Phi_{A_w, \af, \bt}((f,a)))(1)&=&\diag(s(a)\otimes e_{r_1}, s(a)\otimes e_{r_2})\\
&=&s(a)\otimes e_{ r_1+r_2}=s(a)\otimes 1.
\eneq

Therefore, indeed, $\Phi_{A_w, \af, \bt}$ defines a \hm\, from $A(W, \af)$ to  $A(W,\bt).$
It is injective.  {{We also check that
$\Phi_{A_w,\af, \bt}(a_\af)$ is a strictly positive element of $A(W, \bt)$ (recall \eqref{Dconsteaf}).}}

 {{Furthermore $\Phi_{A_w, \af, \bt}^*: \mathrm{\tilde{T}}(A(W,\bt)) (\cong {{\mathcal M}_+}(0,1)\times \R_+)\to \mathrm{\tilde{T}}(A(W,\af)) (\cong {{\mathcal M}_+}(0,1)\times \R_+)$ is given by
$$\Phi_{A_w, \af, \bt}^*(\mu, s)=(r_1\mu, r_2(\int_0^1d \mu)+s),$$
which takes $\mathrm{{T}}(A(W,\bt))$  to $\mathrm{{T}}(A(W,\af))$.}}

Fix any $a\in A_+$ with $\|a\|=1.$
Define $f(t)=(1-t)(s(a)\otimes e_\af)+t(s(a)\oplus 1).$
Then $(f,a)\in A(\af,1)$ is a {{full}} positive element.  Note that $\Phi_{A_w, \af, \bt}((f,a))$
is also a {{full}} positive element.


 Let $m, m'$ be two positive integers such that ${{m|m'}}$. Let $\frac{m'}m=a+1$. Let  $F_2=M_{m'}(\C)$,  $F_1=M_m(\C)$, and $\phi_0,\phi_1:F_1\to F_2$ be defined by
$$\phi_0(x)=\diag(\underbrace{x,\cdots, x}_{a},0), ~~~\mbox{and}~~~\phi_0(x)=\diag(\underbrace{x,\cdots, x}_{a+1}).$$
Denote that $$A(m, m')= A(F_1, F_2, \phi_0, \phi_1)=\{(f,x)\in C([0,1], M_m(\C)\otimes M_{a+1}(\C))\oplus M_m(\C): $$
$$~~f(0)=x\otimes \diag(\underbrace{1,\cdots 1}_a, 0)~~f(1)=x\otimes \diag(\underbrace{1,\cdots 1}_{a+1})\}.$$ Then $A(m, m')\in {\cal C}_0^0$  {{with $\lambda_s(A(m,m'))= \frac{a}{a+1}.$}}

  In {{ \cite{Jb},}}
the author constructed a simple  inductive limit
${\cal W}=\lim W''_i=\lim (A(m_i, (a_i+1)m_i), \omega_{i,j})$ such that $K_0({\cal W})=0=K_1({\cal W})$ and $T({\cal W})=\{pt\}$,
{{In the construction, one has $a_i+1=2(a_{i-1}+1)$ and $m_i=a_im_{i-1}$. Consequently $\lim_{i\to \infty}a_i=\infty$. }}{{ From the construction in \cite{Jb}, the map $\omega_{i,j}$ takes strictly positive elements to strictly positive elements, and $\omega_{i,j}^*$ maps tracial state space $T{{(W''_j)}}$ to tracial state space $T{{(W''_i)}}$.}} {{ Furthermore, $A_i\in  {\cal C}_0^0$ with $\lambda_s(A_i)=\frac{a_i}{a_{i+1}}\to 1$ as $i \to \infty$.}}

{{ Note that  ${\cal W}\otimes Q\cong {\cal W}$. Identify $Q\otimes M_m$  and $Q\otimes M_{a+1}$ with $Q$,  we can identify }}
$A(m, ({{a}}{{+1}})m){{\otimes Q}}$ {{with $R(\af, 1)$ for $\af=a/(a+1)$.  {{Moreover,}} ${\cal W}=\lim(W'_n=R(\af_n, 1), \imath'_{W,n})$, where $\imath'_{W,n}: R(\af_n,1)\to R(\af_{n+1}, 1)$ are injective.}} {{Again, we have that $(\imath'_{W,n})^*$ takes $T{{(R(\af_{n+1}, 1))}}$ to $T{{(R(\af_n,1))}}$.}}


Let $C$ be a unital AF-algebra so that $T({{C}})=T.$
We write $C=\lim_{n\to \infty} (F_n, \imath_{F,n}),$ where
${\rm dim}(F_n)<\infty$ and $\imath_{F,n}: F_n\to F_{n+1}$
are unital injective \hm s.

Let ${\cal W}$ be as before.
Write
$$
W_T={\cal W}\otimes C.
$$
{{Then}} $T(W_T)=T$ and $W_T$ has continuous scale.

Suppose that
$$
F_n=\bigoplus_{i=0}^{k(n)} M_{n_i}.
$$
By identifying $R(\af_n,1)$ with $R(\af_n,1)\otimes M_{n_i}$  and $R(\af_n,1)\otimes Q$,  we may write that
$$
W_T=\lim_{n\to\infty} (W_n, \imath_n),
$$
where $W_n$ is a  direct sum of  $k(n)$  {{ summand}} of $R(\af_n,1):$
$
W_n=\bigoplus_{i=0}^{k(n)}R(\af_n{{,}} 1)^{{(i)}},
$
where $\af_1<\af_2<\cdots <1.$ {{Again, we have that $\imath_n^*$ takes $T{{(W_{n+1})}}$ to $T{{(W_n)}}$.
}}

We write
$$
W_n=R_{0,n}\bigoplus  D_n{{,}}
$$
where $R_{0,n}=R(\af_n,1)^{(0)}$ and
$$
D_n=\bigoplus_{i=1}^{k(n)}R(\af_n{{,}} 1)^{(i)}.
$$

In the case that $W_n$ has only one summand, we understand that
$W_n=R_{0,n}$ and $D_n=\{0\}.$
We also use
$$
P_{0,n}: W_n\to R_{0,n}\andeqn P_{1,n}: W_n\to D_n
$$
for the {{projection map}}, i.e., $P_{0,n}(a\oplus b)=a$ and
$P_{1,n}(a\oplus b)=b$ for all $a\in R_{0,n}$ and $b\in D_n.$

Consider
$$
B_n=W_n\oplus   M_{(n!)^2}(A(W, \af_n)),\,\,\,n=1,2,
$$

Let $r_n=\frac{1}{2^{n+1}k(n)},$ $n=1,2,....$

Let us define a \hm\, $\Psi_{n,n+1}: B_n\to B_{n+1}$ as follows.

On $M_{(n!)^2}(A(W, \af_n))$
define $\Psi_{n,n+1,A,A}: M_{(n!)^2}(A(W, \af_n))\to M_{((n+1)!)^2}(A(W, \af_{{n+1}}))$ by
$$
\Psi_{n,n+1,A,A}(a)=\diag(\Phi_{A_w, \af_n, \af_{n+1}}(a),\overbrace{0,0,...,0}^{((n+1)!)^2-(n!)^2})\rforal a\in M_{(n!)^2}(A(W, \af_n))
$$
and define
$
\Psi_{n,n+1, A, W}: M_{(n!)^2}(A(W, \af_n))\to R_{0,n+1}\otimes e_{r_n}{{Qe_{r_n}}}$
by
$$
\Psi_{n,n+1, A, W}= q_{r_n}\circ  \imath_{W,n}'\circ {\widetilde{\iota^{M_{(n!)^2}}}}\circ (\phi_{A,R,\af_n}\otimes {\rm id}_{M_{(n!)^2}}).
$$
(Recall that ${\widetilde{\iota^{Q}}}: R(\af, 1)\otimes Q\to R(\af,1)$ is an isomorphism and   ${\widetilde{\iota^{M_m}}}: R(\af,1)\otimes M_m\to R(\af,1)$  is defined
by ${\widetilde{\iota^{M_m}}}={\widetilde{\iota^{Q}}}|_{R(\af,1)\otimes M_m}.$)
It is injective.

On $W_n$ define
$\Psi_{n,n+1,W,W}: W_n\to  {{R_{0,{n+1}} \otimes (1-e_{r_n})Q(1-e_{r_n}) \oplus D_{n+1}
\subset W_{n+1}}}$
by,
 for $a\in R_{0,n},$ $b\in D_n,$
\beq\nonumber
&&\hspace{-0.3in}\Psi_{n,n+1,W,W}((a\oplus b)){{=\Psi^0_{n,n+1,W,W}((a\oplus b))\oplus \Psi^1_{n,n+1,W,W}((a\oplus b))}}=\\\nonumber
&&\hspace{-0.1in}
q_{1-r_n}((P_{0,{n+1}}\circ \imath_{n,n+1}(a))
\oplus  (P_{0,{n+1}}\circ \imath_{n,n+1}(b)))\\
&&\hspace{0.4in}\oplus (P_{1, n+1}\circ \imath_{n, n+1}(a)
\oplus P_{1,n+1}\circ  \imath_{n,n+1}(b).
\eneq

Suppose that $a, b\ge 0.$  Then, for any $t\in T(W_{n+1}),$
\beq\label{eomega1}
t(\Psi_{n,n+1,W,W}(a\oplus b))\ge  (1-r_n)t(\imath_{n, n+1}(a\oplus b)).
\eneq

Define
$\Psi_{n,n+1, W,A}:  R_{0,n}\to  M_{((n+1)!)^2}(A(W, \af_{n+1}))$ by
$$
\Psi_{n,{{n+1}}, W,A}(a)=\diag(0, \overbrace{ (\phi_{R, A,\af_{n+1}}\circ \imath_{W,n}')(a), ...,
(\phi_{R, A, \af_{n+1}}\circ \imath_{W,n}')(a)}^{((n+1)!)^2-(n!)^2}).
$$

Now if $(a\oplus b)\oplus c\in   W_n \oplus A(W, \af_n)$
(with $a\in M_{(n!)^2}(R_{0,n}), $ $b\in D_n,$ and $c\in A(W,\af_n),$
define
$$
\Phi_{n,n+1}((a\oplus b)\oplus c)=d \oplus c',
$$
where
$$
d={{{\widetilde{\iota^{Q}}}\big(}}\Psi_{n,n+1, A, W}(c)\oplus \Psi^{{0}}_{n,{n+1},W,W}(a\oplus b){{\big)}}{{\oplus \Psi^1_{n,{n+1},W,W}(a\oplus b)}}\in W_{n+1}$$
({{$\Psi_{n,n+1, A, W}(c)\in R_{0,n+1}\otimes(e_{r_n}Qe_{r_n}), \Psi^0_{n,n+1,W,W}(a\oplus b)\in R_{0,n+1}\otimes(e_{1-r_n}Qe_{1-r_n})$, and  }}\\
$\Psi^{{1}}_{n,n+1,W,W}(a\oplus b){{\in}} D_{n+1}$) and
$$
c'=\diag(\Psi_{n,n+1, A,A}(c),\Psi_{n,n+1, W,A}(a))\in M_{((n+1)!)^2}(A(W, \af_{n+1}).
$$

Since all partial map{{s}} of $\Phi_{n,n+1}$ take the strictly positive elements to the strictly positive elements in corresponding {{corners}},  $\Phi_{n,n+1}$ itself takes strictly positive elements to strictly positive elements.
{{This also implies that}} $\Phi_{n,n+1}^*(T{{(B_{n+1})}})\subset  {{T(B_n)}}$.
Note {{also}} that $\Phi_{n, n+1}$ maps full elements to full elements and it is injective.

Define
$$
B_T=\lim_{n\to\infty} (B_n, \Phi_{n,n+1}).
$$

\end{df}

\begin{rem}\label{Zstable}
In the construction above, $C^*$ algebras $A$ and $Q$ are ${\cal Z}$-stable, one can also choose the homomorphism $s: A\to Q$ and $j: Q\to A$ to be of the form $s'\otimes id_{\cal Z}: A\otimes {\cal Z}\to Q\otimes {\cal Z}$ and $j'\otimes id_{\cal Z}: Q\otimes {\cal Z}\to A\otimes {\cal Z}$ respectively, when one identifies $A\cong A\otimes {\cal Z}$ and $Q\cong Q\otimes {\cal Z}$. Then $R(\af,1)$, $A(W,\af_n)$, $W_n$, $B_n$ are all ${\cal Z}$-stable. One can also make the map $\Phi_{n,n+1}: B_n\otimes {\cal Z}\to B_{n+1}\otimes {\cal Z}$ to be of form of $\Phi'\otimes {\rm id}_{\cal Z}$. In such a way, we will have that $B_T$  is {{${\cal Z}$}}-stable.

 By section 4 of \cite{EG}, one can  write
 $A=\lim_{n\to\infty}(A_n, \phi_n),$ where each $A_n=M_{k(n)}(C(X_n)),$ where each $X_n$ is a finite
 CW complex with dimension no more than 3.  Let $s:A\to Q$ be at the beginning of \ref{Dcc1}.
Then, by the proof of 4.7 (and using 2.29) of \cite{EG},  there exists a sequence of $M_{l(n)}\subset Q$ and \hm s
 $s_n: A_n\to M_{l(n)}$ such that, for each  fixed $m,$
 \beq\label{63-1}
 \lim_{n\to\infty} s\circ \phi_{m, \infty}(a)=\lim_{n\to\infty} s_n\circ\phi_{m,n}(a)\rforal a\in A_m.
 \eneq
 This also follows from the following.
 Note $s_{*i}(G_i)=0,$ $i=0,1.$ Since $K_1(Q)=\{0\}$ and $K_0(Q)=\Q$ which is divisible,  by Theorem 3.9
 of \cite{HLX},  for each fixed $m,$ there exists a sequence of \hm s $\psi_k: A_m\to Q$
 such that $\psi_k(A_m)$ has finite dimension and
 $\lim_{k\to\infty}\psi_k(a)=s\circ \phi_{m, \infty}(a)$ for all $a\in A_m.$ Since finite dimensional
 \CA s are semiprojective, one also obtains \eqref{63-1}.
Then for any finite set  ${{{\cal F}}}\subset A(W, \af)$ and any $\ep>0,$ there is  a  {{\CA}}\, of the form
$${{D_n}}'=\big\{(f,a)\in C([0,1], M_{l(n)}\otimes M_{l(n)})\oplus A_n:~~f(0)=s_n(a)\otimes \diag(\underbrace{1,\cdots,1}_{\af l(n)},0),~~~~~~~~~~~$$
$$~~~~~~~~~~~~~~~~~~~~~~~~~~~~~~~~~~~~~~f(1){{=}}s_n(a)\otimes \diag(\underbrace{1,\cdots,1}_{l(n)})\big\}$$
such that ${{\cal F}}\subset_{\ep} D_{{n}}'$, where $\af l(n)$ is an integer.
Put $D_n=D_n'\oplus W_n.$ Then
that $D_{{n}}$ is a sub-homogeneous \CA s with 3-dimensional spectrum.
Moreover, $D_n\in \overline{{\cal D}_2}$ defined in 4.8 of \cite{GLN}.

{{H}}ence $B_T$ has {{the}} decomposition rank at most three.
(In fact, one can prove that $B_T$ is an inductive limit sub-homogeneous \CA s with spectrum having dimension no more than $3$, but we do not need this fact.) \end{rem}

\begin{lem}\label{LBTsimplR}
Suppose that $a\in (W_n)_+.$
Then, for any integer $k\ge 1$ and any $t\in T(W_{n+{{k}}}),$
\beq
t(\Psi_{n,n+k,W,W}(a))\ge  (1-\sum_{j=0}^{k-1} r_{n+j})t(\imath_{n, n+k}(a)).
\eneq
\end{lem}

\begin{proof}
Note
${{\tau}}\circ {{\Phi}}_{n+1, n+2}$ is in $T(W_{n+1})$ for all $\tau\in T(W_{n+2}).$
Thus this lemma follows from \eqref{eomega1} and induction immediately.
\end{proof}


\begin{lem}\label{LfullW}
Let $n\ge 1$ be an integer.
There is a strictly positive element $e_0'\in W_n$ with $\|e_0'\|=1$ such
that $\imath_{n, \infty}(e_0')$ is a strictly positive element.
Moreover,
for any  $a\in (W_n)_+\setminus \{0\},$ there exists $n_0\ge n,$
$x_1,x_2,..., x_m\in W_{n_0}$ such that
$$
\sum_{i=1}^mx_i^*\imath_{n, n_0}(a)x_i=\imath_{n, n_0}(e_0').
$$
Moreover,
$$
t(\imath_{n, m}(e_0'))\ge 7/8 \tforal t\in T(W_{{{m}}}) \tand \tforal m\ge n_0,
$$
$$
\tand \tau(\imath_{n, \infty}(e_0'))>15/16\rforal \tau\in T(W_T).
$$
\end{lem}

\begin{proof}
To simplify the notation, \wilog, we may let $n=1.$
Since $W_T$ is simple,  pick a strictly positive  element in $e_0'\in (W_1)_+$ with $\|e_0'\|=1$
so that $e'=\imath_{1, \infty}(e_0')$ is a strictly positive in $W_T.$
By replacing $e_0'$ by $g(e_0')$ for some $g\in C_0((0,1])_+$ we may assume that
$$
\tau(e_0')>15/16\rforal \tau\in T(W_T).
$$

There is an integer $n_0'\ge 1$ such that
that
\beq\label{Tsimple-1}
t(\imath_{1, n}(e_0'))\ge 7/8\rforal n\ge n_0'\andeqn t\in T(W_n).
\eneq

Note that this implies that
\beq\label{Tsimple-2}
t(\imath_{1, n}(f_\eta(e_0'))\ge 3/4\rforal n\ge n_0 \andeqn t\in T(W_n)
\eneq
whenever $1/16>\eta>0.$

{{Fixed $a\in (W_1)_+\setminus \{0\}$.}} Since $W_T$ is simple, there exists $n_0\ge n_0'\ge 1$ and $x_1',x_2',...,x_{m'}'\in W_{n_0}$
such that
\beq\label{Tsimple-3}
\|\sum_{i=1}^{m'} (x_i')^* \imath_{1,n_0}(a)x_i'-\imath_{1,n_0}(e_0')\|<1/128.
\eneq
It follows from Lemma 2.2 of \cite{Rr11}
that there are $y_1',y_2',...,y_m'\in W_{n_0}$ such that
\beq\label{Tsimple-4}
\sum_{i=1}^{m'}(y_i')^*\imath_{1,n_0}(a)y_i'=\imath_{1,n_0}(f_\eta(e_0'))
\eneq
for some $1/16>\eta>0.$
By \eqref{Tsimple-2}, $\imath_{1, n_0}(f_\eta(e_0'))$ is full in $W_{n_0}.$
Therefore there are $x_1, x_2,...,x_m\in W_{n_0}$ such that
\beq\label{Tsimple-5}
\sum_{i=1}^m x_i^* \imath_{1, n_0}(a)x_i=\imath_{1,n}(e_0').
\eneq
\end{proof}


\begin{prop}\label{PBTsimple}
$B_T$ is a simple \CA.
\end{prop}

\begin{proof}
It suffices to show that every element in $(B_T)_+\setminus \{0\}$ is full in $B_T.$
It suffices to show that every non-zero positive element in
$\cup_{n=1}^{\infty}\Phi_{n, \infty}(B_n)$ is full.
Take $b\in \cup_{n=1}^{\infty}\Phi_{n, \infty}(B_n)$ with $b\ge 0$ and $\|b\|=1.$
To simplify notation , \wilog, we may assume
that there is $b_0\in B_1$ such that
$\Phi_{1, \infty}(b_0)=b.$

Write $b_0=b_{00}\oplus b_{0,1},$ where $b_{00}\in  (W_1)_+$ and $b_{0,1}\in (A(W,\af_1))_+.$

First suppose that $b_{00}\not=0.$

By applying \ref{LfullW}, one obtains an integer $n_0>1,$ $x_1,x_2,...,x_m\in W_{n_0}$
such that
\beq\label{Psimple-n1}
\sum_{i=1}^m x_i^*(\imath_{1, n_0}(b_{00}))x_i=\imath_{1,n_0}(e_0').
\eneq

Let $M=\max\{\|x_i\|: 1\le i\le m\}.$ The above implies that
\beq\label{Psimple-n1+}
t((\imath_{1, n_0}(b_{00}))\ge {7\over{8mM^2}}\rforal t\in T(W_{n_0}).
\eneq

Let $P_{W,m}: B_{m}\to W_{m}$  and $P_{A, m}: B_{m}\to M_{(m!}(A(W, \af_{m}))$ be the projections
($m\ge 1$).  Then, by
 \ref{LBTsimplR},
\beq\label{Psimple-n2}
t(P_{W,n_0}(\Phi_{1, n_0}(b_{00}))) &\ge& t(\Psi_{1,n_0,W,W}(b_{00})))\\
&\ge & (1-\sum_{j=0}^{n_0-1} r_{1+j})t(\imath_{1, n_0}(b_{00}))\\
 &\ge &   (1-\sum_{j=0}^{n_0-1} r_{1+j}) ( {7\over{8mM^2}}) \rforal t\in T(W_{n_0}).
\eneq
It follows that $P_{W,n_0}(\Phi_{1, n_0}(b_{00})))$ is full in $W_{n_0}.$
Put $b_{00}'=P_{W,n_0}(\Phi_{1, n_0}(b_{00}))).$
By applying \ref{LBTsimplR} again,
one concludes that $P_{W, n_0+1}\circ \Phi_{n_0,n_0+1}(b_{00}')$ is full in $W_{n_0+1}.$

Since $b_{00}'$ is full in $W_{n_0},$  $P_{0,n_0}(b_0')$ is full in $R_{0,n_0}=R(\af_{n_0}, 1).$
Since $\phi_{R, A, \af_{n+1}}\circ \imath_{W,n}'$ maps full elements
of $R_{\af_{n_0},1}$ to full elements  in $A(W, \af_{n_0+1}),$
$P_{A, n_0+1}\circ \Phi_{n_0,n_0+1}(b_{00}')$ is full in $M_{(n+1)!}(A(W, \af_{n_0+1}).$
It follows that $\Phi_{n_0, n_0+1}(b_{00}')$ is full in $B_{n_0+1}.$

Note that what has been proved:
for any $b'\in (W_n)_+\setminus \{0\},$ there is $m_0\ge 1$
such that $\Phi_{n,m_0}(b')$ is full in  $B_{m_0}.$
Therefore $\Phi_{n, m}(b')$ is full in $B_m$ for all $m\ge m_0.$

In particular, this shows that $\Phi_{n,\infty}(b_{00})$ is full. Therefore  $b\ge \Phi_{n, \infty}(b_{00})$
is full.

Now consider the case that $b_{00}=0.$
Then $b_{1,0}\not=0.$ Since $\Psi_{1, 2, A,W}$ is injective,
$P_{W,1}(\Phi_{1,2}(b_{1,0})\not=0.$ Applying what has been proved,
$\Phi_{2, \infty}(P_{W,1}(\Phi_{1,2}(b_{1,0}))$ is full in $B_T.$
But
$$
\Phi_{1, \infty}(b_{1,0})\ge \Phi_{2, \infty}(P_{W,1}(\Phi_{1,2}(b_{1,0})).
$$
This shows that, in all cases, $b$ is full in $B_T.$
Therefore $B_T$ is simple.
\end{proof}

\begin{prop}\label{PWTtrace}
$B_T\in {\cal D}_0$ and $T(B_T)=T.$  In particular, $B_T$ has continuous scale.
Moreover $B_T$  {{is locally approximated by
sub-homogenous \CA s with spectrum having
dimension no more than 3,}}
has finite nuclear dimension,
${\cal Z}$-stable and has stable rank one.
\end{prop}


\begin{proof}

Let us first  show that $T(B_T)=T.$
Recall ${ \tilde{T}(A)}$ is the set of all lower semi-continuous traces on $A$ and ${ T(A)}$ is the set of tracial states on $A$. In {{the rest of the}} proof, for all  $C^*$ algebras $A= B_n$ and $A= W_n$ , we have that $0<\af_n\leq\inf\{d_\tau(a):  \tau\in \overline{T(A)}^w\}$, and that all traces $\tau\in { \tilde{T}(A)}$ are bounded trace.


Note the homomorphisms $\Phi_{n,n+1}: B_n\to B_{n+1}$ and $\imath_{n,n+1}: W_n\to W_{n+1}$ induce  maps
$\Phi_{n,n+1}^*: { \tilde{T} (B_{n+1})\to \tilde{T}(B_{n})}$ and $\imath_{n,n+1}^*: { \tilde{T}(W_{n+1})\to \tilde{T} (W_{n})}$. From the construction above,
(see  also {{\cite{Jb}}}),
since $\Phi_{n, n+1}$ and $\imath_{n, n+1}$  map strictly positive elements to
strictly positive elements,
$\Phi_{n,n+1}^*$ and $\imath_{n,n+1}^*$ take tracial states to tracial states. That is, $\Phi_{n,n+1}^*: {{T}(B_{n+1})\subset  {T}(B_{n})}$ and ${\imath_{n,n+1}^*: {T}(W_{n+1}) \subset {T}(W_{n})}$.  Consequently for any $\tau\in { {\tilde T}(B_{n+1})}$ (or $\tau \in { {\tilde T}(W_{n+1})}$), we have $\|\Phi_{n,n+1}^*(\tau)\|=\|\tau\|$ (or $\|\imath_{n,n+1}^*(\tau)\|=\|\tau\|$).

 Hence we have the following inverse limit systems of compact convex spaces:
$$
\xymatrix{
{\overline{T(B_1)}}^w  &   {\overline{T(B_2)}}^w\ar[l]_{\Phi_{1,2}^*} &    {\overline{T(B_3)}}^w\ar[l]_{\Phi_{2,3}^*} &  \cdots\cdots\ar[l] &  { \lim\limits_{\leftarrow}{\overline{T(B_n)}}^w}\ar[l]~~,
 }
$$
\vspace{-0.2in}$$
\xymatrix{
{\overline{T(W_1)}}^w  &    {\overline{T(W_2)}}^w\ar[l]_{\imath_{1,2}^*} &    {\overline{T(W_3)}}^w\ar[l]_{\imath_{2,3}^*} &    \cdots\cdots\ar[l] &    { \lim\limits_{\leftarrow}{\overline{T(W_n)}}^w}\ar[l]~~.
 }
$$
Here we write that
$$ \lim\limits_{\leftarrow}{\overline{T(B_n)}}^w=\{(\tau_1,\tau_2,\cdots, \tau_n,\cdots)\in \Pi_n {\overline{T(B_n)}}^w : \Phi_{n, m}^*(\tau_m)=\tau_n\},$$
which is a subspace of the product space $\Pi_n {\overline{T(B_n)}}^w $ with product toplogy. On the other hand, since all the map $\Phi_{n, m}^*$ are affine map, $\lim\limits_{\leftarrow}{\overline{T(B_n)}}^w$ has a natural affine structure defined by
$$t(\tau_1,\tau_2,\cdots, \tau_n,\cdots)+(1-t)(\tau'_1,\tau'_2,\cdots, \tau'_n,\cdots)=(t\tau_1+(1-t)\tau'_1,\tau_2+(1-t)\tau'_2,
\cdots, \tau_n+(1-t)\tau'_n),$$
for any $(\tau_1,\tau_2,\cdots, \tau_n,\cdots), (\tau'_1,\tau'_2,\cdots, \tau'_n,\cdots)\in \lim\limits_{\leftarrow}{\overline{T(B_n)}}^w$ and $t\in (0,1)$.


Note that each element in  $\lim\limits_{\leftarrow}{\overline{T(B_n)}}^w$ is given by $(\tau_1, \tau_2,\cdots \tau_n,\cdots, )$ with $\Phi_{n, m}^*(\tau_m)=\tau_n$, for $m>n$. This element decides a unique element $\tau \in { {\tilde T}(B)}$ defined by $\tau|_{B_n}=\tau_n$.  { However, since $\|\tau_n\|\geq \af_n$ and $\lim\limits_n \af_n=1,$
$\tau\in T(B_T).$}
On the other hand,
 each element $\tau \in { T(B_T)}$ defines a sequence $\{\tau_n= \tau|_{B_n}\in {{\tilde T}(B_n)}\}_n$. Since $\cup_n B_n$ is dense in $B$, $\|\tau\|=\lim\limits_{n\to \infty}\|\tau_n\|$. From $\|\Phi_{n,n+1}^*(\tau')\|=\|\tau'\|$  for any $\tau'\in { {\tilde T}(B_{n+1})}$, we know that $\|\tau_n\|=\|\tau_{n+1}\|.$  { {Consequently $\|\tau_n\|=\|\tau\|=1$ for all $n$.}}

  Hence $\tau_n\in T{{(B_n)}}\subset {\overline{T(B_n)}}^w$. Consequently,
$T{{(B_T)}}=\lim\limits_{\leftarrow}{\overline{T(B_n)}}^w$. Similarly, $T{{(W_T)}}=\lim\limits_{\leftarrow}{\overline{T(W_n)}}^w$.
(Note that the map $T(B_T)\to {\overline{T(B_n)}}^w$  from the reverse limit is the same as
 $\Phi_{n,\infty}^*: T(B_T)\to {\overline{T(B_n)}}^w$, the restrict map. That is, $\tau\in T(B_T)$ corresponds to
 the sequence $$(\Phi_{1,\infty}^*(\tau),\Phi_{2,\infty}^*(\tau),\cdots, \Phi_{n,\infty}^*(\tau), \cdots, )=(\tau|_{B_1},\tau|_{B_2},\cdots, \tau|_{B_n}, \cdots).$$
In other word, the homeomorphism between $T(B_T)$ and $\lim\limits_{\leftarrow}{\overline{T(B_n)}}^w$ also preserve the convex structure.)

Similarly, we also have  the following  inverse limit systems of the topological cones:
$$
\xymatrix{
{{{\tilde T}(B_1)}} &   {{{\tilde T}(B_2)}} \ar[l]_{\Phi_{1,2}^*} &    {{{\tilde T}(B_3)}} \ar[l]_{\Phi_{2,3}^*} &     \cdots\cdots\ar[l]   & {{{\tilde T}(B_T)}} \ar[l]~~,
 }
$$
$$
\xymatrix{
{{{\tilde T}(W_1)}}  &   {{{\tilde T}(W_2)}} \ar[l]_{\imath_{1,2}^*} &    {{{\tilde T}(W_3)}} \ar[l]_{\imath_{2,3}^*} &     \cdots\cdots\ar[l] &   {{{\tilde T}(W_T)}} \ar[l]~~.
 }
$$
(Again, the reverse limit is taking in the category of topological space in weak* topology, but it automatically preserves cone structure)


 Let $\pi_n: B_n=W_n\oplus   M_{(n!)^2}(A(W, \af_n))\to  W_n$ be the projection and let $\tilde{\Phi}_{n,n+1}=\Phi_{n,n+1}|_{W_n}$, then we have the following (not commutative) diagram:
$$
\xymatrix{
B_1\ar[r]^{\Phi_{1,2}}\ar[d]_{\pi_1}  &    B_2\ar[r]^{\Phi_{2,3}}\ar[d]_{\pi_2}  &    B_3\ar[r]^{\Phi_{3,4}}\ar[d]_{\pi_3}    &  \\
    W_1\ar[r]^{\imath_{1,2}}\ar[ru]^{\tilde{\Phi}_{1,2}}&    W_2\ar[r]^{\imath_{2,3}}     \ar[ru]^{\tilde{\Phi}_{2,3}}&  W_3\ar[r]^{\imath_{3,4}}~~  &  .  \\
 }
$$

Even though the diagram is not commutative,  {{from the constrction,}} it induces an approximate commuting diagram
$$
\xymatrix{
{{{\tilde T}(B_1)}} &   {{{\tilde T}(B_2)}} \ar[l]_{\Phi_{1,2}^*}\ar[ld]_{\tilde{\Phi}_{1,2}^*} &   {{{\tilde T}(B_3)}} \ar[l]_{\Phi_{2,3}^*}\ar[ld]_{\tilde{\Phi}_{2,3}^*}  &     \cdots\cdots\ar[l] &   {{{\tilde T}(B_T)}} \ar[l]~~ \\
   {{{\tilde T}(W_1)}}\ar[u]^{\pi_1^*}  &   {{{\tilde T}(W_2)}} \ar[l]_{\imath_{1,2}^*}\ar[u]^{\pi_2^*} &   {{{\tilde T}(W_3)}} \ar[l]_{\imath_{2,3}^*}\ar[u]^{\pi_3^*} &    \cdots\cdots\ar[l] &   {{{\tilde T}(W_T)}} \ar[l]~~.
  \\
 }
$$

That is
$$|\big(\tilde{\Phi}_{n,n+1}^*(\pi_{n+1}^*(\tau))\big)(g)-\big(\imath_{n, n+1}^*(\tau)\big)(g)|
\leq k(n)r_n\|g\|\|\tau\|~~~\mbox{for all}~ g\in W_n,~ \tau \in {\tilde T}(W_{n+1}); ~\mbox{and}$$
$$|\big(\pi_n^*(\tilde{\Phi}_{n,n+1}^*(\tau))\big)(f)-\big({\Phi}_{n,n+1}^*(\tau)\big)(f)|\leq\big(\frac{1}{(n+1)^2}+k(n)r_n\big)\|f\|\|\tau\|~~~\mbox{for all}~ f\in B_n,~ \tau \in {\tilde T}(B_{n+1}).$$
(Note that $k(n)r_n=\frac{1}{2^{n+1}}.$)

 Note that from the above, for $\tau_{n+1}\in {\tilde T}(W_{n+1})$ if $\tau_n=\imath_{n, n+1}^*(\tau_{n+1})$, then
\beq\label{PWtrace-g1}
\|\pi_{n+1}^*(\tau_{n+1})\| \geq (1-\frac{1}{2^{n+1}})\|\tau_n\|
\eneq
So, we have the following fact:\\
{{if}} $(\tau_1, \tau_2,\cdots, \tau_n,\cdots, )\in \Pi_n {\tilde T}(W_{n+1})$ satisfies $\tau_n=\imath_{n, n+1}^*(\tau_{n+1})$, then $$\lim\limits_{n\to \infty }\|\tau_n\|=\lim\limits_{n\to \infty }\|\pi_n^*(\tau_n)\|.$$

The approximate intertwining induces an affine homeomorphisms $\Pi:  {\tilde T}(W_T)\to {\tilde T}(B_T)$  as follows.

For each $\tau \in {\tilde T}(W_T)$, for fixed $n,$ we define
a sequence of $\{\sm_{n,m}\}_{m> n}\subset {\tilde T}(B_{n})$ by $$ \sm_{n,m}=\big({\Phi}_{n,m}^*\circ \pi_m^*\circ \imath_{m, \infty}^*\big)(\tau)\in {\tilde T}(B_{n}).$$
Recall that each element in ${\tilde T}(B_n)$ is a bounded trace, whence it
is a positive linear functional of $B_n.$
From the above inequalities for approximately commuting diagram, one concludes that
$\{\sm_{n,m}\}_{m> n}$ is a Cauchy sequence  (in norm) in the dual space of $B_n.$

For each $n,$ let
$\tau_n=\lim\limits_{m\to \infty}\sm_{n,m}$. Evidently, from the inductive system above, $\tilde{\Phi}_{n,n+1}^*(\tau_{n+1})=\tau_n$. Hence the sequence $(\tau_1, \tau_2,...,\tau_n,...)$  determines an element $\tau'\in {\tilde T}(B_T)$. Let $\Pi(\tau)=\tau'$.    From \eqref{PWtrace-g1} and the above mentioned fact, we know that $\Pi$  preserves the norm and  $\Pi$ maps $T(W_T)$ to $T(B_T)$ 
Moreover, it is clear that $\Pi$ is also an affine map on $T(W_T).$


We can define  $\Pi': {\tilde T}(B_T) \to {\tilde T}(W_T)$  in exactly same way by replacing  ${\Phi}_{n,m}^*$ by $\imath_{n, m}^*$,   replacing  $\pi_m^*$ by $\tilde{\Phi}_{m,m+1}^*$, and  $\imath_{m, \infty}^*$ by ${\Phi}_{m+1,\infty}^*$.

We now show that
 both $\Pi$ and $\Pi'$ are continuous on $T(W_T)$ and $T(B_T),$ respectively.
Let $\{s_\lambda\}\subset T(W_T)$ be a net which converges to $s\in T(W_T)$
point-wisely on $W_T.$  Write
$s_\lambda=(s_{\lambda, 1}, s_{\lambda,2},...,s_{\lambda,n},...)$ and
$s=(s_1, s_2,...,s_n,...).$
Since $s_{\lambda, n}=\imath_{n, n+1}^*(s_{\lambda, n+1})$ and
$s_n=\imath_{n, n+1}^*(s_{n+1}),$
 for each $n,$ $s_{\lambda, n}$ converges to $s_n$ on $W_n.$
 Write $\Pi(s_\lambda)=(\tau_{\lambda,1},\tau_{\lambda,2},...,\tau_{\lambda,n},...)$
 and $\Pi(s)=(\tau_1,\tau_2,...,\tau_n,...).$
 Then, by the definition,
 \beq
 \tau_{\lambda, n}&=&\lim_{m\to\infty} \sm_{\lambda,n,m}=\lim_{m\to\infty}\big({\Phi}_{n,m}^*\circ \pi_m^*\circ \imath_{m, \infty}^*\big)(s_{\lambda})\\
 &=&\lim_{m\to\infty}\big({\Phi}_{n,m}^*\circ \pi_m^*)(s_{\lambda,m})\andeqn\\
 \tau_n&=&\lim_{m\to\infty} \sm_{n,m}=\lim_{m\to\infty}\big({\Phi}_{n,m}^*\circ \pi_m^*\circ \imath_{m, \infty}^*\big)(s)\\
 &=&\lim_{m\to\infty}\big({\Phi}_{n,m}^*\circ \pi_m^*)(s_{m}).
 \eneq

For   $b \in B_n$ and $m> n,$
\beq
\big({\Phi}_{n,m}^*\circ \pi_m^*)(s_{\lambda,m})(b)&=&s_{\lambda,m}(\pi_m\circ \Phi_{n,m}(b))\andeqn\\
\big({\Phi}_{n,m}^*\circ \pi_m^*)(s_{m})(b)&=&s_{m}(\pi_m\circ \Phi_{n,m}(b)).
\eneq
Let $\ep>0$ and let ${\cal F}\subset B_n$ be a finite subset. We may assume that ${\cal F}$ is in the unit ball of $B_n.$

There exists $m_0\ge 1$ such that, for all $m\ge m_0,$
\beq
&&|s_{\lambda, n}(\pi_m\circ \Phi_{n,m}(b))-\tau_{\lambda,n}(b))|<\ep/3\andeqn\\
&&|s_{n}(\pi_m\circ \Phi_{n,m}(b))-\tau_n(b))|<\ep/3
\eneq
for all $b$ in the unit ball of $B_n.$

Since $s_{\lambda, n}\to s_n$ on $B_n$ point-wisely,
There exists $\lambda_0$ such that, for all $\lambda>\lambda_0,$
\beq
|s_{\lambda,n}(\pi_{m_0}\circ \Phi_{n,m_0}(b))-s_{n}(\pi_{m_0}\circ \Phi_{n,m_0}(b))|<\ep/3
\eneq
for all $b\in {\cal F}.$  It follows that, when $\lambda>\lambda_0,$ for all $b\in {\cal F},$
\beq
\hspace{-0.4in}|\tau_{\lambda,n}(b)-\tau_n(b)| &\le&  |\tau_{\lambda, n}(b)-s_{\lambda, n}(\pi_{m_0}\circ \Phi_{n,m_0}(b))|\\
 && +|s_{\lambda, n}(\pi_{m_0}\circ \Phi_{n,m_0}(b))-s_{n}(\pi_{m_0}\circ \Phi_{n,m_0}(b))|\\
 &&+|s_{n}(\pi_{m_0}\circ \Phi_{n,m_0}(b))-\tau_n(b)|<\ep/3+\ep/3+\ep/3=\ep.
\eneq
This verifies that $\Pi(s_\lambda)$ converges to $\Pi(s)$ on $B_n$ for each $n.$
Since $\cup_{n=1}B_n$ is dense in $B_T,$ it is easy to see that
$\Pi(s_{\lambda})$ converges to $\Pi(s)$  point-wisely. It follows that $\Pi$ is  weak*-continuous
on $T(W_T).$     A similar argument verifies that $\Pi'$ is weak*-continuous on $T(B_T).$
From the definition, one can also verify that $\Pi$ and $\Pi'$ are inverse each other.
Consequently,  they induce the homeomorphism between $T(W_T)$ and $T(B_T)$. Hence $T(B_T)=T(W_T)=T.$

From Remark \ref{Zstable}, we know that $B_T$ is locally approximated by
sub-homogenous \CA s with spectrum having
dimension no more than 3, has finite nuclear dimension and
${\cal Z}$-stable.  It follows from  a theorem of R\o rdam
{{(see 3.5 of \cite{eglnp1})}} that
$B_T$ has strictly comparison for positive elements.
Since $T$ is compact, it follows from
{{ 5.3 of \cite{eglnp1}}} that $B_T$ has continuous scale.

It remains to show that $B_T\in {\cal D}_0.$
Since $B_T$ has continuous scale, to prove $B_T\in {\cal D}_0.$
let $a\in A_+$ be a strictly  positive element with $\|a\|=1.$
\Wlog, we may assume
that $\tau(f_{1/2}(a))\ge 15/16$ for all $\tau\in T(B_T).$
We choose $a$ such that $a=(a_a,a_w)\in B_1=A(W, \af_1)\oplus W_1$
such that
\beq\label{PWtrace-n0}
t(a_w)>3/4, \,\,\,t(f_{1/2}(a_w))>3/4 \rforal t\in T(W_1).
\eneq
Choose ${\mathfrak{f}}_a=5/16.$
Let $b\in A_+\setminus \{0\}$ and let ${\cal F}\subset B_T$ be a finite set and  $\ep>0$.
Let $\dt>0.$ With out lose of generality, we may assume $F\cup \{a, b\} \subset B_n$ for $n$ large enough, and let $\Lambda: B_T\to B_n$ be a \cpc\,
such that
\beq
\|\Lambda(b)-b\|<\min\{\ep/2, \dt\} \rforal b\in {\cal F}.
\eneq
We choose $\dt$ so small that
\beq\label{TWtrace-1-1}
\|f_{1/2}(\Lambda(a))-\Lambda(f_{1/2}(a))\|<1/16\andeqn \|f_{1/2}(\Lambda(a)-\Lambda(f_{1/2}(a))\|<1/16.
\eneq
 Let $P_A: B_n\to M_{(n!)^2}(A(W, \af_n))$ and $P_W: B_n \to W_n$
 be the canonical projections.
 We choose $n\ge 1$ such that
 \beq\label{PWtrace-n1}
 \frac{1}{(n+1)^2}< \mathrm{inf} \{\tau (b): \tau \in T(B_T)\}/2.
 \eneq
 We will choose the algebra $D\in {\cal C}_0^0$ to be $D=\Psi_{n,n+1, W,A}(W_n)\oplus W_{n+1}$ and the map $\phi:B_T\to B_T$ and $\psi: B_T\to D$ be defined by
$$\phi=\Phi_{n+1,\infty}\circ \Psi_{n,n+1,A,A} \circ P_A \circ\Lambda~~~~~~~~~~~~\mbox{and}~~$$
$$ \psi= \Psi_{n,n+1, W,A}\circ P_W \circ\Lambda \oplus \mathrm{diag}(\Psi_{n,n+1,A,W} \circ P_A \circ\Lambda,\Psi_{n,n+1,W,W} \circ P_W \circ\Lambda).$$
Put
\beq
\psi'=\Psi_{n,n+1, W,A}\circ P_W \oplus \mathrm{diag}(\Psi_{n,n+1,A,W} \circ P_A, \Psi_{n,n+1,W,W} \circ P_W)
\eneq
from $B_n$ to $D.$
Since $\Psi_{n, n+1, W, A}$ is injective on $W_n,$ $D\in C_0^0.$
Since $\Phi_{n, \infty}$ is injective, we will identify $D$ with $\Phi_{n, \infty}(D).$
With this identification, we have
\beq\label{PWtrace-n2-1}
\|x-\diag(\phi(x), \psi(x))\|<\ep\rforal x\in {\cal F}.
\eneq
It follows from \ref{LBTsimplR} that
\beq\label{PWtrace-n2}
P_W(\Phi_{1,n}(f_{1/2}(a))) &\ge & \Phi_{1, n}(a_W){{\andeqn}}\\
 t(P_W(\Phi_{1.n}(f_{1/2}(a)))) &\ge & t(\Phi_{1,n}(f{1/2}(a_W)))\ge (1-\sum_{j=0}^{n-1}r_{1+j})t(\imath_{1,n}(f_{1/2}(a_W)))
\eneq
for all $t\in T(W_n).$  Since $t\circ \imath_{1,n}$ is a tracial state  on $W_1$ as proved above,
{{by}} \eqref{PWtrace-n0},
\beq\label{PWtrace-n3}
t(P_W(\Phi_{1,n}(f_{1/2}(a))))\ge (1/2) (3/4)=3/8\rforal t\in T(W_n).
\eneq
Since $\Psi_{n, n+1, W,A}$ sends strictly positive elements of $W_n$ to
those of $\Psi_{n, n+1, W,A}(W_n),$
{{ any $t'\in T(\Psi_{n, n+1, W,A}(W_n))$ gives a tracial state of $W_n,$ therefore}}
\beq\label{PWtrace-n4}
t'(\Psi_{n,n+1,W, A}(P_W(\Phi_{1,n}(f_{1/2}(a))))\ge 3/8\rforal {{\tau'\in T(\Psi_{n, n+1, W,A}(W_n)).}}
\eneq
For any $t\in T(W_{n+1}),$ by applying \eqref{PWtrace-n0} again,
\beq\label{PWtrace-n5}
t(\Psi_{n,n+1, W,W}(P_W(f)))\ge  (1-\sum_{j=0}^nr_{1+j})t(\imath_{1, n+1}(f_{1/2}(a_W)))\ge (1/2)(3/8)=3/8.
\eneq
Combining \eqref{PWtrace-n4} and \eqref{PWtrace-n5}, we have
that
\beq\label{PWtrace-n6}
t(\psi'(\Phi_{1,n}(f_{1/2}(a))))\ge t(\psi'(P_W(\Phi_{1,n}(f_{1/2}(a))))\ge 3/8\rforal t\in T(D).
\eneq
It follows that,  for all $t\in T(D),$
\beq\label{PWtrace-n7}
t(f_{1/2}(\psi(a)))\ge t(\psi'(\Phi_{1,n}(f_{1/2}(a))))\ge t(\psi'(P_W(\Phi_{1,n}(f_{1/2}(a)))))-1/16\ge 5/16={\mathfrak{f}}_a.
\eneq
On the other hand, from the construction, for any $c\in \Psi_{n,n+1,A,A}(M_{(n!)^2}(A(W, \af_n)))_+$
with $\|c\|\le 1,$
\beq\label{PWtrace-n8}
\tau(c)\le \frac{1}{(n+1)^2}\rforal \tau\in T(M_{((n+1)!)^2}(A(W, \af_{n+1}))).
\eneq
Therefore,  for any integer $m\ge 1,$
\beq\label{PWtrace-n9}
\tau(\phi(a)^{1/m}))< \frac{1}{(n+1)^2}\rforal \tau\in T(B_T).
\eneq
{{Consequently,}} by \eqref{PWtrace-n1},
\beq\label{PWtrace-n10}
d_\tau(\phi(a))\le \frac{1}{(n+1)^2} <\inf\{d_\tau(b):\tau\in T(B_T)\}
\eneq
Since we have proved that $B_T$ has strict comparison for positive elements,
\eqref{PWtrace-n10} implies that
\beq\label{PWtrace-n11}
\phi(a)\lesssim b.
\eneq
It follows from \ref{DD0}, \eqref{PWtrace-n2-1}, \eqref{PWtrace-n11} and \ref{PWtrace-n7} that $B_T\in {\cal D}_0.$
Since $B_T\in {\cal D}_0,$ it follows from
{{11.5 of \cite{eglnp1}}} that  $B_T$ has stable rank one.
This completes the proof of this proposition.

\vspace{0.2in}

\end{proof}

\begin{prop}\label{PBTKT}
$K_0(B_T)={\rm ker}\rho_{B_T}=G_0$ and $K_1(B_T)=G_1.$
\end{prop}

\begin{proof}
{{ Let $I=C_0\big( (0,1), Q\otimes Q\big)$ be the canonical ideal of $A(W,\af_n)$. Then the short exact sequence
$$0\to I \to A(W,\af_n)\to A\to 0$$
induces six term exact sequence $$
\xymatrix{
K_0(I) \ar[r] & K_0(A(W,\af_n)) \ar[r] & K_0(A) \ar[d]^{\partial} \\
K_1(A) \ar[u] & K_1(A(W,\af_n)) \ar[l]  & K_1(I) . \ar[l]
}
$$
{{Note that $K_0(I)=\{0\}$ and $K_1(I)=K_0(Q)=\Q.$
Moreover, $K_0(A)=\Q\oplus G_0$ and $K_1(A)=G_1.$
}}
The map $\partial: K_0(A)\to K_1(I)\cong K_0(Q\otimes Q) $ is given by
$\partial=(1-\af_n)s_{*0}$ (defined by $\partial(x)=(1-\af_n)s_{*0}(x)\in \Q$
for all $x\in K_0(A)$)
as the difference of two induced \hm s at the end points  {{(recall that $s_{*0}(r, g)=r$ for all $(r,g)\in \Q\oplus G_0,$ see  the beginning of \ref{Dcc1}).}}
Then one checks
$(1-\af_n)s_{*0}$  is surjective as {{$K_1(I)=\Q.$}}
{{From the six-term exact sequence above, one computes that}}
 $K_0(A(W,\af_n))= {\rm ker}\partial ={\rm ker} s_{*0}=G_0={\rm ker}\rho_{A}$ and $K_1(A(W,\af_n))\cong K_1(A)=G_1$. }}
 {{We also note that  $\partial$ gives an isomorphism on $\Q.$}}
{{Recall $B_n=W_n\oplus   M_{(n!)^2}(A(W, \af_n))$.}} Since $K_*(W)=\{0\}$, {{one has}}
\beq
&&K_0(B_n)={\rm ker}\rho_{B_n}={\rm ker}\rho_{M_{(n!)^2}(A(W, \af_n))}={\rm ker}\rho_{A}=G_0, {{\andeqn}}\\
&&K_1(B_n)=K_1(M_{(n!)^2}(A(W, \af_n)))=K_1(A)=G_1.
\eneq
Hence $\Phi_{n,n+1,*}: K_*(B_n) \to K_*(B_{n+1})$ is completely {{decided}} by its partial map\\
 $\Phi':  M_{(n!)^2}(A(W, \af_n)) \to  M_{((n+1)!)^2}(A(W, \af_{n+1}))$. Also this  partial map sends \\
 $(f,a)\in M_{(n!)^2}(A(W, \af_n))$ to $({ g}, \diag(a,0,\cdots,0))\in M_{((n+1)!)^2}(A(W, \af_{n+1})),$
where \\ $g=P_f( \Phi_{A_w, \af_n, \af_{n+1}}((f,a)))$  is
  as  in the Definition 7.2.   {{Therefore $\Phi'$ maps $M_{(n!)^2}(I)$ to $M_{((n+1)!)^2}(I)$ and
  it induces a \hm\, $\Phi'': M_{(n!)^2}(A)\to M_{(n+1)!)^2}(A)$ which is
  given by $a\mapsto \diag(a,0,\cdots 0)$ for all $a\in M_{(n!)^2}(A).$ The latter map
  induces the  identity map $\id$ on $K_i(A),$ $i=0,1.$  Thus we have the following commutative
  diagram:
   $$
\xymatrix{
0 \ar[rr]\ar[rrd]  &&& K_0(A(W,\af_n)) \ar[rr]\ar[d]_{\Phi'_{*0}} &&  G_0\oplus \Q \ar[ld]_{{\id}} \ar[ddd]^{\partial} \\
 &&0 \ar[r] & K_0(A(W, \af_{n+1})) \ar[r]  &G_0\oplus \Q \ar[d]^{\partial} &\\
&& G_1 \ar[u] & K_1(A(W,\af_{n+1})) \ar[l]  & \Q \ar[l]_0 &\\
G_1 \ar[uuu]\ar[rru]^{\id}&&& K_1(A(W,\af_n))\ar[u]^{\Phi'_{*1}} \ar[ll]_{\cong}  && \Q \ar[ll]_0\ar[ul]^{\id}\,.
}
$$}}
{{This commutative diagram shows that $\Phi'_{*0}$ is the identity map on $G_0$ and
$\Phi'_{*1}$ is the identity map on $G_1.$}}  {{ Since $K_i(W)=\{0\},$ this shows
that $(\Phi_{n, n+1})_i: K_i(B_n)\to K_i(B_{n+1})$ ($i=0,1$) is the identity map for each $n.$
It follows that  $K_0(B_T)=G_0$ and $K_1(B_T)=G_1.$}}




\end{proof}


\begin{lem}\label{DC_0} Let $G_0$ be a torsion free abelian group and let $A$ be the unital  AF algebra with
$$(K_0(A), K_0(A)_+, [1_A]) =
\big(\Q\oplus G_0, (\Q\oplus G_0)_+,
(1,0)\big),$$
{{where
$(\Q\oplus G_0)_+=\{(r, g): r\in \Q_+\setminus \{0\}, g\in G_0\}\cup\{(0,0)\}.$}}
Let $\gamma: K_0(A) \to K_0(Q)$ be given by sending $(r,x)\in \Q\oplus G_0$ to $r\in \Q=K_0(Q)$.  Then one can write AF inductive limits $A=\lim\limits_{n} (A_n, \phi_{n,m})$ {{with
injective $\phi_{n,m}$}} and $Q=\lim\limits_{n} (M_{l(n)}(\C), \psi_{n,m})$  such that there are {{injective}} homomorphism{{s}} $s_n:  A_n\to M_{l(n)}(\C)$ satisfying the following conditions:

(1) $(s_n)_*: K_0(A_n) \to K_0(M_{l(n)}(\C))$ is surjective;

(2) $s_{n+1}\circ \phi_{n,n+1}=\psi_{n,n+1}\circ s_n$ and the commutative diagram
$$
\xymatrix{
A_1\ar[r]^{\phi_{1,2}}\ar[d]_{s_1}  &    A_2\ar[r]^{\phi_{2,3}}\ar[d]_{s_2}  &    A_3\ar[r]^{\phi_{3,4}}\ar[d]_{s_3}    &\cdots \cdots ~~A  \\
    M_{l(1)}\ar[r]^{\psi_{1,2}}&    M_{l(2)}\ar[r]^{\psi_{2,3}}      & M_{l(3)}\ar[r]^{\psi_{3,4}}   &\cdots\cdots ~~Q   \\
 }
$$
induces $s: A\to Q$ satisfy $s_*=\gamma$.

\end{lem}

\begin{proof}
By the classification theory of AF algebras due to Elliott, there is   a {{one-sided}} intertwining
$$
\xymatrix{
{ F_1}\ar[r]^{\phi'_{1,2}}\ar[d]_{\af_1}    &  F_2\ar[r]^{\phi'_{2,3}}\ar[d]_{\af_2}    &  F_3\ar[r]^{\phi'_{3,4}}\ar[d]_{\af_3}     &\cdots \cdots ~~A  \\
    M_{m(1)}\ar[r]^{\psi'_{1,2}}  &  M_{m(2)}\ar[r]^{\psi'_{2,3}}      & M_{m(3)}\ar[r]^{\psi'_{3,4}}   &\cdots\cdots ~~Q,   \\
 }
 $$
which induces a homomorphism $\af: A\to Q $ with $\af_*=\gamma$, where
{{$F_n$ are finite dimensional \CA s,}} all {{homomorphisms}} $\af_n$,  $\phi'_{n,n+1}$ and $\psi'_{n,n+1}$ are unital and injective.  We need to modify the diagram to make the condition (1) holds.

We will define {{subsequence}} $F_{k_n}$ and for each $n$ construct a matrix algebra $M_{l(n)}$, unital
{{injective}} homomorphisms
$s_n: F_{k_n}\to M_{l(n)}$, $\xi_n: M_{l(n)} \to M_{m(k_n)}$ and $\bt_{n-1}:M_{m(k_{n-1})} \to M_{l(n)}$ (if $n>1$) to satisfy the following conditions:

(i): $(s_n)_*: K_0(F_{k_n}) \to {{K_0(M_{l(n)})}}$ is surjective;

(ii): $\xi_n\circ s_n=\af_{k_n}$ and $\bt_{n-1}\circ \af_{k_{n-1}}= s_n \circ \phi'_{k_{n-1},k_n}$.

Let $k_1=1$. By identifying $K_0(M_{m(k_1)})$ with $\Z$, there is a positive integer $j |m(k_1)$ such that $(\af_{k_1})_* (K_0(F_{k_1}))= j \cdot \Z$. Let $l(1)=\frac{m(k_1)}{j}$. Choose a homomorphism $s_1: F_{k_1}\to M_{l(1)}$ to satisfy that
$(s_1)_*=\frac{(\af_{k_1})_*}{j}: K_0(F_{k_1})\to K_0(M_{l(1)})=Z$ (which is surjective).  Note that for any finite dimensional $C^*$ algebra $F$ and a matrix algebra $M_k$,  a homomorphism $\bt: F\to M_k$ is injective if and only if $\bt_*(K_0(F)_+\setminus\{0\}) \subset K_0(M_k)_+\setminus\{0\}$. Hence the injectivity of $\af_{k_1}$ implies the injectivity of $s_1$.  Let  $\xi'_1: M_{l(n)} \to M_{m(k_n)}$ be any unital embedding. Then $ (\xi'_1\circ s_1)_*= {{(\af_{k_1})_*}}$. There is a unitary $u\in M_{m(k_1)}$ such that $\mbox{Ad} u\circ \xi'_1\circ s_1= \af_{k_1}$. Define $\xi_1= \mbox{Ad} u\circ \xi'_1$ to finish the initial step $n=1$ for the induction.

Suppose that we {{have}} already carried out the construction until step $n$. There is {{a $k_{n+1}$}} such that
$$(\psi'_{k_n, k_{n+1}})_*(K_0(M_{m(k_n)})) \subset (\af_{k_{n+1}})_* (K_0(F_{k_{n+1}}))\subset K_0(M_{m(k_{n+1})}).$$
Again,  there is a positive integer $j |m(k_{n+1})$ such that $$(\af_{k_{n+1}})_* (K_0(F_{k_{n+1}})= j \cdot \Z\subset \Z(=K_0(M_{m(k_{n+1})})).$$
Let $l(n+1)=\frac{m(k_{n+1})}{j}$.  As  {{what}} we {{have done}} in the case for $k_n=k_1$, there are two {{injective}} unital homomorphisms $s_{n+1}: F_{k_{n+1}} \to M_{l(n+1)}$ and $\xi_{n+1}: M_{l(n+1)}\to M_{m(k_{n+1})}$ such that \\
 $\xi_{n+1}\circ s_{n+1}=\af_{k_{n+1}}$.  {{Note that $\xi_{n+1}$ has to be injective as
 $M_{l(n+1)}$ is simple.}} Since the map $(\psi'_{k_n, k_{n+1}})_*:~K_0(M_{m(k_n)})) \to K_0(M_{m(k_{n+1})})$ factors through $K_0(M_{l(n+1)})$ by  $(\xi_{n+1})_*$, one can find a homomorphism ${{\bt_n'}}: M_{m(k_n)}\to
 M_{l(n+1)}$ such that $ (\xi_{n+1})_*\circ {{(\bt_n')}}_*=  (\psi'_{k_n, k_{n+1}})_*$. Since $ (\xi_{n+1})_*$ is injective, we know that  $({{\bt_n'}}\circ \af_{k_{n}})_*= (s_{n+1} \circ \phi'_{k_{n},k_{n+1}})_*$. Hence we can choose a unitary $u\in  M_{l(n+1)}$  such that $\mbox{Ad} u\circ {{\bt_n'}}\circ \af_{k_{n}}= s_{n+1} \circ \phi'_{k_{n},k_{n+1}}$.
 {{In particular, $\bt_n'$ is injective.}} Choose $\bt_n=\mbox{Ad} u\circ {{\bt_n'}}$, we conclude {{that}} the {{inductive}} construction of $F_{k_n}$, $M_{l(n)}$, $s_n: F_{k_n}\to M_{l(n)}$, $\xi_n: M_{l(n)} \to M_{m(k_n)}$ and $\bt_{n-1}:M_{m(k_{n-1})} \to M_{l(n)}$ to satisfy (i) and (ii) for all $n$.  (Warning: we do not require that
 $\xi_n \circ \bt_{n-1}= \psi'_{k_{n-1}, k_n}$.)

 Finally, let $A_n=F_{k_n}$, $\phi_{n,n+1}= \phi'_{k_{n},k_{n+1}}$ and $\psi_{n, n+1}:M_{l(n)}\to M_{l(n+1)} $ be defined by $\psi_{n, n+1}= \bt_n \circ \xi_{n}$.  {{Therefore both $\phi_{n, n+1}$ and $\psi_{n, n+1}$ are
 injective.}} Then
 $$s_{n+1}\circ \phi_{n,n+1}=\bt_n\circ \af_{k_{n}}=\bt_n\circ \xi_n\circ s_n= \psi_{n,n+1}\circ s_n.$$
Since $m(k_n) | l(n+1)$, we have $\lim(M_{l(n)},  \psi_{n, m})=Q$.

\end{proof}

\begin{lem}\label{DC_1} Let $G_0$ be torsion free and $A$ be the AF algebra as in \ref{DC_0} with $K_0(A)=Q\oplus G_0$. Let $a$ be {{a}}  positive integer and $\af=\frac{a}{a+1}$. Let $A(W,\af)$ be defined in  \ref{Dcc1}. Then $A(W,\af)$
is an inductive limit of a sequence of $C^*$-algebras $C_n\in  {{{\cal C}_0}}$ with ${{\lambda_s(C_N)=\af}}$
{{and with injective connecting maps.}}

\end{lem}

\begin{proof}
Let $s: A\to Q$ be as in \ref{DC_0}. By Lemma \ref{Dcc0}, $A(W,\af)$ is isomorphic to the
{{\SCA}}\, of $C([0,1], Q\otimes M_{a+1})\oplus A$ defined by
\vspace{-0.12in}$$
C=\big\{(f,x)\in C([0,1], Q\otimes M_{a+1})\oplus A:
\begin{matrix}f(0)=s(x)\otimes \diag(\overbrace{1,\cdots,1}^a,0),\\
\hspace{-0.1in}f(1){{\,\,=\,}}s(x)\otimes \diag(\underbrace{1,\cdots,1}_{a},1)\end{matrix}\big\}.
$$
%
{{Let $A=\lim\limits_{n} (A_n, \phi_{n,m})$ {{with injective $\phi_{n,m}$,}}  $Q=\lim\limits_{n} (M_{l(n)}(\C), \psi_{n,m})$, and $s_n: A_n\to M_{l(n)}(\C)$ be described as in \ref{DC_0}. }} Evidently $C$ is an inductive limit of 
$$
C_n=\big\{(f,x)\in C([0,1], M_{l({ {n}})}{{(\C)}}\otimes M_{a+1})\oplus A_n:
\begin{matrix}f(0)=s_n(x)\otimes \diag(\overbrace{1,\cdots,1}^a,0),\\
\hspace{-0.1in}f(1){{\,\,=\,}}s_n(x)\otimes \diag(\underbrace{1,\cdots,1}_{a},1)\end{matrix}\big\}{{,}}
$$
{{with connecting homomorphism $\Phi_{n,n+1} :C_n  \to C_{n+1}$ {{given}} by
$$\Phi_{n,n+1}(f,x) =(g,y) ~~~~~~~\mbox{for} ~~~(f,x)\in C_n,$$
where ${{g(t)=(\psi_{n,n+1}\otimes id_{a+1}) (f(t))}}$ and }}$y={{\phi_{n,n+1} }}(x)$.
{{Since both $\phi_{n, n+1}$ and $\psi_{n, n+1}$ are injective, so is $\Phi_{n, n+1}.$}}
The short exact sequence
$$
{{0 \to C_0\big( (0,1), M_{l(n)}(\C)\otimes M_{a+1}\big) \to C_n \to A_n \to 0}}
$$
 induces the six term exact sequence of K-theory. Since ${{(s_n)_{*0}}}: K_0(A_n)\to K_0(M_{l(n)}(\C))$ is surjective, 
 {{exactly as}} the beginning of proof of Proposition \ref{PBTKT}, we have
 $K_0(C_n)= \ker {{((s_n)_{*0})}} \subset K_0(A_n)$ and $K_1(C_n)=0$. From {{a}} standard calculation (see section 3 of \cite{GLN}), we know that $K_0(C_n)_+= {{\ker (s_n)_{*0}}} \cap K_0(A_n)_+$. On the other hand, since $s_n$ is injective, $\ker {{(s_n)_{*0})}}\cap K_0(A_n)_+=\{0\}$. {{In fact, if $x\in \ker {{(s_n)_{*0}}}\cap K_0(A_n)_+\setminus \{0\},$ then there exists a projection $p\in M_r(A_n)$ such that $[p]=x.$
However, since $s_n$ is injective, $s_n(p)=q$ is a non-zero projection in $M_r(M_{l(n)})$ which
is a non-zero element in $K_0(M_{l(n)}),$ whence $x\not\in \ker((s_n)_{*0}).$}} {{This proves that $K_0(C_n)_+=\{0\}.$ Thus
$C_n\in {\cal C}_0.$
 Since $s_n$ are unital, from the very definition (see Definition 3.5), we have
 $\lambda_s(C_n)=\af.$}}

\end{proof}

Summarize the above, we {{obtain}} the following main theorem of this section:

\begin{thm}\label{MainModel} Let $G_0$, $G_1$ be any  {{countable a}}belian groups {{and}}  $T$ be any compact
metrizable Choquet simplex, then there is a simple \CA\, $B\in {\cal D}_0$ with continuous scale such that $K_0(B)=\mathrm{ker}(\rho_B)=G_0$, $K_1(B)=G_1$ and $T{{(B)}}=T$.


Furthermore, if, {{in addition,}} $G_0$ is torsion free and $G_1=0$, then $B=\lim\limits_{n\to \infty} (C_n,\imath_n)$ with each $C_n\in {\cal C}_{0}$, and $\imath_n$
{{map}} strictly positive elements to strictly positive elements.
Moreover, $B$ is locally approximated by \CA s in ${\cal C}_0.$

\end{thm}

\begin{proof} We only need to prove the additional part. But in this case, by Lemma \ref{DC_1}, we know all $B_n$ in the construction of inductive limit of $B$ in \ref{Dcc1} are
{{inductive limits of \CA s}} in ${\cal C}_{0}$ {{with injective connecting maps.} }
Therefore $B$ is locally approximated by \CA s in ${\cal C}_0$ and  $B\in {\cal D}.$
Since the $C^*$-algebras in ${\cal C}_{0}$ are  {{semi-projective,
$B$ itself is an inductive limit  of \CA s in ${\cal C}_0.$}}
\end{proof}
\begin{cor}\label{Ctsang}  Let $G_0$, $G_1$ be any {{countable a}}belian groups.
Let ${\tilde T}$ be a topological cone with a base $T$ which is a metrizable Choquet simplex
and let $\gamma: { T}\to (0, \infty]$ be a lower semi-continuous function and ${\tilde \gamma}: { \tilde T}\to [0, \infty]$
be the extension of  $\gamma$ defined by ${\tilde \gamma} (s\tau)=s\gamma(\tau)$ for any $s\in \R_+$ and $\tau\in T$.
Then there exist a non-unital simple \CA\, $A$, which is stably isomorphic to a $C^*$-algebra
{{with the form $B_T$ (in \ref{PWTtrace}) which is}} in ${\cal D}_0$ such that
$$
(K_0(A), K_1(A), {\tilde T}(A), \Sigma_A, \rho_A)\cong (G_0, G_1, {\tilde T}, {\tilde \gamma}, 0)
$$
(Note that $\rho_A=0$ is equivalent to $K_0(A)=\mathrm{ker}(\rho_A)$.)
\end{cor}

\begin{proof}

Let $B$ be the $C^*$-algebra in \ref{MainModel} with $K_0(B)=\mathrm{ker}(\rho_B)=G_0$, $K_1(B)=G_1$ and
$T{{(B)}}=T$.
There is a positive element   (see 6.2.1 of \cite{Rl}, for example)
$a\in B\otimes {\cal K}$ such that $d_\tau(a)=\gamma(\tau)$ for all $\tau\in T=T(B).$
Let $A=\overline{a(B\otimes {\cal K})a}.$
Then $A$ is stably isomorphic to $B\in{\cal D}_0$ and
$$
(K_0(A), K_1(A), {\tilde T}(A), \Sigma_A, \rho_A)\cong (G_0, G_1, {\tilde T}, {\tilde \gamma}, 0).
$$

\end{proof}

\begin{rem}\label{RTchoquet}
{{We would like to recall the following facts:}}

{{Let $A$ be a separable \CA\, with $T(A)\not=\emptyset$ and  {{${\rm Ped}(A))=A.$}}
Then $T(A)$ forms a base for the cone ${\tilde T}(A).$
It follows from 3.3 of \cite{PedMI} and 3.1 of \cite{PedMIII}
that ${\tilde T}(A)$ forms a vector lattice. Therefore, if $T(A)$ is compact, then
$T(A)$ is always a metrizable Choquet simplex.}}

\end{rem}

\begin{df}\label{DBT}
In what follows we will use ${\cal B}_T$ for the class of \CA s with the form $B_T.$
Note that  if $A\in {\cal B}_T$ then $A$ is ${\cal Z}$-stable with weakly unperforated $K_0(A)$
(see \ref{Tweakunp}).
\end{df}
















\section{\CA s  ${\mathcal Z}_0$ and class ${\cal D}_0$}

\begin{df}\label{DZ0}

{{Let $\zo=B_T$ be as constructed  in the previous section with
$G_0=\Z$ and $G_1=\{0\}$ and with unique tracial state. Note also $\zo$ is ${\cal Z}$-{{stable.}}}}



\end{df}

From Theorem \ref{MainModel} and Corollary  \ref{Ctsang}, we have the following fact.

\begin{prop}\label{DZ0C_0}
{{$\zo$ is locally approximated by \CA s in ${\cal C}_0.$}}
In fact that $\zo=\lim_{n\to\infty}(C_n, \imath_n),$ where
each $C_n\in {\cal C}_0,$ $\imath_n$
maps
strictly positive elements to strictly positive elements.
\end{prop}
(See \ref{Czo} bellow for the uniqueness of $\zo.$)

\begin{lem}\label{Ltensorscale}
Let $A$ be a  separable exact simple
\CA\, with continuous scale. Then
$A\otimes \zo$ also has continuous scale and $A\otimes \zo$ is ${\cal Z}$-stable.
\end{lem}

\begin{proof}
Since $\zo$ is ${\cal Z}$-stable, so is $A\otimes \zo.$
Therefore, by \cite{Rrzstable}, $A\otimes\zo$ is purely infinite or is stably finite.
Since every separable purely infinite simple \CA\, has continuous scale (\cite{Lncs1}),
we assume that $A\otimes \zo$ is stably finite.
In particular,
$T(A\otimes \zo)\not=\emptyset.$
Since ${\cal Z}$ is unital, it is easy to see that $A\otimes {\cal Z}$ has continuous
scale.   It follows that $T(A)$ is compact. Since
$\zo$ has a unique tracial state, $T(A\otimes \zo)$ is also compact.
The lemma follows if we also assume that $A$ is exact by  {{ 5.3 of \cite{eglnp1}}}.
{{However, the proof 5.3 of \cite{eglnp1} also shows that $T(A\otimes \zo)$ is compact.}}

For general cases,
let $B=A\otimes {\cal Z}.$ We may write $A\otimes \zo=B\otimes \zo.$
We also note that $B$ has strict comparison {{(by Theorem 4.5 of \cite{Rrzstable}).}}

Let $\{e_n\}$ be an approximate identity for $B$ such that
$e_{n+1}e_n=e_ne_{n+1}=e_n,$ $n=1,2,...$
Let $\{{{b_n}}\}$ be an approximate identity for $\zo$ such that
${{b_{n+1}b_n=b_nb_{n+1}=b_n}},$ $n=1,2,....$
It follows that ${{c_n}}=e_n\otimes b_n$ is an approximate identity for $B\otimes \zo$ such that
\beq
c_{n+1}c_n=(e_{n+1}e_n)\otimes (b_{n+1}b_n)=e_n\otimes b_n=c_n,\,\,\, n=1,2,....
\eneq
Fix any $d\in B\otimes \zo.$
Put
\beq
\sigma=\inf\{d_\tau(d): \tau\in T(B\otimes \zo)\}>0.
\eneq
Since $B$ has continuous scale, there exists an integer $n_0\ge 1$
such that
\beq
\tau(e_n-e_m)<\sigma/4\rforal \tau\in T(B)
\eneq
when $n>m\ge n_0.$
Let $t_Z$ be the unique tracial state of $\zo.$
There is $n_1\ge 1$ such that
\beq
t_Z(b_n-b_m)<\sigma/4\rforal n>m\ge n_1.
\eneq

We have, for $n>m\ge n_0+n_1,$
\beq
c_n-c_m&=&e_n\otimes b_n-e_m\otimes b_m= (e_n-e_m)\otimes b_n+(e_m\otimes b_n-e_m\otimes b_m)\\
&=& (e_n-e_m)\otimes b_n+(e_m\otimes (b_n-b_m))
\eneq
Therefore, for $n>m\ge n_0+n_1,$
\beq
(\tau\otimes t_Z)(c_n-c_m)<\sigma/2\rforal \tau\in T(B).
\eneq
By the strict comparison for positive element,
the above inequality implies that $c_n-c_m\lesssim d.$
It follows that $A\otimes \zo$ has continuous scale.
\end{proof}


{{Now we are ready to state the following theorem which is a variation of \ref{Ctsang}:}}

\begin{thm}\label{Mainrange}
For any {{separable}}  finite simple amenable  \CA\, $A$, there is a $C^*$-algebra $B$ which is stably isomorphic to a  \CA\, of the form $B_T$ in ${\cal D}_0$ such that $\mathrm{Ell}(B)\cong \mathrm{Ell}(A\otimes {\cal Z}_0)$
\end{thm}

\begin{proof}
{{Note that,
by 6.2.3 of \cite{Rl} {{(see also 7.3 of \cite{eglnp1})}},
one may write
\beq
Cu^{\sim}(\zo)=\Z\sqcup {\rm LAff}_+^{\sim}({\tilde T}(\zo))\andeqn Cu^{\sim}({\cal W})={\rm LAff}_+^{\sim}({\tilde T}({\cal W})).
\eneq
Since both $\zo$ and ${\cal W}$ are monotracial, ${\rm LAff}_+^{\sim}({\tilde T}(\zo))={\rm LAff}_+^{\sim}({\tilde T}({\cal W})).$
Since $K_0(\zo)={{{\rm ker}}}_{\rho_{\zo}}$, one has an ordered semi-group
\hm\, $\Lambda: \Z\sqcup{\rm LAff}_+^{\sim}({\tilde T}(\zo))\to {\rm LAff}_+^{\sim}({\tilde T}({\cal W}))$
which maps $\Z$ to zero and identity on ${\rm LAff}_+^{\sim}({\tilde T}(\zo))=\R_+^{\sim}.$
In particular, $\Lambda$ maps $1$ to $1.$ It follows from \ref{DZ0C_0} and \cite{Rl}
that there is a \hm\, $\phi_{z,w}: \zo\to {\cal W}$ which maps strictly  positive elements to strictly positive elements.
Let $t_Z$ and $t_{\cal W}$ be the unique tracial states of $\zo$ and ${\cal W},$ respectively.
Then $t_{\cal W}\circ \phi_{z,w}=t_Z,$ since $\zo$ has only one tracial state. }}

Since ${\cal Z}\otimes \zo\cong \zo,$ \wilog, we may assume that $A$ is ${\cal Z}$-stable.
 Let $a\in {\rm Ped}(A)_+$ be such that
$\overline{aAa}$ has continuous scale (see 5.2 of \cite{eglnp1}).
Put $B=\overline{aAa}\otimes \zo.$ It is easy to verify that $B$ is a hereditary \SCA\,
of $A\otimes \zo.$
Every tracial state of $B$  has the form
$\tau\otimes t_Z,$ where $\tau\in T(\overline{aAa}).$
Fix $\tau\in T(\overline{aAa}),$ then
\beq\label{Mainrange-2}
(\tau\otimes t_z)(a\otimes z)=\tau(a)t_Z(z)=\tau(a)(t_{\cal W}\circ \phi_{z,w}(z))\rforal a\in A\andeqn z\in \zo.
\eneq
 Let $\psi={\rm id}_A\otimes \phi_{z,w}: A\otimes \zo\to A\otimes {\cal W}$ and let $s=\tau\otimes t_z\in T(B).$
Then, by \eqref{Mainrange-2},
$s=(\tau\otimes t_{\cal W})\circ {{\psi}}.$
Since $\zo$ satisfies the UCT, by {{the}} K\"unneth formula (\cite{RS}),
 $K_i(\overline{aAa}\otimes {\cal W})=0,$
$i=0,1.$
Therefore, for any $x\in K_0(B),$
$s(x)=0.$ This implies that ${\rm ker}\rho_B=K_0(B).$
Since $A$ is separable, simple and  $B$ is a hereditary \SCA\, of $A\otimes \zo,$
by \cite{Br1}, $(A\otimes \zo)\otimes {\cal K}\cong B\otimes {\cal K}.$
It follows that $K_0(A\otimes \zo)={\rm ker}\rho_{A\otimes \zo}.$

{{Note that (see \ref{RTchoquet}) $T(B)$ is a metrizable Choquet simplex.
By \ref{Ctsang}, there is a $C^*$-algebra $C$ which is stably isomorphic to a  \CA\, of the form
$B_T$ in ${\cal D}_0$ such that $\mathrm{Ell}(C)\cong \mathrm{Ell}(A\otimes {\cal Z}_0).$}}

\end{proof}







\begin{thm}\label{D0kerrho}
{{Let $A$ be a separable \CA\, which is stably isomorphic to a \CA\, in
${\cal D}_0.$ Then $K_0(A)={\rm ker}\rho_A.$}}
\end{thm}


\begin{proof}
\Wlog, we may assume that $A\in {\cal D}_0.$
By {{12.3 of \cite{eglnp1},}}
it suffices to show that every tracial state of $A$ is a ${\cal W}$-trace.
By
{{12.2 of \cite{eglnp1}}}, it suffices to produce a sequence of
\cpc s $\{\phi_n\}$ from $A$ into $D_n\in {\cal C}_0^{0'}$ such that
such that
\beq\nonumber
&&\lim_{n\to\infty}\|\phi_n(ab)-\phi_n(a)\phi_n(b)\|=0\rforal a,\,b\in A\andeqn\\\label{Dokerrho-1}
&&\tau(a)=\lim_{n\to\infty}t_n(\phi_n(a))\rforal a\in A,
\eneq
where $t_n\in T(D_n).$

This, of course, follows directly from the definition of ${\cal D}_0.$
In fact, in the proof of  {{9.1 of \cite{eglnp1}}}
$\phi_{1,n}$ would work (note,
we assume that $A\in {\cal D}_0$ instead in ${\cal D},$ therefore
\CA s $D_n\in {\cal C}_0^{0'}$ instead in ${\cal C}_0'$).
Note also that,
{{9.1 of \cite{eglnp1}}} shows that $QT(Q)=T(A).$
Thus \eqref{Dokerrho-1} follows from
{{(e.9.9) of \cite{eglnp1}.}}
\end{proof}

\begin{thm}\label{Ttrzstable}
{{Let $A$ be a separable simple \CA\, in ${\cal D}$ with continuous scale.
Then the map from $Cu(A)$ to ${\rm{LAff}}_+(T(A))$ is  a Cuntz semigroip isomorphism.}}
\end{thm}

\begin{proof}
This follows from
{{11.8 of \cite{eglnp1} (see 15.8 of \cite{GLp1})}} immediately
{{(since $\overline{T(A)}^w=T(A),$ as $A$ is assumed to have continuous scale).}}
\end{proof}

\begin{cor}\label{Ccunsim}
Let $A$ be a separable simple \CA\, in ${\cal D}.$
Then $Cu^{\sim}(A)=K_0(A)\sqcup {\rm LAff}_+^{\sim}({\tilde T}(A)).$
\end{cor}

\begin{proof}
Note,
{{by 11.5 of \cite{eglnp1},}} $A$ has stable rank one.
This follows from \ref{Ttrzstable} {{(see 7.3 of \cite{eglnp1})}}.
\end{proof}

\begin{thm}\label{TD0=D}
Let $A$ be a separable simple \CA\, in ${\cal D}$ with ${\rm ker}\rho_A=K_0(A).$
Then $A\in {\cal D}_0.$
Moreover,
There exists  $e_A\in A_+$ with $\|e_A\|=1$ and  $0<\sigma_0<1$ which satisfy the following:
For any  $\ep>0,$ $\eta>0$ and   any  finite subset ${\cal F}\subset A,$
there are ${\cal F}$-$\ep$-multiplicative \cpc s $\phi: A\to A$ and  $\psi: A\to D$  for some
\SCA\, $D\in {\cal  R}$ (see \ref{DfC1}) with $\phi(A)\perp D$ such that
\vspace{-0.12in}\beq\label{DNtr1div-1}
&&\|x-(\phi(x)\oplus \psi(x))\|<\ep\tforal x\in {\cal F},\\
&&
d_\tau(\phi(e_A))<\eta\tforal \tau\in T(A)\tand\\
&&t(f_{1/4}(\psi(e_A)))\ge 1-\sigma_0\tforal t\in T(D).
\eneq

\end{thm}

\begin{proof}
We may assume, \wilog,  that $A$ has continuous scale, by considering a hereditary \SCA\,
of $A$ {{(see 11.9 of \cite{eglnp1}).
We will use the facts that \CA s in ${\cal D}$ have stable rank one and strict comparison as well as have
the property described in  \ref{Ccunsim}.}}

{{Let $e_A\in A$ be a strictly positive element with $\|e_A\|=1.$ Note that $d_\tau(e_A)$ is now assumed
to be continuous on ${\tilde T}(A).$ Fix any integer $m\ge 2,$ by \ref{Ccunsim}, there is a positive element $e_{00}\in A$
such that $d_\tau(e_{00})=(1/m) d_\tau(e_A)$ for all $\tau\in {\tilde T}(A).$
Moreover, as in the proof of \ref{Pwwdivisible}, $A\cong M_m(\overline{e_{00}Ae_{00}}).$
Let $\Lambda_{0,m}: {\rm Cu}^\sim(A)\to {\rm Cu}^\sim(A)$ be defined by
$(\Lambda_{0,m})|_{K_0(A)}=\id_{K_0(A)}$ and $(\Lambda_{0,m})|_{ {\rm LAff}_+^{\sim}({\tilde T}(A))}=(1/m)\id_{ {\rm LAff}_+^{\sim}({\tilde T}(A))}.$  Then, since $K_0(A)={\rm ker}\rho_A,$ one sees that $\Lambda_{0,m}$ is a morphism in ${\bf {\rm Cu}}.$
}}

{{Claim: For any \SCA\, $D'{{\subset}} A$ with $D'\in {\cal C}_0,$ there is a \hm\, $j_{0, D'}: D'\to A_0:=\overline{e_{00}Ae_{00}}$ such that ${\rm Cu}^\sim(j_{0,D'})=\Lambda_{0,m}\circ  {\rm Cu}^\sim(\iota_{D'}).$}}
{{To see the claim, note that $A_0$ has stable rank one (see 11.5 of \cite{eglnp1}) and $\Lambda_{0,m}\circ  {\rm Cu}^\sim(\iota_{D'})$ is a morphism
in ${\bf Cu}.$ Then, by Theorem 1.0.1. of \cite{Rl}, such  \hm\, $j_{0,D'}$ exists.}}

Let $1>\sigma_0>0.$
We may assume that
$\tau(e_A)\ge 1-\sigma_0/64$ for all $\tau\in T(A).$
Suppose also that
$\tau(f_{1/2}(e_A))>1-\sigma_0/32$ for all $\tau\in T(A).$ Let ${\mathfrak{f}}=1-\sigma_0/4.$
Let $\ep>0$ and let ${\cal F}\subset A$ be a finite subset.
Let $\sigma_0/32>\eta>0.$ Choose $m\ge 2$ such that $1/m<\min\{\sigma_0/2^{12}, \eta/2\}.$

Let $T=T(A).$ {{Let}} $W_T$ be the separable simple
amenable \CA\, with  $K_i(W_T)=\{0\},$ $ i=0,1,$ and $T(W_T)=T$ as in
{{2.8 of \cite{eglnkk0}.}}
Therefore ${\rm LAff}_+^{\sim}({\tilde T}(W_T))={\rm LAff}_+^{\sim}({\tilde T}(A)).$
Let $\Gamma:{\rm LAff}_+^{\sim}({\tilde T}(W_T))\to {\rm LAff}_+^{\sim}({\tilde T}(A))$
be the order semi-group isomorphism.
By \ref{Ccunsim}, $Cu^{\sim}(A)=K_0(A)\sqcup {\rm LAff}_+^{\sim}({\tilde T}(A)).
$ Since $K_0(A)={\rm ker}\rho_A,$ the map $\Gamma^{-1'}: Cu^{\sim}(A)\to Cu^{\sim}(W_T)$
which maps $K_0(A)$ to zero and $\Gamma^{-1'}|_{ {\rm LAff}_+^{\sim}({\tilde T}(A))}={{({m-1\over{m}})}}\Gamma^{-1}$
is  {{a morphism in ${\bf Cu}.$}}

Fix $0<\ep_1<\min\{\ep/4, \sigma/2^{10}\}$ and a finite subset ${\cal F}_1\supset {\cal F}.$
There are ${\cal F}_1$-$\ep_1$-multiplicative \cpc s $\phi_0: A\to A$ and  $\psi: A\to D$  for some
\SCA\,  $D\subset A$  such that $\phi_0(A)\perp D,$
\beq\label{DNtr1div-1}
&&\|x-(\phi_0(x)\oplus \psi(x))\|<\ep_1/4\rforal x\in {\cal F}_1\cup \{e_A\},\\\label{DNtrdiv-2+}
&&D\in {\cal C}_0 ({\rm see}\,\, \ref{RDd}),\,\,\,
d_\tau(\phi_0(e_A))<\eta\rforal \tau\in T(A)\andeqn\\\label{DNtrdiv-4+}
&&t(f_{1/4}(\psi(e_A)))\ge 1-\sigma_0/16\rforal t\in T(D).
\eneq
Let $\imath_D: D\to A$ be the embedding.
Consider $\Gamma^{-1'}\circ Cu^\sim(\imath_D).$ Then, by \cite{Rl}, there exists
a \hm\, $\psi_1: D\to W_T$ such that $Cu^\sim(\psi_1)=\Gamma^{-1'}\circ  Cu^\sim(\imath_D).$
Let $e_d\in D$ be a strictly positive element of $D$ with $\|e_d\|=1$ and let $W_1=\overline{\psi_1(e_d)W_T{{\psi_1(e_d)}}}.$
By \cite{Rl} again, there exists a \hm\, $\psi_{w,a}: W_T\to A$
such that $Cu^\sim(\psi_{w,a})=\Gamma.$

{{Note that $\la \psi_{w,a}\circ \psi_1(e_d)\ra \le (1/n)\la e_A\ra.$  By the claim above,
we may assume that  there also exists a \hm\, $j_{D}: D\to A$ such that ${\rm Cu}^\sim(j_D)=\Lambda_{0,m}\circ {\rm Cu}^\sim (\iota_D)$
and $j_D(D)\perp \psi_{w,a}\circ \psi.$}}
{{Put $\Psi:=j_D\oplus \psi_{w,a}\circ \psi_1.$ Then $Cu^\sim(\imath_D)=Cu^\sim(\Psi).$ }}

By \cite{Rl},
there exists a sequence of  unitaries $u_n\in {\tilde A}$ such that
\beq\label{TD0=D-12}
\lim_{n\to\infty}\|\imath_D(g)-u_n^*\Psi(g)
u_n\|=0\rforal g\in D.
\eneq
Let $\dt>0,$ ${\cal G}\subset D$ be a finite subset, and
$e_n=\Psi(e_d).$
Choose $1/4>\sigma>0$
such that
\beq
\|f_\sigma(e_d)gf_\sigma(e_d)-g\|<\dt/2\rforal g\in {\cal G}.
\eneq
By
\eqref{TD0=D-12},  with sufficiently small
$\dt,$ by Prop.1 of \cite{CEI}, there is $n_0\ge 1$ and  unitaries  $v_n\in {\tilde A},$ for all $n\ge n_0,$
\beq
\hspace{-0.1in}\|\imath_D(g)-
v_n^*f_\sigma(e_n)\Psi'(g)
f_{\sigma}(e_n)v_n\|<\ep_1\rforal g\in {\cal G}\andeqn
v_n^*f_\sigma(e_n)v_n\in \overline{DAD}.
\eneq
Put $\Phi': D\to A$  by
$\Phi'(c)=v_{n_0}^*\psi_{w,a}\circ \psi_1(c)
v_{n_0}$ for all $c\in D,$ and
$\phi: A\to A$ by $\phi(a)=\phi_0(a)\oplus v_n^*f_\sigma(e_d)j_D(a)f_\sigma(e_d)v_n$
for all $a\in A.$
Let   $W_0=v_{n_0}^*\psi_{w,a}(W_1)v_{n_0}.$
By
{{the choice of $m$ and by \eqref{DNtrdiv-4+},}} we may also assume that
\beq
t(f_{1/4}(\Phi'(\psi(e_A))))>
1-\sigma_0/8\rforal t\in T(W_0).
\eneq
{{Note $W_0\perp \phi(A)A\phi(A).$}}
Moreover, with sufficiently small $\dt$ and large ${\cal G},$
\beq
\|x-(\phi(x)\oplus \Phi'(\psi(x))\|<\ep\rforal x\in {\cal F}.
\eneq
Note that $W_0=\overline{\cup_{n=1}^{\infty} F_n\otimes{\cal W}},$ where $F_n\subset F_{n+1}$ are finite dimensional \CA s,
and ${\cal W}=\overline{\cup_{n=1}^{\infty} R_n},$ where $R_n\in {\cal R}.$
Since $D$ is semiprojective, there exists a sequence of \hm s\, $\psi_{0,n}: D\to R_n$ such that
\beq
\lim_{n\to\infty}\|\psi_{0,n}(g)-\Phi'(g)\|=0\rforal g\in D.
\eneq
Then, passing to a subsequence,  applying a weak*-compactness argument,
if necessarily,  we may assume that, for all sufficiently large $n,$
\beq
t(f_{1/4}\psi_{0,n}(\psi(e_A))))>1-\sigma_0/8\ge {\mathfrak{f}}\ge \rforal t\in T(R_n),
\eneq
Moreover,
$$
\|x-(\phi(x)\oplus \psi_{0,n}(\psi(x))\|<\ep\rforal x\in {\cal F}.
$$
The lemma then follows.

\end{proof}

\begin{prop}\label{Ptensor}
{{Let $A$ be a separable simple \CA\, and $B$ be a separable simple  \CA\, which is  tracially approximate
divisibility {\rm (}see definition
{{{\rm{10.1 of \cite{eglnp1}}}}}{\rm ).}
Suppose {{that}} both $A$ and $B$ have  continuous scale, and $B$ has strict comparison.
Let $C=A\otimes B$ (minimal tensor product) be such that $C$ has continuous scale and  also has strict comparison.
Then $A\otimes B$ is tracially approximate divisible.}}

\end{prop}

\begin{proof}
Let $\ep>0,$ let ${\cal F}\subset A\otimes B$ be a finite subset, let $c\in (A\otimes B)_+\setminus \{0\}$ and let
$n\ge 1$ be an integer.
\Wlog, we may assume that
$${\cal F}=\{a\otimes b:  a\in {\cal F}_A \andeqn b\in {\cal F}_B\},$$
where ${\cal F}_A\subset A$ and ${\cal F}_B$ are finite subsets. We may further assume
that $\|a\|, \|b\| \le 1$ for all $a\in {\cal F}_A$ and $b\in {\cal F}_B.$
Since $A\otimes B$ is simple and  $T(C)$ is compact, as $C$ has continuous scale,
\beq
\inf\{d_\tau(c): \tau\in T(C)\}=d>0.
\eneq
 Choose $b_0\in B_+\setminus \{0\}$ with $\|b_0\|=1$ such that
 $d_\tau(b_0)<d/2$ for all $\tau\in T(B).$

Since $B$ has traically approximate divisible property, there are \SCA s  $B_0$ and $B_1$
of $A$ such that
\beq
{\rm dist}(b, D_d)<\ep/2\rforal  y\in {\cal F}_B,
\eneq
where
$D_d\subset  D\subset B$ which has
the form
$$
D_d=\{d_0\oplus \diag(\overbrace{d_1, d_1,...,d_1}^n)\in B_0\oplus M_n(B_1): d_0\in B_0, d_1\in B_1\},
$$
where $D=B_0\oplus M_n(B_1).$  Moreover, $b_{e0}\lesssim b_0,$
where $b_{e0}$ is a strictly positive element of $B_0.$

Now let $A_0=A\otimes B_0,$ $A_1=A\otimes B_1$ and  $A_3=A_0\oplus M_n(A_1).$
Also let $A_d=A\otimes D_d.$
Then
\beq
{\rm dist}(x, A_d)<\ep\rforal x\in {\cal F}.
\eneq
We also compute that
\beq
e_a\otimes b_{e0}\lesssim e_a\otimes b_0\lesssim c.
\eneq
This implies that $C$ is tracially approximate divisible.
\end{proof}

\begin{prop}\label{PtensorB_T}
$B_T\otimes \zo\in {\cal D}_0.$
\end{prop}
\begin{proof}
\CA\, $B_T$ has finite nuclear dimension and so does $\zo.$
{{By Proposition 2.3 of \cite{WZ}, $B_T\otimes \zo$ has finite nuclear dimension.
  By \ref{Ltensorscale}, $B_T\otimes \zo$ has continuous scale and
  ${\cal Z}$-stable. Therefore, by  \cite{Rrzstable},  it has
strict comparison.  It follows from \ref{Ptensor} and \ref{Pdivisiblehere} that every
hereditary \SCA\, has tracially approximate divisible property.
It follows from {{6.5 of \cite{eglnkk0}}}
that every tracial state of $B_T\otimes \zo$ is a ${\cal W}$-trace. It follows from
{{6.6 of \cite{eglnkk0}}}
that $B_T\otimes \zo$ is in ${\cal D}_0.$}}
\end{proof}

In the appendix (\ref{TMST}), we will show that
\begin{thm}[\ref{TMST}]\label{Dzstable}
Let $A$ be a separable amenable \CA\, in ${\cal D}.$ 
 Then $A\otimes {\cal Z}\cong A.$
\end{thm}




\begin{df}\label{DWZmaps}
By \cite{Rl},
there exists a \hm\, $
{{\phi_{w,z}:}} {\cal W}\to {\cal Z}_0$  which maps
the strictly positive elements to strictly positive elements,
Since $K_0(\zo)={\rm ker}\rho_{\zo},$
 by \ref{DZ0C_0}  and by \cite{Rl},
 there exists  also a \hm\, $
{{\phi_{z,w}:}} {\cal Z}_0\to {\cal W}$ which maps
the strictly positive elements to strictly positive elements. Note as in the proof of  \ref{Mainrange}
we also have 
{{$t_Z=t_{\cal W}\circ \phi_{z,w}$ and $t_{\cal W}=t_Z\circ \phi_{w,z},$}} where
$t_Z$
and $t_{\cal W}$ are the unique tracial states of 
{{ $\zo$ and ${\cal W}$ respectively.}}

There exists also an isomorphism $\phi_{w21}: M_2({\cal W})\to {\cal W}$
and an isomorphism $\phi_{z21}: M_2(\zo)\to \zo$
such that $(\phi_{z21})_{*0}={\rm id}_{K_0(\zo)}.$
We will fixed these four \hm s.

\end{df}

\begin{df}\label{Tnegative}

Let $\kappa^o_0: K_0({\mathcal Z}_0)\to K_0({\mathcal Z}_0)$ be a \hm\, by sending
$x$ to $-x$ for all $x\in K_0(\zo)={\rm ker}\rho_{\zo}.$
Denote also by $\kappa^o$ the automorphism on $Cu^{\sim}({\mathcal Z}_0)$
such that $\kappa^o|_{K_0({\mathcal Z}_0)}=\kappa^o_0$ and
identity on ${\rm LAff}^{\sim}({\tilde T}({\mathcal Z}_0))$ which maps function $1$ to function $1.$
 It follows from \cite{Rl}
that there is an endomorphism $j^{\circledast'}: {\mathcal Z}_0\to {\mathcal Z}_0$ such that
$Cu^{\sim}(j^{\circledast'})=\kappa^o$ and $j^{\circledast'}(a)$ is a strictly positive element of ${\mathcal Z}_0$
for some strictly positive element $a.$
By \cite{Rl} again, $j^{\circledast'}(\zo)$ is isomorphic to $\zo,$ say, given by $j: j^{\circledast'}(\zo)\to \zo.$
Then
$j^{\circledast}=j\circ j^{\circledast'}$ is an automorphism.
The automorphism  $j^{\circledast}$ will be also used in later sections.
\end{df}

\begin{lem}\label{LZ0group}
Define
$\Phi,\,\Psi: \zo \to M_2(\zo) $ by
$$
\Phi(a)=\diag(a, j^{\circledast}(a))\andeqn \Psi(a)=(\phi_{wz}\otimes {\rm id}_{M_2})(\diag(\phi_{zw}(a), \phi_{zw}(a))\rforal a\in \zo.
$$
Then
$\Phi$ is approximately unitarily equivalent to  $\Psi,$
i.e., there exists a sequence of unitaries $\{u_n\}\subset {\widetilde{M_2(\zo)}}$
such that
$$
\lim_{n\to\infty} {\rm Ad}\, u_n\circ \Phi(a)=(\phi_{wz}\otimes {\rm id}_{M_2})\circ \diag( \phi_{zw}(a), \phi_{zw}(a))\rforal a\in \zo.
$$
In particular, $j^{\circledast}_{*0}(x)=-x$ for $x\in K_0(\zo).$

Moreover
$\phi_{z21}\circ \Phi$ is approximately unitarily equivalent to
$\phi_{z21}\circ \Psi.$
\end{lem}

\begin{proof}
Using 6.1.1 of \cite{Rl} (see also 7.3 of \cite{eglnp1}), one computes that
$$
Cu^{\sim}(\Phi)=Cu^{\sim}(\Psi).
$$
It follows from \cite{Rl} that $\Phi$ is approximately
unitarily equivalent to $\Psi.$

\end{proof}


\section{${\cal E}(A,B)$}

\begin{df}\label{DE(AB)}
Let $A$ be a separable amenable  \CA\, and let $B$ be another \CA.
We use $B^{\vdash}$ for the \CA\, obtained by adding a unit to $B$
(regardless $B$ has a unit or not).  We will  continue to use the embedding
$\phi_{wz}.: {\cal W}\to {\mathcal Z}_0.$
  Without causing confusion, we will identify
${\cal W}$ with $\phi_{wz}({\cal W})$ from time to time.

An {\it asymptotic sequencial morphism}\, $\phi=\{\phi_n\}$ from
$A$ to $B$ is a
sequence of  approximately multiplicative \cpc s $\phi_n: A\to B^{\vdash}\otimes \zo\otimes {\mathcal K}$
which satisfies the following condition:
there is  $\af\in Hom_{\Lambda}(\underline{K}(A), \underline{K}(B^{\vdash}\otimes {\mathcal Z}_0\otimes {\mathcal K}))$ and  there are  two sequences of approximately  multiplicative \cpc s\\
$h_n,\, h_n': A\to \C \cdot 1_{B^{\vdash}}\otimes \zo\otimes {\mathcal K}$
such that,
for any finite subset ${\mathcal P}\in \underline{K}(A),$ there exists $n_0\ge 1$ such that
\beq\label{sec9-1205-af}
[\phi_n]|_{\cal P} +[h_n]|_{\mathcal P}=\af|_{\mathcal P}  +[h_n']|_{\cal P} \rforal n\ge n_0.
\eneq
Let $\phi=\{\phi_n\}$ and $\psi=\{\psi_n\}$ be two asymptotic sequential morphisms from $A$ to $B.$
We say $\phi$ and $\psi$ are equivalent and write $\phi\sim \psi$ if there exist two sequences
of approximately multiplicative \cpc s $h_n, h_n': A\to \C\cdot 1_{B^{\vdash}}\otimes \zo\otimes {\mathcal K}$ and
a sequence of unitaries $u_n\in {\widetilde{B^{\vdash}\otimes {\mathcal Z}_0\otimes {\mathcal K}}}$ such that
$$
\lim_{n\to\infty}\|u_n^*\diag(\phi_n(a), h_n(a))u_n-\diag(\psi_n(a), h_n'(a))\|=0\rforal a\in A.
$$
Denote by $\la \phi\ra $  the equivalence class of asymptotic sequential morphisms represented by $\phi,$ and
by ${\mathcal E}(A,B)$ the set of all equivalence classes of asymptotic sequential morphisms
from $A$ to $B.$

{{Consider the split short exact sequence
$$
0\to B\otimes \zo\otimes {\cal K}{\stackrel{\imath}{\longrightarrow}} B^{\vdash}\otimes \zo\otimes {\mathcal K}
{\stackrel{\pi}{\rightleftarrows}}_s \C\cdot 1_{B^{\vdash}}\otimes \zo\otimes {\mathcal K}\to 0.
$$
It gives the following split short exact sequence:
\beq\label{DbetaA-2}
0\to KL(A, B\otimes \zo){\stackrel{[\imath]}{\longrightarrow}} KL(A, B^{\vdash}\otimes \zo)
{\stackrel{[\pi]}{\rightleftarrows}}_{[s]} KL(A, \C\cdot 1_{B^{\vdash}}\otimes \zo)\to 0.
\eneq
Define $\lambda_B: {\rm Hom}_{\Lambda}(\underline{K}(A), \underline{K}(B^{\vdash}\otimes \zo))\to
{\rm Hom}_{\Lambda}(\underline{K}(A), \underline{K}(B\otimes \zo)
$
by
\beq\label{lambda-def-91}
\lambda_B(x)=x-[s]\circ [\pi](x)\rforal x\in KL(A, B^{\vdash}\otimes \zo).
\eneq
Note that
$$
s\circ \pi\circ g_n=g_n
$$
for any \cpc \, $g_n: A\to
\C\cdot 1_{B^{\vdash}}\otimes \zo\otimes {\mathcal K}\subset B^{\vdash}\otimes \zo\otimes {\mathcal K}.$}}

{{Let $\la \phi\ra\in {\cal E}(A,B)$ be represented by $\{\phi_n\}$ and
let $\af$ be as in \eqref{sec9-1205-af}.}}
{{Then, for any fixed finite subset ${\cal P}\subset \underline{K}(A),$
\beq\label{DbetaA-5}
\lambda_B\circ ([\phi_n]|_{\cal P}+[h_n]|_{\cal P}-[h_n']|_{\cal P})=[\phi_n]|_{\cal P}-[s\circ \pi\circ \phi_n]|_{\cal P}=
\lambda_B\circ \af|_{\cal P}
\eneq
for all  $n\ge n_0({\cal P})$ for some integer $n_0({\cal P}).$
If $\{\psi_n\}$ is another representation of $\la \phi\ra,$
then,  there exist two sequences
of approximately multiplicative \cpc s $g_n, g_n': A\to \C\cdot 1_{B^{\vdash}}\otimes \zo\otimes {\mathcal K}$ and
a sequence of unitaries $u_n\in {\widetilde{B^{\vdash}\otimes {\mathcal Z}_0\otimes {\mathcal K}}}$ such that
$$
\lim_{n\to\infty}\|u_n^*\diag(\phi_n(a), g_n(a))u_n-\diag(\psi_n(a), g_n'(a))\|=0\rforal a\in A.
$$
}}
{{Thus there is an integer $n_1({\cal P})\ge 1$ such that
\beq\label{DbetaA-6}
[\phi_n]|_{\cal P}+[g_n]|_{\cal P}&=&[\psi_n]|_{\cal P}+[g_n']|_{\cal P}\andeqn\\
{[}s\circ \pi \circ \phi_n{]}|_{\cal P}+[g_n]|_{\cal P}&=&[s\circ \pi \circ \psi_n]+[g_n']|_{\cal P}\rforal n\ge {{n_1({\cal P}).}}
\eneq
Therefore
\beq\label{DbetaA-7}
[\psi_n]|_{\cal P}-[s\circ \pi\circ \psi_n]|_{\cal P})&=&
([\phi_n]|_{\cal P}+[g_n]|_{\cal P}-[g_n']|_{\cal P})\\
&&\hspace{0.2in}-([s\circ \pi\circ \phi_n]|_{\cal P}+[s\circ \pi\circ g_n]|_{\cal P}-
[s\circ \pi\circ g_n']|_{\cal P})\\
&=&[\phi_n]|_{\cal P}-[s\circ \pi\circ \phi_n]|_{\cal P}=\lambda_B\circ \af|_{\cal P}
\eneq
for all $n\ge \max \{n_0({\cal P}), n_1({\cal P})\}.$
Thus {{$\bt_A: {\cal E}(A, B)\to Hom_{\Lambda}(\underline{K}(A), \underline{K}(B\otimes {\mathcal Z}_0\otimes {\mathcal K}))$ given by }} $\bt_A(\la \phi\ra)={{\lambda_B}}\circ \af$,   is {{well}} defined.}}

If $\phi$ and $\psi$ are two asymptotic sequential morphisms from $A$ to $B,$
we define $\phi\oplus \psi$ by $(\phi\oplus \psi)(a)=\diag(\phi(a), \psi(a))$ for all $a\in A.$
Here we identify $M_2({\cal K})$ with ${\cal K}$ in the usual way.
We define  $\la \phi\ra +\la \psi \ra =\la \phi\oplus \psi \ra.$ This clearly defines an addition in
${\mathcal E}(A,B).$ Let $\la \psi \ra \in {\mathcal E}(A,B)$ be represented
by $\{\psi_n\}$ whose images are in $\C \cdot 1_{B^{\vdash}}\otimes \zo\otimes {\mathcal K}.$
Then, for any $\la \phi\ra \in  {\mathcal E}(A,B),$
$\la \phi\oplus \{\psi_n\}\ra =\la \phi \ra.$ In other words
that ${\mathcal E}(A,B)$ is a semigroup with zero represented by zero asymptotic morphism.

\end{df}

\begin{df}\label{Dsyhomtop}
Denote $C=B^{\vdash}\otimes \zo\otimes {\cal K}.$
Let $C_{\infty}=l^{\infty}(C)/c_0(C).$
If $\phi=\{\phi_n\}$ is an asymptotic sequential morphism, then
we may view $\phi$ as a \hm\, from $A$ to $C_{\infty}.$
Two asymptotic sequential morphisms $\phi$ and $\psi$ are {\it  homotopy}
if there is a \hm\, $H: A\to C([0,1], C_{\infty})$ such that
$\pi_0\circ H=\phi$ and $\pi_1\circ H=\psi,$
where $\pi_t: C([0,1], C_{\infty})\to C_{\infty}$ is the point-evaluation at $t\in [0,1].$
Since we assume that $A$ is amenable, there exists a \cpc\,
$L: A\to  C([0,1], l^{\infty}(C))$ such  that $\Pi\circ L=H,$
where $\Pi: l^{\infty}(C)\to C_{\infty}$ is the quotient map.
Denote by $P_n: l^{\infty}(C)\to C$ the $n$-th coordinate map.
Define $\Phi_n'=P_n\circ L,$  $n=1,2,....$
Define $\phi_n'=\pi_0\circ \Phi_n'$ and $\psi_n'=\pi_1\circ \Phi_n'.$
Note that
\beq
\lim_{n\to\infty}\|\phi_n(a)-\phi_n'(a)\|=0
\andeqn
\andeqn
\lim_{n\to\infty}\|\psi_n(a)-\psi_n'(a)\|=0\rforal a\in A.
\eneq
Therefore we may assume, \wilog, as far as in this section,
that $\phi_n$ and $\psi_n$ are homotopy for each $n.$
{{Fix a finite subset ${\cal F}\subset A$ and $\ep>0.$
There is a partition $0=t_0<t_1<\cdots t_m=1$ such
that
\beq\label{92-1}
\|\pi_t\circ L(a)-\pi_{t_i}\circ L(a)\|<\ep/2\rforal a\in \{cd, c, d: c, d\in {\cal F}\}.
\eneq
Since $\Pi\circ L=H,$ there is  $n_0>1$ such that $\pi_{t_i}\circ P_n\circ L$ is  ${\cal F}$-$\ep/2$-multiplicative
for all $a, b\in {\cal F}$ for all $n\ge n_0,$  $i=0,1,...,m.$
It follows from \eqref{92-1} that $\pi_t\circ P_n\circ L$ is ${\cal F}$-$\ep$-multiplicative.
In other words, $\phi_n$ and $\psi_n$ are connected by a path of ${\cal F}$-$\ep$-multiplicative \cpc s
for all large $n.$
}}

\end{df}

\begin{df}\label{DFixA}
We now fixed a separable amenable \CA\, $A$ satisfying the UCT with the following property:
There is a map $T: A_+\setminus \{0\}\to \N\times \R_+\setminus \{0\}$ and
a sequence of approximately multiplicative \cpc s $\phi_n: A\to {\cal W}$ such
that, for any finite subset ${\mathcal H}\subset A_+\setminus \{0\},$
there exists an integer $n_0\ge 1$ such that
$\phi_n$ is
$T$-${\mathcal H}$-full (see 5.5 of \cite{eglnp1} and 7.8 of \cite{GLp1}) for all $n\ge n_0.$

{\it For the rest of this section, $A$ is as above.}
\end{df}

\begin{lem}\label{LEabsorb}
Let $\{\phi_n\}$ be  an asymptotic  sequential morphism
from
$A$ to $B^{\vdash}\otimes \zo\otimes {\cal K}$ such that
the image of $\phi_n$ are all contained in $B^{\vdash}\otimes {\cal W}\otimes {\cal K}.$
Then
%
$\la \{\phi_n\} \ra=0.$
\end{lem}

\begin{proof}
Let $\ep>0$ and ${\cal F}\subset A$ be a finite subset.
Let $T$ be given in \ref{DFixA}.
Write $T(a)=(N(a), M(a))$ for all $a\in A_+\setminus \{0\}.$
We will  apply \ref{CLuniq}.

Let $\dt>0,$ ${\cal G}$ be a finite subset, ${\cal H}\subset A_+\setminus \{0\}$ be a finite subset
and let $K \ge 1$ be an integer {{as}} required by \ref{CLuniq}
for
$T.$

Suppose that $\phi_n: A\to B^{\vdash}\otimes {\cal W}\otimes {\cal K}$ is a
${\cal G}$-$\dt$-multiplicative \cpc. We may assume, \wilog, that
the image of $\phi_n$ lies in\linebreak
 $M_{k(n)}(B^{\vdash}\otimes {\cal W}).$
Choose  an asymptotic sequential morphism $\{\psi_n\}$ from $A$ to  ${\cal W}$
given by \ref{DFixA}.
We may assume that $\psi_n$ is
${\cal G}$-$\dt$-multiplicative and is $T$-${\cal H}$-full.
Let $\psi_0: {\cal W}\to \C\cdot 1_{B^{\vdash}}\otimes {\cal W}\otimes {\cal K}$ be a
\hm. By replacing $\{\psi_n\}$ by $\{\psi_0\circ \psi_n\},$ we
assume that the image of $\psi_n$ are in $\C\cdot 1_{B^{\vdash}}\otimes {\cal W}\otimes {\cal K}.$
Define ${\bar \psi_n}: A\to M_{k(n)}({\cal W})$ by
\vspace{-0.1in}$$
{\bar \psi}_n(a)=\diag(\overbrace{\psi_n(a), \psi_n(a),...,\psi_n(a)}^{k(n)})\rforal a \in A.
$$
By viewing ${\bar \psi}_n$ as a map from $A$ to $M_{k(n)}((\C \cdot 1_{B^{\vdash}})\otimes {\cal W}),$
it is easy to check that
${\bar \psi}_n$ is $T$-${\cal H}$-full
in $M_{k(n)}((\C \cdot 1_{B^{\vdash}})\otimes {\cal W}) $ (see the proof of \ref{CLuniq}).

{{Then, by \ref{CLuniq},
there exists an integer $K$ and a unitary $v\in M_{(K+1)k(n)}(B^{\vdash}\otimes {\cal W})^\sim
\subset M_{(K+1)k(n)}(B^{\vdash}\otimes \zo)^\sim$
such that
$$
\|v^*\diag(\phi_n(a), \Psi_n(a))v-\diag(0,
\Psi_n(a))\|<\ep\rforal a\in {\cal F},
$$
{{ where $\Psi_n(a)={\bar \psi}_n(a)\otimes 1_K.$ Lemma then follows.}}}}

\end{proof}

\begin{prop}\label{PEab1}
${\mathcal E}(A,B)$ is an abelian group.

\end{prop}

\begin{proof}
Define  an endomorphism  $\iota^{\circledast}$ on $B^{\vdash}\otimes \zo \otimes {\mathcal K}$
by
$$
\iota^{\circledast}(a\otimes b\otimes c)=a\otimes j^{\circledast}(b)\otimes c\rforal a\in B^{\vdash}, b\in {\mathcal Z}_0
\andeqn c\in {\cal K}
$$
(see \ref{Tnegative} for the definition of $j^{\circledast}$).
Let $\phi=\{\phi_n\}$ be an asymptotic sequencial morphism from $A$ to $B^{\vdash}\otimes {\mathcal Z}_0\otimes {\mathcal K}.$
Let $\psi_n: A\to B^{\vdash}\otimes \zo\otimes {\cal K}$ be defined by
$$
\psi_n(a)=\iota^{\circledast}\circ \phi_n(a)\rforal a\in A.
$$
Define ${{H:}} B^{\vdash}\otimes \zo \otimes {\mathcal K}\to M_2(B^{\vdash}\otimes \zo \otimes {\mathcal K}))$
by
$$
H(a\otimes b\otimes c)=a\otimes (\phi_{wz}\otimes {\rm id}_{M_2}) (\diag(\phi_{zw}(b), \phi_{zw}(b))\otimes c \rforal a\in B^{\vdash}, b\in \zo \andeqn c\in {\cal K}.
$$
It follows from \ref{LZ0group} that there exists a sequence of unitaries
$\{u_n\}\subset {\widetilde{B^{\vdash}\otimes \zo\otimes {\cal K}}}$ such that
$$
{\rm Ad}\, u_n\circ H({{\phi_n(x)}})=\lim_{n\to\infty}\diag(\phi_n(x), \psi_n(x))\rforal {{x\in A.}}
$$
It follows that $\{\phi_n\oplus \psi_n\}$ is approximately unitarily equivalent to $\{H\circ \phi_n\}.$
By \ref{LEabsorb}, $\la \phi_n\oplus \psi_n\ra =0.$

\end{proof}

\begin{df}\label{Dfunctor}  Fixed $A$ as in \ref{DFixA}, we will consider ${\mathcal E}(A, B)$ for separable \CA\, $B$, and denote ${\mathcal E}(A, B)$ by ${\mathcal E}_A(B).$
Suppose that  $B$ and $C$ are separable \CA s and $h: B\to C$ is a \hm.
Define ${{h_*:}}={\mathcal E}_A(h): {\mathcal E}_A(B)\to {\mathcal E}_A(C)$ by
${\mathcal E}_A(h)(\la \phi\ra)=\la \{h\circ \phi_n\}\ra,$ where $\{\phi_n\}$ is a representation
of $\la \phi \ra$ and where we also use $h$ for $h^{\sim}\otimes {\rm id}_{\zo\otimes {\cal K}}.$ This gives a \hm\, from the abelian group ${\mathcal E}_A(B)$ to the abelian
${\mathcal E}_A(C).$
Clearly ${\mathcal E}_A({\rm id}_B)={\rm id}_{{\mathcal E}_A(B)}.$
If $D$ is another \CA\, and $h_1: C\to D$ is a \hm, then
one checks that ${\mathcal E}_A(h_1\circ h)={\mathcal E}_A(h_1)\circ {\mathcal E}_A(h).$
\end{df}

\begin{thm}\label{Thomtopy}
  ${\mathcal E}(A,-)={\mathcal E}_A(-)$ is a covariant functor from separable \CA s to abelian groups which is
  homotopy invariant and stable, i.e., ${\cal E}_A(D)={\cal E}_A(D\otimes {\cal K}).$
\end{thm}

\begin{proof}
From \ref{PEab1} and \ref{Dfunctor}, {{${\mathcal E}_A(-)$}} is a covariant functor from separable \CA s to abelian groups.
It is obviously stable. We will show it is homotopy invariant.

Fix a \CA\, $B.$ Set $C=B^{\vdash}\otimes \zo\otimes {\cal K}.$
Let $\phi$ and $\psi$ be two homotopy asymptotic sequential morphisms from
$A$ to $C.$
Let $\dt>0$ and ${\cal G}\subset A.$

Fix a large integer $n.$ As discussed in \ref{Dsyhomtop}, we may assume that
there exists ${\cal G}$-$\dt$-multiplicative \cpc\, $L_n: A\to C([0,1], C)$ which is  such that
$\pi_0\circ L_n= \phi_n$ and $\pi_1\circ L_n=\psi_n.$

Let $\ep>0$ and ${\cal F}\subset A$ be a finite subset.

Let ${\cal F}_1$ be a finite subset which contains ${\cal F}.$
Let ${\cal P}: 0=t_0<t_1<\cdots <t_k=1$ be a partition such that
\beq\label{Thmtp-0}
\|\pi_t\circ L_n(g)-\pi_{t_i}\circ L_n(g)\|<\ep/4\rforal g\in {\cal F}_1
\eneq
for all $t\in [t_{i-1}, t_{i+1}],$ $i=1,2,...,k.$
Put $\gamma_i=\pi_{t_i}\circ L_n,$ $i=0,1,2,...,k.$
Define
$\Phi_n, \Psi_n, \Phi_n',\Psi_n': A\to M_{2k+1}(C)$ as follows.
\beq\label{Thmtp-1}
\Phi_n(a) &=& \diag(\gamma_0(a), \iota^{\circledast}\circ \gamma_1(a), \gamma_{1}(a), ...,\iota^{\circledast}\circ \gamma_{k}(a),
\gamma_{k}(a)),\\
{{\Phi_n'(a)}}&=&\diag(\gamma_0(a), \iota^{\circledast}\circ \gamma_0(a), \gamma_{1}(a), ...,\iota^{\circledast}\circ \gamma_{k-1}(a),
\gamma_k(a)),\\
\Psi_n'(a)&=&\diag(\gamma_k(a),   \iota^{\circledast}\circ \gamma_0(a),\gamma_0(a),..., \iota^{\circledast}\circ \gamma_{k-1}(a),\gamma_{k-1}(a)),\\
\Psi_n(a)&=&\diag(\gamma_k(a), \iota^{\circledast}\circ \gamma_1(a), \gamma_1(a),...,\iota^{\circledast}\circ\gamma_{k}(a), \gamma_{k}(a))
\eneq
for all $a\in A.$
We estimate that, by \eqref{Thmtp-0},
\beq\label{Thmtp-10}
\|\Phi_n(g)-\Phi_n'(g)\|<\ep/4
\andeqn
\|\Psi_n(g)-\Psi_n'(g)\|< \ep/4\rforal g\in {\cal F}_1.
\eneq
There is also a unitary $u\in M_{2k+1}({\tilde{C}})$ such that
\beq\label{Thmtp-11}
\|{\rm Ad}\, u\circ \Phi_n'(g)-\Psi_n'(g)\|<\ep/4\rforal g\in {\cal F}_1.
\eneq
It follows that
\beq\label{Thmtp-12}
\|{\rm Ad}\, u\circ \Phi_n(f)-\Psi_n(f)\|<3\ep/4\rforal f\in {\cal F}.
\eneq
Define   $\Theta: A\to M_{2k}(C)$ by
$$
\Theta(a) =\diag(\iota^{\circledast}\circ \gamma_1(a), \gamma_1(a),...,\iota^{\circledast}\circ\gamma_{k}(a), \gamma_{k}(a))
$$
for all $a\in A.$  Then \eqref{Thmtp-12} becomes
\beq\label{Thmtp-13}
\|{\rm Ad}\, u\circ\diag(\phi_n(g), \Theta(g))-\diag(\psi_n(g), \Theta(g))\|<3\ep/4\rforal g\in {\cal F}_1.
\eneq
On the other hand, by \ref{LZ0group},
there exists a \hm\, $H: B^{\vdash}\otimes \zo\otimes {\cal K}\to B^{\vdash}\otimes {\cal W}\otimes {\cal K}$
and ${\cal G}$-$\dt$-multiplicative \cpc\, ${{\Gamma}}_n: A\to C$
such that
\beq\label{Thmtp-14}
\|H\circ {{\Gamma}}_n(g)-\Theta(g)\|<\ep/8\rforal g\in {\cal F}_1
\eneq
{{($\Gamma_n=\diag(\gamma_1, \gamma_2,...,\gamma_k)$).}}
Finally, we obtain that
$$
\|{\rm Ad}\, u\circ \diag(\phi_n(f), H\circ {{\Gamma}}_n(f))-\diag(\psi_n(f),H\circ {{\Gamma}}_n(f)\|<\ep
$$
for all $f\in {\cal F}.$
Since the image of $H\circ \Gamma_n$ are  in  $B^{\vdash}\otimes {\cal W}\otimes {\cal K},$
the above implies that
that $\la \phi\ra =\la \psi \ra.$
\end{proof}

The proof of the following is essentially the same as that in 6.1.4 of \cite{Lnbk}.
\begin{prop}\label{PEexact}
If
\beq\label{split-exact}
0\to J {\stackrel{j}{\rightarrow}} D{\stackrel{\pi}{\rightarrow}} D/J\to 0
\eneq
 is a split short exact sequence of separable \CA s, then
$$
{\cal E}(A, J){\stackrel{j_*}{\longrightarrow}} {\cal E}(A, D){\stackrel{\pi_*}{\longrightarrow}} {\cal E}(A, D/J)
$$
is exact in the middle.
\end{prop}

\begin{proof}
Suppose that $\la \phi \ra \in {\cal E}(A, J)$ which can be represented by an asymptotic
sequential morphism
$\{\phi_n\}$ which maps $A$ to $J^{\vdash}\otimes \zo\otimes {\cal K}.$
Then $\pi\circ j\circ \phi_n$ has image in $\C\cdot 1_{(D/J)^{\vdash}}\otimes \zo\otimes {\cal K}.$
It follows from the definition that $\pi_*\circ j_*=0.$

Now assume that $\la \phi \ra \in {\cal E}(A,D)$ which is represented by $\{\phi_n\}.$
\Wlog, we may assume that ${\rm im}\phi_n\in M_{k(n)}(D^{\vdash}\otimes\zo
)$
for some sequence $\{k(n)\}.$

Suppose that $\pi_*(\la \phi \ra)=0.$ Thus we may assume that there exist
two asymptotic sequential morphisms $h_n, h_n': A\to \C\cdot 1_{(D/J)^{\vdash}}\otimes \zo\otimes {\cal K}$ and a
sequence of unitaries
$u_n\in ((D/J)^{\vdash}\otimes \zo\otimes  {\cal K})^{\sim}$ such that
\beq\label{PExact-2}
\lim_{n\to\infty}\|u_n^*\diag(\pi\circ \phi_n(a), h_n(a))u_n-h_n'(a)\|=0\rforal a\in A.
\eneq
By the proof of \ref{PEab1}, by adding $\iota^{\circledast}\circ h_n'$ (to both terms),  we
may assume that $[h_n']|_{\cal P}=0$ for all $n\ge n_0.$
We  also assume that there exists $\af\in Hom_{\Lambda}(\underline{K}(A), \underline{K}(D^{\vdash}\otimes {\mathcal Z}_0)$ such that, for any finite subset ${\cal P}\subset \underline{K}(A)$ and for all $n\ge n_0$ (for some $n_0\ge1$),
\beq\label{916-nnn}
[\phi_n]|_{\cal P} +[h_n]|_{\mathcal P}=[J_D]\circ\lambda_D\circ \af|_{\mathcal P}  +[h_n'']|_{\cal P} \,\,\,{\text{(see \eqref{lambda-def-91}
for the definition of $\lambda_D$)}},
\eneq
{{where
$\{h_n''\}$
 is a sequence  approximately multiplicative
\cpc s from $A$ to $\C\cdot 1_{(D)^{\vdash}}\otimes \zo\otimes {\cal K},$ and
we also view
$h_n$ and $h_n'$ as maps from $A$ to  $\C\cdot 1_{(D)^{\vdash}}\otimes \zo\otimes {\cal K},$
and where $J_D: D\otimes \zo\to D^{\vdash}\otimes \zo$ is the embedding.}}
{{Thus, combining  with \eqref{916-nnn}, $[\pi]\circ (\lambda_D(\af))=0.$}}

Denote $\Pi_{D/J}: ((D/J)^{\vdash}\otimes \zo\otimes  {\cal K})^{\sim}\to \C$ the quotient map.
\Wlog, we may assume that ${\rm im}\,(\phi_n\oplus h_n), {\rm im}\, h_n'\subset M_{K(n)}(\C\cdot 1_{(D/J)^{\vdash}}\otimes \zo).$
We may further assume that
$K(n)=2k(n).$
We may view  $\diag(u_n, u_n^*)\in ((D/J)^{\vdash}\otimes  \zo \otimes {\cal K})^{\sim}.$
Replacing $u_n$ by $\diag(u_n, u_n^*),$ we may assume that $u_n\in U_0(((D/J)^{\vdash}\otimes {\cal K})^{\sim}).$
Therefore, we may assume that
there exists a unitary
$z_n\in U((D^{\vdash}\otimes \zo\otimes {\cal K})^{\sim})$ such that
$\pi(z_n)=u_n.$

By identifying $\C\cdot 1_{(D/J)^{\vdash}}\otimes \zo\otimes {\cal K}$ with
$\C\cdot 1_{D^{\vdash}}\otimes \zo \otimes {\cal K}$ and $\C\cdot 1_{J^{\vdash}}\otimes \zo \otimes {\cal K},$
we may view $h_n, h_n': A\to \C\cdot 1_{D^{\vdash}}\otimes \zo\otimes {\cal K}$
as
well as $h_n, h_n': A\to \C\cdot 1_{J^{\vdash}}\otimes \zo\otimes {\cal K},$ whichever
it is convenient.

Let $\Pi: l^{\infty}(D^{\vdash}\otimes \zo\otimes {\cal K})\to l^{\infty}(D^{\vdash}\otimes \zo\otimes {\cal K})/c_0(D^{\vdash}\otimes \zo\otimes {\cal K})$ be the quotient map.
Let
$$
U=\{z_n\},\,\,\, Z=\Pi(U), \Phi=\{\phi_n\},\,\,\, H=\{h_n\}, H'=\{h_n'\},
$$
where we view $\Phi, H, H': A\to l^{\infty}(D^{\vdash}\otimes \zo\otimes {\cal K}).$
Then, by {{\eqref{PExact-2}}}
$$
Z^* (\Pi(\Phi(a)\oplus H(a)))Z-\Pi\circ H'(a)\in l^{\infty}(J^{\vdash}\otimes \zo \otimes {\cal K})/c_0(J^{\vdash}\otimes \zo \otimes {\cal K})
$$
for all $a\in A.$ Since $\Pi\circ H'(a),\, \Pi\circ H(a) \in \C\cdot 1_{J^{\vdash}}\otimes \zo\otimes {\cal K},$  it follows that
$$
Z^* (\Pi(\Phi(a)\oplus  0))
Z\in l^{\infty}(J^{\vdash}\otimes \zo\otimes {\cal K})/c_0(J^{\vdash}\otimes \zo\otimes {\cal K})
$$
for all $a\in A.$
Since $A$ is amenable, by \cite{CE}, there exists a \cpc\,
$L=\{l_n\}: A\to J^{\vdash}\otimes \zo \otimes {\cal K}$ such that
$\Pi\circ L={\rm Ad}\, U\circ (\Phi).$
Also
\beq\label{ln-n12/4}
\lim_{n\to\infty}\|\diag(l_n(a), h_n(a))-z_n^*(\diag(\phi_n(a), h_n(a)))z_n\|=0\rforal a\in A.
\eneq
Since \eqref{split-exact} is split exact,  by Proposition 5.2 of \cite{LnTrans04}, there is a unique $\bt\in {\rm Hom}_{\Lambda}(\underline{K}(A), \underline{K}(J\otimes \zo))$
such that $[j]\circ \bt=\lambda_D\circ \af.$
It follows (by \eqref{916-nnn} and \eqref{ln-n12/4}) that, viewing $l_n$ as maps from $A$ to $J^{\vdash}\otimes \zo\otimes {\cal K},$
 there exist two sequences of approximately multiplicative  \cpc s $H_n, H_n'': A\to \C\cdot 1_{J^{\vdash}}\otimes \zo\otimes {\cal K},$
for any finite subset ${\cal P}\subset \underline{K}(A),$
such that, for all $n\ge n_1$ (for some $n_1\ge 1$),
$$
[l_n]|_{\cal P} +[H_n]|_{\mathcal P}=[J_J]\circ \bt |_{\mathcal P}  +[H_n'']|_{\cal P}\,\,\,\,\,\,\,{\text{($J_J: J\otimes  \zo\to J^{\vdash}\otimes \zo
$ is the embedding)}}.
$$
So $\la \{l_n\}\ra$ is an asymptotic sequential morphism in ${\cal E}(A, J)$
and (by \eqref {ln-n12/4}) $j_*\la \{l_n\}\ra =\la \phi_n\ra$ which implies that
$\la \phi \ra$ is in the $j_*({\cal E}(A, J)).$

\end{proof}

\begin{prop}\label{PEsplit}
${\cal E}_A(-)$ is split exact.
\end{prop}

\begin{proof}
This is standard from \ref{Thomtopy} and \ref{PEexact} (see \cite{Hg}).
Let
$$
0\to J {\stackrel{j}{\longrightarrow}} D{\stackrel{\pi}{\longrightarrow}} D/J\to 0
$$
be a short exact sequence of separable \CA s.

Let us first assume that $D/J$ is contractible.
Then by \ref{Thomtopy}, ${\cal E}_A(D/J)=\{0\}.$
It follows from \ref{PEexact} that  $j_*$ gives {{a}} surjective  map from
${\cal E}_A(J)$ onto ${\cal E}_A(D).$

For \CA\, $C,$ denote by $S(C)=C_0((0,1), C).$
Then, by \ref{Thomtopy},
$$
{\cal E}_A(D/J)=0={\cal E}_A(S(D/J))
$$
Put
\vspace{-0.1in}\beq\label{PEsplit-2}
S(D, D/J)&=&\{(a,b)\in D\oplus C_0([0,1), D/J): \pi(a)=b(0)\}\andeqn\\
Z(J, D)&=&\{x\in C([0,1], D): x(0)\in J\}.
\eneq
We have the following exact sequence:
\beq\label{PEsplit-3}
0={\cal E}_A(S(D/J))\longrightarrow {\cal E}_A(S(D, D/J))\longrightarrow {\cal E}_A(D).
\eneq
Define $\pi': Z(J, D)\to C_0([0,1), D/J)$ by
$\pi'(f)(t)=\pi(f)(1-t)$ for $t\in [0,1).$ Note $\pi'(f)(1)=\pi(f)(0)=0$
for all $f\in Z(J,D).$
Define $\chi: Z(J,D)\to S(D, D/J)$ by
$$
\chi(f)=(f(1), \pi'(f))\rforal f\in Z(J, D).
$$
One obtains the short exact sequence:
$$
0\to C_0([0,1), J)\to Z(J, D)\to S(D, D/J)\to 0.
$$
This gives the following exact sequence:
\beq\label{PEsplit-4}
0={\cal E}_A(C_0([0,1), J)\longrightarrow {\cal E}_A(Z(J, D))\longrightarrow {\cal E}_A(S(D, D/J)).
\eneq
From \eqref{PEsplit-3} and \eqref{PEsplit-4}, it follows that  composition map
${\cal E}_A(Z(J, D))\to {\cal E}_A(S(D, D/J))\to {\cal E}_A(D)$
is injective.

However, $Z(J, D)$ is homotopically equivalent to $J.$ Moreover, one sees that the composition
$J\to Z(J,D)\to S(D, D/J)\to D$ coincides with $j.$  It follows
that $j_*$ is injective.

Thus we show that, when $D/J$ is contractible,   $j_*$
is an isomorphism from ${\cal E}_A(J)$ onto ${\cal E}_A(D).$

In general, let $\imath: J\to S(D, D/J)$ be defined by
$\imath(b)=(b,0)$ for $b\in J.$   Then
$S(D, D/J)/\imath(J)\cong C_0([0,1), D/J)$ which is contractible. So, from
what has been proved, $\imath_*$ is an isomorphism.

To see that ${\cal E}_A(-)$ is split exact, consider
the short exact sequence of separable \CA s:
$$
0\to J {\stackrel{j}{\longrightarrow}} D{\stackrel{\pi}{\rightleftarrows}}_s D/J\to 0.
$$
By \ref{PEexact},
$$
{\cal E}_A(J) {\stackrel{j}{\longrightarrow}} {\cal E}_A(D){\stackrel{\pi}{\longrightarrow}} {\cal E}_A( D/J)
$$
is exact in the middle. Since $\pi\circ s={\rm id}_{D/J},$
we check that $\pi_*\circ s_*=({\rm id}_{D/J})_*.$

It remains to show that $j_*$ is injective. Using the exact sequence
$$
{\cal E}_A(S(D/J))\to {\cal E}_A(S(D, D/J))\to {\cal E}_A(D),
$$
and identifying ${\cal E}_A(J)$ with ${\cal E}_A(S(D,D/J)),$ we see that ${\rm ker}\, j_*\subset
{\rm im}\, (\imath_1)_*$ where
$\imath_1: S(D/J)\to S(D,D/J)$ is the embedding.

 Let
 $$
 I=\{(s(b(0)), b)\in S(D, D/J): b\in C_0([0,1), D/J)\},
 $$
 where $s$ is the split map given above.
 Since $\pi\circ s={\rm id}_{D/J},$ $I\cong C_0([0,1), D/J)$ which is contractible.
 On the other hand, ${\rm im}\, \imath_1\subset I.$  Therefore $(\imath_1)_*=0.$
 Thus ${\rm ker}\, j_*=0.$ In other words, $j_*$ is injective.

\end{proof}

\section{An existence theorem}

\begin{df}\label{DbetaA}
Fix $A$ as in \ref{DFixA}. We assume that $A$ satisfies the UCT.
There is a \hm\, $\bt_A^B$ from
${\cal E}_A(B)$ to
$KL(A,B)$ defined as follows.

We will identify $KL(A,C)$ with  ${\rm Hom}_{\Lambda}(\underline{K}(A), \underline{K}(C))$
for any separable \CA\, $C$ (see \cite{DL}).
Let $\la \phi\ra\in {\cal E}_A(B):={\cal E}(A,B)$ be
represented  by an asymptotic morphism $\{\phi_n\}.$
{{Recall  (see \ref{DE(AB)}) that  $\beta_A(\la \phi_n\ra)=\lambda_B\circ \af$ is well defined which defines a \hm\, $\bt_A^B: {\cal E}_A(B)
\to KL(A,B).$}}
Consequently $\beta_A^B$ is a morphism which maps \CA\,
$B$ to the abelian
group
$ {\rm Hom}_{\Lambda}(\underline{K}(A), \underline{K}(B\otimes \zo)).$
If $B$ and $C$ are two \CA s and $h: B\to C$ is a \hm\, we have the following
commutative diagram:
 \begin{displaymath}
\xymatrix{
{\mathcal E}_A(B) \ar[r]^{{\mathcal E}_A(h)} \ar[d]_{\bt_A^B} & {\mathcal E}_A(C) \ar[d]^{\bt_A^C}\\
Hom_{\Lambda}(A, B\otimes \zo\otimes {\cal K})
\ar[r]_{[{\rm h}]} & Hom_{\Lambda}(A, C\otimes \zo\otimes {\cal K}).
}
\end{displaymath}
It follows  that
$$
\bt: {\cal E}_A(-)\to Hom_{\Lambda}(\underline{K}(A), \underline{K}(-\otimes \zo))
$$
is a natural transformation.

\end{df}

\begin{thm}\label{Thigson}
The transformation $\bt_A$ maps ${\cal E}_A(B)$ onto $Hom_{\Lambda}(A, B\otimes \zo)$
for
each separable \CA\, $B,$  if $A$ satisfies the UCT.
\end{thm}

\begin{proof}
By a theorem of Higson (Theorem 3.7 of \cite{Hg}),  since
${\cal E}_A(-)$ is a covariant functor from separable \CA s to abelian groups
which is homotopy invariant, stable and split exact (Section 8),
there is a unique transformation
$$
\af: KK(A,-)\to {\cal E}_A(-)
$$
such that
$
\af_A([{\rm id}_A]_{KK})=\la {\rm id}_A\ra.
$
Let $\gamma: KK(A, -)\to KL(A, -)$ be the natural transformation induced by
$\Gamma: KK(A,B)\to KL(A,B).$
We have
$$
\bt_A\circ \af_A([{\rm id}_A])=[{\rm id}_A]_{KL},
$$
where $\bt$ is defined in \ref{DbetaA}.
Since $\gamma([{\rm id}_A])=[{\rm id}_A],$ (the first $[{\rm id}_A]$ is in $KK(A,A)$
and the second is in $KL(A,A)$),  by the uniqueness of Higson's theorem (3.7 of \cite{Hg}),
$$
\bt\circ \af=\gamma.
$$
Since $\gamma(KK(A, B))=Hom_{\Lambda}(A, B\otimes \zo\otimes {\cal K}),$ if $A$ satisfies the UCT,
 $\bt_A: {\cal E}_A(B)\to KL(A,B)$
must be surjective for each $B.$

\end{proof}


\begin{lem}\label{LabsorbW}
Let $B$ a non-unital and separable simple \CA\, with continuous scale.
Let $\phi_0,\phi_1, \phi_2: {\cal W}\to M(B)/B$ be three non-zero
\hm s.
Then, for any $\ep>0,$ and any finite subset ${\cal F}\subset {\cal W},$
there exists a unitary $U\in M_2(M(B))$ such that
$$
\|\pi(U)^*\diag(\phi_1(a), \phi_0(a))\pi(U)-\diag(\phi_2(a), \phi_0(a))\|<\ep\tforal a\in {\cal F}.
$$
\end{lem}

\begin{proof}
It follows from \cite{Lncs1} that $M(B)/B$ is simple and purely infinite.

Fix a strictly positive element $a_W\in {\cal W}$ with $\|a_W\|=1.$
Let $b_0=\phi_0(a_W)$ and let $B_0=\overline{b_0(M(B)/B)b_0}.$

Since ${\cal W}$ and $B_0$ are both simple, there is  a map $T: {\cal W}_+\setminus\{0\}\to \N\times \R_+\setminus \{0\}$
such that $\phi_0: {\cal W}\to B_0$ is
$T$-${\cal W}_+\setminus\{0\}$-full.
Let $W_0=\phi_0({\cal W}).$ So $b_0\in W_0.$

Let ${\cal H}\subset {\cal W}_+\setminus \{0\}$ be a finite subset and $K\ge 1$  be an integer  as required by
{{Cor. 3.16 of \cite{eglnkk0}}}
for the
above given $T,$ $\ep/2$ (in place of $\ep$) and ${\cal F}.$

Note that ${\cal W}\otimes Q\cong {\cal W}.$ Moreover, the map from ${\cal W}$ to ${\cal W}\otimes {{1_Q}}$ which maps  $a$ to  $a\otimes {{1_Q}}$
then to ${\cal W}$ is approximately inner. To simplify notation, \wilog, we may assume
that $\phi_0: {\cal W}\to W_0\otimes Q$ has the form $\phi_0(a)\otimes {{1_Q}}.$
Let $e_1, e_2,...,e_K\in Q$ be mutually orthogonal and mutually equivalent projections
such that
$\sum_{i=1}^Ke_i=1_Q.$
Define $\phi_{0,i}: {\cal W}\to W_0\otimes e_i$ by
$$
\phi_{0,i}(a)=\phi_0(a)\otimes e_i\rforal a\in {\cal W}.
$$
Put $B_{0,1}=\overline{(b_0\otimes e_1){{(M(B)/B)}}(b_0\otimes e_1)}.$

Let $b_1=\phi_1(a_W), \, b_2=\phi_2(a_W)\in M(B)/B.$ Since ${\cal W}$ is projectionless,
${\rm sp}(a_W)=[0,1].$  Thus, since ${\cal W}$ is simple,
$b_1$ can not be invertible in
$M(B)/B.$  {{Let $D_1=\overline{b_1Ab_1}.$}} By Pedersen's double orthogonal complement theorem (Theorem 15 of \cite{PedSW}),
there is a projection $E_1\in {{M(B)/B}}$
such that $1_{M(B)/B}-E_1\in {{D^{\perp}}}$ {{is not zero}}
and $E_1b_1=b_1E_1=b_1.$
Similarly, one obtains a projection $E_2\in M(B)/B$
such that $1_{M(B)/B}-E_2\not=0$ and $E_2b_2=b_2E_2=b_2.$
Using the fact that $M(B)/B$ is   purely infinite simple again, one obtains a  unitary $w_1\in M(B)/B$ such that
$$
w_1^*E_2 w_1\le E_1.
$$
Thus, \wilog,  {{replacing $\phi_2$ by ${\rm{Ad}}\, w_1\circ \phi_2$,}} one may assume that
$E_2\le E_1.$

Since $M(B)/B$ is purely infinite and simple,   ${{E_1}}\lesssim p_0'$ for some
projection $p_0'\in B_0.$  Thus we obtain a  unital hereditary \SCA\,
$B_{00}\subset M(B)/B$ such that, we may view that
$\phi_1, \phi_2: {\cal W}\to B_{00}$ and $\phi_{0,1}: {\cal W}\to B_{00}$ is a
$T$-${\cal W}_+\setminus\{0\}$-full.
Moreover, we view
$$
\phi_0(a)=\diag(\overbrace{\phi_{0,1}(a),\phi_{0,1}(a),...,\phi_{0,1}(a)}^K)\rforal a\in {\cal W}.
$$
Furthermore, $M_{K+1}(B_{00})$ is a unital \SCA\, of $M_2(M(B)/B)$
such that $1_{M_{K+1](B_{00})}}$  is not the unit of $M_2(M(B)/B).$
By applying {{3.16 of \cite{eglnkk0},}}
%
there is a unitary $u\in M_{K+1}(B_{00})\subset M_2(M(B)/B)$
such that
$$
\|u^*(\diag(\phi_1(a), \phi_0(a)))u-\diag(\phi_2(a), \phi_0(a))\|<\ep\rforal a\in {\cal F}.
$$
Since $1_{M_2}-1_{M_{K+1}(B_{00})}\not=0$ and $M_2(M(B)/B)$ is purely infinite and simple,
there exists a unitary $v\in (1_{M_2}-1_{M_{K+1}(B_{00})})(M(B)/B)(1_{M_2}-1_{M_{K+1}(B_{00})})$
such that $u\oplus v\in U_0(M_2(M(B)/B)).$
Thus we may assume that $u$ is a unitary in $U_0(M_2(M(B)/B)).$
{{Hence}} there is a unitary $U\in M_2(M(B))$ such that
$\pi(U)=u.$

\end{proof}

\begin{NN}[{\bf Construction of $\phi_W$}]\label{ConstphiW}
Let $B$ be a non-unital separable  simple \CA\, with  stable rank one, with $T(B)\not=\emptyset$ and with continuous scale.

Let $\{e_n\}\subset B\otimes {\cal Z}_0$ be an approximate identity
with
$$
e_{n+1}e_n=e_ne_{n+1}=e_n\tforal n\in \N.
$$
We may assume that $e_{n+1}-e_n\not=0$ for all $n\ge 1.$
Choose  $k(n)\ge 1$ such that
$$
\inf\{d_\tau(e_{4n}-e_{4n-1}): \tau\in T(B\otimes \zo)\}>{1\over{k(n)}},\,\,\, n=1,2,....
$$
Note that $\sum_{n=1}^{\infty}{1\over{k(n)}}<1.$
Put  $B_n:=\overline{(e_{4n}-e_{4n-1})(B\otimes \zo) (e_{4n}-e_{4n-1})}.$
Fix a strictly positive element $a_w\in {\cal W}$ with $\|a_w\|=1.$


It follows from \cite{Rl} that there is a \hm\,
$\phi_{0,n}: {\cal W}\to B_n$ such that
$$
d_\tau(\phi_{0,n}(a_w))={1\over{k(n)}}\rforal \tau\in T(B).
$$
Let $\phi_{even}, \phi_{odd}, \phi_W:
{\cal W}\to M(B\otimes \zo)$ be defined
by
\beq
&&\phi_{even}=\sum_{n=1}^{\infty} \phi_{0,2n},\,\,\,
\phi_{odd}=\sum_{n=1}^{\infty} \phi_{0, 2n+1}\andeqn\\
&&\phi_W=\sum_{n=1}^{\infty} \phi_{0,n}=\diag(\phi_{even}, \phi_{odd}).
\eneq

%
\end{NN}

\begin{prop}\label{AnuniqW}
Let $B$ be a non-unital separable  simple \CA\, with  stable rank one, with $T(B)\not=\emptyset$ and with continuous scale.
Fix an integer $k\ge 1.$
Let $j_{w,z}: {\cal W}\to  M_k(\zo)$ be an embedding  which maps strictly positive elements
to strictly positive elements
and $d: \zo\to \C\cdot 1_{M_k({\tilde B})}\otimes \zo\subset
M_k({\tilde B}\otimes \zo)\subset M(M_k(B\otimes \zo))$ be the embedding defined by $d(z)=1\otimes z$ for all $z\in \zo.$

Let $\ep>0$ and ${\cal F}\subset {\cal W}$ be a finite subset.
Then there is an integer $K\ge 1$ and a unitary $u\in M_{K+1}(M(M_k(B\otimes \zo)))$ such that
$$
\|u^*\diag(d_K\circ j_{w,z}(a),0)u-(d_K\circ j_{w,z}(a)\oplus \phi_{odd}(a))\|<\ep\tforal a\in {\cal F},
$$
where
$$
d_K(z)=\diag(\overbrace{d(z),d(z),...,d(z)}^K)\tforal z\in \zo.
$$

\end{prop}

\begin{proof}
The proof has the same spirit as that of \ref{LabsorbW}.
Keep in mind  that $B$ has continuous scale.
Therefore $M(M_k(B\otimes \zo))$ has only one (closed) ideal $M_k(B\otimes \zo)$
{{(see \cite{Lncs1}).}}
Since ${\cal W}$ is simple and
$d\circ j_{w,z}$ maps a strictly positive element
to that of $\C \cdot 1_{M_k({\tilde B})}\otimes \zo$  which is not in
$M_k(B\otimes \zo),$ $d\circ j_{w,z}(a)$ is full in $M(M_k(B\otimes \zo))$ for every
$a\in {\cal W}_+\setminus \{0\}.$
There is a map $T: {\cal W}_+\setminus \{0\}\to \N\times \R_+\setminus \{0\}$
such that
$d\circ j_{w,z}$ is
$T$-${\cal W}_+\setminus \{0\}$-full in $M(M_k(B\otimes \zo)).$

Let $K\ge 1$ be the integer required by
{{Cor. 3.16 of \cite{eglnkk0}}} for $\ep/2$ (in place of $\ep$),
${\cal F}$ and $T.$
By applying {{3.16 of \cite{eglnkk0}}}
(and considering  $\phi_{odd}$ and zero map), one obtains (note that $M(M_k(B\otimes \zo))$ is unital)
a unitary $v\in M_{K+1}(M(M_k(B\otimes \zo)))$ such that
$$
\|u^*\diag(d_K\circ j_{w,z}(a),0)u-(d_K\circ j_{w,z}(a)\oplus \phi_{odd}(a))\|<\ep\tforal a\in {\cal F}.
$$

\end{proof}

%


\begin{lem}\label{Lesthalf}
For any $\ep>0,$ there is $\dt>0$ satisfying the following:
for any $e\in A_+$ with $\|e\| \le 1$ and any $a\in A$ with
$\|a\| \le 1,$
$$
\|e^{1/2}ae^{1/2}-ea\|<\ep
$$
whenever
$
\|ea-ae\|<\dt.
$
\end{lem}

In the following statement  {{and the proof}} we keep notations in \ref{ConstphiW} and \ref{AnuniqW}.

\begin{thm}\label{Tappmul}
Let $A$ be a non-unital separable  amenable \CA.
Let $\ep>0$ and ${\cal F}\subset A$ be finite subset.

There exists $\dt>0$ with $\dt<\ep/2,$  a finite subset ${\cal G}\subset A$
with ${\cal F}\subset {\cal G}$ and an integer $K\ge 1$ satisfying the following:
For any ${\cal G}$-$\dt$-multiplicative \cpc\, $\phi: A\to
M_k({\tilde B}\otimes \zo)$ (for any non-unital separable
simple \CA\, $B$ with continuous scale and any integer $k\ge 1$) such that
if  there are  \hm s
$\psi_{z,w}: M_k(\zo)\to {\cal W}$ and $\psi_{w,z}: {\cal W}\to M_k(\C\cdot 1_{{\tilde B}}\otimes \zo)\cong M_k(\zo)$
which
map strictly positive elements to strictly positive elements such that
$$
\|\pi\circ (\phi(a))-(\psi_{w,z}\circ \psi_{z,w}\circ \pi\circ (\phi(a))\|<\dt\rforal a\in {\cal G},
$$
where $\pi: M_k({\tilde B}\otimes \zo)\to M_k(\C\cdot 1_{{\tilde B}}\otimes \zo)$
is the quotient map,
then there exists an ${\cal F}$-$\ep$-multiplicative \cpc\, $L_0: A\to M_{K+2}(M_k(B\otimes \zo))$
and an ${\cal F}$-$\ep$-multiplicative \cpc\, $L_1: A\to M_{K+2}(M_k({\tilde B}\otimes \zo))$
such that
$$
\|L_0(a)\oplus L_1(a)-\phi(a)\oplus d_K\circ s\circ \phi^{\pi}(a)\|<\ep\rforal a\in {\cal F},
$$
where  $\phi^{\pi}=\psi_{w,z}\circ \psi_{z,w}\circ \pi\circ \phi,$ $s: M_k(\C\cdot 1_{\tilde{B}}\otimes \zo)\to M_k({\tilde B}\otimes \zo)$
is  the nature embedding, {{and furthermore, the following are true:}}
$$
\hspace{-1in}{{(1)}}~~~~~~~~~~~~~~~~~~L_0(a)=p_m^{1/2}(\phi(a)\oplus d_K\circ s\circ \phi^{\pi}(a))p_m^{1/2}
\tforal a\in A
$$
for some $m\ge m_0,$  where  $\{p_m\}$ is an approximate identity for  $M_{K+2}(M_k(B\otimes \zo))$ and,

{{{\rm{(2)}}}}~~~
there  are  ${\cal G}$-$\dt$-multiplicative \cpc\, $L_{00}: A\to {\cal W}$ and    $L_{0,0}({\cal F})$-$\ep/2$-multiplicative  \cpc\,
$L_{w,b}: {\cal W}\to M_{K+2}(M_k({\tilde B}\otimes \zo))$ such that
$L_1=L_{w,b}\circ L_{00}.$
\end{thm}

\begin{proof}
Fix $1/2>\ep>0$ and a finite subset ${\cal F}\subset A.$ We may assume
that ${\cal F}\subset A^{\bf 1}.$

Let ${\cal G}=\{ab: a, b\in {\cal F}\}\cup {\cal F}.$
Let $\{e_n\}\subset M_k(B)$ be  an approximate identity as described in \ref{ConstphiW}.
Let $\dt_1>0$ (in place of $\dt$) be in \ref{Lesthalf} for $\ep/64.$

Let $\dt=\min\{\dt_1/{{2^{12}}}, \ep/{{2^{12}}}\}.$
We view $M_k({\tilde B}\otimes \zo)$ as a  \SCA\, of $M(M_k(B\otimes \zo)).$
Suppose that $\phi: A\to M_k({\tilde B}\otimes \zo)$ is ${\cal G}$-$\dt$-multiplacative \cpc.
Suppose that there are  \hm s
$\psi_{z,w}:  M_k(\zo) \to {\cal W}$ and $\psi_{w,z}: {\cal W}\to M_k(\C\cdot 1_{{\tilde B}}\otimes \zo)$
such that
\beq\label{Tappmul-9}
\|\pi\circ \phi(a)-(\psi_{w,z}\circ \psi_{z,w}\circ \pi\circ (\phi(a)))\|<
\dt\rforal a\in {\cal G}.
\eneq

Recall that  $\phi^\pi=\psi_{w,z}\circ \psi_{z,w}\circ \pi\circ \phi.$
Put $\phi^W=\psi_{z,w}\circ \pi\circ \phi.$ Thus $\psi_{w,z}\circ \phi^W=\phi^{\pi}.$
Let $K$ be  the integer in \ref{AnuniqW} associated with
$\dt$ (in place of $\ep$) and ${{\phi^W({\cal G})}}\subset {\cal W}$ (in place of ${\cal F}$).

By applying \ref{LabsorbW},
 we obtain a unitary
$U_1\in M_{K+2}(M(M_k(B\otimes \zo)))$ such that
\beq\label{Tappmul-10}
\hspace{-0.3in}\|\Pi(U_1)^*\Pi\circ \phi_W(\phi^W(a))\Pi(U_1)-\diag(\Pi\circ d_{K+1}\circ\psi_{w,z}\circ \phi^W(a)), \Pi\circ \phi_{odd}(\phi^{W}(a))\|<\dt
\eneq
for all $ a\in {\cal G},$ {{where $\Pi: M_{K+2}(M(M_k({\tilde B}\otimes \zo)))\to
M_{K+2}(M((M_k(B\otimes \zo))/(M_k(B\otimes \zo))))$ is the quotient map.}}


Let $s: M_k(\C\cdot 1_{\tilde{B}}\otimes \zo)\to M_k({\tilde B}\otimes \zo)$
be the embedding
such that
$$
\pi\circ s(a)=a\rforal a\in M_k(\C\cdot 1_{\tilde{B}}\otimes \zo).
$$
Consider $L_{1,1}: A\to
{{M(M_k(B\otimes \zo)))}}$
defined by
$L_{1,1}= \phi_W\circ  \phi^{W}$  and  $L_{1,0}': A\to M_{K+2}({{M(M_k(B\otimes \zo))}})$ defined by
$$
L_{1,0}'(a)=\diag(d'_{K+1}\circ s\circ\psi_{w,z}\circ \phi^{W}(a)), \phi_{odd}(\phi^{W}(a))\rforal a\in A,
$$
where  notation $d_m'(c)$ means the following:
$$
d'_{m}(c)=\diag(\overbrace{c,c,...,c}^{m}).
$$
By \ref{AnuniqW},  there is another unitary $U_2\in M_{K+2}(M(M_k(B\otimes \zo)))$
such that
\beq\label{Tappmul-n10}
\|U_2^*L_{1,0}'(a)U_2-\diag(d'_{K+1}\circ s\circ\psi_{w,z}\circ \phi^W(a),0)\|<\dt
\rforal a\in {\cal G}.
\eneq
Define {{a \hm\, $L_{1,0}: A\to M_{K+1}(M_k({\tilde B}\otimes \zo))$ by}}
$$
L_{1,0}(a)=d'_{K+1}\circ s\circ\phi^{\pi}(a)\rforal a\in A.
$$
Put
$\Phi=\phi\oplus d'_{K}\circ s\circ \phi^{\pi}$
and $U=U_1U_2.$
By \eqref{Tappmul-10} {{and \eqref{Tappmul-n10},}} for each $a\in {\cal G},$ there exist $b(a),
b'(a)
\in M_{K+2}(M_k(B\otimes \zo))$  with
$\|b(a)\|\le 1,$
$\|b'(a)\|\le 1$
such that
\beq\label{Tappmul-11}
&&\|U^*L_{1,1}(a)U- L_{1,0}(a)+b(a)\|<2\dt\andeqn\\
&& \|U^*L_{1,1}(a)U-\Phi(a) +b'(a)\|<2\dt\rforal a\in {\cal G}.
\eneq

\vspace{-0.05in}
Put ${\bar e}_n=\diag(\overbrace{e_n,e_n,...,e_n}^{K+2}),$ $n=1,2,....$
Let $p_n=U^*{\bar e_n}U,$ $n=1,2,....$
Then $\{p_n\}$ is an approximate identity for $M_{K+2}(M_k(B\otimes \zo)).$
Let $S=\N\setminus {{\{4n-1,\, 4n: n\in \N\}}}.$
If $m\in S,$
\beq\label{Tappmul-12+1}
(1-p_m)(p_{4n}-p_{4n-1})&=&\begin{cases}  (p_{4n}-p_{4n-1})\,\,\, &\text{if}\,\,\, m<4n-1;\\
                                                                       0 \,\,\, &\text{if} \,\,\, {{m>4n}}\end{cases}\andeqn\\\label{Tappmul-12+2}
      p_m(1-p_m)(p_{4n}-p_{4n-1})&=&0\rforal m{\in {\cal S}}.
                                                                      \eneq

There is $N\ge 1$ such that, for any $m\ge N$ and $m\in S,$
\beq\label{Tappmul-13}
&&\|(1-p_m)(U^*L_{1,1}(a)U)-(1-p_m)L_{1,0}(a)\|<4\dt,\\\label{Tappmul-13+1}
&&\|(U^*L_{1,1}(a)U)(1-p_m)-L_{1,0}(a)(1-p_m)\| <4\dt,\\\label{Tappmul-13+2}
&&\|(1-p_m)(U^*L_{1,1}(a)U)-(1-p_m)\Phi(a)\|<4\dt\andeqn\\\label{Tappmul-13+3}
&&\|(U^*L_{1,1}(a)U)(1-p_m)-\Phi(a)(1-p_m)\| <4\dt \rforal a\in {\cal G}.
\eneq
Note that, by the construction of $\phi_W$ and \eqref{Tappmul-12+2}, {{if $m\in S,$}}
\beq\label{Tappmul-14}
(1-p_m)(U^*L_{1,1}(a)U)&=&(U^*L_{1,1}(a)U)(1-p_m)\\\label{Tappmul-14+1}
&=&(1-p_m)(U^*L_{1,1}(a)U)(1-p_m)\rforal a\in A.
\eneq
It follows from \eqref{Tappmul-13}, {{\eqref{Tappmul-13+1}, \eqref{Tappmul-13+2}, \eqref {Tappmul-13+3}}} and \eqref{Tappmul-14}, for all
{{$m\ge N$ and  $m\in S,$}}
\beq\label{Tappmul-15}
\hspace{-0.3in}\|p_m\Phi(a)-\Phi(a)p_m\|<8\dt{{\andeqn \|(1-p_m)L_{1,0}(a)-L_{1,0}(a)(1-p_m)\|<8\dt\,\, \text{for\,\,all}\,\, a\in {\cal G}.}}
\eneq
{{By the choice of $\dt_1$ and \ref{Lesthalf}, for all $a\in {\cal G},$
\beq\label{191026-1}
&&\|p_m^{1/2}\Phi(a)p_m^{1/2}-p_m\Phi(a)\|<\ep/64\andeqn\\\label{191026-2}
&&\|(1-p_m)^{1/2}L_{1,0}(a)(1-p_m)^{1/2}-(1-p_m)L_{1,0}(a)\|<\ep/64.
\eneq}}
Moreover, the map $a\mapsto (1-p_m)(U^*L_{1,1}(a)U)$ is ${\cal G}$-$\dt$-multiplicative.
{{By \eqref{191026-2} and \eqref{Tappmul-13},
$a\to (1-p_m)^{1/2}L_{1,0}(a)(1-p_m)^{1/2}$ is ${\cal F}$-$\ep$-multiplicative.}}
Define
$$
L(a)=p_m\Phi(a)+(1-p_m)(U^*L_{1,1}(a)U)\rforal a\in A.
$$
Then, by \eqref{Tappmul-13+2},
\beq\label{Tappmul-15+}
\|L(a)-\Phi(a)\|<4\dt\rforal a\in {\cal G}.
\eneq
{{Consequently,}}
\beq\label{Tappmul-16}
\|L(ab)-L(a)L(b)\|<{{8\dt \rforal a, b\in {\cal F}.}}
\eneq
We compute that
\beq\label{Tappmul-17}
L(ab)=p_m\Phi(ab)+(1-p_m)(U^*L_{1,1}(ab)U)\rforal a, b\in A,
\eneq
and, for all $a, b\in {\cal G},$ by \eqref{Tappmul-12+2}, \eqref{Tappmul-14+1} and \eqref{Tappmul-15},
\beq\nonumber
&&\hspace{-1in}L(a)L(b)=(p_m\Phi(a)+(1-p_m)(U^*L_{1,1}(a)U))(p_m\Phi(b)+(1-p_m)(U^*L_{1,1}(b)U))\\\nonumber
&=& p_m\Phi(a)p_m\Phi(b)+((1-p_m)((U^*L_{1,1}(a)U))(1-p_m)(U^*L_{1,1}(b)U))\\\nonumber
\hspace{0.2in}&\approx_
{8\dt+\dt}
& p_m\Phi(a)\Phi(b)p_m+(1-p_m)(U^*L_{1,1}(ab)U).
\eneq
Combining this with \eqref{Tappmul-17}, \eqref{Tappmul-16}
\beq\label{Tappmul-18}
\|p_m\Phi(ab)-p_m\Phi(a)\Phi(b)p_m\|<8\dt+{{8\dt+\dt=17}}\dt\rforal a, b\in {\cal F}.
\eneq
Therefore {{(see \ref{Lesthalf})}}
\beq\label{Tappmul-19}
\|p_m^{1/2}\Phi(ab)p_m^{1/2}-p_m^{1/2}\Phi(a)p_m^{1/2}p_m^{1/2}\Phi(b)p_m^{1/2}\|<
{{17\dt}}+3\ep/64<\ep/16.
\eneq
Define $L_0(a)=p_m^{1/2}\Phi(a)p_m^{1/2}$ and
$L_1(a)=(1-p_m)^{1/2}L_{1,0}(a)(1-p_m)^{1/2}.$
{{By \eqref{Tappmul-19}, $L_0$ is ${\cal F}$-$\ep$-multiplicative.}}
By \eqref{Tappmul-15+}, \eqref{Tappmul-13}, {{\eqref{Tappmul-13+2},}}  and the choice of $\dt_1,$
we finally have
$$
\|(L_0(a)+L_1(a))-\Phi(a)\|<\ep\rforal a\in {\cal F}.
$$
{{Let $L_{00}=\phi^W: A\to {\cal W}$ and $L_{w,b}: {\cal W}\to M_{K+2}(M_k({\tilde B}\otimes \zo))$
be defined by $L_{w,b}(b)=(1-p_m)^{1/2}(d_K'\circ s\circ \psi_{w,z}(b))(1-p_m)^{1/2}$ for $b\in {\cal W}.$
Then $L_1=L_{00}\circ L_{w,b}.$}}
\end{proof}


%
%

\begin{thm}\label{TExistence}
Let $A$ be a non-unital separable  amenable \CA\, which satisfies the UCT
and  satisfies the condition in \ref{DFixA} and let $B$ be
a separable simple \CA\, with continuous scale.
For any $\af\in KL(A,B),$
there exists an asymptotic sequential morphism $\{\phi_n\}$ from
$A$ into $B\otimes \zo\otimes {\cal K}$ such that
$$
[\{\phi_n\}]=\af.
$$
\end{thm}

\begin{proof}

Let ${\cal P}\subset \underline{K}(A)$ be a finite subset.  Let $\ep>0$ and ${\cal F}\subset A$ be a finite subset.
We assume that any ${\cal F}$-$\ep$-multiplicative \cpc\, $L$ from $A,$
$[L]|_{\cal P}$ is well-defined.

If follows from \ref{Thigson} that there exist  sequences of approximately multiplicative \cpc s
$\Phi_n: A\to B^{\vdash}\otimes \zo \otimes {\cal K}$   and $\Psi_n: A\to \C \cdot 1_{B^{\vdash}}\otimes \zo\otimes {\cal K}$
such that, for any
finite subset ${\cal Q}\subset \underline{K}(A),$
$$
[\Phi_n]|_{\cal Q}=\af|_{\cal Q}+[\Psi_n]|_{\cal Q}
$$
for all sufficiently large $n,$  where $\Psi_n=s\circ \pi\circ \Phi_n$
(without loss of generality)
and $\pi: B^{\vdash}\otimes \zo \otimes {\cal K}\to \C \cdot 1_{B^{\vdash}}\otimes \zo\otimes {\cal K}$ be the quotient
map.
Fix a sufficiently large $n.$

Let $\{e_{i,j}\}$ be a system of matrix unit for ${\cal K}$ and
let $E$ be the unit of the unitization of  $1_{B^{\vdash}}\otimes \zo.$
By considering   maps $a\mapsto (E\otimes \sum_{i=1}^k e_{i,i})\Phi_n(a)(E\otimes \sum_{i=1}^k e_{i,i})$
and maps $a\mapsto (E\otimes \sum_{i=1}^k e_{i,i})\Psi_n(a)(E\otimes \sum_{i=1}^k e_{i,i}),$
\wilog, we may assume that
the image of $\Phi_n$ is in $M_k(B^{\vdash}\otimes \zo)$ and
that of $\Psi_n$ is also in $M_k(\C\cdot 1_{B^{\vdash}}\otimes \zo)$ for some
sufficiently large $k.$

Define $\imath^{\circledast}: B^{\vdash}\otimes \zo\otimes {\cal K}\to  B^{\vdash}\otimes \zo\otimes {\cal K}$
by defining $\imath^{\circledast}(b\otimes z\otimes k)=b\otimes j^{\circledast}(z)\otimes k$
for all $b\in B^{\vdash},$ $z\in \zo$ and $k\in {\cal K}$ {{(see \ref{Tnegative} for  $j^{\circledast}$).}}
Note that
$$
s\circ \pi(\Phi_n\oplus s\circ \pi\circ i^{\circledast}\circ \Phi_n)=\Psi_n\oplus s\circ \pi \circ i^{\circledast}\circ \Phi_n.
$$

Let $\dt>0$ and let ${\cal G}\subset A$ be a finite subset.

It follows from virtue of \ref{LZ0group},
replacing $\Phi_n$ by $\Phi_n\oplus  s\circ \pi\circ i^{\circledast}\circ \Phi_n$ and
replacing $\Psi_n$ by $\Psi_n\oplus  s\circ \pi\circ i^{\circledast}\circ \Phi_n ,$
and by implementing a unitary in unitization of $M_k(\C\cdot 1_{B^{\vdash}}\otimes \zo),$
we may assume that
$$
\|\pi\circ \Phi_n(g)-\phi_{w,z}\circ \phi_{z,w}\circ \pi(\Phi_n(a))\|<\dt\rforal g\in {\cal G}.
$$
and  $\Psi_n=s\circ \pi\circ \Phi_n$  approximately factors through ${\cal W},$ in particular,
$[\Psi_n]|_{\cal P}=0.$
In other words,
\vspace{-0.13in}\beq\label{Texist-n1}
[\Phi_n]|_{\cal P}=\af|_{\cal P}.
\eneq
By applying \ref{Tappmul}, we obtains  an integer $K\ge 1,$ ${\cal F}$-$\ep$-multiplicative \cpc s
$L_{0,n}: A\to M_k(B\otimes \zo),$ $L_{1,n}: A\to M_{(K+2)k}(B^{\vdash}\otimes \zo)$
and $L_{2,n}: A\to M_{(K+1)k}(B^{\vdash}\otimes \zo)$  such that
\beq\label{TE-9}
\|L_{0,n}(a)\oplus L_{1,n}(a)-\Phi_n(a)\oplus L_{2,n}(a)\|<\ep\rforal a\in {\cal F},
\eneq
where $L_{1,n}$ and $L_{2,n}$ factor through ${\cal W}.$
In particular,
\beq\label{TE-10}
[L_{1,n}]|_{\cal P}=[L_{2,n}]|_{\cal P}=0.
\eneq
It follows that, using \eqref{Texist-n1} and \eqref{TE-9}
\beq\label{TE-11}
[L_{0,n}]|_{\cal P}=\af|_{\cal P}.
\eneq

Choose $\phi_n=L_{0,n}$ (for all sufficiently large $n$).

\end{proof}

\section{Existence Theorem for  determinant maps}

\begin{lem}\label{Lcompactcon}
Let $A$ be a stably projectionless simple \CA\, such that
$Cu(A)={\rm LAff}_+({\tilde T}(A))$ with strict comparison for positive elements and with continuous scale.
Suppose $a, b\in A\otimes {\cal K}_+.$ Then
$\la a\ra \ll \la b\ra$ ($\la a \ra $ is compact contained in $\la b\ra$) if and only if,  there exists $\dt>0,$ for any ${{t}}\in T(A),$  there exists
a  neighborhood $O({{t}})\subset T(A)$
such that
\beq\label{Lcompact-1}
d_{ t}(b)>d_\tau(a)+\dt\rforal \tau\in O({{t}}).
\eneq
\end{lem}

\begin{proof}
The proof of ``if" part is a standard compactness argument (see, for example  5.4 of \cite{Lnloc}).
{{Recall that $T(A)$ is compact in this case (see \cite{Lncs2}).}}
Suppose that \eqref{Lcompact-1} holds.  Let $f_n\in {\rm LAff}_+({{\tilde{T}(A)}})$
such that $f_n\nearrow \sup f_n\ge  \la b\ra.$
Then, for each $t\in T(A),$ there exist $n_{t}$ such that
\beq
f_{{(n_t)}}(t)>d_{t}(b)-\dt/8.
\eneq
Since each $f_{n_t}$ is lower semi-continuous, there is a neighborhood
$U(t)\subset O(t)$ such that
\beq
f_{{(n_t)}}(\tau)>d_t(b)-\dt/4\rforal \tau\in U(t).
\eneq
It follows that
\beq
f_{n_t}(\tau)>d_t(b)-\dt/4>d_\tau(a)+\dt/2\rforal \tau\in U(t).
\eneq
There are finitely many such  $U(t_1), U(t_2),...,U(t_m)$ covers $T(A).$
Put $n_0=\max\{n_{t_i}: 1\le i\le m\}.$
Then, if $\tau\in U(t_j),$
\beq
f_{n_0}(\tau)> f_{n_{t_j}}(\tau)>d_\tau(a)+\dt/2.
\eneq
This implies that  $f_{n_0}>\la a\ra$ in ${\rm LAff}_+({\tilde T}(A)),$ {{which means $\la a\ra \ll \la b \ra.$}}


For the converse,  as in Lemma 2.2 of \cite{BT}
{{(see
7.2 of \cite{eglnp1}),}}
there exists a sequence of continuous $f_n\in \Aff_+(T(A))$ such that
$f_n\nearrow  b.$  {{Let $g_n=f_n-\frac{1}{n}$. Then $g_n\nearrow  b.$ }}The assumption
that $\la a\ra \ll \la b\ra$ implies that, for some $n_0\ge 1,$
$\la a\ra <{{g_{n_0}=f_{n_0}-\frac{1}{n_0}}}$ in $Cu(A).$
{{Hence}}
\beq
f_{n_0}(\tau)>d_\tau(a)+{{\frac{1}{n_0}}}\rforal \tau\in T(A).
\eneq
Since $f_{n_0}$ is continuous,  for each $t\in T(A),$ there is a
 neighborhood $O(t)$ such that
\beq
f_{n_0}(t)>d_\tau(a)+{{\frac{1}{2n_0}}}\rforal \tau\in O(t).
\eneq
Therefore
\beq
d_t(b)\ge  f_{n_0}(t)> d_\tau(a)+{{\frac{1}{2n_0}}}\rforal \tau\in O(t).
\eneq

%

\end{proof}




\begin{thm}\label{Tdert1}
Let $A$ be a stably projectionless simple  exact  \CA\, with strictly comparison for positive elements,
with stable rank one and with continuous scale such that $Cu(A)={\rm LAff}_+({\tilde T}(A)).$
Fix $1>\af>0$ and $1>\eta\ge3/4.$
Let
$$
h_\eta\in \{f\in C([0,1],\R): f(0)=\af f(1)\}
$$
such that $h_\eta$ is strictly increasing on $[0,\eta],$ $0\le h_\eta\le 1,$ $h_{\eta}(0)=0=h_{\eta}(1),$ and
$h_{\eta}(\eta)=1.$

Let $c\in A_+$ with $\|c\|=1$ and $b\in \overline{cAc}_+$ with $\|b\|=1.$
Suppose that there is a  non-zero \hm\, $\phi: R(\af,1)\to \overline{cAc}.$

Then, for any $\ep>0,$  there exists a \hm\, $\psi: R(\af,1)\to B:=\overline{cAc}$
such that
$$
\sup\{|\tau(\psi(h_\eta))-\tau(b)|: \tau\in T(A)\}<\ep.
$$

\end{thm}

\begin{proof}
Let $\ep>0.$
Since $A$ is stably projectionless, we may assume that ${\rm sp}(b)=[0, 1].$

Note that $(h_\eta)|_{[0, \eta]}: [0,\eta]\to [0,1]$ is a bijection. Define $h_\eta^{-1}: [0,1]\to [0, \eta]$ to be the inverse of $(h_\eta)|_{[0, \eta]}.$
Note that $h_\eta\circ h_\eta^{-1}={\rm id}_{[0,1]}.$
For each $f\in C([0,1], \R)_+,$ define
$
\gamma(f)(\tau)=\tau(f\circ h_\eta^{-1}(b))$ for all $\tau\in T(A).$

The $\gamma$ above gives an affine continuous map from $C([0,1], \R)\to \Aff(T(A)).$
Note that $\Aff({\tilde T}(R(\af, 1)))$ and ${\rm LAff}({\tilde T}(R(\af, 1)))_+$ are identified with
\beq\nonumber
&&\{(f,{{s}})\in C([0,1], \R)\oplus \R: f(0)={{s}}\af\andeqn f(1)={{s}}\}=
\{f\in C([0,1], \R): f(0)=\af f(1)\}\\\nonumber
&&\andeqn LSC([0,1], \R_+^{\sim})\oplus_\af \R_+^{\sim}
\eneq
(see \ref{DRaf1}),
respectively.
Let $\gamma_1=\gamma|_{{\rm Aff}({\tilde T}(R(\af,1)))_+}.$
Then
\beq\label{dert1-n2}
\gamma_1(h_\eta)(\tau)=\tau(h_\eta\circ h_\eta^{-1}(b))=\tau(b).
\eneq
It induces an {{order}} semi-group \hm\,
$\gamma_1: {\text{LAff}}({\tilde T}(R(\af,1)))_+\to {\text{LAff}}({\tilde T}(A))_+.$ Note
 $\gm_1$ takes continuous functions to continuous functions.
Let $r: Cu(R(\af,1))\to {\text{LAff}}({\tilde T}(R(\af,1)))_+$ be the rank function defined in \ref{DRaf1}.
Define  an order semi-group \hm\,
$\gamma_2: Cu(R(\af,1))\to  {\text{LAff}}_+({\tilde T}(A))$
by
\beq\label{dert1-n1}
\gamma_2(\la (f,{{s}})\ra)=(1-\ep/4)\gamma_1(r((f,{{s}}))) +(\ep/4) Cu(\phi)((f,{{s}}))).
\eneq

We  verify that $\gamma_2$ is a morphism in ${\bf Cu}.$
Since the rank function {{$r$}} preserves the suprema of increasing sequences,
it is easy to check that $\gamma_2$ also preserves the suprema
of increasing sequences.
Suppose that $\la f \ra{{=\la (f, s_f)\ra  \ll \la g\ra=\la (g, s_g)\ra}}$ in $Cu(R(\af ,1)).$
There is a sequence of $c_n\in Cu(R(\af, 1))$ such that
$r(c_n)$ is continuous and $r(c_n)\nearrow r(\la (g, {{s_g}})\ra)$ (see \ref{DRaf1} {{and \ref{CcuR}}}).
Note that $c_n$ can be identified with an element in $LSC([0,1],  (\R^{\sim}\setminus\{0\}\sqcup \Q)_+)\oplus_\af (\R^{\sim}\setminus \{0\}\sqcup \Q)_+$, at each point $t$, we identify $r(c_n)(t)$ with the corresponding values of $c_n(t)$ in $\R^{\sim}_+$---that is, $[{{s}}]\in \Q_+$ is regarded as ${{s}}\in \R_+$.
 Put $c=\sup_n r(c_n).$  {{Then $c=r(\la a\ra).$}}
 For any $\ep_1>0,$ $(1+\ep_1)r(c)\ge {{r(\la g\ra)}}.$
 Since $\la f\ra \ll \la g \ra,$ there exists
 $n_0\ge 1$ such that
 \beq
 (1+\ep_1)r(c_{n_0})\ge  {{r( \la f \ra).}}
 \eneq
 This, in particular, implies
 that $r(\la f\ra)$ is a bounded function.

Now let $z_n\in Cu(R(\af,1))$ such that $z_n\nearrow \sup z_n\ge \gamma_2((g,{{s_g}})).$
By \ref{Lcompactcon}, there exists $\dt>0$ such that, for each
$t\in T(A),$  there is a neighborhood $U(t)$ such that
\beq\label{Tdert1-n11}
d_t(\phi(g))>d_\tau(\phi(f))+\dt\rforal \tau\in {{U(t)}}.
\eneq
Choose $0<\ep_1 <\ep\cdot \dt/16(M+1).$
Then, for some $n_0\ge 1,$
\beq
(1-\ep/4)(1+\ep_1)\gamma_1(r(c_{n_0}))>(1-\ep/4)(1+\ep_1) \gamma_1( r(\la f\ra)).
\eneq
Since $r(c_{n_0})$ is continuous, $\gamma_1(r(c_{n_0}))$ is also continuous.
Therefore, for each $t\in T(A),$  there is a neighborhood $O(t)$ such that
\beq\label{Tdert1-n12}
(1-\ep/4)\gamma_1(r(c_{n_0}))(t)>(1-\ep/4) \gamma_1(r(\la f\ra))(\tau)-\ep_1\rforal \tau\in O(t).
\eneq
Put $N(t)=O(t)\cap U(t).$  Then, by \eqref{Tdert1-n11} and \eqref{Tdert1-n12} as well as \eqref{dert1-n1},
\beq\label{Tdert11-n13}
\gamma_2(\la g\ra)(t)>\gamma_2(r(\la f\ra ))(\tau)+\ep\dt/2\rforal \tau\in N(t).
\eneq
It follows from \ref{Lcompactcon} that $\gamma_2(\la f\ra )\ll \gamma_2(\la g\ra).$
This shows that $\gamma_2$ is a morphism in ${\bf Cu}.$
Since $K_0(R(\af,1))=\{0\},$ it induces a morphism $\gamma_2^{\sim}: Cu^{\sim}(R(\af,1))\to Cu^{\sim}(A)$
(see 7.3 of \cite{eglnp1}).

It follows from \cite{Rl} that there exists a \hm\, $\psi: R(\af, 1)\to {{B=\overline{cAc}}}$
 such
that
\beq\label{dert1-10}
d_\tau(\psi(g))=\gamma_2(\la g \ra)(\tau)\tforal \tau\in T(A)
\eneq
and for all $g\in R(\af,1)_+.$
There is ${{f}}\in R(\af, 1)_+$ such that $d_\tau({{f}})=\tau(h_\eta)$ for all $\tau\in T(R(\af, 1))$
(see \ref{2Rg15}).
Therefore
\beq
d_\tau(\psi({{f}}))&=&\lim_{n\to\infty}\tau(\psi({{f}}^{1/n}))=\lim_{n\to\infty}\tau\circ \psi({{f}}^{1/n})\\
 &=&d_{\tau\circ \psi}({{f}})=(\tau\circ \psi)(h_\eta)\rforal \tau\in T(A).
\eneq
Then, by \eqref{dert1-n1} and \eqref{dert1-10},
\beq\label{dert1-10+1}
|d_t(\psi({{f}}))-\gamma_1(r({{f}}))(t)| <\ep/4\rforal t\in T(R(\af,1)).
\eneq
Since $\gamma_1(r({{f}}))=\gamma_1(\hat{h_\eta}),$
we estimate that
$$
\sup\{|\tau\circ \psi(h_\eta)-\tau(b)|:\tau\in T(A)\}<\ep.
$$

The lemma follows.
\end{proof}

\begin{NN}\label{DdertF}
Let $A$ be the AH-algebras of real rank zero
with unique tracial {{state}} as associated with $B_T$ in section 6.
So $B_T=\lim_{n\to\infty}(B_n, \Phi_n).$
Write
$$
B_n=W_n\oplus  E_n\andeqn E_n=M_{(n!)^2}(A(W, \af_n)),\,\,\,n=1,2,....
$$
We may write $A=\overline{\cup_{n=1}^{\infty}C_n},$
where $C_n=C_{n,1}\oplus C_{n,2},$
$C_{n,1}\oplus C_{n,2}\subset C_{n+1,1}\oplus C_{n+1,2}$ and
$C_{n,1}$ is a circle algebra and $C_{n,2}$ is a homogeneous \CA\, with torsion $K_1.$
In fact, $C_{n,2}$ may be written as $M_{r(n)}(C(X_n)),$ where $X_n$ is a finite CW complex
with dimension no more than 3 and $r(n)\ge 6$ (see \cite{EG}).
In particular (by \cite{Ref}),  $K_1(C_{n,2})=U(C_{n,2})/U_0(C_{n,2}).$
We use $j_n: C_n\to C_{n+1}$ for the embedding.


 Fix a finitely generated subgroup $F_0\subset K_1(B_T).$
We may assume that $F_0'\subset K_1(B_n)$ such that
$(\Phi_{n, \infty})_{*1}(F_0')=F_0.$
Write $B_n=E_n\oplus W_n,$  where $E_n=M_{(n!)^2}(A(W, \af_n){{)}}.$
We also write
$$
C_{k,1}=M_{r(k(1))}(C(\T))\oplus M_{r(k(2))}(C(\T))\oplus\cdots M_{r(k(m_f)))}(C(\T)).
$$
with the identity of each summand {{being}} $p_j,$ $j=1,2,...,k(m_f){{=m_f,}}${{---here we denote $m_f$ by $k(m_f)$ to emphasis that it is correspond to $C_{k}$.}}
We choose $n\ge 1$ so that
$n\ge m_f.$
Put $F_1''=\pi'_{n*1}(F_0'),$ where $\pi_n': B_n\to A$ defined by
$\pi_n'(a\oplus b)=\pi(a)$ for all $a\in {{M_{(n!)^2}}}(A(W, \af_n))$ and $b\in W_n,$
where $\pi_n: {{M_{(n!)^2}}}(A(W,\af_n))\to {{M_{(n!)^2}(}}A{{)}}$ is the quotient map.
Note that $\pi_{n*1}: K_1(B_n)\to K_1(A)$ is an isomorphism.
We may assume that $F_1''\subset (j_{k, \infty})_{*1}(K_1(C_{k})).$
Let ${\tilde F}=\pi_{n*1}^{-1}((j_{k, \infty})_{*1}(K_1(C_{k})))$ and $F=(\Phi_{n,\infty})_{*1}({\tilde F}).$ {{(Here, we identify $K_1(M_{(n!)^2} (A))$ with $K_1(A)$ and $K_1(M_{(n!)^2} (C_k))$ with $K_1(C_k)$.)}}

The subgroup $F$ may be called {\it the standard subgroup} of $K_1(B_T).$

In what follows ${\rm tr}$ is the unique tracial state on $Q.$
We will define an injective  \hm\, $j_{F,u}: F\to U(B_T)/CU(B_T).$
We identify ${\widetilde{A(W,\af_n)}}$ with the following \CA\,(recall $s: A\to Q$ is defined in the beginning
of \ref{Dcc1}):
$$
\{(f_\lambda, a)\in C([0,1], Q\otimes Q)\oplus A: f_{{\ld}}(0)=(s(a-\lambda)\otimes e_{\af_n})+ \lambda\cdot 1_{Q\otimes Q}\andeqn
f_{{\ld}}(1)=s(a-\lambda)\otimes 1_Q+\lambda\cdot 1_{Q\otimes Q}\},
$$
where $\lambda\in \C$ and $a-\lambda=a-\lambda\cdot 1_A\in A.$
Note that $(f,1_A),$ where $f(t)=1_Q\otimes 1_Q,$ is added to $A(W,\af).$

Write $F=\Z^{k(m_f)}\oplus \Z/k_1\Z\oplus \cdots \Z/k_{m_t}\Z.$
Put $m=k(m_f)+k(m_t).$
Let $x_1,x_2,...,x_{k(m_f)}$ be the free cyclic generators for $\Z^{k(m_f)}$
and $x_{0,j}$ be cyclic generators for each $\Z/k_j\Z,$
$j=1,2,...,k(m_t),$ respectively.

Fix unitaries $z_1',z_2',...,z_{k(m_f)}', z_{0,1}', z_{0,2}',...,z_{0,k(m_t)}'\in C_{k}$
such that $[z_i']=x_i,$ $i=1,2,...,k(m_f)$ and $[z_{0,j}']=x_{0,j},$ $j=1,2,...,{{k(m_t)=m_t}}.$
{{Note that $(z_{0,j}')^{k_j}\in U_0(C_{n,2}).$ We may choose $z_{0,j}'$ so
that $(z_{0,j}')^{k_j}\in CU(C_{n,2}).$}}
We further assume that  {{$z_j'=\diag(z_j^{(0)}, 1,...,1),$  where
$z_j^{(0)}$ is the standard unitary generator for $C(\T),$}}
$j=1,2,...,k(m_f).$

We write $s(z_j')=\exp(ih_{j,0}')\exp(ih_{j,1}'),$
where $h_{j,0}', h_{j,1}'\in {{s(p_j)}}Q_{s.a.}{{s(p_j).}}$ (Note that here we use the fact that the exponential rank
for 
{{$Q$}} is $1+\ep$ (see \cite{LnFU}){{)}}.   Let $h{{''}}_{j,0},\, h{{''}}_{j,1}\in \R$ such
that $h{{''}}_{j,l}={\rm tr}(h_{j,l}'),$ $l=0,1.$ Put $z_j=z_j'\exp(-{{2}}i\pi h{{''}}_{j,1})\exp(-{{2}}i\pi  h{{''}}_{j,0}){{,}}$ $j=1,2,...,m(k).$ {{Then $[z_j]=[z'_j]=x_j$.}}
Note that $s(z_j)=\exp(2i\pi h_{j,0})\exp(2i \pi h_{j,1})$ such that
$h_{j,0}, h_{j,1}\in (s(p_j)Qs(p_j))_{s.a.}$ and
${\rm tr}(h_{j,0})+{\rm tr}(h_{j,1})=0,$ $j=1,2,...,k(m_f).$
We also choose $z_{0,j}$  and $s(z_{0,j})=\exp(ih_{j,0,0})\exp(i h_{j,0,1})$ such that
${\rm tr}(h_{j,0,0})+{\rm tr}(h_{j,0,1})=0${{, and $[z_{0,j}]=x_{0,j}.$}}



Define
$
u_j=(f_j, z_j)
$
as follows.
\beq\label{fromuj}
f_j(t)={{(}}s(z_j)\otimes e_{\af_n}{{)}}\oplus {{(}}(\exp(i 2t\pi h_{j,0})\exp(i 2t \pi h_{j,1}){{)}}\otimes (1-e_{\af_n}){{)}}\rforal t\in [0,1].
\eneq
Note that
\beq
&&\hspace{-0.2in}f_j(0)= {{(}}s(z_j)\otimes e_{\af_n}{{)}} \oplus {{(1}}\otimes (1_Q-e_{\af_n}) {{)}}\andeqn\\
&&\hspace{-0.2in} f_j(1)={{(}}s(z_j)\otimes e_{\af_n} {{)}}\oplus  {{(}} \exp(i 2\pi h_{j,0})\exp(i  2\pi h_{j,1})\otimes (1-e_{\af_n}){{)}}=s(z_j)\otimes 1_Q.
 \eneq

In fact
\beq
f_j(t)=\exp(2i\pi d_{j,0}(t))\exp(2i \pi d_{j,1}(t)),
\eneq
where
\beq
d_{j,0}(t)&=&h_{j,0}\otimes e_{\af_n} +t h_{j,0}\otimes (1_Q-e_{\af_n})\andeqn\\
d_{j,1}(t)&=& h_{j,1}\otimes e_{\af_n}+t  h_{j,1}\otimes (1_Q-e_{\af_n}).
\eneq
In particular, $({{f_j}}, z_j)\in {\widetilde{A(W, {\af_n})}}$ and $u_j\in U({\widetilde{A(W, {\af_n})}}),$
$j=1,2,...,k(m_f).$

Write $u_j=\zeta_j+\mu(u_j),$ where $\zeta_j\in A(W,\af_n)$ and $\mu(u_j)$ is a scalar.
Since $d_{j,0}, d_{j,1}\in A(W,\af_n)_{s.a.},$  $\mu(u_j)=1.$
In particular, $(f, z_j)\in {\widetilde{A(W, {\af_n})}}$ and $u_j\in U({\widetilde{A(W, {\af_n})}}),$
$j=1,2,...,k(m_f).$


Let $u_{0,j}=(f_{0,j}, z_{0,j}) \in {\widetilde{A(W, {\af_n})}}$
be defined as follows:
\beq
f_{0,j}(t)=\exp(2i\pi d_{j,0,0}(t))\exp(2i \pi d_{j,0,1}(t)),
\eneq
where
\beq
d_{j,0,0}(t)&=&h_{j,0,0}\otimes e_{\af_n} +t h_{j,0,0}\otimes (1_Q-e_{\af_n})\andeqn\\
d_{j,0,1}(t)&=& h_{j,0,1}\otimes e_{\af_n}+t  h_{j,0,1}\otimes (1_Q-e_{\af_n}).
\eneq
One has, for some $\zeta_{0,j}\in A(W, \af_n),$
$$
u_{0,j}=\zeta_{0,j}+1_{{\widetilde{A(W, {\af_n})}}}.
$$

The map $J_{u,F,n}: {\tilde F}\to U({\widetilde{B_{n}}})/CU({\widetilde{B_{n}}})$
defined by $x_j\mapsto  {\overline{u_{j}}}$ and $x_{0,j}\mapsto {\overline{u_{0,j}}}$ is an injective \hm\,
and define $J_{u,F}: F\to   U({\widetilde{B_T}})/CU({\widetilde{B_{T}}})$
by identifying ${\bar u}_j$ with $\overline{\Phi_{n, \infty}(u_j)}$
and ${\bar u}_{0,j}$ with $\overline{\Phi_{n, \infty}(u_{0,j})}.$
{{It should be noted, by our choice, $k_j{\overline{u_{0,j}}}=0.$}}

\end{NN}

\begin{NN}\label{Dert-Large}
We keep notation used in \ref{DdertF}.
Define
$$
E_{n,k}=\{(f, a)\in M_{(n!)^2}(C([0,1], Q\otimes Q)\oplus {{M_{(n!)^2}(C_k)}}: f(0)=s(a){{\otimes}} e_{\af_n}\andeqn f(1)\in s(a)\otimes 1_Q\},
$$
$n=1,2,...${{, where $s:M_{(n!)^2}(C_k)\to M_{(n!)^2} (Q)$ is the restriction of $s:M_{(n!)^2}(A)\to M_{(n!)^2} (Q)$ to $M_{(n!)^2}(C_k)\subset M_{(n!)^2}(A)$.}}
Fix $\ep>0$ and a finite subset ${\cal F}\subset B_T.$
\Wlog, we may assume that ${\cal F}\subset B_n.$
Denote by ${\cal F}^{Aw}=q_{E_n}({\cal F}),$ where
$q_{E_n}: B_n\to E_n=M_{(n!)^2}(A(W, \af_n))$ is the projection map.
Let
$$
C_{k,1}=\bigoplus_{i=1}^{k(m_f)} M_{r(k(i))}(C(\T)).
$$
Now write $u_1, u_2,...,u_{{k(m_f)}}\in {\tilde E_n}$ which represent the free generators of $K_1(E_{n,k}).$
We may assume that $\pi_n(u_j)=z_j,$ the unitary generator for $M_{r({{k(j)}})}(C(\T)),$ ${{j}}=1,2,...,k(m_f),$
and where $\pi_n: E_n\to {{M_{(n!)^2}(A)}}$ is the quotient map.
We also assume that $z_j$ and $u_j$ have the form \eqref{fromuj}.

Fix $\ep/2>\dt>0$ and a finite subset ${\cal G}'\subset C_k$ with ${\cal G}'\supset \pi_{n,k}({\cal F}^{Aw}),$
where $\pi_{n,k}: E_{n,k}\to {{M_{(n!)^2}(C_k)}}$ is the quotient map.
Choose a finite subset ${\cal F}_1\supset {\cal F}^{Aw}$ such that
$\pi_{n, k}({\cal F}_1)\supset {\cal G}'.$

We also assume
that there is an ${\cal G}'$-$\dt$-multiplicative \cpc\, $L: A\to C_k$ such
that
\beq\label{Large-1}
\|L(a)-a\|<\dt/4\rforal a\in {\cal G}'{{,}}
\eneq
{{where we also use $L$ to denote $L\otimes \id_{M_{(n!)^2}}:  M_{(n!)^2}(A) \to M_{(n!)^2}(C_k)$.}}
Choose $\dt>\dt_0>0$  such that, for any $(f,a)\in {\cal F}_1,$  if $|t-t'|<2\dt_0,$
\beq
\|f(t)-f(t')\|<\dt/16\rforal t, t'\in [0,1].
\eneq
Define ${\tilde L}: E_n\to E_{n,k}$ as follows:
${\tilde L}((f,a))=(g, L(a)),$
where
\beq\label{DdertF-2}
g(t)=\begin{cases} ((1-2t/\dt_0)s(L(a))\otimes e_{\af_n}+ {2t\over{\dt_0}}s(a)\otimes e_{\af_n}) & \rforal t\in [0,\dt_0/2],\\\nonumber
                              f({t-\dt_0/2\over{1-\dt_0}}) & \rforal t\in (\dt_0/2, 1-\dt_0/2],\\\nonumber
                              {1-t\over{\dt_0/2}}s(a){{\otimes 1_Q}}+
                              {{{t-(1-\dt_0/2)\over{\dt_0/2}}}}s(L(a))\otimes 1_Q &\rforal t\in (1-\dt_0/2, 1].\end{cases}
                              \eneq
One verifies that ${\tilde L}$ is an ${\cal F}_1$-$\dt/2$-multiplicative \cpc\, from
$E_n$ into $E_{n,k}.$

We now assume that $\af_n<\af_{n+1}.$
Let $r_1={1-\af_{n+1}\over{1-\af_n}}$ and $r_2={\af_{n+1}-\af_n\over{1-\af_n}}.$
Let $1>\eta>3/4$ and $\mu_j\ge 0,$ $j=1,2,...,k(m_f).$
Let $\omega_j=\mu_j/{\rm tr}(s(p_j)),$ $j=1,2,...,k(m_f).$

Fix a continuous increasing surjective function $g_1:[0, \eta]\to [0,1]$ such that $g_1(0)=0, g_1(\eta)=1$ and
decreasing surjective function $g_2:[\eta,1]\to [0,1]$ such that $g_2(\eta)=1,\,g_2(1)=0.$
Define $h|_{[0, \eta]}=g_1$ and $h|_{[\eta, 1]}=g_2.$ In particular,
$h=(0)=0$ and $h(1)=0.$

Define a  \hm\, $\phi_{c,R}^f:
{{M_{(n!)^2}(C_{k,1})}}\to M_{(n!)^2}(C([0,1], Q)\otimes e_{r_2})$
{{such that}}
\beq
\phi_{c, R}^f(z_j)(t)=
s(z_j)\exp(i2\pi (\omega_j/r_2) h(t))s(p_j)\otimes e_{r_2}\rforal t\in [0,1].
\eneq
Define $\phi_{c,R}=\phi_{c, R}^f|_{{M_{(n!)^2}(C_{k,1})}}\oplus (\phi_{c,R}^t)|_{{M_{(n!)^2}(C_{k,2})}}: {{M_{(n!)^2}(C_{k})}}\to M_{(n!)^2}(C([0,1], Q)\otimes e_{r_2}),$
where
\beq
\phi_{c,R}^t(a)(t)=s(a){{\otimes e_{r_2}}}\rforal t\in [0,1].
\eneq

Let $\phi_{A,R}: E_n\to {{M_{(n!)^2}(C([0,1], Q)\otimes e_{r_2})}}$ {{be defined}}
by $\phi_{A,R}((f,a))=\psi_{c, R}\circ {{L}}\circ \pi_A(a)$ for all $a\in E_n$
and where $\pi_A: M_{(n!)^2}(A(W, \af_n))\to M_{(n!)^2}(A)$ is the quotient map.


Now define a \cpc\, $\Psi: E_n\to M_{(n!)^2}(A(W,\af_{n+1})$
as follows. We will use some of the notation
in section 7.
Define  (see section {{7}} for the notation)
\beq
&&P_a( \Psi((f,a)))=L(a)\andeqn\\\nonumber
&&P_f( \Psi((f,a)))=\diag( P_f\circ \phi_{R,r_1}\circ \phi_{A,R, \af_n}({\tilde L}(f,a)), (\phi_{A,R}(f{{,a)}}))\\
&&=\diag( P_f\circ \phi_{R,r_1}\circ \phi_{A,R, \af_n}(g,L(a)), (\phi_{A,R}(f{{,a)}})).
\eneq
Note that
\beq
P_f(\Psi(f,a))(0)&=&\diag(s(L(a))\otimes e_{\af_nr_1}, s(L(a))\otimes e_{r_2})=s(L(a))\otimes e_{\af_{n+1}}\andeqn\\
P_f(\Psi(f,a))(1)&=&\diag(s(L(a))\otimes e_{r_1}, s(L(a))\otimes e_{r_2})=s(L(a))\otimes 1_Q.
\eneq
Let
\beq\nonumber
W_j(t)&=&(\exp(i2\pi h_{j,0})\exp(i2\pi h_{j,1})\otimes e_{\af_nr_1}\oplus (\exp(i2\pi t h_{j,0})\exp(i2\pi t h_{j,1})\otimes (e_{r_1}-e_{\af_nr_1})
\\\nonumber
&&\hspace{0.5in}+s(z_j)\exp(i2\pi (\omega_j/r_2) h(t))s(p_j)\otimes e_{r_2},\hspace{0.5in}\,\,\,j=1,2,...,k(m_f).
\eneq
{{Let}} $E_{n+1}':=M_{(n!)^2}(A(W, \af_{n+1}))${{, then in ${\tilde E'_{n+1}}$ }}(with large ${\cal G}'$),
\beq
\|\Psi(u_j)-(W_j, z_j)\|<\dt,\,\,\, j=1,2,...,k(m_f).
\eneq
{{(Here the unitalization of $\Psi$ is also denoted by $\Psi$.)}}
Therefore there exists $H_{j,00}\in (E_{n+1}')_{s.a.}$ with
$\|H_{j,00}\|\le 2\arcsin(\dt/2)$ such that
\beq
\lceil \Psi(u_j)\rceil=\exp(i2\pi  H_{j,00})(W_j, z_j),\,\,\, j=1,2,...,k(m_f).
\eneq
Put
\beq
H_{j,0}(t)&=&h_{j,0}\otimes (e_{\af_n r_1} \oplus e_{r_2})\oplus th_{j,0}\otimes (e_{r_1}-e_{\af_n r_1}),\\
H_{j,1}(t)&=& h_{j,1}\otimes (e_{\af_n r_1}\oplus e_{r_2})\oplus th_{j,1}\otimes (e_{r_1}-e_{\af_n r_1})\andeqn\\
H_{j,2}(t)&=& (\omega_j/r_2) h(t))s(p_j)\otimes e_{r_2}.
\eneq
Noting $h(0)=0$ and $h(1)=0,$ we see that
$H_{j,l}\in M_{(n!)^2}(R(\af_{n+1}, 1)).$
Therefore
\beq
{{\phi}}_{A, R, \af_{n+1}}(\lceil \Psi(u_j)\rceil)=\exp(i2\pi H_{j,00})\exp(i2\pi H_{j,0})\exp(i2\pi H_{j,1})\exp(i2\pi H_{j, 2}).
\eneq

{{Note that (recall that ${\rm tr}(h_{j,0})+{\rm tr}(h_{j,1})=0,$ $j=1,2,...,k(m_f)$), for all $t\in [0,1],$
\beq\label{11-nn-1}
{\rm tr}(H_{j,0}+H_{j,1})(t)=0
\eneq}}
We  {{then}} compute that, for all $t\in [0,1],$
\beq\label{Large-20}
{\rm tr}(H_{j,00}+H_{j,0}+H_{j,1}+H_{j,2})(t)&=&{\rm tr}(H_{j,00})+(\omega_j/r_2)h(t)\cdot {\rm tr}(s(p_j)){\rm tr}(e_{r_2})\\
&=&{\rm tr}(H_{j,00})+\mu_j h(t).
\eneq
It follows that, in $E{{''}}_{n+1}=M_{(n!)^2}(R(\af_{n+1}, 1)),$  for all $t\in [0,1],$
\beq\label{Fdert1}
|D_{E{{''}}_{n+1}}({{\phi}}_{A, R, \af_{n+1}}(\lceil \Psi(u_j)\rceil))(t)-\mu_jh(t)|<\dt.
\eneq

Let
\beq\nonumber
W_{0,j}(t)&=&(\exp(i2\pi h_{j,0,0})\exp(i2\pi h_{j,0,1})\otimes e_{\af_nr_1}\oplus (\exp(i2\pi t h_{j,0,0})\exp(i2\pi t h_{j,0,1})\otimes (e_{r_1}-e_{\af_nr_1})
\\\nonumber
&&\hspace{0.5in}+s(z_j)s(p_j)\otimes e_{r_2},\hspace{0.5in}\,\,\,j=1,2,...,m_t.
\eneq
A similar computation shows that
\beq\label{Fdert2}
|D_{E{{''}}_{n+1}}({{\phi}}_{A, R, \af_{n+1}}(\lceil \Psi(u_{{{0,}}j})\rceil))(t)|<\dt.
\eneq

\end{NN}

We will keep notations in \ref{DdertF} and \ref{Dert-Large} in the following statement.

\begin{lem}\label{L215}
Let $C$ be a non-unital separable simple \CA\, in ${\cal D}$  with  continuous scale
such that ${\rm ker}\rho_C=K_0(C)$ and let
$B=B_T$ be as constructed in \ref{Dcc1}.

Let $\ep>0, $ ${\cal F}\subset B$ be a finite subset,  let
${\cal P}\subset \underline{K}(B)$ be a finite subset  and let $1/2>\dt_0>0.$

For any finitely generated standard subgroup
$F$ (see \ref{DdertF}),  any finite subset $S\subset F,$
there exists an integer $n\ge 1$  with the following property:\\
for
any finite  subset  ${\cal U}\subset  U({\tilde B_T})$ such that
${\overline{\cal U}}\subset
J_{F,u}(F)\subset J_{F,u}((\Phi_{n, \infty})_{*1}(K_1(E_n)))$ (see the end of  \ref{DdertF}) and $\Pi(\overline{{\cal U}})=S,$
where $\Pi: U({\tilde B})/CU({\tilde B})\to K_1(B)$ is the quotient map, {{for}}
any \hm\, \\ $\gamma: J_{F,u}((\Phi_{n, \infty})_{*1}(K_1(E_n)))\to \Aff(T({{\tilde C}}))/\Z,$
such that $\gamma|_{{\rm Tor}(J_{u,F}((\Phi_{n, \infty})_{*1}(K_1(E_n)))}=0$ and
any $c\in C_+$ with $\|c\|=1,$
there exists ${\cal F}$-$\ep$-multiplicative \cpc\,
$\Phi:  B_T
\to \overline{cCc}$
such that
\beq\label{L215-1}
[\Phi]|_{\cal P}=0\tand
{\rm dist}(\Phi^{\dag}({\bar z}),\gamma({\bar z}))<\dt_0\tforal z\in {\cal U}
\\\nonumber
\hspace{1in}{{{\rm(in}\,\, U_0({\tilde C})/CU({\tilde C})\cong \Aff(T({\tilde C}))/\Z).}}
\eneq

\end{lem}
({{here we assume $\dist(\Phi^\dag({\bar z}), \overline{\lceil \Phi(z)\rceil})<\dt_0/4$
for all $z\in {\cal U}$}}--see \ref{DLddag} for the definition of $\Phi^{\dag}.$)

\begin{proof}
Fix $\ep>0,$ ${\cal F}$ and ${\cal P}$ as described by this lemma.
Fix $\dt_1>0,$  a finite subset ${\cal G}\subset B_T.$ We assume that ${\cal F}\subset {\cal G}.$
Choose $n_0\ge 1$ such that there exists  finite subset ${\cal G}'\subset B_{n_0}$
such that, for any $b\in {\cal G},$ there exists $b'\in {\cal G}'$
such that
\beq
\|b-\Phi_{n, \infty}(b')\|<\dt_1/64.
\eneq
We assume that $\dt_1<\min\{\dt_0/{{16}}, \ep/16\}.$

Choose $k\ge 1$ as in \ref{DdertF} and
write $F=\Z^{m_f}\oplus \Z/{k_1}\Z\oplus \cdots \Z/k_{m_t}\Z.$
Fix a set of generator $S$ of $F.$

To simplify notation, \wilog, we may assume that ${\cal G}\subset \Phi_{n_0, \infty}({\cal G}').$
We also assume, \wilog, that ${\cal P}\subset [\Phi_{n_0, \infty}](B_n).$
Let ${\cal P}'\subset \underline{K}(B_{n_0})$ be a finite subset such that
${\cal P}\subset [\Phi_{n_0, \infty}]({\cal P}').$

{{Let $\overline{{\cal U}}\subset J_{F,u}(F)$ and,
let $z_j$ and $u_j,$  $j=1,2,...,m_f,$ and $z_{0,j},$ and $u_{0,j},$  $j=1,2,...,m_t,$ be as described in
 \ref{DdertF}.  \Wlog, we may assume that
 $\overline{{\cal U}}=\{{\bar u}_1,{\bar u}_2, ...,{\bar u}_{m_f}, {\bar u}_{0,1},...,{\bar u}_{0,m_t}\}.$}}

We also assume that there exists a \cpc\, $L: B_T\to B_n$ such that, for all $n\ge n_0,$
\beq
\|L(\Phi_{n, \infty}(b'))-b'\|<\dt_1/64\rforal b'\in {\cal G}'
\eneq
We further assume that $\dt_1$ is sufficiently small and ${\cal G}$ is sufficiently large
so that $[L']|_{\cal P}$ is well defined for any ${\cal G}$-$\dt_1$-multiplicative \cpc\, from ${{B_T}}.$
{{Moreover, ${L'}^\dag$ can be defined so that ${\rm dist}({L'}^\dag({\bar z}), \overline{\lceil L'(z)\rceil})<\dt_0/4$
for all $z\in {\cal U}$ (see \ref{DLddag}).}}

Choose $\dt={\dt_1\over{{{16(m_f+m_t+2)}}}}$  and
choose $n\ge n_0+m_f+m_t+2$ as in \ref{Dert-Large} associated with $\dt/64$ (in place $\ep$)
and ${\cal G}$ (in place of ${\cal F}$).

Choose  non-zero elements $c_{i,l}\in \overline{cCc}_+$ which are
mutually orthogonal, $i=1,2,...,m_f,$ $l=1,2.$

Choose $1>\eta_0>0$ such that
$$
\eta_0\le \inf\{d_\tau(c_{j,l}): \tau\in T(C)\},
{{1\le j\le m_f\andeqn
l\in \{1,2\}.}}$$


Choose $g_{j,+}, g_{j,-}\in \Aff(T(C))_+$ and $\lambda_{j,+}, \lambda_{j,-}\in \R_+$  such that
\beq\label{L215-2}
0<g_{j+}(\tau)\le \eta_0, 0<g_{j-}(\tau)\le \eta_0\rforal \tau\in T(C)\andeqn\\
\gamma({\bar u_j})=\lambda_{j,+}g_{j+}-\lambda_{j,-}g_{j-},\,\,\, j=1,2,...,m_f.
\eneq


Let $P_n: B_n\to E_n$ be the projection map, and let  ${\cal G}''\subset E_n$ be a finite subset such that ${\cal G}''\supset P({\cal G}').$

Define $\phi_{j,l}':
M_{(n!)^2}({{A(W, \af_n)}})\to M_{(n!)^2}(R(\af_{n+1}, 1))$
be as defined  (denoted by ${{\phi}}_{A,R,\af_{n+1}}\circ \Psi$ there) in \ref{Dert-Large} {{(with $\mu_j=\ld_{j,+}$ and $\mu_i=0$ if  $i\not=j$  (for $\phi_{j,1}'$); and with $\mu_j=\ld_{j,-}$ and $\mu_i=0$ if  $i\not=j$  (for $\phi_{j,2}'$))}} such that
\beq
\hspace{-0.2in}\lceil \phi_{j,l}'(u_j) \rceil =\exp(\sqrt{-1}2\pi H_{j,00})\exp(\sqrt{-1}2\pi H_{j,0})\exp(\sqrt{-1}2\pi H_{j,1})\exp(\sqrt{-1}2\pi H_{j, 2,l}),
\eneq
where $H_{j,00}, H_{j,0}, H_{j,1}, H_{j,2,l}\in M_{(n!)^2}(R(\af_{n+1}, 1),$ $l=1,2,$ such that
\beq
&&\hspace{-0.4in}{\rm tr}(H_{j,00}(t)+H_{j,0}(t)+H_{j,1}(t)+H_{j,2,l}(t))={\rm tr} (H_{j,00}(t)) +{\rm tr}(H_{j,2,l})(t),\,\,\, l=1,2,\\
&&{\rm tr}(H_{j,2,1})(t)=\lambda_+h(t),\,\,\, {\rm tr}(H_{j,2,2})(t)=\lambda_-h(t)\andeqn\\
&&|{\rm tr}(H_{j,00}(t))|<\dt/4
\eneq
for all $t\in [0,1],$
where
$h(t)$ is  $C([0,1])_+$ such that
$h(0)=0,$ $h(3/4)=1,$ $h(1)=0,$ $h(t)$ is strictly increasing
on $[0, 3/4]$ and strictly decreasing on $[3/4, 1].$ Moreover $\phi_{j,l}'$ is ${\cal G}''$-$\dt/8(m_f)$-multiplicative,
\beq
\hspace{-0.2in}\lceil \phi_{j,l}'(u_i) \rceil &=&\exp(2\pi \sqrt{-1} H_{{{i}},00})\exp(2\pi \sqrt{-1}H_{{{i}},0})\exp(2\pi \sqrt{-1} H_{{{i}},1}),{\rm  if}\,\, i\not=j\andeqn\\
{[}\phi_{j,l}'{]}|_{\cal Q}&=&0,
\eneq
where ${{{\cal Q}}}=[P_n\circ \Phi_{n, \infty}]({\cal P}').$
(Note that $K_i(R(\af_{n+1}, 1))=\{0\},$ $i=0,1$).
Note since $C\in {\cal D},$
for each $j,$  there exists a non-zero \hm\, ${{\phi''_{j,l}}}: M_{(n!)^2}(R(\af_{n+1}, 1))\to C_{j,l}:= \overline{c_{j,l}Cc_{j,l}},$
$j=1,2,...,m_f.$
It follows from \ref{Tdert1} that there is, for each $j$ and $l,$ a \hm\,
$\phi_{j,l}'': M_{(n!)^2}(R(\af_{n+1}, 1))\to C_{j,l}$ such that
\beq\label{L215-4}
&&\sup\{|\tau\circ \phi_{j,1}''(h)-g_{j,+}(\tau)|: \tau\in T(C)\}<\dt/2\andeqn\\
&&\sup\{|\tau\circ \phi_{j,2}''(h)-g_{j,-}(\tau)|:\tau\in T(C)\}<\dt/2.
\eneq
Let  $\phi_{j,l}=\phi_{j,l}''\circ \phi_{j,l}': E_n\to C_{j,l}.$
{{Recall $\phi_{j,l}'$ is of the form $\phi_{A,R,\af_{n+1}}\circ \Psi$, w}}e compute  that (also using
{{\eqref{11-nn-1} and}} \eqref{Fdert1},
{{and,  see \ref{Dceil} for the notation $\lceil\,{\mathbf{\cdot}} \,\rceil$}}),
\beq\label{L215-10}
D_{{\tilde C}}(\sum_{l=1}^2\lceil \phi_{j,l}(u_j)\rceil)&\approx_{2{{\dt_1}}/16(m_f)}&
(\lambda_{j+}g_{j,+}-\lambda_{j-}g_{j,-})\hspace{0.2in} ({\rm in}\,\,\Aff(T({\tilde C}))/\Z)\\
&=& \gamma({\bar u}_j)\hspace{1in}({\rm in}\,\,\Aff(T({\tilde C}))/\Z)\\
D_{{\tilde C}}(\lceil \phi_{j,l}(u_{{i}})\rceil)
&\approx_{2{{\dt_1}}/16(m_f)} &
0\,\hspace{1in}({\rm in}\,\,\Aff(T({\tilde C}))/\Z),\,\,\,i\not=j.
\eneq
{{Similarly, using  \eqref{Fdert2}, we have}}

\beq\label{L215-11}
D_{\tilde C}(\sum_{l=1}^2(\lceil \phi_{j,l}(u_{0,i})\rceil )\approx_{2{{\dt_1}}/16(m_f)} 0\,\,\,({\rm in}\,\,\Aff(T({\tilde C}))/\Z).
\eneq

Now define
$\Phi': E_n\to \bigoplus_{j=1}^{m_f}(\bigoplus_{l=1}^2 C_{j,l})$
by
$
\Phi'=\sum_{j=1}^{m_f} (\sum_{l=1}^2\phi_{j,l}).
$
From the above estimates,
\beq
{\rm dist}(\Phi^{\dag}({\bar z}), \gamma({\bar z}))<{{\dt_0}}\rforal z\in {\cal U}.
\eneq
Moreover, since $\Phi'$ factors through $M_{(n!)^2}(R(\af_{n+1},1)),$
\beq
[\Phi']|_{\cal Q}=0.
\eneq
 Define $\Phi=\Phi'\circ P_n\circ  L.$  We check that $\Phi$ meets the requirements.


%

\end{proof}

\begin{lem}\label{Lderttorsion}
Let $C$ be a
non-unital separable \CA.
Suppose that $u\in U(M_s({\tilde C}))$ (for some integer $s\ge 1$) with $[u]\not=0$ in $K_1(C)$
but $u^k\in CU(M_s({\tilde C}))$
 for some $k\ge 1.$
 Suppose that $\pi_C(u)=e^{2i \pi \theta}$
 for some $\theta\in (M_s)_{s.a.},$
 {{where $\pi_C: M_s({\tilde C})\to M_s$ is the quotient map.}}
Then ${{sk\cdot {\rm tr}(\theta)}}\in \Z,$ where ${\rm tr}$ is the tracial state
of $M_s.$

Let $B$ be a stably projectionless simple separable \CA\, with ${\rm ker}\rho_B=K_0(B)$
and with continuous scale.
For any $\ep>0,$ there exists $\dt>0$ and finite subset ${\cal G}\subset C$ satisfying the following:
If $L_1, L_2: C\to B$ are two ${\cal G}$-$\dt$-multiplicative \cpc s such that $[L_1](u)=[L_2](u)$
in $K_1(B),$
then
\beq\label{Lderttor-1}
{\rm dist}(\overline{\lceil L_1(u)\rceil}, \overline{\lceil L_2(u)\rceil} )<\ep,
\eneq
{{where $u$ is as in  the first paragraph.}}

\end{lem}

\begin{proof}
Write  $u=e^{2{{\sqrt{-1}}} \pi\theta}+\zeta,$  where $\zeta\in M_s(C)$
and $\theta\in (M_s)_{s.a.}.$
Therefore, if $u^k\in CU(M_s({\tilde C})),$ then $sk{\rm tr}(\theta)\in \Z.$

Note $L_i$ is originally defined on $C$ and the extension
$L_i: M_s({\tilde C})\to M_s({\tilde B})$ has the property  that
$L_i(u)=e^{2{{\sqrt{-1}}}  \pi \theta}+L_i(\zeta),$  $i=1,2.$ To simplify notation,
\wilog, we may assume that $\lceil L_1(u)\rceil \cdot \lceil L_2(u^*)\rceil\in U_0(M_s({\tilde B})).$  {{Note that $$\pi_B(\lceil L_1(u)\rceil \cdot \lceil L_2(u^*)\rceil)=e^{2\sqrt{-1}  \pi \theta } e^{-2\sqrt{-1}  \pi \theta}=1$$}}
{{(where $\pi_B: M_s({\tilde B})\to M_s$  is the quotient map).}}
{{ By \ref{Thelemma}, we may write}}
$$
\lceil L_1(u)\rceil \cdot \lceil L_2(u^*)\rceil =\prod_{j=1}^n \exp(2{{\sqrt{-1}}}  \pi h_j)\rforal \text{some}\,\,h_1,h_2,...,h_n\in M_s({\tilde B})_{s.a}\,\,\text{with}
$$
\vspace{-0.1in}
\beq\label{1802}
{{\pi_B(h_j)=0~\mbox{and}}}~\pi_B( \exp(2\sqrt{-1} \pi h_j))=1 ~{\mbox{ for all}~j.}
\eneq
It follows from 14.5  of \cite{Lnmemhp} that, by choosing small $\dt$ and large ${\cal G}$
(independent of $L_1$ and $L_2$)
there is $h_0\in M_s({\tilde B})_{s.a.}$ such that
$\|h_0\|<{{\min\{1, \ep\}/4s(k+1)}}$ and
\beq\label{116-nn1}
((\exp(2i\pi h_0))(\prod_{j=1}^n \exp(2i \pi h_j)))^k\in CU(M_s({\tilde B})).
\eneq
By \eqref{1802}, $\pi_B(\exp(2i h_0))\in CU(M_s).$
Then $st_0(h_0)\in \Z.$ However, since $\|h_0\|<1/4s(k+1),$
$t_0(h_0)<1/4s(k+1).$
This implies that $t_0(h_0)=0.$ Note also $U({\tilde B})/CU({\tilde B})=\Aff(T({\tilde B}))/\Z.$
{{Therefore (by \eqref{116-nn1}), t}}here is an integer $m\in \Z$ such that, for any $\tau\in T({\tilde B}),$
\beq
k(\tau(\sum_{j=1}^n h_j+h_0))=m.
\eneq
For any $\tau_0\in T(B)$ and any $0<\af<1,$
$t=\af \tau_0+(1-\af)t_0$ is a tracial state of ${\tilde B}.$
Then
(by \eqref{1802}),
\beq\label{ntorsion-1}
kt(\sum_{j=1}^n h_j+h_0)=k(\af\tau_0(\sum_{j=1}^n h_j) +\af\tau_0(h_0))=m.
\eneq
So
$
k\af\tau_0(\sum_{j=1}^nh_j +h_0)=m
$
for any $0<\af<1.$ It follows that
\beq
\tau_0(\sum_{j=1}^n h_j+h_0)=0\rforal \tau_0\in T({\tilde B}).
\eneq
It follows that
\beq
|\tau(\sum_{j=1}^n h_j)|<\ep/2(k+1)\rforal \tau\in T({\tilde B}).
\eneq
Thus \eqref{Lderttor-1} holds.
\end{proof}

\section{Construction of \hm s}

\begin{prop}\label{PADw}
Let $A$ be a separable simple \CA\, in ${\cal D}.$  Suppose
that ${\rm ker}\rho_A=K_0(A).$
Then there exists a sequence of approximately multiplicative \cpc s
$\{\phi_n\}$ from $A$ to ${\cal W}$ which maps  strictly positive elements to strictly positive elements.

\end{prop}

\begin{proof}
Fix $\tau\in T(A).$
Define $\gamma: T({\cal W})\to T(A)$ by
$\gamma(t_{\cal W})=\tau,$ where $t_{\cal W}$ is the unique tracial state of ${\cal W}.$
 Then $\gamma$ induces an order semi-group
\hm\,  from ${\rm LAff}({\tilde T}(A))$ onto ${\rm LAff}({\tilde T}({\cal W})).$ Since ${\rm ker}\rho_A=K_0(A)$ { {and $K_0({\cal W})=0$,}}
this in turn induces a \hm\,
$\Gamma: Cu^{\sim}(A)\to  Cu^{\sim}({\cal W})$ {{(see 7.3 of \cite{eglnp1})}}.
Fix a strictly positive element $a_0\in A$ with $\|a\|=1.$
Let $\mathfrak{f}_{a_0}>0$ be the associated number (see \ref{DD0}).
There exists a sequence of approximately multiplicative \cpc s $\psi_n: A\to D_n$
such that
$\psi_n(a_0)$ is a strictly positive element of $D_n,$
$t(f_{1/4}(a_0))\ge \mathfrak{f}_{a_0}$ for all $t\in T(D_n),$
{{(where $D_n$ is the same as constructed in the proof of 9.1 of \cite{eglnp1} and
$\psi_n$ is the same as $\phi_{1,n}$).}}
Moreover,
$$
\lim_{n\to\infty}\sup\{|\tau(a)-\tau\circ \psi_n(a)|: \tau\in \overline{T(A)}^w\}=0\rforal a\in A
$$
(see the proof of
9.1 of \cite{eglnp1}).
In particular, this implies
that $\lim_{n\to\infty}\|\psi_n(x)\|=\|x\|$ for all $x\in A.$
For each $n,$ let $\imath_n: D_n\to A$ be the embedding.

Let $\lambda_n=\Gamma\circ  (Cu^{\sim}(\imath_n)).$
It follows from \cite{Rl} that there is a \hm\,
$h_n: D_n\to {\cal W}$ such that
$$
Cu^{\sim}(h_n)=\lambda_n,\,\,\, n=1,2,....
$$
By passing a subsequence if necessary,  we may assume
that
$$
\lim_{n\to\infty}\|h_n\circ \psi_n(ab)-h_n\circ \psi_n(a)h_n\circ \psi_n(b)\|=0\rforal a, b\in A.
$$
By using an argument used in the proof   of
{{12.4 of \cite{eglnp1}),}}
we {{may}} also assume that $h_n\circ \psi_n(a_0)$ is a strictly positive element of ${\cal W}.$
\end{proof}

\begin{rem}\label{11R1}
In the absence of the condition $K_0(A)={\rm ker}\rho_A,$ the proof of  \ref{PADw} shows
that the conclusion of \ref{PADw} holds if the assumption
is changed to 
{{that}} $A$ has at least one non-zero ${\cal W}$-trace.
The proof of Lemma \ref{PADw}  also shows that every tracial state of simple \CA s in ${\cal D}$ with
$K_0(A)={\rm ker}\rho_A$ is a ${\cal W}$-trace.
Lemma \ref{PADw} can also be obtained from the proof of \ref{TD0=D}.
\end{rem}

The following is a number theory lemma which may be known.

\begin{lem}\label{NumberT}
Let $a_1, a_2, \cdots, a_n$ be non-zero integers  such that  at least one of them is positive and  one of them is negative. Then, for any $d\in \Z,$ if $a_1x_1+a_2x_2+\cdots+a_n x_n=d$ has an integer solution, then it must have a positive  integer solution.
\end{lem}

\begin{proof}
 We will prove it by induction.
Suppose that $a, b\in \Z$ such that $a>0$ and $b<0.$  Suppose also there are $x_0, y_0\in \Z$ such that
$ax_0+by_0=d.$
Then, for any integer $m\in \Z,$ and any $x=x_0+bm$ and $y=y_0-am,$
\beq
a(x_0+bm)+b(y_0-am)=d.
\eneq
Thus, by choosing negative integer $m$ with large $|m|,$ both $x_0+bm$ and $y_0-am$ are positive.
This prove the case $n=2.$

 Suppose the lemma holds for $n-1$ {{with}} $n\ge 3.$
Without lose of generality, let us first assume that $a_1$ and $a_2$ have different signs. Suppose $\{x_1^0,x_2^0,\cdots x_n^0\}$ is an integer solution for $a_1x_1+a_2x_2+\cdots+a_n x_n=d$, Let $k=a_1x_1^0+a_2x_2^0+\cdots+a_{n-1} x_{n-1}^0$. Now we divided it into two cases:

Case 1: $k$ and $a_n$ have opposite signs.
By induction assumption there are positive integers  $x'_1, x'_2, \cdots, x'_{n-1}$ such that
\beq
k=a_1x'_1+a_2x'_2+\cdots+a_{n-1} x'_{n-1},
\eneq
since $a_1a_2<0$ and $n\ge 3.$ On the other hand, by applying the case $n=2,$  we have integers $x>0$ and $y>0$ such that $kx+a_ny=d$.

Let $x_i=x x'_i$ for $i \in \{1,2...n-1\}$  and $x_n=y$ to get desired positive integer solution for
\beq
\sum_{i=1}^n x_i a_i={{x\sum_{i=1}^{n-1} {{x'_i}}a_i+a_ny}}=d.
\eneq

Case 2: $k$ and $a_n$ have the same sign.

By the induction assumption there are positive integers: $x'_1, x'_2, \cdots, x'_{n-1}$ such that
\beq
-k=a_1x'_1+a_2x'_2+\cdots+a_{n-1} x'_{n-1}
\eneq
(recall $a_1a_2<0$).
On the other hand, {{applying}} the case $n=2$ (note that $-k$ and $a_n$ have opposite signs),
 we have $x>0$ and $y>0$ such that $-kx+a_ny=d$. Finally let $x_i=x x'_i$ for $i\in\{1,2...n-1\}$  and $x_n=y$ to get
 the desired positive integer solution.

\end{proof}

\begin{NN}\label{KKzo=Z} {{Recall {{that,}} from \ref{Dcc1}, $\zo$ is an inductive limit of $B_m=W_m \oplus M_{(m!)^2}(A(W, \af_m))$ and r}}ecall that $K_0(\zo)=\Z$ and $K_1(\zo)=\{0\}.$
Let $E_m=M_{(m!)^2}(A(W, \af_m))$ be as in \ref{Dcc1}. For any $m,$ $K_0(E_m)=\Z$ and
$K_1(E_m)=\{0\}.$  Let ${\rm id}: K_0(\zo)\cong K_0(E_m)$ {{be the isomorphism.}}
Then it induces a unique element in $KK(\zo, E_m)$ and will be denote by id.
Let $z_{\Z}=[1]\in \Z=K_0(\zo)$ be the generator of $K_0(\zo).$
Suppose that $C$ is a separable amenable \CA\, satisfies the UCT.
Denote by {{$\kappa_{\zo}\in KK(C\otimes \zo, C)$ the element such that}} $(\kappa_{\zo})_{*i}: K_i(C\otimes \zo)\to K_i(C)\otimes \Z=K_i(C)$ the isomorphism
such that $(\kappa_{\zo})_{*i}(x\otimes z_{\Z})=x$ for $x\in K_i({{C}}),$
given by Kunneth's formula,
$i=0,1.$
\end{NN}

\begin{lem}\label{11Extmz}
Let $C$ be a separable amenable \CA\,  which satisfies the condition in \ref{DFixA} and
which satisfies the UCT.
There exists a sequence of approximate multiplicative \cpc s
$\phi_n: C\otimes \zo\to C\otimes M_{k(n)}$ (for some subsequence $\{k(n)\}$)
which maps strictly positive elements to strictly positive elements such that
\beq\label{11Extmz-1}
[\phi_n]|_{\cal P}=(\kappa_{\zo})|_{\cal P},
\eneq
where $\kappa_{\zo}\in KK(C\otimes \zo, C)$ is an invertible element  which induces
$(\kappa_{\zo})_{*i},$
for every finite subset ${\cal P}\subset \underline{K}(C)$ and all sufficiently large $n.$
\end{lem}

\begin{proof}
Let $\ep>0$ and let ${\cal F}\subset C$ be a finite subset.

\Wlog, we may assume that $[L]|_{\cal P}$ is well -defined
for any ${\cal F}$-$\ep$-multiplicative \cpc\, from $C.$
{{We may}} also assume that ${\cal P}$ generates the subgroup
$$
G_{\cal P}\subset K_0(C)\bigoplus K_1(C)\bigoplus \bigoplus_{i=1,0}\bigoplus_{j=1}^m K_i(C, \Z/j\Z)\,\,\,
{\text{for\,\, some\,\,}} m\ge 2.
$$
Let $\dt>0$ and ${\cal G}\subset A$ be a finite subset.
Let $A$ be a unital simple AF-algebra with $K_0(A)=\Q\oplus \Z$ and with ${\rm ker}\rho_A=\Z.$
Write
$$
A=\overline{\cup_{n=1}^{\infty} F_n},
$$
where $1_A\in F_n\subset F_{n+1}$ is a sequence of finite dimensional \CA s. {{Recall that there is an identification of $K_0(\zo)$ with ${\rm ker}\rho_A\cong \Z\subset K_0(A).$}}
Therefore there are sequences of pair of projections
$p_n, q_n\in F_n$ such that
$$
(j_{n, \infty})_{*0}([p_n]-[q_n])=z_\Z,
$$
where $j_{n, \infty}: F_n\to A$ is the embedding and $z_\Z$ is $[1]$ in $\Z\cong {\rm ker}\rho_A.$
\Wlog\,  we may assume that
\beq
[p_n]\not= [q_n]\in K_0(F_n)\rforal n\ge 1.
\eneq
Write
$$
F_n=M_{k_1}\oplus M_{k_2}\oplus \cdots M_{k_l}.
$$
Note that $l\ge 3$ (see 7.7.2 of \cite{BKT}).
Let $P_i: F_n\to M_{k_i}$ be the projection map.
Let $x_i=[P_i(p_n)]-[P_i(q_n)]\in \Z{{=K_0(M_{k_i}),}}$ $i=1,2,...,l.$
Then some of $x_i>0$ and some of $x_i<0.$
{{Otherwise,}} we may assume that
\beq
x_i\ge 0\rforal i\in \{1,2,...,l\}.
\eneq
Then $[p_n]-[q_n]\ge 0$ for all $n.$ It follows
that, for all $k\ge 1,$
\beq
(j_{n, n+k})_{*0}([p_n]-[q_n])\ge 0\andeqn (j_{n, \infty})_{*0}([p_n]-[q_n])\ge 0.
\eneq
{{That is $(j_{n, \infty})_{*0}([p_n]-[q_n])\in K_0(A)_+$.}} This contradicts the fact  that $(j_{n, \infty})_{*0}([p_n]-[q_n])=z_{\Z}.$

Note that, as constructed in section 6,  with $A$ above,
\beq\label{11Extmz-1+}
\zo=\lim_{m\to\infty} (E_m\oplus W_m),
\eneq
where $W_m$ is a single summand of the form $R(\af_m,1)$
for some $0<\af_m<1$ and $E_m=M_{(m!)^2}(A(W, \af_m)).$
Note that $K_i(W_m)=\{0\},$ $i=0,1,$ and
$K_0(A(W,\af_m))=\Z$ and $K_1(A(W,\af_m))=\{0\}.$
Let ${\rm id}\in KK(\zo, E_m)$ be as described in \ref{KKzo=Z}.
Let $\kappa_{00}\in KK(C\otimes \zo, C\otimes E_m)$ be the invertible
element given by $[{\rm id}_C]$ and ${\rm id}.$

By \eqref{11Extmz-1+},  there exists a ${\cal G}$-$\dt$-multiplicative \cpc\,
$\Phi: C\otimes \zo\to C\otimes E_m$ (for sufficiently large $m$)
such that
\beq\label{11Extmz-2}
[\Phi]|_{\cal P}=(\kappa_{00})|_{\cal P}
\eneq
which maps strictly positive elements to strictly positive elements.
Consider the short exact sequence
$$
0\to C_0((0,1), Q)\to E_m\to M_{(m!)^2}(A)\to 0.
$$
Let $\phi_{qa}: E_m\to M_{(m!)^2}(A)$ be the quotient map.
Note that
$(\phi_{qa})_{*0}$ gives an isomorphism from $\Z=K_0(A(W, \af_m))$ onto
${\rm ker}\rho_A.$
Let $\phi_q: C\otimes E_m\to C\otimes M_{(m!)}(A)$ be defined by
${\rm id}_C\otimes \phi_{qa}.$
 Let $\phi_1: C\otimes \zo\to C\otimes M_{(m!)^2}(A)$
be defined by
$\phi_1=\phi_q\circ \Phi.$
For any $\dt_1>0$ and finite subset ${\cal F}_A\subset M_{(m!)^2}(A),$
there is a unital ${\cal F}_A$-$\dt_1$-multiplicative {{completely positive map,}} $\Phi_A: M_{(m!)^2}(A)\to {{M_{(m!)^2}(F_n)}}$
such that
$
[\Phi_A]|_{{\rm ker}\rho_A}
$
is injective.   Note that $\Phi_A$ maps strictly positive elements of $A$  to strictly
positive elements of $F_n.$
{{Recall}}
$$
F_n=M_{k_1}\oplus M_{k_2}\oplus \cdots M_{k_l}
$$
{{and}}
$x_i=[P_i(p_n)]-[P_i(q_n)]\in \Z,$ $i=1,2,...,l.$
\Wlog, we may assume that
\beq
x_i> 0\,\,\mbox{{for}}\,\, i=1,2,..., m^+,\, x_i<0\,\,\mbox{{for}}\,\, i=m^++1,...,l',\,\mbox{{and}}\, \,x_i=0\,\mbox{for}\, 1=l'+1,...,l.
\eneq
We claim that $x_1, x_2,...,x_{l'}$ are relatively prime.
If not,  $x_i=Nx_i',$  {{for some integer $x_i',$}} $i=1,2,...,l,$ for some $N\ge 2.$
Then $N(j_{n,\infty})_{*0}((x_1', x_2',...,x_l'))=
{{z_\Z.}}$ This is impossible since
$K_0(A)=\Q\oplus \Z$ {{and ${{z_\Z}}=[1]$ in the summand $\Z.$}}
It follows from \ref{NumberT} that there are positive integers
$N_1, N_2,...,N_l$ such that
\beq
\sum_{i=1}^l N_i x_i=1.
\eneq
Let $r=\sum_{i=1}N_ik_i.$ Define
$\imath: {{M_{(m!)^2}(F_n)\to M_{(m!)^2}(M_r)}}$ by
\vspace{-0.14in}\beq
\imath((f_1, f_2,...,f_l))=\bigoplus_{i=1}^l \imath_i(f_i),
\eneq
where $\imath_i: M_{k_i}\to M_r$ is defined by
\beq
\imath_i(f_i)=\diag(\overbrace{f_i, f_i,...,f_i}^{N_i})\rforal f_i\in M_{k_i},
\,\,\,i=1,2,...,l.
\eneq
Let $\kappa_{\zo}\in KL(C\otimes \zo, C)$ be defined by, for $j=2,3,...,$
\beq
\kappa_{\zo}(x\otimes z_{\Z})=x\rforal x\in K_i(C\otimes \zo)\oplus K_i(C\otimes \zo, \Z/j\Z),\,\,i=0,1.
\eneq
Note that $(\imath_{*0})([p_n]-[q_n])=[1]\in \Z=K_0({{M_{(m!)^2r}}}).$
Let
$
L=({\rm id}_C\otimes \imath)\circ ({\rm id}_C\otimes \Phi_A)\circ \phi_1: C\otimes \zo\to C\otimes M_{{(m!)^2r}}.
$
By choosing $\dt$ and $\dt_1$ sufficiently small and
${\cal G}$ and ${\cal G}_A$ sufficiently large,
$L$ is ${\cal F}$-$\ep$-multiplicative. Moreover, we compute
that
$$
[L]|_{\cal P}=[\kappa_{\zo}]|_{\cal P}.
$$

\end{proof}

\begin{lem}\label{11Ext1}
Let $A$  and $B$ be {{separable}}  simple \CA s in ${\cal D}$ with $K_0(A)={\rm ker}\rho_A$ and
$K_0(B)={\rm ker}\rho_B,$  respectively,  which
have continuous scale and
 satisfy the UCT. 
Suppose that there is $\kappa\in KL(A, B)$ and
an affine continuous map
$\kappa_T: T(B)\to T(A).$
Then,
there exists a  sequence of approximate multiplicative  \cpc s $\phi_n: A\to B$ such that
\beq
&&[\{\phi_n\}]=\kappa\tand\\
&&\lim_{n\to\infty}\sup \{|\tau\circ \phi_n(a)-\kappa_T(\tau)(a)|\}=0\tforal a\in A_{s.a.}.
\eneq
\end{lem}

\begin{proof}

Let $\ep>0, \eta>0,$ ${\cal F}\subset A$ be a finite subset and ${\cal H}\subset A_{s.a}$ be
a finite subset.


Fix a finite subset ${\cal P}\subset \underline{K}(A).$
We may assume that, for some $m\ge 1,$
$$
{\cal P}\subset K_0(A)\bigoplus K_1(A)\bigoplus_{j=1}^m (K_0(A, \Z/j\Z)\bigoplus K_1(A, \Z/j\Z).
$$
Moreover, $m!x=0$ for all $x\in {\rm Tor}(K_0(A))\cap {\cal P}.$
Let $G_{0, {\cal P}}$ be the subgroup generated by $K_0(A)\cap {\cal P}.$
We may write $G_{0, {\cal P}}:=F_0\oplus G_0,$ where $F_0$ is free and $G_0$ is torsion.
In particular, $m!x=0$ for all $x\in G_0.$

Choose $\dt>0$ and finite subset ${\cal G}\subset A$ so that $[L]|_{{\cal P}}$ is well defined
for any
${\cal G}$-$\dt$-multiplicative \cpc\, $L$ from $A.$
We may assume that $\dt<\ep$ and ${\cal F}\cup {\cal H}\subset {\cal G}.$
Since both $A$ and $B$ have continuous scale,
$T(A)$ and $T(B)$ are compact
{{(5.3 of \cite{eglnp1}).}}
Choose  $a_0\in A_+$ such that $\|a_0\|=1$ and
\beq\label{11Ext1-n2}
d_\tau(a_0)<\min\{\eta, \dt\}/4\rforal \tau\in T(A).
\eneq
Let  $e_0\in A$ be a strictly positive element of $A$ with $\|e_0\|=1$
such that $\tau(e_0)>15/16$ for all $\tau \in T(A).$


Since $A\in {\cal D}_0$ (see \ref{TD0=D}),   by 10.7 of \cite{eglnp1},
there are ${\cal G}$-$\dt/4$-multiplicative \cpc s
$\phi_0: A\to \overline{\phi_0(A)A\phi_0(A)}$ and $\psi_0: A\to D\subset A$ with $D\in 
{{{\cal C}_0^{(0)}}}$ and
$M_{m!}(D)\perp \phi_0(A)$  such that
\beq\label{11Ext1-n2+}
&&\|x-(\phi_0(x)\oplus \diag(\overbrace{\psi_0(x), \psi_0(x),...,\psi_0(x)}^{m!})\|<\dt/16\rforal x\in {\cal G},\\
&&\phi_0(e_0)\lesssim a_0,\\
 &&t(f_{1/4}(\psi_0(e_0)))>1/4\rforal t\in T(D).
\eneq
Let $\Psi_{0}: A\to M_{m!}(D)\subset A$
be defined by
\beq\label{11Ext1-5}
\Psi_0(a)&=& \diag(\overbrace{\psi_0(x), \psi_0(x),...,\psi_0(x)}^{m!})
\rforal a\in A.
\eneq
Let ${\cal P}_1=[\phi_0]({\cal P})$ and ${\cal P}_2=[\Psi_0]({\cal P}).$
Put ${\cal P}_3={\cal P}\cup {\cal P}_1\cup {\cal P}_2.$
Note
that, since $K_i(D)=\{0\}$ ($i=0,1$),  $\Psi_0|_{{\cal P}\cap K_i(A)}=0,$ $i=0,1.$
Moreover, by \eqref{11Ext1-5},
\beq\label{11Ext1-6}
[\Psi_0]|_{{\cal P}\cap K_i(\Z/j\Z)}=0,\,\,\, i=0,1,\,\, j=2,....,m.
\eneq
Set
\vspace{-0.14in}\beq\label{11Ext1-6+1}
d=\inf\{d_\tau(\phi_0(e_0)): \tau\in T(A)\}.
\eneq
We also have
\vspace{-0.14in}\beq\label{11Ext1-6+2}
[\phi_0]|_{F_0}=[{\rm id}_A]|_{F_0}.
\eneq



Let ${\cal G}_1={\cal G}\cup \phi_0({\cal G}).$
Choose $0<\dt_1<\dt$  and finite subset ${\cal G}_1\subset A$ such that
$[L']|_{{\cal P}_4}$ is well defined for
any ${\cal G}_1$-$\dt_1$-multiplicative \cpc\, from $A.$

It follows from \ref{TExistence}, \ref{PADw}  and \ref{11Extmz} that there exists a ${\cal G}_1$-$\dt_1/4$-multiplicative \cpc\,
$L: A\to B\otimes M_K$
 for some
integer $K$ such that
\beq\label{11Ext1-2}
[L]|_{{\cal P}_3}=\kappa_{\zo}\circ (\kappa_{\zo}^{-1}\circ \kappa)|_{{\cal P}_3}=\kappa|_{{\cal P}_3},
\eneq
{{where $\kappa_{\zo}\in KK(B\otimes \zo, B)$ is as in \ref{KKzo=Z} with $B=C$.}}  \Wlog, we may assume that ${\cal G}_1\subset A^{\bf 1}.$

Let $b_0\in B$ with $\|b_0\|=1$ such that
\beq\label{11Ext1-3}
\tau(b_0)<\min\{\eta, \dt_1, d\}/16(K+1)\rforal \tau\in T(B).
\eneq
Let $e_b\in   B\otimes M_K$
be a strictly positive element of $B\otimes M_K$
such that
\beq\label{11Ext1-3+}
\tau(e_b)>7/8\rforal \tau\in T(B\otimes M_K).
\eneq


Let $Q\subset \underline{K}(B)$ be a finite subset
which contains $[L]({\cal P}_4).$
We assume that
\beq\label{11Ext1-3+1}
{\cal Q}\subset K_0(B)\bigoplus K_1(B)\bigoplus\bigoplus_{i=0,1}\bigoplus_{j=1}^{m_1}K_i(B, \Z/j\Z)
\eneq
for some $m_1\ge 2.$ Moreover, we may assume
that $m_1x=0$ for all $x\in {\rm Tor}(G_{0,b}),$ where $G_{0,b}$ is the subgroup
generated by ${\cal Q}\cap K_0(B).$ \Wlog, we may assume that
$m|m_1.$

Let ${\cal G}_b\subset B\otimes  M_K$ be a  finite subset and
$1/2>\dt_2>0$ be such that
$[\Phi]|_{\cal Q}$ is well defined for any ${\cal G}_b$-$\dt_2$-multiplicative \cpc\,  $\Phi$ from
$B\otimes M_K.$ Note also, by \ref{TD0=D}, $B\in {\cal D}_0.$
There are  ${\cal G}_b$-$\dt_2$-multiplicative \cpc s $\phi_{0,b}: B\otimes M_K\to \overline{\phi_{0,b}(B\otimes M_K)( B\otimes M_K)\phi_{0,b}(B\otimes M_K)}$ and $\psi_{0,b}: B\otimes M_K\to D_b,$ $M_{(m_1)!}(D_b)\subset B\otimes M_K$
with
$D_b\in 
{{{\cal C}_0^{(0)}}}$ such that
\beq\label{11Ext1-4}
\|b-(\phi_{0,b}(b),\diag(\overbrace{\psi_{0,b}(b), \psi_{0,b}(b),..., \psi_{0,b}(b)}^{(m_1)!})\|<\min\{\dt_2, \ep/16, \eta/16\}\rforal b\in {\cal G}_b\\
\andeqn \phi_{0,b}(e_b)\lesssim b_0\andeqn t(\psi_{0,b})>3/4\rforal t\in T(D_b).
\eneq

Note that $K_1(D_b)=\{0\}=K_0(D_b).$ Moreover, as in \eqref{11Ext1-n2+} and \eqref{11Ext1-6}, we may also assume
that
\beq\label{11Ext1-4+}
[\psi_{0,b}]|_{{\rm Tor}(G_{0,b})}=0\andeqn
[\psi_{0,b}]|_{{\cal Q}\cap K_i(B, \Z/j\Z)}=0,\,\,\, j=2,3,...,m_1.
\eneq
Therefore
\beq\label{11Ext1-5n}
&&[\phi_{0,b}]|_{{\rm Tor}(G_{0,b})}=[{\rm id}_B]|_{{\rm Tor}(G_{0,b})},
{[}\phi_{0,b}{]}|_{{\cal Q}\cap K_1(B)}=[{\rm id}_{B}]|_{{\cal Q}\cap K_1(B)}\andeqn\\\label{11Ext1-5n+2}
&&{[}\phi_{0,b}{]}|_{{\cal Q}\cap K_i(B, \Z/j\Z)}=[{\rm id}_B]|_{{\cal Q}\cap K_i(B, \Z/j\Z)},\,\,\, j=2,3,...,m_1.
\eneq
Let $G_{\cal P}$ be the subgroup
generated by ${\cal P}$ and let
$\kappa'=\kappa-[\phi_{0,b}]\circ [L]\circ [\phi_0]$ be defined on $G_{\cal P}.$

Then, by \eqref{11Ext1-2}, \eqref{11Ext1-5n} and \eqref{11Ext1-5n+2},
we compute that
\beq\label{11Ext1-101}
&&\kappa'|_{G_{0,{\cal P}}}=0,\,\,\,
\kappa'|_{{\cal P}\cap K_1(A)}=0\andeqn\\
&&\kappa'|_{{\cal P}\cap K_i(A, \Z/j\Z)}=0,\,\,\, j=2,3,...,m.
\eneq




Let $\imath: M_{m!}(D)\to A$ be the embedding.

Let $\kappa_T^{\sharp}: \Aff(T(A))\to \Aff(T(B))$  {{be given by $\kappa_T.$}} This induces
an order semigroup \hm\, ${\tilde \kappa}^T: {\rm LAff}_+({\tilde T}(A))\to {\rm LAff}_+({\tilde T}(B)).$
By \ref{Ttrzstable} and \ref{Lcompactcon}, one checks easily that $\kappa_T^{\sharp}$ is a Cuntz semigroup
\hm\, and a morphism in ${\bf Cu}.$

Let $\gamma':
Cu(M_{m!}(D))\to {\rm LAff}_+({\tilde T}(B))$ be
the Cuntz semi-group \hm\, given by $\gamma'=\kappa_T^{\sharp}\circ Cu(\imath).$
Put
$\gamma: Cu{{(M_{m!}(D))}}\to
 {\rm LAff}_+({\tilde T}(B))$
defined by
$\gamma(f)=(1-\min\{\eta, \eta_0\}/2(m!))\gamma'(f)$ for all $f \in Cu(M_{m!}(D)).$



Let  $\gamma_0: Cu^{\sim}(M_{m!}(D))\to Cu^{\sim}(B)$  be the morphism induced by
$\gamma$ (note $K_0(M_{m!}(D))=\{0\}$ {{and see also 7.3 of \cite{eglnp1}}}).

By applying 1.0.1 of \cite{Rl}, one obtains a \hm\, $h_{d}: M_{m!}(D)\to B$
such
that
\beq
(h_{d})_{*0}=\gamma_{00}\andeqn
\tau\circ h_{d}(c)=\gamma(\widehat{c})(\tau)\rforal \tau\in T(B)\andeqn c\in (M_{m!}(D))_{s.a.}
\eneq
Define $h: A\to B$ by $h=h_d\circ \Psi_0.$
Then
\beq
[h]|_{{\cal P}}=\kappa'|_{\cal P},
{[}h{]}|_{{\cal P}\cap K_1(A)}=0\andeqn
{[}h{]}|_{{\cal P}\cap K_i(\Z/j\Z)}=0,\,\,\, i=2,3,...,m.
\eneq
Moreover,
\vspace{-0.15in}\beq
\tau(h(a))=\gamma(\widehat{\Psi_0(a)})\rforal a\in A \andeqn \tau\in T(B).
\eneq


Let $e_{d}\in M_{m!}(D)$ be a strictly positive element with $\|e_d\|=1.$
Then, by \eqref{11Ext1-6+1},
\beq\label{11Ext1-17}
d_\tau(h_d(e_d))<1-d \rforal \tau\in T(B).
\eneq
It follows from \eqref{11Ext1-3} that
\beq\label{11Ext1-18}
d_\tau(h(e_d))+d_\tau(\phi_{0,b}(e_0))<1\rforal \tau\in T(B).
\eneq
Note that $B$ has stable rank one
{{(see 11.5 of \cite{eglnp1} and
15.5 of \cite{GLp1}).}}
By omitting  {{a conjugating}} unitary in $B$ \wilog, we may
assume
that $\phi_{0,b}\circ L\oplus h$ maps $A$ into $B.$
Put $\Phi=\phi_{0,b}\circ L\oplus h.$ Then $\Phi$ is ${\cal G}$-$\dt$-multiplicative.
Moreover, we compute that
\beq
[\Phi]|_{\cal P}=\kappa|_{\cal P}\andeqn
\sup\{|\tau(\Phi(x))-\kappa_T(\tau)(x)|: \tau\in T(B_T)\}<\eta\rforal x\in {\cal H}.
\eneq
The lemma then follows.

\end{proof}

\begin{lem}\label{11EXt2}

Let $A$ be a non-unital simple separable \CA\, in ${\cal D}$ with $K_0(A)={\rm ker}\rho_A$
and with continuous scale
which satisfies the UCT.  Let $B_T$ be as in \ref{Dcc1}.
Suppose that there is $\kappa\in KL(B_T,A),$
an affine continuous map
$\kappa_T: T(A)\to T(B_T)$
and a continuous \hm\, $\kappa_{uc}: U({\tilde B_T})/CU({\tilde B_T})\to U({\widetilde{A}})/CU({\widetilde{A}})$ such that $(\kappa, \kappa_T, \kappa_{uc})$ is compatible.
Then there exists a  sequence of approximate multiplicative  \cpc s $\phi_n: B_T\to A$ such that
\vspace{-0.15in}\beq
&&[\{\phi_n\}]=\kappa\\
&&\lim_{n\to\infty}\sup \{|\tau\circ \phi_n(a)-\kappa_T(\tau)(a)|\}=0\rforal a\in (B_T)_{s.a.}\tand\\
&&\lim_{n\to\infty}{\rm dist}(\kappa_{uc}(z), \phi_n^{\dag}(z))=0\rforal z\in U({\tilde B_T})/CU({\tilde B_T}).
\eneq
\end{lem}

\begin{proof}
Let $\ep>0,$ let $\eta>0$ and let $\sigma>0,$ let ${\cal P}\subset \underline{K}(B_T)$ be a finite subset,
let $S_u\subset U({\tilde B_T})/CU({\tilde B_T})$ be a finite subset, let ${\cal H}\subset (B_T)_{s.a.}$ be a finite
subset and
let ${\cal F}\subset B_T$ be a finite subset.

\Wlog, we may assume that ${\cal F}\subset (B_T)^{\bf 1},$ {{and,}}
$[L']|_{\cal P}$ and $(L')^{\dag}|_{S_u}$ are well-defined for any
${\cal F}$-$\ep$-multiplicative \cpc\, from $B_T.$

Let $G_1\subset K_1(B_T)$ be the subgroup generated by ${\cal P}\cap K_1(B_T).$

Fix $\dt>0$ and a finite subset ${\cal G}\subset B_T.$
We assume that $\dt<\min\{\ep/2, \eta/4, \sigma/16\}.$
To simplify notation, \wilog, we may assume
that $G_1\subset F\subset (\Phi_{n_0, \infty})_{*1}(K_1(B_{n_0}))$ for some $n_0\ge 1,$
where $F$ is a finitely generated standard subgroup {{of $K_1(B_T)$}} (see \ref{DdertF}).
We also choose $n_0$ larger than that required by \ref{L215} for $\dt$ (in place of $\ep$) ${\cal G}$
(in place of ${\cal F}$)
${\cal P}$ and $\sigma/16$ (in place of $\dt_0$).
\Wlog, we may write
\beq
S_u=S_{u,1}\sqcup S_{u,0},
\eneq
where $S_{u,1}\subset J_{F,u}(F)$ and $S_{u,0}\subset U_0({\tilde B_T})/CU({\tilde B_T})=\Aff(T({\tilde B_T}))/\Z$
and both $S_{u,1}$ and $S_{u,0}$ are finite subsets.
For $w\in S_{u,0},$
write
\vspace{-0.1in}\beq
w=\prod_{j=0}^{l(w)}\exp({{\sqrt{-1}}}2\pi h_{w,j}),
\eneq
where  $h_{w,0}\in \R$ and $h_{w,j}\in (B_T)_{s.a.},$ $j=1,2,...,l(w).$
Let
\vspace{-0.1in}\beq
{\cal H}_u=\{h_{w,j}: 1\le j\le l(w),\,\,\, w\in S_{u,0}\}\andeqn
M=\max\{\sum_{i=0}^{l(w)}\|h_{w,j}\|: w\in S_{u,0}\}.
\eneq
To simplify notation further, we may assume that $G_1=F.$

Write $G_1=\Z^{m_f}\oplus {\rm Tor}(G_1)$ and
$\Z^{m_f}$ is generated by cyclic and free generators $x_1,x_2,...,x_{m_f}.$
Let ${\rm Tor}(G)$ be generated by $x_{0,1}, x_{0,2},...,x_{0, m_t}.$
Let $u_1, u_2,..., u_{m_f}, u_{1,0},u_{2,0},...,u_{m_t,0}\in U({\tilde B_T})$
be unitaries such
that $[u_i]=x_i,$ $i=1,2,...,m_f,$ and $[u_{j,0}]=x_{0,j},$ $j=1,2,...,m_t.$
Let $\pi_u: U({\tilde B_T})/CU({\tilde B_T})\to K_1(B_T)$ be the quotient map
and let $G_u$ be the subgroup generated by $S_{u,1}.$
Since $(\kappa, \kappa_T, \kappa_u)$ is compatible, \wilog,
we may assume that $\pi_u(G_u)=\{x_1,x_2,...,x_{m_f}\}\cup \{x_{0,1}, x_{0,2},...,x_{0, m_t}\}$
and $S_{u,1}=\{{\bar u}_1, {\bar u}_2,...,{\bar u}_{m_f}, {\bar u}_{1,0},{\bar u}_{2,0},...,{\bar u}_{m_t,0}\}$
as described in \ref{DdertF}, in particular, $k_j{\bar u}_{j,0}=0$ in $U({\tilde B_T})/CU({\tilde B_T}),$
$j=1,2,...,m_t.$

Let $\phi_n: B_T\to A$ be a sequence of approximately multiplicative \cpc s
given by \ref{11Ext1} such that
\beq\label{11Ext2-3}
[\{\phi_n\}]&=&\kappa\andeqn\\\label{11Ext2-3+}
\lim_{n\to\infty}\sup \{|\tau\circ \phi_n(a)-\kappa_T(\tau)(a)|\}&=&0\rforal a\in (B_T)_{s.a.}.
\eneq

Fix a strictly positive element $e_b\in B_T$ with
$\|e_b\|=1$ and $\tau(e_b)\ge 15/16$  and
$\tau(f_{1/2}(e_b))\ge 15/16$ for all $\tau\in T(B_T).$

Let ${\cal F}_b\subset B_T$ be a finite subset which contains ${\cal F}\cup {\cal H}\cup {\cal H}_u.$
and let  $\dt_b>0.$
There are ${\cal F}_a$-$\dt_b$-multiplicative \cpc s $\Phi_0:B_T \to D_b\subset B_T$ with
$D_b\in C_0^{0},$  $\Phi_1: B_T\to B_T$ and $\Phi_1(B_T)\perp D_b$ such that
\beq\label{11Ext2-11}
\|b-(\Phi_0(b)\oplus \Phi_1(b))\|<\dt_b/2\rforal b\in {\cal F}_b\andeqn\\\label{11EXT2-11+}
0<d_\tau(\Phi_0(e_b))<\min\{\eta, \sigma/16\}/4(M+1)\rforal \tau\in T(B_T).
\eneq
Note that $K_0(D_b)=K_1(D_b)=\{0\}.$  Therefore, for any sufficiently large $n,$
\beq\label{11Ex2-12}
&&[\phi_n\circ \Phi_0]|_{\cal P}=0,\,\,\, [\phi_n\circ \Phi_1]|_{\cal P}=\kappa|_{\cal P}\andeqn\\\label{w11Ex2-12+}
&&d_{\tau}(\phi_n(\Phi_0(e_b)))<\min\{\eta, \sigma/16\}/2(M+1)\rforal \tau\in T(A).
\eneq

Fix a sufficiently large $n.$
Define $\lambda=\kappa|_{G_u}-(\phi_n\circ \Phi_1)^{\dag}|_{G_u}: G_u\to U({\widetilde{A}})/CU({\widetilde{A}}).$ Since $(\kappa, \kappa_T, \kappa_u)$ is compatible,
$\pi_u\circ \lambda({\bar u_i})=0$ and $\pi_u\circ \lambda({\bar u}_{0,j})=0,$
$i=1,2,...,m_f$ and $j=1,2,...,m_t.$


Let ${\cal F}_1={\cal F}\cup {\cal H}.$
It follows from \ref{L215} that there exists ${\cal F}_1$-$\min\{\ep/4, \eta/4\}$-multiplicative \cpc\,
$L: B_T\to \overline{cAc},$ where $c=\phi_n\circ \Phi_0(e_b),$
such that
\beq\label{11Ext2-15}
[L]|_{\cal P}=0\andeqn {\rm dist}(L^{\dag}({\bar u}_j), \lambda({\bar u}_j))<{{\sigma/4}},\,\,\,j=1,2,...,m_f.
\eneq

Define $\Psi: B_T\to A$ by
\beq
\Psi(a)=L(a)\oplus \phi_n\circ \Phi_1(a)\rforal a\in B_T.
\eneq
Then $\Psi$ is ${\cal F}$-$\ep$-multiplicative if $n$ is sufficiently large.
{{By \eqref{11EXT2-11+},  by \eqref{11Ext2-3+} and by choosing sufficiently large $n,$
\beq\label{127-nn1}
&&\sup \{|\tau\circ \phi_n(a)-\kappa_T(\tau)(a)|\}<\min\{\sigma/16, \eta\}/(M+1)\rforal a\in {\cal H}_u\andeqn\\
&&\sup\{|\tau(\Psi(b))-\kappa_T(\tau)(b)|: \tau\in T(A)\}<\min\{\sigma/16, \eta\}\rforal b\in {\cal H}.
\eneq
}}
It follows from \eqref{11Ex2-12}, \eqref{11Ext2-15} and the definition of $\lambda$ that
\beq
[\Psi]|_{\cal P}=\kappa|_{\cal P}\andeqn
{\rm dist}(\Psi^{\dag}({\bar u}_j), \kappa_{uc}({\bar u}_j))&<&{{\sigma/2}},\,\,\, j=1,2,...,m_f.
\eneq
By \ref{Lderttorsion}, we may also have
\beq
{\rm dist}(\Psi^{\dag}({\bar u}_{j,0}), \kappa_{uc}({\bar u}_{j,0}))&<&{{\sigma}},\,\,\, j=1,2,...,m_t.
\eneq
By the choice of $M$ and ${\cal H}_u,$  \eqref{11EXT2-11+}, {{and \eqref{127-nn1},}}
and by the assumption
that $(\kappa, \kappa_T, \kappa_{uc})$ is compatible,
\beq
{\rm dist}(\Psi^{\dag}({\bar w}), \kappa_{u,c}({\bar w}))<\sigma\rforal w\in S_{u,0}.
\eneq

\end{proof}

\begin{thm}\label{11ExtT1}
Let $A$ be a separable amenable simple \CA\, in ${\cal D}_0$
with continuous scale
which satisfies the UCT.  Let $B_T$ be as in \ref{Dcc1}.
Suppose that there is $\kappa\in KL(B_T,A),$
an affine continuous map
$\kappa_T: T(A)\to T(B_T)$
and a continuous \hm\, $\kappa_{uc}: U({\tilde B_T})/CU({\tilde B_T})\to U({\widetilde{A}})/CU({\widetilde{A}})$ such that $(\kappa, \kappa_T, \kappa_{uc})$ is compatible.
Then there exists a  \hm\, $\phi: B_T\to A$
 such that
\beq\label{11ExtT1-1}
&&[\phi]=\kappa,\,\,\,
\tau\circ \phi(a)=\kappa_T(\tau)(a)\tforal a\in (B_T)_{s.a.}\tand
\phi^{\dag}=\kappa_{uc}.
\eneq
\end{thm}

\begin{proof}
Let $e_b\in B_T$ be a strictly positive element of $B_T$ with $\|e_b\|=1.$
Since $A$ has continuous scale, \wilog, we may assume that
\beq\label{11ExtT1-5}
\min\{\inf\{\tau(e_b): \tau\in T(B_T)\},\inf\{\tau(f_{1/2}(e_b)):\tau\in T(B_T)\}\}>3/4.
\eneq
Let $T: (B_T)_+\setminus \{0\}\to \N\times \R_+\setminus \{0\}$ be given by
{{Theorem 5.7 of \cite{eglnp1}.}}

By \ref{11EXt2}, there exists a sequence of approximately multiplicative  \cpc s $\phi_n: B_T\to A$ such that
\beq
&&[\{\phi_n\}]=\kappa\\
&&\lim_{n\to\infty}\sup \{|\tau\circ \phi_n(a)-\kappa_T(\tau)(a): \tau\in T(A)|\}=0\rforal a\in (B_T)_{s.a.}\andeqn\\
&&\lim_{n\to\infty}{\rm dist}(\kappa_{uc}(z), \phi_n^{\dag}(z))=0\rforal z\in U({\tilde B_T})/CU({\tilde B_T}).
\eneq

Let $\ep>0$ and ${\cal F}\subset B_T$ be a finite subset.

We will apply \ref{TUNIq}. Note that $K_0({\tilde A})$
is weakly unperforated (see \ref{Tweakunp} and \ref{Dzstable}).
 $\dt_{1,1}>0$ (in place of $\dt$), $\gamma_1>0$ (in place of $\gamma$),
$\eta_1>0$ (in place of $\eta$), let ${\cal G}_{1,1}\subset B_T$ (in place of ${\cal G}$) be a finite subset,
${\cal H}_{1,1}\subset (B_T)_+\setminus \{0\}$ (in place of ${\cal H}_1$) be a finite subset, ${\cal P}_1\subset \underline{K}(B_T)$ (in place of ${\cal P}$),
${\cal U}_1\subset U({\tilde U})$ (in place of ${\cal U}$) with
$\overline{{\cal U}}={\cal P}\cap K_1(B_T)$  and let ${\cal H}_{1,2}\subset (B_T)_{s.a.}$ (in place of
${\cal H}_2$) {{be as}} required by Theorem \ref{TUNIq} for $T,$ $\ep$ and ${\cal F}$ (with $T(k,n)=n,$ see \ref{RbTuniq}).

\Wlog, we may assume that ${\cal H}_{1,1}\subset (B_T)_+^{\bf 1}\setminus \{0\}$
and $\gamma_1<1/64.$

Let ${\cal G}_{1,2}\subset B_T$  (in place of ${\cal G}$)  be a finite subset and let
$\dt_{1,2}>0$ be {{as}} required by {{Theorem 5.7 of \cite{eglnp1}}}
for the above ${\cal H}_{1,1}$
(in place of ${\cal H}_1$).
Let $\dt_1=\min\{\dt_{1,1}, \dt_{1,2}\}$  and ${\cal G}_1={\cal G}_{1,1}\cup {\cal G}_{1,2}.$

Choose $n_0\ge 1$ such that
$\phi_n$ is ${\cal G}_1$-$\dt_1/2$-multiplicative,  for all $n\ge n_0,$
\beq
&&[\phi_n]|_{{\cal P}_1}=\kappa|_{{\cal P}_1},\\
&&\sup\{|\tau\circ \phi_n(a)-\kappa_T(\tau)(a)|:\tau\in T(B_T)\}<\gamma_1/2\rforal  a\in {\cal H}_{1,2},\\
&&\tau(f_{1/2}(\phi_n(e_a)))>3/8\rforal \tau\in T(B_T)\andeqn\\
&&{\rm dist}(\phi_n^{\dag}({\bar u}), \kappa_{uc}({\bar u}))<\eta/2\rforal u\in {\cal U}.
\eneq
By applying
5.7 of \cite{eglnp1}, $\phi_n$ are all
$T$-${\cal H}_{1,1}$-full.
By applying Theorem \ref{TUNIq}, we obtain a unitary $u_n\in {\tilde B_T}$ (for each $n\ge n_0$)
such that
\beq\label{11ExtT1-15}
\|u_n^*\phi_n(a)u_n-\phi_{n_0}(a)\|<\ep\rforal a\in {\cal F}.
\eneq

Now let $\{\ep_n\}$ be an decreasing sequence of positive elements
such that $\sum_{n=1}^{\infty}\ep_n<\infty$ and let
$\{{\cal F}_n\}$ be an increasing sequence of finite subsets of $B_T$
such that $\cup_{n=1}^{\infty}{\cal F}_k$ is dense in $B_T.$

By what have been proved, we obtain a subsequence $\{n_k\}$ and
a sequence of unitaries $\{u_k\}\subset {\tilde B_T}$ such that
\beq\label{11ExtT1-16}
\|{\rm Ad}\, u_{k+1}\circ \phi_{n_{k+1}}(a)-{\rm Ad}\, u_{k}\circ \phi_{n_k}(a)\|<\ep_k\rforal a\in {\cal F}_k,
\eneq
$k=1,2,....$
Since $\cup_{n=1}^{\infty}{\cal F}_k$ is dense in $B_T,$  by \eqref{11ExtT1-16},
$\{{\rm Ad}\, u_k\circ \phi_{n_k}(a)\}$ is a Cauchy sequence.
Let
\beq
\phi(a)=\lim_{k\to\infty}{\rm Ad}\, u_k\circ \phi_{n_k}(a)\rforal a\in B_T.
\eneq
Then $\phi: B_T\to A$ is a \hm\, which satisfies \eqref{11ExtT1-1}.

\end{proof}

\begin{lem}\label{11Ext2n}

Let $A$ be a non-unital simple separable \CA\, in ${\cal D}$ with $K_0(A)={\rm ker}\rho_A$
and with continuous scale
which satisfies the UCT.  Let $B_T$ be as in \ref{Dcc1}.
Suppose that there is $\kappa\in KL(A,B_T),$
an affine continuous map
$\kappa_T: T(B_T)\to T(A)$
and a continuous \hm\, $\kappa_{uc}: U({\tilde A})/CU({\tilde A})\to U({\widetilde{B_T}})/CU({\widetilde{B_T}})$ such that $(\kappa, \kappa_T, \kappa_{uc})$ is compatible.
Suppose also that $\kappa|_{K_1(A)}$ is injective.

Then there exists a  sequence of approximate multiplicative  \cpc s $\phi_n: A\to B_T$ such that
\beq\label{11Ext2n-1}
&&[\{\phi_n\}]=\kappa,\\\label{11Ext2n-2}
&&\lim_{n\to\infty}\sup \{|\tau\circ \phi_n(a)-\kappa_T(\tau)(a)|\}=0{{\tforal}} a\in A_{s.a.}\tand\\
&&\lim_{n\to\infty}{\rm dist}(\kappa_{uc}(z), \phi_n^{\dag}(z))=0{{\tforal}} z\in U({\tilde A})/CU({\tilde A}).
\eneq
\end{lem}

\begin{proof}
Denote by $\Pi: U({\tilde A})/CU({\tilde A})\to K_1(A)$
{{the}} quotient map and fix
a splitting map $J_u: K_1(A)\to U({\tilde A})/CU({\tilde A}).$
Since $(\kappa, \kappa_T, \kappa_{uc})$ is compatible, it suffices to show
that there are $\{\phi_n\}$ which satisfies \eqref{11Ext2n-1} and \eqref{11Ext2n-2} and
\beq
\lim_{n\to\infty}{\rm dist}(\kappa_{uc}(J_u(\zeta)), \phi_n^{\dag}(J_u(\zeta)))=0\rforal \zeta\in K_1(A).
\eneq

It follows from \ref{11Ext1} that there exists $\{\phi_n\}$  which  satisfies \eqref{11Ext2n-1} and \eqref{11Ext2n-2}.
Let $G_1\subset K_1(A)$ be a finitely generated subgroup.

Choose some sufficiently large $n,$
then $\phi_n^{\dag}$ induces a \hm\, on the {{group}} $J_u(G_1).$
Since $\kappa|_{K_1(A)}$ is injective and $(\kappa, \kappa_T, \kappa_{uc})$ is compatible,
$\phi_n^{\dag}|_{J_u(G_1)}$ has an inverse
$\gamma.$
Let $G_b=\phi_n^{\dag}(J_u(G_1))$ and let $\Pi_b: U({\tilde B_T})/CU({\tilde B_T})\to K_1(B_T)$ be the quotient map. Again, using the fact that $(\kappa, \kappa_T, \kappa_{uc})$
is compatible, $(\Pi_b)|_{G_b}$ is injective.
Let $J_{ub}:K_1(B_T)\to U({\tilde B_T})/CU({\tilde B_T})$ be a {{\hm\,}}  such that
$\Pi_b\circ J_{uc}={\rm id}_{K_1(B_T)}.$

Put
\vspace{-0.15in}\beq\label{11ExtTn2-1}
\lambda_0=((\kappa_{uc}\circ \gamma)\circ J_{uc}-(\phi_n)^{\dag}\circ \gamma\circ J_{uc})|_{\Pi_b(G_b)}.
\eneq

Then, since $(\kappa, \kappa_T, \kappa_{uc})$ is compatible,
$\Pi_b\circ \lambda_0=0.$
Therefore $\lambda_0$ maps from $\Pi_b(G_b)$ to $\Aff(T({\tilde B_T}))/\overline{\rho_{B_T}(K_1({\tilde B_T}))}.$
However, $\Aff(T({\tilde B_T}))/\overline{\rho_{B_T}(K_1({\tilde B_T}))}$ is divisible.
Therefore there is a \hm\, $\lambda_1: K_1(B_T)\to \Aff(T({\tilde B_T}))/\overline{\rho_{B_T}(K_1({\tilde B_T}))}$
such that
\beq
(\lambda_1)|_{\Pi_b(G_b)}=\lambda_0.
\eneq
Now defined $\Lambda: U({\tilde B_T})/CU({\tilde B_T})\to U({\tilde B_T})/CU({\tilde B_T})$  as follows.
\beq
\Lambda|_{\Aff(T({\tilde B_T}))/\overline{\rho_{B_T}(K_1({\tilde B_T}))}}={\rm id}_{\Aff(T({\tilde B_T}))/\overline{\rho_{B_T}(K_1({\tilde B_T}))}},\\
\Lambda|_{J_{ub}(K_1(B_T))}=\lambda_1\circ \Pi_b+({\rm id}_{B_T})^{\dag}.
\eneq

Note that $([{\rm id}_{B_T}], ({\rm id}_{B_T})_T, \Lambda)$ is compatible.
It follows from \ref{11EXt2} that there exists a \hm\, $\psi_n: B_T\to B_T$
such that
\beq
[\psi_n]=[{\rm id}_{B_T}], \,\,\, (\psi_n)_T=({\rm id}_{B_T})_T\andeqn \psi_n^{\dag}=\Lambda.
\eneq
Now let $\Phi_n=\psi_n\circ \phi_n.$
Then, for $z\in J_u(G_1),$   by \eqref{11ExtTn2-1},
\beq
\Phi_n^{\dag}(z)&=&\psi_n^{\dag}\circ \phi_n^{\dag}(z)=\lambda_1\circ \Pi_b\circ \phi_n^{\dag}(z)+\phi_n^{\dag}(z)\\
&=&\lambda_0\circ \phi_n^{\dag}(z)+\phi_n^{\dag}(z)=\kappa_{uc}(z).
\eneq
The lemma  follows immediately from this construction of $\Phi_n.$


\end{proof}

\begin{lem}\label{11Ext2n2}

Let $A$ be a non-unital simple separable \CA\, in ${\cal D}_0$
with continuous scale
which satisfies the UCT.  Let $B_T$ be as in \ref{Dcc1}.
Suppose that there is $\kappa\in KL(A, B_T),$
an affine continuous map
$\kappa_T: T(B_T)\to T(A),$
and a continuous \hm\, $\kappa_{uc}: U({\tilde A})/CU({\tilde A})\to U({\widetilde{B_T}})/CU({\widetilde{B_T}})$ such that $(\kappa, \kappa_T, \kappa_{uc})$ is compatible.
Suppose also that $\kappa|_{K_1(A)}$ is injective.

Then there exists a  \hm\, $\phi: A\to B_T$ such that
\beq\label{11Ext3n-1}
[\phi]=\kappa,\,\,\, \phi_T=\kappa_T\andeqn \phi^{\dag}=\kappa_{uc}.
\eneq
\end{lem}

\begin{proof}
The proof is exactly the same as that of \ref{11ExtT1} but applying \ref{11Ext2n} instead of \ref{11EXt2}.
\end{proof}

\section{The Isomorphism Theorem for $\zo$-stable \CA s}

\begin{thm}\label{Misothm}
Let $A$ and $B$ be two separable simple  amenable \CA s in ${\cal D}$ with continuous scale which satisfy the UCT.
Suppose that ${\rm ker}\rho_A=K_0(A)$ and ${\rm ker}\rho_B=K_0(B).$
Then $A\cong B$ if and only if
\beq
(K_0(A), K_1(A), T(A))\cong (K_0(B), K_1(B), T(B)).
\eneq
Moreover,
let $\kappa_i: K_i(A)\to K_i(B)$ be an isomorphism as abelian groups ($i=0,1$) and
let $\kappa_T: T(B)\to T(A)$ be an affine homeomorphism.
Suppose that $\kappa\in KL(A,B)$ which gives $\kappa_i$ and
$\kappa_{cu}: U({\tilde A})/CU({\tilde A})\to U({\tilde B})/CU({\tilde B})$ is
a continuous affine isomorphism so that
$(\kappa, \kappa_T, \kappa_{cu})$ is compatible.
Then there is an isomorphism $\phi: A\to B$ such that
\beq\label{Misothm-1-1}
[\phi]=\kappa\,\,\,{\rm (} i=0,1,\,\, \phi_T=\kappa_T\andeqn \phi^{\dag}=\kappa_{cu}
\eneq
\end{thm}

\begin{proof}
Note it follows from \ref{TD0=D} that $A,\, B\in {\cal D}_0.$
It follows from \ref{MainModel} that there is a non-unital simple \CA\, $B_T$ constructed
in section {{7}} such that
\beq\label{Misothm-1}
K_0(B_T)=K_0(B), K_1(B_T)=K_1(B)\andeqn T(B_T)=T(B).
\eneq
Let $\kappa\in KL(A, B)$ be an invertible element which gives
$\kappa_i$ ($i=0,1$). Let $\kappa_T: T(B)\to T(A)$ be an affine homeomorphism.
By the assumption, $(\kappa, \kappa_T)$ is always compatible.
Choose any $\kappa_{cu}$ so that $(\kappa, \kappa_T, \kappa_{cu})$ is compatible.
Note that there is always at least one: $\kappa_{cu}|_{J_c(K_1(A))}=J_c\circ \kappa|_{K_1(A)}\circ \pi_{cu},$
where $\pi_{cu}: U({\tilde A})/CU({\tilde A})\to K_1(A)$ is the quotient map
and $\kappa_{cu}|_{\Aff(T(A))/\Z}$ {{is}} induced by $\kappa_T.$

Therefore it suffices to show that there is  an isomorphism $\phi: A\to B$ such that
\eqref{Misothm-1-1} holds. We will use the Elliott intertwining argument.

Let $\{{\cal F}_{a,n}\}$ be an increasing sequence of finite subsets of $A$ such that
$\cup_{n=1}^{\infty}{\cal F}_{a,n}$ is dense in $A,$ let $\{{\cal F}_{b,n}\}$ be an increasing
sequence of finite subsets of $B$ such that $\cup_{n=1}^{\infty} {\cal F}_{b,n}$ is dense in
$B.$
Let $\{\ep_n\}$ be a sequence of  decreasing positive numbers  such that
$\sum_{n=1}^{\infty}\ep_n<1.$

Let $e_a\in A$ and $e_b\in B$  be  strictly positive elements of $A$  and $B,$ respectively, with $\|e_a\|=1$
and with $\|e_b\|=1.$
Note that $d_\tau(e_a)=1$ for all $\tau\in T(A)$ and $d_\tau(e_b)=1$ for
all $\tau\in T(B).$

It follows from \ref{11Ext2n2} that there is a \hm\, $\phi_1: A\to B$ such that
\beq\label{Misothm-5}
[\phi_1]=\kappa,\,\,\, (\phi_1)_T=\kappa_T\andeqn \phi_1^{\dag}=\kappa_{cu}.
\eneq
Note that $d_\tau(\phi_1(e_a))=1.$ Therefore
$\phi_1$ maps $e_a$ to a strictly positive element of $B.$
It follows from \ref{11EXt2} that there is a \hm\,
$\psi_1': B\to A$ such that
\beq\label{Misothm-6}
[\psi_1']=\kappa^{-1},\,\, (\psi_1')_T=\kappa_T^{-1}\andeqn (\psi_1')^{\dag}={\rm id}_A^{\dag}\circ ( \phi_1^{\dag})^{-1}.
\eneq
Thus
\vspace{-0.1in}\beq
[\psi_1'\circ \phi_1]=[{\rm id}_A],\,\,\, (\psi_1'\circ \phi_1)_T={\rm id}_{T(A)}\andeqn
(\psi_1'\circ \phi_1)^{\dag}={\rm id}_{U({\tilde A})/CU({\tilde A})}.
\eneq
It follows from \ref{TUNIq} (see also \ref{Rsec53}) that there exists a unitary $u_{1,a}\in {\tilde A}$ such that
\beq\label{Misothm-7}
{\rm Ad}\, u_{1,a}\circ \psi_1'\circ \phi_1\approx_{\ep_1} {\rm id}_A\,\,\,{\rm on}\,\,\, {\cal F}_{a,1}.
\eneq
Put $\psi_1={\rm Ad}\, u_{1,a}\circ \psi_1'.$
Then we obtain the following diagram
\vspace{-0.12in} \begin{displaymath}
\xymatrix{
A \ar[r]^{{\rm id}_A} \ar[d]_{\phi_1} & A\\
B \ar[ur]_{\psi_1}
}
\end{displaymath}
which is approximately commutative on the subset ${\cal F}_{a,1}$ within $\ep_1.$

By applying \ref{11Ext2n2},  there exists a \hm\, $\phi_2': A\to B$ such that
\beq\label{Misothm-8}
[\phi_2']=\kappa,\,\,\, (\phi_2')_T=\kappa_T\andeqn  (\phi_2')^{\dag}={\rm id}_B^{\dag}\circ (\psi_1^{\dag})^{-1}=\kappa_{cu}.
\eneq
Then,
\vspace{-0.1in}\beq
[\phi_2'\circ \psi_1]=[{\rm id}_B],\,\,\, (\phi_2'\circ \psi_1)_T\andeqn (\phi_2'\circ \psi_1)^{\dag}={\rm id}_{U({\tilde B})/CU({\tilde B})}.
\eneq
It follows from \ref{TUNIq} (and \ref{Rsec53}) that there exists a unitary $u_{2,b}\in {\tilde B}$ such that
\beq\label{Misothm-7+}
{\rm Ad}\, u_{2,b}\circ \phi_2''\circ \psi_1\approx_{\ep_2} {\rm id}_B\,\,\,{\rm on}\,\,\, {\cal F}_{b,2}\cup \phi_1({\cal F}_{a,1}).
\eneq
Put $\phi_2={\rm Ad}\, u_{2,b}\circ \phi_2.$ Then we obtain the following diagram:
 \begin{displaymath}
\xymatrix{
A \ar[r]^{{\rm id}_A} \ar[d]_{\phi_1} & A \ar[d]^{\phi_2}\\
B\ar[ur]_{\psi_1}\ar[r]_{{\rm id}_B} & B
}
\end{displaymath}
with the upper triangle approximately commutes on $\mathcal F_{a,1}$ within  $\ep_1$ and the lower triangle approximately commutes on ${\cal F}_{b,2}\cup \phi_1({\cal F}_{a,1})$ within $\ep_2.$
Note also
\beq\label{Misothm-10}
[\phi_2]=\kappa,\,\,\, (\phi_2)_T=\kappa_T\andeqn  (\phi_2)^{\dag}=\kappa_{cu}.
\eneq

We then continue this process, and, by the induction, we obtain an approximate intertwining:
 \begin{displaymath}
 \xymatrix{
A  \ar[r]^{{\rm id}_A}\ar[d]_{\phi_1}  &   A  \ar[r]^{{\rm id}_A}\ar[d]_{\phi_2}  &   A  \ar[r]^{{\rm id}_A}\ar[d]_{\phi_3} &   \cdots \cdots A \\
  B  \ar[r]_{{\rm id}_B}\ar[ru]^{\psi_1}&    B  \ar[r]_{{\rm id}_B}     \ar[ru]^{\psi_2}&   B\ar[r]_{{\rm id}_B}&  \cdots \cdots B  \\
 }
\end{displaymath}

By the Elliott approximate intertwining argument, this implies
that $A\cong B$ and the isomorphism $\phi$  produced by the above diagram meets the requirements
of \eqref{Misothm-1-1}.
\end{proof}

The following theorem and its proof gives the proof of Theorem \ref{TTT1}.

\begin{thm}\label{T1main}
Let $A$ and $B$ be two  stably  projectionless  separable simple amenable \CA s
with
$gTR(A)\le 1$ and $gTR(B)\le 1$  and which satisfy
the UCT.   Suppose that $K_0(A)={\rm ker}\rho_A$ and $K_0(B)={\rm ker}\rho_B.$
Then $A\cong B$ if and only if
\beq
(K_0(A), K_1(A), {\tilde T}(A), \Sigma_A)\cong (K_0(B), K_1(B), {\tilde T}(B), \Sigma_B).
\eneq

\end{thm}

\begin{proof}
Let
\hspace{-0.12in}\beq
\Gamma:(K_0(A), K_1(A), {\tilde T}(A), \Sigma_A)\to  (K_0(B), K_1(B), {\tilde T}(B), \Sigma_B)
\eneq
be an isomorphism.
Let $\Gamma_T: {\tilde T}(A)\to {\tilde T}(B)$ be the cone  homeomorphism
such that
\beq
\Sigma_B(\Gamma_T(\tau))=\Sigma_A(\tau)\rforal \tau\in {\tilde T}(A).
\eneq
Let $e_A\in  {\rm Ped}(A)_+$ such that $\|e_A\|=1$ such that
$A_0:=\overline{e_AAe_A}$  has continuous scale {{(see
5.3 of
\cite{eglnp1}).}}
Choose $b_0\in P(B)_+\setminus \{0\}$ with $\|b_0\|=1$ such that
$B':=\overline{b_0Bb_0}$ has continuous scale.
Then $T(A_0)$  and $T(B')$ are  metrizable Choquet simpleces.
Moreover $T(A_0)$  and $T(B')$ can be identified with
\beq
T_A=\{\tau\in {\tilde T}(A): d_\tau(a_A)=1\}\andeqn \{s\in {\tilde T}(B'): d_s(b_0)=1\},
\eneq
respectively.
Let $g(t)=d_{\Gamma^{-1}(t)}(e_A)\in {\rm LAff}_f({\tilde T}(B)).$
Since $d_\tau(e_A)$ is continuous and $\Gamma^{-1}$ is a cone homeomorphism,
$g(t)$ is continuous and $g\in \Aff_+(T(B')).$  Since $\Aff_+(T(B'))$ is compact,
$g$ is also bounded. By identifying $B'\otimes {\cal K}$ with
$B\otimes {\cal K},$ we find a positive element $b_{00}=\diag(b_0,...,b_0)\in B\otimes {\cal K},$
where $b_0$ repeats $m$ times so that ${{d_s(b_{00})}}> g(s)$ on $T(B').$
Then $g$ is continuous on $T(B''),$ where $B'':=\overline{{{b_{00}}}(B\otimes {\cal K})b_{00}}.$
It follows \ref{Ttrzstable} that there is $e_B\in B''_+\subset B\otimes {\cal K}$ with $\|e_B\|=1$ such that
$d_s(e_B)=g|_{T(B'')}.$
 Since $B$ has strictly comparison, $B_0:=\overline{e_BBe_B}$ has continuous scale
(see
5.3 of \cite{eglnp1}).
Let
\beq
T_B &=& \{t\in {\tilde T}(B): d_t(e_B)=1\}.
\eneq
Then  $T(A_0)=T_B.$
It follows that $\Gamma$ induces the following isomorphism
\beq
(K_0(A_0), K_1(A_0), T(A_0)\cong (K_0(B_0), K_1(B_0), T(B_0)).
\eneq
It follows from \ref{Misothm} that there is an isomorphism  $\phi_0: A_0\to  B_0$
which induces $\Gamma$ on\\ $(K_0(A_0), K_1(A_0), T(A_0).$
By \cite{Br1}, $\phi_0$ gives an isomorphism from $A_0\otimes {\cal K}$ onto
$B_0\otimes {\cal K}.$
Let $a\in A_+$ with $\|a\|=1$ be a strictly positive element.
Then
\beq
\hat{a}(\tau)=\Sigma_A(\tau)\tforal \tau\in {\tilde T}(A).
\eneq
Let $b\in (B_0\otimes {\cal K})_+$ such that $\phi(a)=b.$
Then
\beq
d_t(b)=\lim_{n\to\infty} t\circ \phi(a^{1/n})\tforal   t\in {\tilde T}(B).
\eneq
Note $\Sigma_B(t)=d_t(b).$ Since $B$ is simple and has stable rank one, this implies  that
$B\cong \overline{b(B_0\otimes {\cal K})b}.$  The theorem follows.
\end{proof}

\begin{cor}\label{TTT2}
Let $A$ and $B$ be in ${\cal D}_0$ which are amenable and  satisfy the UCT.
Then $A\cong B$ if and only if
\beq
{\rm Ell}(A)\cong {\rm Ell}(B).
\eneq
\end{cor}

\begin{proof}
Since $A$ and $B$ are in ${\cal D}_0,$ by \ref{D0kerrho}, $K_0(A)={\rm ker}\rho_A$
and $K_0(B)={\rm ker}\rho_B.$ Therefore  Theorem \ref{T1main} applies.
\end{proof}

\begin{cor}\label{zoselfabsorbing}
Let $A$ be a  stably projectionless simple separable amenable \CA\, which satisfies the UCT and
$gTR(A)\le 1.$  Suppose that $K_0(A)={\rm ker}\rho_A.$
Then $A\otimes \zo\cong A.$

In particular, $\zo\otimes \zo\cong \zo.$
\end{cor}

\begin{proof}
Recall that  $K_0(\zo)=\Z={\rm ker}\rho_{\zo},$
$K_1(\zo)=\{0\}$ and $T(\zo)$ has exactly one point.
{{Let $A_0=\overline{eAe}$ for some $e\in A_+\setminus \{0\}$ such that $A_0\in {\cal D}.$
Since $K_0(A)={\rm ker}\rho_A,$ $A_0\in {\cal D}_0,$ by \ref{TD0=D}.
By {{12.5 of \cite{eglnp1} and 6.6 of \cite{eglnkk0}
(or by 18.5 and 18.6 of \cite{GLp1}),}}
$A_0\otimes \zo\in {\cal D}_0.$
Therefore $gTR(A\otimes \zo)\le 1.$
Moreover,
$K_0(A\otimes \zo)\cong K_0(A)={\rm ker}\rho_A,$
$K_1(A\otimes \zo)\cong K_1(A),$  ${\tilde T}(A\otimes \zo)={\tilde T}(A)$ and $\Sigma_A=\Sigma_{A\otimes\zo}.$
Thus \ref{T1main} applies.}}
\end{proof}

\section{A homotopy lemma}

The purpose of this section is to present  \ref{Chomotopy} which will be used in next section.
The following is known, a proof for the unital case can be found in 12.4 of  \cite{GLN}
\begin{lem}\label{FfullDelta}
Let $C$ be a  separable C*-algebra, and let $\Delta: C_+^{q, {\bf 1}}\setminus \{0\}\to (0,1)$ be an order preserving
map.  There exists  a map  $T: C_+\setminus \{0\}\to \R_+\setminus \{0\}\times \N$
satisfying the following:
For any finite subset ${\cal H} \subset C_+^{\bf 1}\setminus \{0\}$ and any $\sigma$-unital  C*-algebra $A$
 with the strict comparison of positive elements   which is quasi-compact, if $\phi: C \to A$ is a unital \morp\ satisfying
\hspace{-0.1in}\beq\label{Fullm-1}
\tau\circ \phi(h)\ge \Delta(\hat{h})\tforal h\in {\cal H} \tforal \tau\in\mathrm T(A),
\eneq
then $\phi$ is
$T$-${\cal H}$-full.
\end{lem}

 {{Recall the class of sub-homogeneous $C^*$-algebras $\overline{D_r}$ is defined in 4.8 of \cite{GLN}.
The following is a non-unital version
of 8.4 of \cite{GLN} (see 5.2.7 of \cite{Lncbms}).

\begin{thm}\label{UniqN1}
Let $A_0$ be a non-unital \CA\, such that $A:={\tilde{A_0}}\in  
{{\overline{D_r}}}$
 with finitely generated $K_i(A)$ ($i=0,1$).
Let ${\cal F}\subset A$ be a finite subset, let
$\ep>0$ be a positive number and let $\Delta: A_+^{q, {\bf 1}}\setminus \{0\}\to (0,1)$  be an order preserving map. There exists  a finite subset ${\cal H}_1\subset A_+^{\bf 1}\setminus \{0\},$
$\gamma_1>0,$ $\gamma_2>0,$ $\dt>0,$ a finite subset
${\cal G}\subset A$ and a finite subset ${\cal P}\subset \underline{K}(A),$ a finite subset ${\cal H}_2\subset A$, a finite subset ${\cal U}\subset J_c(K_1(A))$ {\rm (see \eqref{CUsplit}  in \ref{DJc} for the definition of $J_c$)}
for which $[{\cal U}]\subset {\cal P}$
satisfying the following:
For any unital ${\cal G}$-$\dt$-multiplicative \morp s $\phi, \psi: A_0\to C$
for some $C\in {\cal C}_0$ such that
\begin{equation}\label{Uni1-1}
[\phi^{\sim}]|_{\cal P}=[\psi^{\sim}]|_{\cal P},
\end{equation}
\begin{equation}\label{Uni1-2}
\tau(\phi^{\sim}(a))\ge \Delta(\hat{a}),\,\,\, \tau(\psi^{\sim}(a))\ge \Delta(\hat{a}),\quad \textrm{for all $\tau\in T(C) {{\tand}} a\in {\cal H}_1$},
\end{equation}
\begin{equation}\label{Uni1-3}
|\tau\circ \phi^{\sim}(a)-\tau\circ \psi^{\sim}(a)|<\gamma_1 \tforal a\in {\cal H}_2,\tand
\end{equation}
\begin{equation}\label{Uni1-3+1}
{\rm dist}((\phi^{\sim})^{\dag}(u), (\psi^{\sim})^{\dag}(u))<\gamma_2 \tforal u\in {\cal U},
\end{equation}
there exists a unitary $W\in {\tilde C}$ such that
\begin{equation}\label{Uni1-4}
\|W(\phi^{\sim}(f))W^*-(\psi^{\sim}(f))\|<\ep,\tforal f\in {\cal F},
\end{equation}
where $\phi^{\sim}, \psi^{\sim}$ are the unital extension of $\phi$ and $\psi$
from $A$ to ${\tilde C}.$
\end{thm}

\begin{proof}
\Wlog, we may assume that $A$ is  infinite dimensional.


Since $K_*(A)$ is finitely generated,  there is $n_0$ such that $\kappa\in\mathrm{Hom}_\Lambda(\underline{K}(A), \underline{K}(C))$ is determined by its restriction to 
 {{$K_*(A, \Z/n\Z)$,}} $n=0,..., n_0$.



Let ${\cal H}_1'\subset A_+\setminus \{0\}$ (in place of ${\cal H}_1$),
$\dt_1>0$ (in place of $\dt$), ${\cal G}_1\subset A$ (in place of ${\cal G}$) be a finite subset and let ${\cal P}_0\subset \underline{K}(A)$ (in place of ${\cal P}$) be a finite subset required by  4.4.5 of \cite{Lncbms}
(6.7 of \cite{GLN})
for $\ep/32$ (in place of $\ep$), ${\cal F}$ and   {{$\Delta.$}}
We may assume that $\dt_1<\ep/32$ and $(2\dt_1, {\cal G}_1)$
is a $KK$-pair (see the end of
2.12 of \cite{GLN}).

Moreover, we may assume that $\dt_1$ is sufficiently small  that
if $\|uv-vu\|<3\dt_1,$ then the Exel formula
$$
\tau({\rm bott}_1(u,v))={1\over{2\pi\sqrt{-1}}}(\tau(\log(u^*vuv^*))
$$
holds for any pair of unitaries $u$ and $v$ in any unital \CA\, $C$ with tracial rank zero and any $\tau\in T(C)$
(see Theorem 3.6 of \cite{Lnamj}). Moreover
if $\|v_1-v_2\|<3\dt_1,$ then
$$
{\rm bott}_1(u,v_1)={\rm bott}_1(u,v_2).
$$

Let  $g_1, g_2,...,g_{k(A)}\in U(M_{m(A)}(A))$ $(m(A)\ge 1$ is an integer) be a finite subset such that
$\{\bar{g_1}, \bar{g_2},..., \bar{g}_{k(A)}\}\subset J_c(K_1(A))$ and such that $\{[g_1], [g_2],...,[g_{k(A)}]\}$ forms  a set of generators for $K_1(A).$
Let ${\cal U}=\{\bar{g_1}, \bar{g_2},..., \bar{g}_{k(A)}\}\subset J_c(K_1(A))$ be a finite subset.

Let ${\cal U}_0\subset A$ be a finite subset
such that
$$
\{g_1, g_2,...,g_{k(A)}\}\subseteq\{(a_{i,j}): a_{i,j}\in {\cal U}_0\}.
$$

Let $\dt_u=\min\{1/256m(A)^2, \dt_1/16m(A)^2\},$
${\cal G}_u={\cal F}\cup{\cal G}_1\cup {\cal U}_0$ and
let ${\cal P}_u={\cal P}_0
{{\cup \{[g_1],[g_2],...,[g_{k(A)}]\}}}.$


Let $\dt_2>0$ (in place of $\dt$), ${\cal G}_2\subset A$ (in place of ${\cal G}$), ${\cal H}_2'\subset A_+\setminus \{0\}$ (in place of ${\cal H}$), $N_1\ge 1$ (in place of $N$) be the finite subsets and the constants as required by {{7.3}} 
of \cite{GLN}
for
$\dt_u$ (in place of $\ep$), ${\cal G}_u$ (in place of ${\cal F}$), ${\cal P}_u$ (in place of ${\cal P}$)  and
$\Delta$ and with $\bar{g}_j$ (in place of $g_j$), $j=1,2,...,k(A)$ (with $k(A)=r$).

Let $\dt_3>0$  and let ${\cal G}_3\subset A\otimes C(\T)$
be a finite subset satisfying the following:
For any ${\cal G}_3$-$\dt_3$-multiplicative \morp\, $L': A\otimes C(\T)\to C'$ (for any unital \CA\, $C'$ with $T(C')\not=\emptyset$),
 \beq\label{Uni-10}
 |\tau([L']({\boldsymbol{\bt}}(\bar{g}_j))|<1/8N_1,\,\,\,j=1,2,...,k(A).
\eneq
Without loss of generality, we may assume
that
$$
{\cal G}_3=\{g\otimes f: g\in {\cal G}_3'\andeqn f\in \{1,z, z^*\}\},
$$
where
${\cal G}_3'\subset A$ is a finite subset  {{containing $1_A$}} (by  choosing a smaller $\dt_3$ and large ${\cal G}_3'$).

  Let $\ep_1'=\min\{d/27N_1m(A)^2, \dt_u/2,  \dt_2/2m(A)^2, \dt_3/2m(A)^2\}$ and let ${\bar \ep}_1>0$ (in place of $\dt$) and ${\cal G}_4\subset A$ (in place of ${\cal G}$) be a finite subset as
required by  6.4 of \cite{GLN}
 for $\ep_1'$ (in place of $\ep$) and ${\cal G}_u\cup {\cal G}_3'.$
Put
$
\ep_1=\min\{\ep_1', \ep_1'',{\bar \ep_1}\}.
$
Let ${\cal G}_5={\cal G}_u\cup {\cal G}_3'\cup {\cal G}_4.$




Let $\mathcal H_3'\subseteq A_+^{\bf 1}\setminus \{0\}$ (in place of $\mathcal H_1$), $\dt_4>0$ (in place of $\dt$),  ${\cal G}_6\subset A$  (in place of ${\cal G}$), ${\cal H}_4'\subset A_{s.a.}$ (in place of $\mathcal H_2$), ${\cal P}_1\subset \underline{K}(A)$ (in place of ${\cal P}$) and $\sigma>0$
be the finite subsets and constants  as required by  {{Theorem 5.8}} 
of \cite{GLN}
with respect to $\ep_1/{{16}}$ (in place $\ep$) and ${\cal G}_5$ (in place of ${\cal F}$) and $\Delta$.

Choose $N_2\ge N_1$ such that $(k(A)+1)/N_2<1/8N_1.$
 Choose ${\cal H}_5'\subset A_+^{\bf 1} \setminus \{0\}$ and
 $\dt_5>0$ and a finite subset ${\cal G}_7\subset A$ such that, for any $M_m$ and unital ${\cal G}_7$-$\dt_5$-multiplicative \morp\, $L': A\to M_m,$
 if
 ${\rm tr}\circ L'(h)>0\tforal h\in {\cal H}_5',$
 then
 $m\ge 16N_2.$

Put $\dt=\min\{\ep_1/16,  \dt_4/4m(A)^2, \dt_5/4m(A)^2\},$
${\cal G}={\cal G}_5\cup {\cal G}_6\cup {\cal G}_7$, and  ${\cal P}={\cal P}_u\cup {\cal P}_1.$
%
Put
$$
{\cal H}_1={\cal H}_1'\cup {\cal H}_2'\cup {\cal H}_3'\cup {\cal H}_4'\cup {\cal H}_{{5}}'
$$
and let ${\cal H}_2={\cal H}_4'.$
Let $\gamma_1=\sigma$ and let
$0<\gamma_2<\min\{d/16N_2m(A)^2, \dt_u/9m(A)^2, 1/256m(A)^2\}.$

Now suppose that $C\in {\cal C}_0$ and $\phi, \psi: A\to C$ are two unital
${\cal G}$-$\dt$-multiplicative \morp s satisfying the condition of the theorem for the given $\Delta,$ ${\cal H}_1,$ $\dt,$ ${\cal G},$ ${\cal P},$ ${\cal H}_2,$ $\gamma_1,$ $\gamma_2$ and ${\cal U}.$

We write $C=A(F_1, F_2, h_0, h_1),$
$F_1=M_{m_1}\oplus M_{m_2}\oplus \cdots \oplus M_{m_{F(1)}}$ and
$F_2=M_{n_1}\oplus M_{n_2}\oplus \cdots \oplus M_{n_{F(2)}}.$ By the choice of ${\cal H}_5',$ one has that
\beq\label{Uni11+1}
n_j\ge 16 N_2\andeqn m_s\ge 16 N_2,\,\,\, 1\le j \le  F(2), \,\,\, 1\le s\le F(1).
\eneq
Let $q_{F_1,0}=h_0(1_{F_1})$ and $q_{F_1, 1}=h_1(1_{F_1}).$
Define $h_0^{\sim}: {{F_1^\sim:=}}F_1\oplus \C\to F_2$ by
$h_0^{\sim}((a,\lambda))=h_0(a)\oplus \lambda(1-q_{F_1,0})$
and $h_1^{\sim}((a, \lambda))=h_1(a)\oplus \lambda(1-q_{F_1, 1}).$
Then ${\tilde C}=A(F_1\oplus \C, F_2, h_0^{\sim}, h_1^{\sim}).$
Put $\pi^{C\sim}: {\tilde C}\to \C.$
{{Note that $\pi^{C\sim}\circ \phi(a)=0=\pi^{C\sim}\circ \psi(a)\rforal a\in A_0\subset A$, and that $\pi^{C\sim}\circ \phi(1_{A^\sim})=1_\C=\pi^{C\sim}\circ \psi(1_{A^\sim})$. Hence}}
\vspace{-0.1in}\beq\label{December-19-2019}
\pi^{C\sim}\circ \phi(a)=\pi^{C\sim}\circ \psi(a)\rforal a\in A.
\eneq
Let
$0=t_0<t_1<\cdots <t_n=1$
be a partition of $[0,1]$ so that
\beq\label{Uni-11}
\|\pi_{t}\circ \phi^{\sim}(g)-\pi_{t'}\circ \phi^{\sim}(g)\|<\ep_1/16\andeqn
\|\pi_{t}\circ \psi^{\sim}(g)-\pi_{t'}\circ \psi^{\sim}(g)\|<\ep_1/16
\eneq
for all $g\in {\cal G},$ provided $t, t'\in [t_{i-1}, t_i],$ $i=1,2,...,n.$

Applying Theorem
{{5.8}} of \cite{GLN}, one obtains a unitary
$w_i\in F_2$
if $0<i<n,$ $w_0\in h_0(F_1),$
such that
\beq\label{Uni-12}
\|w_i\pi_{t_i}\circ {{\phi^\sim}}(g)w_i^*-\pi_{t_i}\circ {{\psi^\sim}}(g)\|<\ep_1/16\tforal g\in {\cal G}_5,
\eneq
Also there is $w_e'\in F_1$
such that
\beq\label{Uni-2n-12}
\|(w_e')^*\pi_e\circ \phi(g)w_e'-\pi_e\circ \psi(g)\|<\ep_1/16 \rforal g\in {\cal G}_5.
\eneq
Let $\pi^{F_1^{\sim'}}: h_0^{\sim}(F_1^{\sim})\to \C$
and let $\pi': h_0(F_1^{\sim})\to h_0(F_1)$ be the quotient maps.
Put $w_0=h_0(w_e')\oplus (1_{F_2}-q_{F_1,0}),$  $w_n=h_1(w_e')\oplus (1_{F_2}-q_{F_1, 1}),$
$w_0'=h_0(w_e')$ and $w_n'=h_1(w_e').$
Then
\beq
\|w_i^*
{{\pi_{t_i}}}\circ \phi^{\sim}(g)w_i-
{{\pi_{t_i}}}\circ \psi^{\sim}(g)\|<\ep_1/16\rforal g\in {\cal G}_5{{,}}
\eneq
$i=0$ and $i=n.$ {{Denote  $w_e=w'_e\oplus 1_{\C}\in F_1\oplus \C.$ Then $w_0={h}_0^{\sim}(w_e),$
$w_n={h}_1^{\sim}(w_e)$.}}

By \eqref{Uni1-3+1}, there is a unitary  ${{\omega'_j}}\in M_{m(A)}({\tilde C})$
such that $\omega_j\in CU(M_{m(A)}({\tilde C}))$ and
\vspace{-0.1in}\beq\label{Uni-13}
\|\lceil (\phi^{\sim}\otimes {\rm id}_{M_{m(A)}}(g_j^*)\rceil \lceil (\psi^{\sim}\otimes {\rm id}_{M_{m(A)}})(g_j)\rceil  -{{\omega'_j}}\|<\gamma_2,\,\,\,j=1,2,...,k(A).
\eneq
{{By \eqref{December-19-2019} and \eqref{Uni-13}, we have $\|(\pi^{C\sim}\otimes \id_{m(A)})(\omega'_j)-1\|<\gm_2.$ Set  $\omega_j=\omega'_j\cdot (\pi^{C\sim}\otimes \id_{m(A)})(\omega'_j)^*$ (viewing $(\pi^{C\sim}\otimes \id_{m(A)})(\omega'_j)\in M_{m(A)}(\C)\subset M_{m(A)}(\tilde{C})$). Consequently, we have
\beq\label{Uni-13+Dec-19}
\|\lceil \phi^{\sim}\otimes {\rm id}_{M_{m(A)}}(g_j^*)\rceil \lceil (\psi^{\sim}\otimes {\rm id}_{M_{m(A)}})(g_j)\rceil  -{{\omega_j}}\|<2\gamma_2,\,\,\,j=1,2,...,k(A),
\eneq
with an extra condition $\pi^{C\sim} \id_{m(A)} (\omega_j)=1_{m(A)}(\C)$.}}
{{As mentioned in \ref{DtensorMr}, we will use $\pi^{C\sim}$ for  $\pi^{C\sim}\otimes \id_{m(A)}$.}}
({\it {{Note}} that we now have $w_i$ as well as $\omega_i$ in the proof}.)  Write
$$
\omega_j=\prod_{l=1}^{e(j)}\exp(\sqrt{-1}a_j^{(l)})
$$
for some {{self adjoint}} element $a_j^{(l)}\in M_{m(A)}({\tilde C}),$
$l=1,2,...,e(j),$ $j=1,2,...,k(A).$
{{In particular, one can choose $a_j^{(l)}$ such that $\pi^{C\sim} (a_j^{(l)})=0\in M_{m(A)}(\C)$ (see \ref{Thelemma}).}}
Write
\beq\nonumber
a_j^{(l)}=(a_j^{(l,1)}, a_j^{(l,2)},...,a_j^{(l,n_{F(2)})})\andeqn \omega_j=(\omega_{j,1},\omega_{j,2},...,\omega_{j,F(2)})
\eneq
in $C([0,1], F_2)=C([0,1],M_{n_1})\oplus \cdots \oplus C([0,1],M_{n_{{F(2)}}}),$
where $\omega_{j,s}=\prod_{l=1}^{e(j)}\exp(\sqrt{-1}a_j^{(l,s)}),$ $s=1,2,...,F(2).$

Then
$$
\sum_{l=1}^{e(j)}{n_s(t_s\otimes {\rm Tr}_{m(A)})(a_j^{(l,s)}(t))\over{2\pi}}\in \Z,\quad t\in (0,1),
$$
where $t_s$ is the normalized trace on $M_{n_s},$
$s=1,2,...,F(2).$
In particular,
\beq\label{Uni-15}
\sum_{l=1}^{e(j)}n_s(t_s\otimes {\rm Tr}_{m(A)})(a_j^{(l,s)}(t))=\sum_{l=1}^{e(j)}n_s(t_s\otimes {\rm Tr}_{m(A)})(a_j^{(l,s)}(t'))
\rforal t, t'\in (0,1).
\eneq
We also have
\vspace{-0.1in}\beq
(1/2\pi)\sum_{l=1}^{e(j)} m_s(t_{es}\otimes {\rm Tr}_{m(A)})(\pi_e
{{(a_j^{(l)})}}\in \Z,
\eneq
where $t_{es}$ is the tracial state on $M_{m_s}.$ {{Note, for $s=F(1)+1$, one has $\pi_{e,s}(a_j^{(l)})=\pi^{C\sim}(a_j^{(l)})=0$.}}

Let $W_i=w_i\otimes {\rm id}_{M_{m(A)}},$ $i=0,1,....,n$
and $W_e=w_e\otimes {\rm id}_{M_{m(A)}(F_1)}.$
Then {{it follows from (\ref{Uni-12}) and (\ref{Uni-13+Dec-19}) that}}
\begin{eqnarray}\label{Ui-16}
\hspace{-0.6in}&&\|\pi_{{t_i}}(\lceil \phi^{\sim}\otimes {\rm id}_{M_{m(A)}})(g_j^*)\rceil) W_i(\pi_{{t_i}}(\lceil \phi^{\sim}\otimes {\rm id}_{M_{m(A)}})(g_j)\rceil)W_i^*-\omega_j(t_i)\|\\
&&\hspace{0.2in}
<3m(A)^2\ep_1+2\gamma_2<1/32.
\end{eqnarray}
We also have (with $\phi_e=\pi_e\circ \phi^{\sim}$)
\beq\label{Uni-17}
\hspace{-0.4in}\|\lceil (\phi_e\otimes {\rm id}_{M_{m(A)}})(g_j^*)\rceil W_e(\lceil \phi_e\otimes {\rm id}_{M_{m(A)}})(g_j)\rceil)W_e^*-\pi_e(\omega_j)\|<3m(A)^2\ep_1+2\gamma_2<1/32.
\eneq

It follows from \eqref{Ui-16} that there exists selfadjoint elements $b_{i,j}\in M_{m(A)}(F_2)$ such that
\beq\label{Uni-18}
\hspace{-0.2in}\exp(\sqrt{-1}b_{i,j})=\omega_j(t_i)^*(\pi_i(\lceil \phi^{\sim}\otimes {\rm id}_{M_{m(A)}})(g_j^*)\rceil)W_i(\pi_i(\lceil \phi^{\sim}\otimes {\rm id}_{M_{m(A)}})(g_j)\rceil) W_i^*,
\eneq
and $b_{e,j}\in M_{m(A)}({{F_1\oplus \C}})$ such that
\beq\label{Uni-19}
\hspace{-0.2in}\exp(\sqrt{-1}b_{e,j})=\pi_e(\omega_j)^*(\pi_e(\lceil \phi^{\sim}\otimes {\rm id}_{M_{m(A)}})(g_j^*)\rceil)W_e(\pi_e(\lceil \phi^{\sim}\otimes {\rm id}_{M_{m(A)}})(g_j)\rceil) W_e^*,
\eneq
and
\vspace{-0.12in}\beq\label{Uni-20}
\|b_{i,j}\|<{{2\arcsin (3m(A)^2\ep_1/2+\gamma_2)}},\,\,\,j=1,2,...,k(A),\,i=0,1,...,n,e.
\eneq
Write
\beq\nonumber
b_{i,j}=(b_{i,j}^{(1)},b_{i,j}^{(2)},...,b_{i,j}^{F(2)})\in {{M_{m(A)}(F_2)}}\andeqn
b_{e,j}=(b_{e,j}^{(1)}, b_{e,j}^{(2)},...,b_{e,j}^{(F(1))}, {{b_{e,j}^{F(1)+1}}})\in {{M_{m(A)}(F_1\oplus \C)}}.
\eneq
{{From $\pi^{C\sim}(\omega_j)=1$ and definition of $W_e$ and $w_e$, we know that $b_{e,j}^{F(1)+1}=0$.}} We have that
\beq\label{Uni-20+}
{{h_0^{\sim}}}(b_{e,j})=b_{0,j}\andeqn {{h_1^{\sim}}}(b_{e,j})=b_{n,j}.
\eneq
Note that
\beq\label{Uni-21}
(\pi_{{t_i}}(\lceil \phi^{\sim}\otimes {\rm id}_{M_{m(A)}}(g_j^*)\rceil))W_i(\pi_{{t_i}}(\lceil \phi^{\sim}\otimes {\rm id}_{M_{m(A)}})(g_j)\rceil) W_i^*
=\pi_{{t_i}}(\omega_j)\exp(\sqrt{-1}b_{i,j}),
\eneq
$j=1,2,...,k(A)$ and $i=0,1,...,n,e.$
Then,
\beq\label{Uni-22}
{n_s\over{2\pi}}(t_s\otimes {\rm Tr}_{M_{m(A)}})(b_{i,j}^{(s)})\in \Z,
\eneq
where $t_s$ is the normalized trace on $M_{n_s},$
$s=1,2,...,F(2),$ $j=1,2,...,k(A),$ and $i=0,1,...,n.$
We also have
\beq\label{Uni-23}
{m_s\over{2\pi}}(t_s\otimes {\rm Tr}_{M_{m(A)}})(b_{e,j}^{(s)})\in \Z,
\eneq
where $t_s$ is the normalized trace on $M_{m_s},$ $s=1,2,...,F(1),$
$j=1,2,...,k(A).$
Put
$$
\lambda_{i,j}^{(s)}={n_s\over{2\pi}}(t_s\otimes {\rm Tr}_{M_{m(A)}})(b_{i,j}^{(s)})\in \Z,
$$
where $t_s$ is the normalized trace on $M_{n_s},$
$s=1,2,...,n,$ $j=1,2,...,k(A)$ and  $i=0,1,2,...,n.$

Put
$$
\lambda_{e,j}^{(s)}={m_s\over{2\pi}}(t_s\otimes {\rm Tr}_{M_{m(A)}})(b_{e,j}^{(s)})\in \Z,
$$
where $t_s$ is the normalized trace on $M_{m_s},$
$s=1,2,...,F(1)$ and $j=1,2,...,k(A).$
Denote
\beq\nonumber
\lambda_{i,j}=(\lambda_{i,j}^{(1)}, \lambda_{i,j}^{(2)},...,\lambda_{i,j}^{(F(2)})\in \Z^{F(2)},\andeqn
\lambda_{e,j}=(\lambda_{e,j}^{(1)},\lambda_{e,j}^{(2)},...,{{\lambda_{e,j}^{F(1)},0}})\in \Z^{{F(1)+1}}.
\eneq
 We have, by (\ref{Uni-20}), for  $j=1,2,...,k(A)$ and  $i=0,1,2,...,n,$
\beq\label{Uni-24}
|{\lambda_{i,j}^{(s)}\over{n_s}}|<1/4N_1, \,\,\,s=1,2,...,F(2),\,\,\,
|{\lambda_{e,j}^{(s)}\over{m_s}}|<1/4N_1,\,\,\,s=1,2,...,F(1),
\eneq
Define $\af_i^{(0,1)}: K_1(A)\to \Z^{F(2)}$ by mapping $[g_j]$ to $\lambda_{i,j},$ $j=1,2,...,k(A)$, $i=0,1,2,...,n,$ and
define
$\af_e^{(0,1)}: K_1(A)\to \Z^{F(1)}\oplus \Z$ by
mapping $[g_j]$ to $(\lambda_{e,j},0),$ $j=1,2,...,k(A).$
We write $K_0(A\otimes C(\T))=K_0(A)\oplus {\boldsymbol{\bt}}(K_1(A)))$
(see
2.10 of \cite{Lnmemhp}
for the definition
of ${\boldsymbol{\bt}}$).
Define $\af_i: K_*(A\otimes C(\T))\to K_*(F_2)$ as follows:
On $K_0(A\otimes C(\T)),$ define
\beq\label{Uni-25}
\af_i|_{K_0(A)}=[\pi_i\circ \phi]|_{K_0(A)},\,\,\,
\af_i|_{{\boldsymbol{\bt}}(K_1(A))}=\af_i\circ {\boldsymbol{\bt}}|_{K_1(A)}=\af_i^{(0,1)}
\eneq
and on $K_1(A\otimes C(\T)),$ define
$\af_i|_{K_1(A\otimes C(\T))}=0,\,\,\,i=0,1,2,...,n.$

Also define  $\af_e\in {\rm Hom}(K_*(A\otimes C(\T)), K_*(F_1\otimes \C)),$ by
\beq\label{Uni-26}
\af_e|_{K_0(A)}=[\pi_e\circ \phi^{\sim}]|_{K_0(A)},\,\,\,
\af_e|_{{\boldsymbol{\bt}}(K_1(A))}=\af_e\circ {\boldsymbol{\bt}}|_{K_1(A)}=\af_e^{(0,1)}
\eneq
on $K_0(A\otimes C(\T))$ and  $(\af_e)|_{K_1(A\otimes C(\T))}=0.$
Note that
\beq\label{Uni-26+0}
(h_0^{\sim})_{*}\circ \af_e=\af_0\andeqn (h_1^{\sim})_{*}\circ\af_e=\af_n.
\eneq
Since $A\otimes C(\mathbb{T})$ satisfies the UCT,  the map $\alpha_e$ can be lifted to an element of $KK(A\otimes C(\mathbb T), F_1\oplus \C)$ which is still denoted by $\alpha_e$. Then define
\beq\label{Uni-26+}
\af_0=\af_e\times [h_0^{\sim}] \andeqn \af_n=\af_e\times [h_1^{\sim}]
\eneq
in $KK(A\otimes C(\mathbb T), F_2)$.
For $i=1, ..., n-1$, also pick a lifting of $\alpha_i$ in $KK(A\otimes C(\mathbb T), F_2)$, and still denote it by $\alpha_i$.
We estimate that
\beq\label{Uni-26+n}
\|(w_{i}^*w_{i+1})\pi_{t_i}\circ \phi^{\sim}(g)-\pi_{t_i}\circ \phi^{\sim}(g)(w_{i}^*w_{i+1})\|<\ep_1/4\tforal g\in {\cal G}_5,
\eneq
$i=0,1,...,n-1.$
Let $\Lambda_{i,i+1}: C(\T)\otimes A\to F_2$ be a unital \morp\, given
by the pair $w_{i}^*w_{i+1}$ and $\pi_{t_i}\circ \phi$  (by
6.4 of \cite{GLN},
see 2.8 of \cite{Lnmemhp}).
Denote $V_{i,j}=\lceil \pi_{t_i}\circ \phi^{\sim}\otimes {\rm id}_{M_{m(A)}}(g_j) \rceil,$ $j=1,2,...,k(A)$ and $i=0,1,2,...,n-1.$

Write
$$
V_{i,j}=(V_{i,j, 1}, V_{i,j,2},..., V_{i,j,F(2)})\in M_{m(A)}(F_2),\,\,\, j=1,2,...,k(A),\,\,\, i=0,1,2,...,n.
$$
Similarly, write
\vspace{-0.12in}\beq
W_i=(W_{i,1},W_{i,2},...,W_{i,F(2)})\in M_{m(A)}(F_2),\,\,\, i=0,1,2,...,n.
\eneq
We have
\vspace{-0.12in}\beq\label{Uni-27}
\|W_{i}V_{i,j}^*W_{i}^* V_{i,j}V_{i,j}^*W_{i+1}V_{i,j}W_{i+1}^*-1\|<1/16\\
\|W_{i}V_{i,j}^*W_{i}^*V_{i,j}V_{i+1,j}^*
W_{i+1} V_{i+1,j}W_{i+1}^*-1\|<1/16
\eneq
and there is a continuous path $Z(t)$
of unitaries  such that $Z(0)=V_{i,j}$ and $Z(1)=V_{i+1,j}.$ Since
$$
\|V_{i,j}-V_{i+1,j}\|<\dt_1/12,\,\,\,j=1,2,...,k(A),
$$
we may assume that $\|Z(t)-Z(1)\|<\dt_1/6$ for all $t\in [0,1].$
We also write
$$
Z(t)=(Z_1(t), Z_2(t), ...,Z_{F(2)}(t))\in F_2\andeqn t\in [0,1].
$$
We obtain a continuous path
$
W_{i}V_{i,j}^*W_{i}^*V_{i,j}Z(t)^*W_{i+1} Z(t)W_{i+1}^*
$
which is in $CU(M_{nm(A)})$
for all $t\in  [0,1]$ and
$$
\|W_{i}V_{i,j}^*W_{i}^*V_{i,j}Z(t)^*W_{i+1} Z(t)W_{i+1}^*-1\|<1/8\tforal t\in [0,1].
$$
It follows that
$$
(1/2\pi\sqrt{-1})(t_s\otimes {\rm Tr}_{M_{m(A)}})[\log(W_{i,s}V_{i,j,s}^*W_{i,s}^*V_{i,j,s}Z_s(t)^*W_{i+1,s}Z_s(t)W_{i+1,s}^*)]
$$
is a constant integer,  where $t_s$ is the normalized trace on $M_{n_s}.$
In particular,
\beq\label{Uni-28}
&&\hspace{-0.4in}(1/2\pi\sqrt{-1})(t_s\otimes {\rm Tr}_{M_{m(A)}})(\log(W_{i,s}V_{i,j,s}^*W_{i,s}^*W_{i+1,s} V_{i,j,s}W_{i+1,s}^*))\\\label{Uni-28+}
&&\hspace{-0.2in}=(1/2\pi\sqrt{-1})(t_s\otimes {\rm Tr}_{M_{m(A)}})(\log(W_{i,s}V_{i,j,s}^*W_{i,s}^*V_{i, j}V_{i+1,j,s}^*W_{i +1,s}V_{i,j,s}W_{i+1,s}^*)).
\eneq
One also has
\vspace{-0.12in}\begin{eqnarray}\label{Uni-29}
&&\hspace{-0.7in}W_{i}V_{i,j}^*W_{i}^*V_{i, j}V_{i+1,j}^*W_{i+1} V_{i+1,j}W_{i+1}^*
=(\omega_j(t_{i})\exp(\sqrt{-1}b_{i,j}))^*\omega_j(t_{i+1})
\exp(\sqrt{-1}b_{i+1,j})\\\label{Uni-30}
&&=\exp(-\sqrt{-1}b_{i,j})\omega_j(t_{i})^*\omega_j(t_{i+1})
\exp(\sqrt{-1}b_{i+1,j}).
\end{eqnarray}
Note that, by (\ref{Uni-13}) and (\ref{Uni-11}), for $t\in [t_i, t_{i+1}],$
\begin{equation}
\|\omega_j(t_{i})^*\omega_j(t)-1\|<2(m(A)^2)\ep_1/16+2\gamma_2<1/32,
\end{equation}
$j=1,2,...,k(A),$ $i=0,1,..., n-1.$
By Lemma 3.5 of \cite{Lin-AU11},
\begin{equation}\label{LeasyApp}
(t_s\otimes {\rm Tr}_{m(A)})(\log(\omega_{j,s}(t_{i})^*\omega_{j,s}(t_{i+1})))=0.
\end{equation}
It follows that (by the Exel formula (see \cite{HL}), using (\ref{Uni-28+}), (\ref{Uni-30}) and (\ref{LeasyApp}))
\beq\label{Uni-31}
&&\hspace{-0.6in}(t\otimes {\rm Tr}_{m(A)})({\rm bott}_1(V_{i,j}, W_{i}^*W_{i+1}))\\
 \hspace{-0.2in}&=&
({1\over{2\pi \sqrt{-1}}})(t\otimes {\rm Tr}_{m(A)})(\log(V_{i,j}^*W_{i}^*W_{i+1}V_{i,j}W_{i+1}^*W_{i}))\\
 \hspace{-0.2in}&=&({1\over{2\pi \sqrt{-1}}})(t\otimes {\rm Tr}_{m(A)})(\log(W_{i}V_{i,j}^*W_{i}^*W_{i+1,s}V_{i,j}W_{i+1}^*))\\
&=&({1\over{2\pi \sqrt{-1}}})(t\otimes {\rm Tr}_{m(A)})(\log(W_{i}V_{i,j}^*W_{i}^*V_{i,j}V_{i+1,j}^*
W_{i+1}V_{i+1,j}W_{i+1}^*))\\
&=& ({1\over{2\pi \sqrt{-1}}})(t\otimes {\rm Tr}_{m(A)})(\log(\exp(-\sqrt{-1}b_{i,j})\omega_j(t_{i})^*
\omega_j(t_{i+1})\exp(\sqrt{-1}b_{i+1,j}))\\
&=& ({1\over{2\pi \sqrt{-1}}})[(t\otimes {\rm Tr}_{{m(A)}})(-\sqrt{-1}b_{i,j})+(t\otimes {\rm Tr}_{{m(A)}})(\log(\omega_j(t_{i})^*\omega_j(t_{i+1}))\\
&&\hspace{1.4in}+(t\otimes {\rm Tr}_{{m(A)}})(\sqrt{-1}b_{i,j})]\\
&=&{1\over{2\pi}}(t\otimes {\rm Tr}_{{m(A)}})(-b_{i,j}+b_{i+1,j})
\eneq
for all $t\in T(F_2).$
In other words,
\beq\label{Uni-32}
{\rm bott}_1(V_{i,j}, W_{i}^*W_{i+1}))=-\lambda_{i,j}+\lambda_{i+1,j}
\eneq
$j=1,2,...,m(A),$  $i=0,1,...,  n-1.$


Define $\beta_0=0,$ $\bt_1=[\Lambda_{0,1}]-\af_1+\af_0+\bt_0,$
\beq\label{betai}
\bt_i=[\Lambda_{i-1, i}]-\af_i+\af_{i-1}+\bt_{i-1},\,\,\,i=2,3,...,n.
\eneq
Then
\vspace{-0.12in}\beq\nonumber
&&\bt_1([g_j])=\Lambda_{0,1}([g_j])-\lambda_{1,j}+\lambda_{0,j}=0,\\\nonumber
&&\bt_2([g_j])=\Lambda_{1,2}([g_j])-\lambda_{2,j}-\lambda_{1,j}+\bt_1([g_j])=0\andeqn\\\nonumber
&&\bt_i([g_j])=\lambda_{i-1, i}([g_j])-\lambda_{i,j}-\lambda_{i-1, j}-\bt_{i-1}([g_j])=0,\,\,\, i=3,...,n.
\eneq
It follows 5.2.5 of \cite{Lncbms}
that there is $\varrho\in {\rm Hom}_{\Lambda}(\underline{K}(A), \underline{K}(F_1\otimes \C))$
such that
\beq\nonumber
&&\varrho({\boldsymbol{\bt}}(K_1(A)))=0\andeqn\\\nonumber
&&\varrho\times ([h_1^{\sim}]-[h_0^{\sim}])|_{{\boldsymbol{\bt}}(\underline{K}(A))}=\bt_n|_{{\boldsymbol{\bt}}(\underline{K}(A))}.
\eneq
Define $\kappa_0=\af_0+\bt_0+\varrho\times [h_0^{\sim}],$ $\kappa_i=\af_i+\bt_i+\varrho\times [h_0^{\sim}],$ $i=1,2,...,n.$
Note
that, on ${\boldsymbol{\bt}}(\underline{K}(A){)},$
\beq
\kappa_n&=&\af_n+\bt_n+\varrho\times [h_0^{\sim}]
					=\af_n+\varrho\times([h_1^{\sim}]-[h_0^{\sim}])+\varrho\times [h_0^{\sim}]\\
					&=&\af_n+\varrho\times [h_1^{\sim}]
					=(\af_e+\varrho)\times [h_1^{\sim}],
					\eneq
and, by \eqref{Uni-26+},
$\kappa_0=\af_0+\varrho\times [h_0^{\sim}]=\af_e\times [h_0^{\sim}]+\varrho\times [h_0^{\sim}].$
We also have, for each $j=1,2,..., k(A),$
\beq\nonumber
&&\kappa_i([g_j])=\lambda_{i,j} +(h_0^{\sim})_{*0}\circ \varrho([g_j])=\lambda_{i,j},\,\,\,i=0,1,...,n\andeqn\\\nonumber
&&(\varrho +\af_e)([g_j])=\lambda_{e,j}.
\eneq
Applying  7.4 of \cite{GLN}
(using \eqref{Uni-24}, \eqref{Uni1-2}), there are unitaries
$z_i\in F_2$, $i=1,2,...,n-1$,  and  $z_e\in F_1\otimes \C$  with $z_e=z_e'\oplus 1$  such that, for $i=1,2,...,n-1,$
\beq\label{Uni-33}
\hspace{-0.4in}\|[z_i,\, \pi_{t_i}\circ \phi^{\sim}(g)]\|<\dt_u\rforal g\in {\cal G}_u,\,
{\rm Bott}(z_i,\, \pi_{t_i}\circ \phi^{\sim})=(\kappa_i)|_{\boldsymbol{\bt}(\underline{K}(A))},\andeqn
\eneq
\vspace{-0.2in}\beq\label{Uni-47}
\hspace{-0.4in}\|[z_e,\, \pi_e\circ \phi^{\sim}(g)]\|<\dt_u\rforal g\in {\cal G}_u\andeqn
{\rm Bott}(z_e,\, \pi_e\circ \phi^{\sim})&=&(\varrho+\af_e)|_{\boldsymbol{\bt}(\underline{K}(A))}.
\eneq
Put $$z_0=h_0(z_e)\otimes (1_{F_2}-h_0(1_{F_1}))\quad \mathrm{and} \quad z_n=h_1(z_e)\oplus (1_{F_2}-h_1(1_{F_1})).$$
Note that, as above,
\beq\nonumber
{\rm Bott}(z_0,\, \pi_{0}\circ \phi^{\sim})=\kappa_0|_{\boldsymbol{\bt}(\underline{K}(A))}\andeqn
{\rm Bott}(z_n,\, \pi_{0}\circ \phi^{\sim})=\kappa_n|_{\boldsymbol{\bt}(\underline{K}(A))}.
\eneq
Let
\vspace{-0.14in}\beq
U_i=z_iw_iw_{i+1}^*z_{i+1},\,\,\,i=0,1,...,n-1.
\eneq
Then, by \eqref{Uni-33}, \eqref{Uni-47} and \eqref{Uni-26+n},
\beq\label{Uni-50}
\|[U_i,\, \pi_{t_i}\circ \phi^{\sim}(g)]\|<2\dt_u+2\ep_1/4<\dt_1/2\rforal g\in {\cal G}_u.
\eneq
We also compute that (using the choice of $\dt_1$ and \eqref{betai})
\beq\nonumber
{\rm Bott}(U_i,\, \pi_{t_i}\circ \phi^{\sim}) &=&  {\rm Bott}(z_i,\, \pi_{t_i}\circ \phi^{\sim})+{\rm Bott}(w_i^*w_{i+1},\,
\pi_{t_i}\circ \phi^{\sim})\\\nonumber
 &=&{{\rm Bott}(z_{i+1},\, \pi_{t_i}\circ \phi^{\sim})}
= \kappa_i+[\Lambda_{i,i+1}]-\kappa_{i+1}\\\nonumber
&=&\af_i+\bt_i+\varrho\times [h_0]+[\Lambda_{i,i+1}]-(\af_{i+1}+\bt_{i+1}+\varrho\times [h_0])\\\nonumber
\hspace{-0.2in}&=&\af_i+\bt_i+[\Lambda_{i,i+1}]-\af_{i+1}-([\Lambda_{i,i+1}]-\af_{i+1} +\af_i +\bt_i)=0,
\eneq
$i=0,1,...,n-1.$
Note that, by the assumption \eqref{Uni1-2},
\beq\label{Uni-38}
t_s\circ \pi_t\circ \phi(h)\ge \Delta(\hat{h})\tforal h\in {\cal H}_1',
\eneq
where $t_s$ is the normalized trace on $M_{n_s},$ $1\le s\le F(2).$
Then, by this, \eqref{Uni-50}, \eqref{Uni-38} and by applying  6.7 of \cite{GLN}
we obtain a continuous path of unitaries $\{U_i(t): t\in [t_i, t_{i+1}]\}\subset F_2$
such that $U_i(t_i)=1_{F_2}$ and $U(t_{i+1})=z_i(w_i)^*w_{i+1}z_{i+1}^*$ and
\beq\label{Uni-39+}
\|[U_i(t),\, \pi_t\circ \phi^{\sim}(f)]\|<\ep/32\rforal f\in {\cal F},
\eneq
$i=0,1,...,n-1.$ Now define $W(t)=w_iz_i^*U_i(t)$ for $t\in [t_i, t_{i+1}],$ $i=0,1,...,n-1.$
Then $W(t)\in C([0,1], F_2)$ but also
$$
W(0)=w_0z_0^*=h_0^{\sim}(w_ez_e^*)\andeqn W(1)=w_nz_n^*=h_1^{\sim}(w_ez_e^*).
$$
Therefore $W\in {\tilde C}.$
One then checks that, by  (\ref{Uni-11}), (\ref{Uni-39+}) , (\ref{Uni-33}) and (\ref{Uni-12}),
\begin{eqnarray}
&&\|W(t)(\pi_t\circ \phi^{\sim})(f)W(t)^*-(\pi_t\circ \psi^{\sim})(f)\otimes 1_{M_N}\|\\
&<&\|W(t)(\pi_t\circ \phi^{\sim})(f)W(t)^*-W(t)(\pi_{t_i}\circ \phi^{\sim})(f)W^*(t)\|\\
 &&+\|W(t)(\pi_{t_i}\circ \phi^{\sim})(f)W(t)^*-W(t_i)(\pi_{t_i}\circ \phi^{\sim})(f) W(t_i)^*\|\\
 &&+\|W(t_i)(\pi_{t_i}\circ \phi^{\sim})(f) W(t_i)^*-(w_i\pi_{t_i}\circ \phi^{\sim})(f)w_i^*\|\\
 &&+\|w_i(\pi_{t_i}\circ \phi^{\sim})(f)w_i^*-\pi_{t_i}\circ \psi^{\sim}(f)\|\\
 && +\|\pi_{t_i}\circ \psi^{\sim}(f)-\pi_t\circ \phi^{\sim}(f)\|\\
 &<&\ep_1/16+\ep/32+\dt_u+\ep_1/16+\ep_1/16<\ep
\end{eqnarray}
 for all $f\in {\cal F}$ and for $t\in [t_i, t_{i+1}]$.

 \end{proof}

\begin{df}
Let $D$ be a non-unital \CA. Denote by
$C(\T, {\tilde D})^o$ the \SCA\, of $C(\T, {\tilde D})$ generated by
$C_0(\T\setminus \{1\})\otimes 1_{\tilde D}$ and $1_{C(\T)}\otimes D.$
The unitization of $C(\T, {\tilde D})^o$ is $C(\T, {\tilde D})=C(\T)\otimes {\tilde D}.$
Let $C$ be another non-unital \CA, $L: C(\T, {\tilde D})^o\to C$ be a \cpc\,
and $L^{\sim}: C(\T)\otimes {\tilde D}\to {\tilde C}$ be the unitization.
Denote by $z$ the standard unitary generator of $C(\T).$  For any finite subset ${\cal F}\subset C(\T)\otimes D,$
any finite subset ${\cal F}_d\subset {\tilde D},$ and $\ep>0,$
there exists a finite subset ${\cal G}\subset D$ and $\dt>0$ such that, whenever
$\phi: D\to  C$ is a ${\cal G}$-$\dt$ -multiplicative \cpc \, (for any \CA\, $C$)  and
$\|[u, \phi(g)]\|<\dt$ for all $g\in {\cal G},$ there exists a ${\cal F}$-$\ep$-multiplicative
\cpc\, $L': C(\T)\otimes {\tilde D}\to {\tilde C}$ such that
\beq
\|L'(z\otimes 1)-u\|<\ep\andeqn \|L'(1\otimes d)-\phi^{\sim}(d)\|<\ep\tforal d\in {\cal F}_d.
\eneq
We will denote such $L'$ by $\Phi_{u, \phi}.$

Conversely,   there exists a finite subset ${\cal G}'\subset C(\T, {\tilde D})^o$ and $\dt'>0,$
if $L: C(\T, D)^o\to C$ is ${\cal G}'$-$\dt'$-multiplicative \cpc,
there is a unitary $u\in {\tilde C}$ such that
\beq
\|{\tilde L}(z\otimes 1)-u\|<\ep
\eneq
and $\phi=L^{\sim}|_{1\otimes D}$ is a \cpc.
\end{df}

{{In what follows, we use ${\cal A}$ for the family of \CA s which can be  approximated  by
\CA s  {{$D\in {\overline{\cal D}_r}$}} for some integer $r\ge 1,$}}
{{Note that $B_T\subset {{{\cal A}}}.$}}

\begin{lem}\label{TUnHT2}
Let $A=C(\T)\otimes {\tilde D},$ where $D\in 
{{\cal {A}.}}$ 
Let ${\cal F}\subset A$ be a finite subset, let
$\ep>0$ be a positive number and let $\Delta: A_+^{q, {\bf 1}}\setminus \{0\}\to (0,1)$  be an order preserving map. There exists  a finite subset ${\cal H}_1\subset A_+^{\bf 1}\setminus \{0\},$
$\gamma_1>0,$ $\gamma_2>0,$ $\dt>0,$ a finite subset
${\cal G}\subset A,$
and a finite subset ${\cal P}\subset \underline{K}(A),$
a finite subset ${\cal H}_2\subset A,$
 a finite subset ${\cal U}\subset J_c(K_1(A))$  for which $[{\cal U}]\subset {\cal P}$
satisfying the following:
For any unital ${\cal G}$-$\dt$-multiplicative \morp s
$\Phi_{u,\phi}, \Phi_{v, \psi}: A\to {\tilde C}$
for some amenable $C\in {\cal D}^d$ with continuous scale,
where $u, v\in U({\tilde C})$ and
$\phi, \psi: D\to C$ are two ${\cal G}_d$-$\dt$-multiplicative \cpc s
{{(${\cal G}_d=\{g: g\otimes 1\in {\cal G}\}$)}}
such that
\begin{equation}
\label{HT2-1}
[\Phi_{u,\phi}]|_{\cal P}=[\Phi_{v,\psi}]|_{\cal P},
\end{equation}
\begin{equation}\label{HT2-2}
\tau(\Phi_{u,\phi}(a))\ge \Delta(\hat{a}),\,\,\, \tau(\Phi_{v,\psi}(a))\ge \Delta(\hat{a}) \tforal \tau\in T(C) \tand a\in {\cal H}_1,
\end{equation}
\begin{equation}\label{Uni1-3}
|\tau\circ \Phi_{u, \phi}(a)-\tau\circ \Phi_{v,\psi}(a)|<\gamma_1 \tforal a\in {\cal H}_2\tand
\end{equation}
\begin{equation}\label{Uni1-3+100}
{\rm dist}(\Phi_{u,\phi}^{\dag}(y), \Phi_{v, \psi}^{\dag}(y))<\gamma_2\tforal y\in {\cal U},
\end{equation}
there exists a unitary $W\in {\tilde C}$ such that
\begin{equation}\label{Uni1-4}
\|W(\Phi_{u,\phi}(f))W^*-(\Psi_{v, \psi}(f))\|<\ep,\tforal f\in {\cal F}.
\end{equation}
\end{lem}

\begin{proof}
Let us first reduce the general case  to the case that $D\in 
{{{\overline{\cal D}_r}}}.$
Fix any finite subset ${\cal F}_d\subset D$ and any $\ep_d>0,$  by \ref{Zstable},  there is $D_n\in {{{\overline{\cal D}_r}}}$
such that
\beq
{\rm dist}(x, D_n)<\ep_d\rforal x\in {\cal F}_d.
\eneq
This effectively allows us to assume that $D\in 
{{{\overline{\cal D}_r}}}.$
It should then be noted that $C(\T, {\tilde D})\in 
{{{\overline{\cal D}_{r+1}}}}.$

Now we assume that $D\in 
{{{\overline{\cal D}_r}}}.$

Let ${\bf L}=8\pi,$ $r_0=0,$ $r_1=0,$  ${\bf T}(n,k)=n$ for all $(n,k),$ $s=1$ and $R=7.$
Let $1/2>\ep>0$ and ${\cal F}\subset A$ be a finite subset.
Let $\Delta_0=\Delta/2.$
Let $F': A_+\setminus \{0\}\to \R\times \N$ be given by \ref{FfullDelta} associated with $\Delta_0.$

Put $A_0=C(\T, {\tilde D})^o.$
Let ${\cal F}_I\subset A_0$ be a finite subset  such that, if $x\in {\cal F},$ then
$x=\lambda+y$ for some $y\in {\cal F}_I.$

Let $\dt_0>0$ (in place of $\dt$),  ${\cal G}_0\subset A_0$ (in place of ${\cal G}$) be finite subset,
${\cal P}_0\subset \underline{K}(A_0)$ (in place of ${\cal P}$),
${\cal U}_0\subset U(M_N(A))$ (for some integer $N\ge 1$)
${\cal H}_0\subset (A_0)_+\setminus \{0\}$
(in place of ${\cal H}$) and $K\ge 1$ be an integer required by
Theorem 3.14 of \cite{eglnkk0}
for $A_0,$ $\ep/16$ (in place of $\ep$), ${\cal F}_I$ (in place of ${\cal F}$), ${\bf L},$ $F',$ (in place
$F$), as well as $r_0, r_1, T,$ $s$ and $R$ above.
As in
{{3.15 of \cite{eglnkk0},}} we can choose
${\cal U}_0=\{g_1, g_2,...,g_{k(A)}\}$ so that
$K_1(A)\cap {\cal P}_0=\{[g_1], [g_2],...,[g_{k(A)}]\}.$

Let  $\gamma_1'>0,$ $\gamma_2'>0,$ $\dt'>0,$
${\cal G}'\subset A,$
${\cal H}_1'\subset  (A)^{\bf 1}_+\setminus \{0\}$,
${\cal P}'\subset \underline{K}(A),$ ${\overline{{\cal U}'}}\subset J_c(K_1(A))$
and ${\cal H}_2'\subset A_{s.a.}$
be finite subsets required by \ref{UniqN1}
 for $\min\{\dt_0/4, \ep/16\}$ (in place of $\ep$)
${\cal G}_0$ (in place of ${\cal F}$) and for $\Delta_0$ (in place of $\Delta$).

Put $\gamma_1=\gamma_1'/4,$ $\gamma_2={1\over{2K+1}}\min\{\gamma_2'/16,\ep/64\},$
$\dt=\min\{\dt'/16, \dt_0/16, \gamma_1/16, \gamma_2/16, \ep/2^{10}\},$ ${\cal H}_1={\cal H}_1',$ ${\cal H}_2={\cal H}_2$
and ${\cal G}={\cal G}'.$

Now suppose that $\Phi_1, \Phi_2: A\to {\tilde C}$ are two
${\cal G}$-$\dt$-multiplicative \cpc s  such that
$\Phi_1=\Phi_{u, \phi}$ and $\Phi_2=\Phi_{v, \psi},$ where $u,v$ and $\phi, \psi$ are as given.
Moreover,
 $\Phi_1, \Phi_2$ satisfy the condition \eqref{HT2-1}, \eqref{HT2-2}, \eqref{Uni1-3},
\eqref{Uni1-3+100} and \eqref{Uni1-3+100} for the above mentioned
$\Delta,$ ${\cal P},$ ${\cal H}_1,$ ${\cal H}_2,$ $\gamma_1, $ $\gamma_2$ and ${\cal U}.$

Since $C\in {\cal D}^d,$ there exists  a sequence of positive elements $\{b_n\}$ of $C,$
a sequence of \SCA s $C_{0,n}\in {\cal C}_0,$
two sequences
of \cpc s $\phi_{0,n}: A\to B_n$ and $\phi_{1,n}: C\to C_{0,n}$
such that $C_{0,n}\perp B_n,$
\beq\label{HT1-10}
&&\hspace{-0.2in}\lim_{n\to\infty}\|\phi_{i,n}(ab)-\phi_{i,n}(a)\phi_{i,n}(b)\|=0\rforal a,\, b\in C,\\\label{BT1-11}
&&\hspace{-0.2in}\lim_{n\to\infty}\|x-(\phi_{0,n}(x)\oplus \diag(\overbrace{\phi_{1,n}(x),\phi_{1,n}(x),...,\psi_{1,n}(x)}^K)\|=0
\rforal x\in
C\\
&&\lim_{n\to\infty} \sup_{\tau\in T(C)}d_\tau(b_n)=0,\,\,\,
t(f_{1/4}(\phi_{1,n}(e_C)))\ge 1/2\rforal t\in T(C_{0,n}),\\
  &&\andeqn \tau(f_{1/4}(\phi_{1,n}(e_C))>1/2
\rforal \tau\in T(C),
\eneq
where $e_C\in C$ is a strictly positive element with $\|e_C\|=1,$
$B_n=\overline{b_nCb_n}$
{{(see 9.2 of \cite{eglnp1}).}}
Put $C_n=M_K(C_{0,n}),$ $n=1,2,...$
It should be noted that $C_n\perp B_n,$ $n=1,2,....$
We may assume, \wilog, for all $n,$
\beq\label{BT1-5}
 \sup_{\tau\in T(C)}d_\tau(b_n)<\min\{\gamma_1/64K, \gamma_2/64K, \min\{\Delta_0(\hat{h}): h\in {\cal H}_1\}/4(K+2)\}.
\eneq

Let $u_i, v_i\in M_N({\tilde C})$ ($i=1,2,...,k(A)$) be two unitaries such
that
\beq\nonumber
\|(\Phi_1\otimes {\rm id}_{M_N})(g_i)-u_i\|<\min\{\ep/2^8, \gamma_2/8\} \andeqn
\|(\Phi_2\otimes {\rm id}_{M_N}(g_i)-v_i\|<\min\{\ep/2^8, \gamma_2/8\}.
\eneq
Let 
{{$w_i\in CU({\tilde C})$}} be such that
\beq
\|u_iv_i^*-w_i\|<(5/4)\gamma_2\andeqn w_i=\prod_{j=1}^{m(i)} w_{i,j},\,\,\,
w_{i,j}=w_{1,i, j}^*w_{2,i, j}^*w_{1,i,j}w_{2,i,j},
\eneq
where
$w_{s,i,j}\in U({\tilde C}),$ $s=1,2,$ $j=1,2,...,m(i)$ and $i=1,2,...,k(A).$
Let $m=\max\{m(i): 1\le i\le k(A)\}.$

Write $w_{s,i,j}=\af_{s,i,j}+c(w_{s,i,j}),$ where
$\af_{s,i,j}\in \T\subset \C$ and
$c(w_{s,i,j})\in C,$ $j=1,2,...,m(i).$
Note that $\|c(w_{s,i,j})\|\le 2,$ $j=1,2,...,m(i),$ $i=1,2,...,k(A).$

Define  $\psi_{1,n}: A\to C_n$ by
$\psi_{1,n}(a)=\diag(\overbrace{\phi_{1,n}(a), \phi_{1,n}(a),...,\phi_{1,n}(a)}^K)$ for all $n.$
Put $\Psi_j=\psi_{1,n}\circ \Phi_j: A\to C_n,$ $j=1,2.$

\noindent
Let ${\cal G}_2={\cal G}\cup \{c(w_{s,i,j}):  s=1,2,\,1\le j\le m(i), 1\le i\le k(A)\}.$
We can choose $n$ large enough so that $\psi_{0,n}$ and $\psi_{1,n}$
are ${\cal G}_2$-${\dt\over{2^{12}mN^2}}$-multiplicative.
In particular,
by (\ref{Uni1-3+100}) and the choice of $\gm_2,$
\beq\label{HT1-20}
&&{\rm dist}(\overline{\lceil \phi_{0,n}^{\sim}(u_i) \rceil }, \overline{\lceil \phi_{0,n}^{\sim}
{{(v_i)}}\rceil}) \le \gamma_2'/4
\,\,\,{\rm in}\,\,\, U({\tilde B}_n)/CU({\tilde B}_n) \andeqn\\\label{HT2-20+}
&&{\rm dist}(\overline{\lceil \psi_{1,n}^{\sim}(u_i) \rceil }, \overline{\lceil \psi_{1,n}^{\sim}
{{(v_i)}}\rceil}) \le \gamma_2'/4
\,\,\,{\rm in}\,\,\,U({\tilde C}_n)/CU({\tilde C}_n).
 \eneq

It is standard to check that, by choosing sufficiently large $n,$ we may
assume that $\Psi_j$ are ${\cal G}$-$\dt$-multiplicative \cpc s satisfying the following:
\beq\label{BT1-15}
t\circ \Psi_1(h)\ge \Delta_0(\hat{h}), \,\,\, t\circ \Psi_2(h)\ge \Delta_0(\hat{h})\rforal h\in {\cal H}_1,\\
|t\circ \Psi_1(g)-t\circ \Psi_2(g)|<\gamma_2'\rforal g\in {\cal H}_2.
\eneq
Combining these with \eqref{HT2-20+},
by applying \ref{UniqN1}, one obtains a unitary
$U_1\in {\tilde C}_n$ such that
\beq\label{BT1-16}
\|U_1^*\Psi_1(x)U_1-\Psi_2(x)\|<\min\{\dt_0/4, \ep/4\}\rforal x\in {\cal G}_0.
\eneq
Write $U_1=\lambda\cdot 1_{{\tilde C}_n}+c(U_1),$ where $\lambda\in \T\subset \C$ and $c(U_1)\in C_n.$
Define $V_1=\lambda\cdot 1_{\tilde C}+c(U_1).$ Then $V_1\in U({\tilde C}).$
Note, since $B_n\perp C_n,$ $V_1^*bV_1=b$ for all $b\in B_n.$

Let  $E_n=\overline{C_{0,n}CC_{0,n}}$ and
 $e_{E_n}$ be a strictly positive element with $\|e_{E_n}\|=1.$
Put $\Lambda: A\to C_{1,n}\subset E_n$ by defining
$\Lambda(a)={\rm Ad}\, V_1\circ \phi_{1,n}\circ \Phi_1(a)$ for all $a\in A_0,$
By \eqref{BT1-15}, $\Lambda$ is
$F'$-${\cal H}_1$-full in  $C_{1,n}.$ It follows
it is
$F'$-${\cal H}_1$-full in $E_n.$
By \eqref{BT1-5}, we may assume that
$b_n\lesssim e_{E_n}.$

Let $L_i=\phi_{0,n}\circ \Phi_i,$ $i=1,2.$
By  \eqref{BT1-11}, we assume that $L_i$ is also ${\cal G}$-$2\dt$-multiplicative and
\vspace{-0.1in}\beq\label{BT1-17}
\|L_i(x)\oplus \Psi_i(x)-\Phi_i(x)\|<\dt\rforal x\in {\cal G}.
\eneq
Since $K_i(C_n)=\{0\},$ $i=0,1,$
we conclude that
\beq
[L_1]|_{\cal P}=[\Phi_1]|_{\cal P}=[\Phi_2]|_{\cal P}=[L_2]|_{\cal P}.
\eneq
It follows from \ref{Ulength} and \eqref{HT1-20} that, in $B_n,$
\beq
{\rm cel}(\lceil L_1(z\otimes 1)\rceil \lceil L_2(z\otimes 1)\rceil ^*)<8\pi={\bf L}.
\eneq
It follows from
3.14 of \cite{eglnkk0} that there exists a unitary $W_1\in {\tilde B}$
such that
\vspace{-0.1in}\beq\label{BT1-22}
\|W_1^*(L_1(a)\oplus  S(a))W_1-(L_2(a)\oplus S(a))\|<\ep/16,
\eneq
where $S(a)=\diag(\overbrace{\Lambda(a), \Lambda(a),...,\Lambda(a)}^K)=V_1^*\Psi_1(a)V_1,$ for all $a\in {\cal F}_I.$
Put $W=V_1W_1.$  One then estimates, by \eqref{BT1-17}, \eqref{BT1-22} and \eqref{BT1-16},
\beq
{\rm Ad}\, W\circ \Phi_1&\approx_\dt & {\rm Ad}\, W\circ (L_1\oplus {\rm Ad}\, V_1\circ \Psi_1)\\
&&\approx_{\ep/16} L_2\oplus  V_1\circ \Psi_1\approx_{\ep/4} L_2\oplus \Psi_2\approx_\dt \Phi_2
\,\,\,{\rm on}\,\,\, {\cal F}_I.
\eneq
Therefore
\vspace{-0.12in}\beq
\|{\rm Ad}\, W\circ \Phi_1(a)-\Phi_2(a)\|<\ep\rforal a\in {\cal F}.
\eneq
\end{proof}

\begin{lem}\label{densesp}
Let $A$  be a  non-unital  C*-algebra and $T(A)\not=\emptyset,$
  let $U$ be an infinite dimensional  UHF-algebra
and $B\subset A$ be a hereditary \SCA\, of $B.$
Suppose that there exists  $e\in A_+$ with $\|e\|=1$ and $eb=be=b$ for all
$b\in B.$ Then there is a unitary $w\in {\tilde A\otimes U}$
with the form $w=\exp(i \pi (e\otimes h))$ for some $h\in U_{s.a.}$  with
$\tau_U(h)=0$ (where $\tau_U$ is the unique tracial state of $U$) such that for
any unitary $u=\lambda +x\in {\tilde A}$   with
$\lambda \in \T\subset \C$ and $x\in B,$  one has, for any $b\in B$ and $f\in C(\T),$
\beq\label{densesp-1}
\tau(bf((u\otimes 1)w))=\tau(b)\tau(f(1_A\otimes \exp(i h)))=\tau(b)\int_{\T}f dm
\eneq
and for all $\tau\in T(A\otimes U),$
where $m$ is the normalized Lebesgue
 measure on $\T.$
Moreover, for any $a\in B$  and $\tau\in T(A\otimes U)$, $\tau ((a\otimes 1)  w^j)=0$ if $j\not=0$.
Furthermore, if $A$ has continuous scale, then, for any $\ep>0,$  and any $N\ge 1,$ one can choose $e$ such that
\beq\label{densesp-2}
|\tau((u\otimes 1)w)^j)|<\ep\rforal  0<|j|\le N.
\eneq
\end{lem}

\begin{proof}
Denote by $\tau_U$ the unique trace of $U$. Then any trace $\tau\in T(A\otimes U)$ is a product trace, i.e.,
$$\tau(a\otimes b)=\tau(a\otimes 1)\otimes\tau_U(b),\quad a\in A, b\in U.$$

Pick a selfadjoint element $h\in U$ such that the spectral measure of the  unitary $w_0=\exp(ih)$
is the Lebesque measure (a Haar unitary).  Moreover, ${\rm sp}(h)=[-\pi, \pi]$ and
$\tau(h)=0.$

Then one has, for each $n\in \Z,$
\vspace{-0.12in}$$
\tau_U(w_0^n)=
\left\{
\begin{array}{ll}
1, & \textrm{if $n=0$},\\
0, & \textrm{otherwise}.
\end{array}
\right.
$$
Put $w=\exp(i (e\otimes h))\in {\widetilde{A\otimes U}}.$
Thus $w=\sum_{k=0}^{\infty} {i e^k\otimes h^k\over{k!}}.$
Hence, for any $\tau\in T(A\otimes U)$, one has, for each $n\in \Z,$ {{and any $b\in B$ (note that $eb=be=b$),}}
$$\tau(b((u\otimes 1) w)^n)=\tau(b(u^n\otimes 1)(e\otimes 1)(1\otimes w_0^n))=\tau(bu^n\otimes 1)\tau_U(w_0^n)=
\left\{
\begin{array}{ll}
\tau(b), & \textrm{if $n=0$},\\
0, & \textrm{otherwise};
\end{array}
\right.
$$
and therefore
$$\tau((b\otimes 1)P(u\otimes 1)w))=\tau(b)\tau(P(1\otimes w))=\tau(b)\int_{\T}P(z)dm$$
for any polynomial $P$. Similarly, $\tau(bP(u\otimes w)^*)=\tau(b)\int_TP({\bar z})dm$
for any polynomial $P.$
Since polynomials of $z$ and $z^{-1}$ are dense in $\mathrm{C}(\T)$, one has
$$\tau((b\otimes 1)f((u\otimes 1)w))=\tau(b)\tau(f(1\otimes w))=\tau(b)\int_{\T}fdm,\quad f\in \mathrm{C}(\T),$$
as desired.

For the second part of this lemma,
assume that $A$ has continuous scale.  Then,  for any $\dt>0$ and any integer $N_1\ge 1,$
we can choose $e_1, e\in A_+$ such that $1\ge e\ge e_1,$ $e_1e=ee_1=e_1$
$\tau(e^k)\ge \tau(e_1^{N_1})>1-\dt$ for all $\tau\in T(A)$ and $k\in \N.$
Fix  $N$ and $\ep>0.$
A simple calculation shows the second part of
the lemma follows by choosing sufficiently small $\dt$ and large $N_1.$

\end{proof}

\begin{prop}\label{densesp3}
Let $C$ be a  non-unital  amenable simple \CA\, and let $U$ be an infinite dimensional UHF-algebra.
{{For any $ \dt>0, \dt_c>0,$  $1>\sigma_1,\, \sigma_2>0,$ any finite subset ${\cal G}\subset
{\tilde C}\otimes C(\T),$  any finite ${\cal G}_c\subset {\tilde C},$ any finite subset ${\cal H}_1\subset C(\T)_+\setminus \{0\}$ and any finite subset ${\cal H}_2\subset  (C\otimes C(\T))_{s.a.}$
and any integer $N\ge 1,$
 there exist $\dt_1>0$ and a finite subset
 ${\cal G}_1\subset C$ satisfying the following:}}
For
 any unital
 ${\cal G}_1$-$\dt_1$-multiplicative \morp\, $L: C\to  A$ and a unitary
 $u\in {\tilde A}$ with $\|L(g),\, u]\|<\dt_1$ for all $g\in {\cal G}_1,$
 where $A$ is another non-unital \CA\, with $T(A)\not=\emptyset$ and with continuous scale,
 there exists  a positive element $e\in A$ with $\|e\|=1$ and $h\in U$
satisfying the following:
there are two unital $\mathcal G$-$\dt$-multiplicative \cpc s $L_1, L_2: {\tilde C}\otimes C(\T)\to {\tilde B}$ such that
 \beq\label{densesp3-1}
 |\tau(L_1(f))-\tau(L_2(f))|<\sigma_1\tforal f\in {\cal H}_2,\ \tau\in T(B),\ \textrm{and}\ \label{densesp3-2} \\
 \tau(g(u\exp({{\sqrt{-1}}} e\otimes h)))\ge \sigma_2(\int g dm)\tforal g\in {\cal H}_1,\ \tau\in T(B),
 \eneq
where $B=A\otimes U$ and $m$ is the normalized Lebesgue
 measure
 on $\T,$  and
  \beq\label{densesp3-n11}
 \|L_i(g\otimes 1_{C(\T)})-L^{\sim}(g)\otimes 1_U\|<\dt_c\tforal {{g}}\in {\cal G}_c,\,\,\, i=1,2,\\\label{densesp3-n12}
 \|L_1(g\otimes z^j)-L^{\sim}(g)(u\exp({{\sqrt{-1}}}e\otimes h))^j\|<\dt_c\tforal {{g}}\in {\cal G}_c\andeqn\\\label{densesp3-n13}
 \|L_2(g\otimes z^j)-L(g)^{\sim} \exp({{\sqrt{-1}}}e\otimes h)^j\|<\dt_c\tforal {{g}}\in {\cal G}_c
 \eneq
 and for all $0<|j|\le N,$ where $L^{\sim}: {\tilde C}\to {\tilde A}$ is
 the unital extension of  $L.$  Moreover, $\tau(e\otimes h)=0$ for all $\tau\in A\otimes U.$
%
\end{prop}

\begin{proof}
\Wlog, we may assume that there are  finite subsets ${\cal G}_c, {\cal H}_{c,1}\subset {\tilde C}$
such that ${\cal G}=\{c\otimes 1_{C(\T)}: c\in {\cal G}_c\}\cup \{1, 1_{\tilde C}\otimes z, 1_{\tilde C}\otimes z^*\}$ and
${\cal H}_2=\{c\otimes 1_{C(\T)}: c\in {\cal H}_{c,1}\}\cup \{1\otimes b: b\in {\cal H}_T\},$
where ${\cal H}_T\subset C(\T)_{s.a.}.$ We may assume
that $1_{\tilde C}\in {\cal G}_c,$ $1_{\tilde C}\in {\cal H}_{c,1}$ and $1_{C(\T)}\in {\cal H}_T.$
We may also assume that $\|a\|\le 1$ for all $a\in {\cal G}_c\cup {\cal H}_2.$
Put
$$
{\cal G}_0=\{cd\otimes gf: c, d\in {\cal G}_c\cup {\cal H}_{c,1},\,\, g,f\in \{z, z^*\}\cup {\cal H}_T\}.
$$

Fix $\dt, \dt_c>0, \sigma_1, \sigma_2>0.$
Put $\ep=\min\{\dt/4, \dt_c/4, \sigma_1/4, \sigma_2/4\}.$

Let $\dt_1'>0$ and ${\cal G}_{0m}\subset {\tilde C}$ be a finite subset
such that there is a ${\cal G}_0$-$\ep$-multiplicative \cpc\, $L': {\tilde C}\otimes C(\T)\to D,$
for any \CA\, $D$ and any ${\cal G}_{0m}$-$\dt_1'$-multiplicative \cpc\, $L'': {\tilde C}\to D,$
such that
\beq
\|L'(g\otimes 1_{C(\T)})-L''(g)\|<\ep\rforal g\in {\cal G}_0.
\eneq

Let ${\cal G}_1={\cal G}_0\cup  {\cal G}_{0m}$ and
$\dt_1=\min\{\dt_1'/4, \ep/4\}.$

Now suppose that $L: {\tilde C} \to {\tilde A}$ is a ${\cal G}_1$-$\dt_1$-multiplicative \cpc\, and
$u\in {\tilde A}$ is a unitary.
\Wlog, we may assume that there are  positive elements
$e_1, e\in A$ with $\|e_1\|=1=\|e\|$ such that
\beq
e_1L(g)=L(g)e=L(g)\rforal g\in C, ee_1=e_1e=e_1\andeqn \tau_U(e_1)>1-\ep.
\eneq
Furthermore, \wilog, we may assume that $L(c)=e_1L(c)e_1$ for all $c\in C.$
Let $h\in U$ be as in  \ref{densesp}.
Let $v=\exp ({{\sqrt{-1}}} e\otimes h).$   Note that $\tau_U(e^j)>1-\ep$ for all $j\in \N.$
We can choose $e$ so that both \eqref{densesp-1} and \eqref{densesp-2} hold. This lemma then follows from an easy  application
of \ref{densesp} and
Lemma 2.8 of \cite{Lnmemhp}
(with $L_1=\Phi_{v_1, L}$ and
$L_2=\Phi_{v_2, L},$ where $v_1=u(\exp(ie\otimes h))$ and $v_2=\exp(ie\otimes h)$).

\end{proof}

\begin{cor}\label{CCTW}
Let $C$ be a non-unital separable \CA.
Suppose that there is an embedding $\phi: C\to {\cal W}.$
Then $C(\T, {\tilde C})^o$
satisfies the condition in \ref{DFixA}.
Moreover, there exists an embedding $\Phi: C(\T,{\tilde C})^o\to {\cal W}$ which maps
strictly positive elements to strictly positive elements.
\end{cor}

\begin{proof}
Let $\{e_n\}$ be an approximate identity for ${\cal W}$ such that
$e_{n+1}e_n=e_n$ fort all $n.$
Let $W_1=\overline{e_n{\cal W}e_n}$ Then there exists an isomorphism $\psi_w:{\cal W}\to W_1.$
Put $\phi_0=\psi_w\circ \phi.$
Therefore there is $e\in {\cal W}_+$ with $\|e\|=1$ such that
$e\phi_1(c)=\phi_1(c)e=\phi_1(c)$ for all $c\in C.$
Let ${\cal U}$ be a UHF-algebra of infinite type. Choose $h\in {\cal U}_{s.a.}$ with ${\rm sp}(h)=[-\pi, \pi]$
and $t_{\cal U}(h)=0,$ where $t_{\cal U}$ is the unique tracial state of ${\cal U}.$
Define $x=
{{\sum_{n=1}^{\infty} {(\sqrt{-1}e\otimes h)^n\over{n!}}}}\in {\cal W}\otimes {\cal U}$ and
$u=1_{{\tilde {\cal W}}}+x\in {\tilde {\cal W}}\otimes {\cal U}.$
Note that $u\phi_1(c)=\phi_1(c)u$ for all $c\in C.$
Define $\Psi: C(\T, {\tilde C})\to {\tilde {\cal W}}\otimes {\cal U}$ by
\beq
\Psi(f\otimes 1_{{\tilde C}})=f(u)\rforal f\in C(\T)\andeqn
\Psi(1_{C(\T)}\otimes c)=\phi_1(c)\rforal c.
\eneq
This gives a \hm\, $\Phi: C(\T, {\tilde C})^o\to {\cal W}\otimes {\cal U}.$
By  the proof of \ref{densesp},  we have, for all $c\in C$ and $f\in C(\T),$
\beq
(t_{\cal W}\otimes t_{\cal U})(\Psi(f\otimes c))=t_{\cal W}(\phi_1(c))\int_\T f dm,
\eneq
where $m$ is the normalized Lebesgue measure on $\T.$
It follows that, for any $f\in (C(\T)\otimes {\tilde C})_+,$
\beq
(t_{\cal W}\otimes t_{\cal U})(\Psi(f))=\int_\T t_w(\phi_1(f(t)))dm.
\eneq
This implies $\Phi$ is injective.  Note that ${\cal W}\otimes {\cal U}\cong {\cal W}.$
By replacing ${\cal W}$ by $\overline{\Phi(C(\T, {\tilde C})^o){\cal W}\Phi(C(\T, {\tilde C})^o)},$
we may assume that $\Phi$ maps strictly positive elements to strictly positive elements.
It follows from
{{5.6 of \cite{eglnp1}}}
that $C(\T, {\tilde C})^o$ satisfies the condition in \ref{DFixA}.
\end{proof}

In what follows, if $A$ is a unital \CA, $u$ is a unitary and $p$ is a projection in $A$
such that $\|[p, u]\|<\dt$ for  some sufficiently small
$\dt,$ then $(1-p)+pu$ is close to a unitary with the form
$(1-p)+v,$ where $v$ is a unitary in $pAp.$   As before
this unitary may be chosen to be  $\lceil (1-p)+pu\rceil$
(see \ref{Dceil}).  Moreover, when $\lceil (1-p)+pu\rceil$
is written we also assume that $\|[p, \, u]\|$ is sufficiently small so the notation makes sense.
Therefore, if $L: A\to B$ is  a map which is $\eta$-${\cal F}$-multiplicative,
 $p\in M_n({\tilde A})$ is a projection and  $u\in {\tilde B}$ is a unitary
such that $\|L(x),\, u]\|<\eta$ for all $x\in {\cal F}$ for some sufficiently small $\eta$ and
some large ${\cal F}\subset A,$  then $L(p)$ is close to a projection and
$\|[L(p), \, u_{\underline{n}}]\|<\dt,$ where
$u_{\underline{n}}=u\otimes 1_{M_n}.$  So \eqref{BHfull-1n1} makes sense. Similar items
will appear again later.

\begin{lem}\label{BHfull}
Let $A\in {\cal A}$ be a separable simple \CA\,   
{{with}} continuous {{scale.}}
For any $1>\ep>0$ and any finite subset ${\cal F}\subset A,$ there exist  $\dt>0,$ $\sigma>0$, a finite subset
${\cal G}\subset A,$ a finite subset $\{p_1,p_2,...,p_k, q_1,q_2,...,q_k\}$ of projections of $M_N({\tilde A})$
(for some integer $N\ge 1$) such that
$\{[p_1]-[q_1],[p_2]-[q_2],...,[p_k]-[q_k]\}$ generates a free subgroup $G_u$ of $K_0(A),$
and a finite subset ${\cal P}\subset \underline{K}(A),$
satisfying the following:

Suppose that $\phi: A\to B\otimes 
{{V}}$ is a \hm\, which maps   strictly
positive elements to  strictly positive elements, where $B\in {\cal D}$ has continuous scale
and 
{{$V$}} is a UHF-algebra of infinite type.
If $u\in U({\widetilde{B\otimes 
{{V}}}})$  is a unitary such that
\beq\label{BHfull-1}
&&\|[\phi(x),\, u]\|<\dt\tforal x\in {\cal G},\\\label{BHfull-1n}
&&{\rm Bott}(\phi,\, u)|_{\cal P}=0{{,}}\\\label{BHfull-1n1}
&&{\rm dist}(\overline{\lceil ((1-\phi^{\sim}(p_i))+\phi^\sim(p_i)u_{\underline{N}})(1-\phi^\sim(q_i))+\phi^\sim(q_i)u_{\underline{N}}^*)\rceil}, {\bar 1})<\sigma\andeqn\\\label{BHfull-1n2}
&&\mathrm{dist}(\bar{u}, \bar{1})<\sigma,
\eneq
(where $u_{\underline{N}}=u\otimes 1_{M_N}$),
then there exists a continuous path of unitaries $\{u(t): t\in [0,1]\}\subset U_0({\widetilde{B\otimes 
{{V}}}})$ such
that
\beq\label{BHTL-3}
&&u(0)=u,\,\,\, u(1)=1\\
&&\|[\phi(a),\, u(t)]\|<\ep\tforal a\in {\cal F}\tand for\,\, all\,\, t\in [0,1].
\eneq
\end{lem}

\begin{proof}

Without loss of generality, one only has to prove the statement with assumption that $u\in CU(B\otimes 
{{V}})$
as \eqref{BHfull-1n2} is assumed.
Since $B\otimes 
{{V\otimes V}}\cong B\otimes 
{{V}},$ to simplify notation, \wilog, we may assume
that $B=B\otimes 
{{V}}.$ In particular, $K_0({\tilde B})$ is weakly unperforated (see \ref{Tweakunp}).

In what follows we will use the fact that every C*-algebra in
${\cal D}$ has stable rank one
(11.5 of \cite{eglnp1}).
We will also use $z$ for the generator unitary function on the unit circle.
Let $A_2=C(\T)\otimes {\tilde A}$ and  $m$ is the normalized Lebesgue measure on the unit circle $\T.$ Define
\beq
\Delta(\hat{h})=(1/4) \inf\{\int_\T\tau(h(t))dm: \tau\in T(A)\}
\eneq
for $h\in (A_2)_+^{\bf 1}\setminus \{0\}$ (note, by the assumption, $T(A)$ is compact).
 Let ${\cal F}_1=\{x\otimes f: x\in {\cal F}, f=1, z, z^*\}.$
To simplify notation, {\wilog}, we may assume that ${\cal F}\subset A^{\bf 1}.$
Let $1>\dt_1>0$ (in place of $\dt$), ${\cal G}_1\subset A_2$ be a finite subset
(in place of ${\cal G}$), $1/4>\gamma_1>0, \,1/4>\gamma_2>0,$ ${\cal P}'\subset \underline{K}(A_2)$
(in place of ${\cal P}$)
be a finite subset, ${\cal H}_1\subset (A_2)_+^{\bf 1}\setminus \{0\}$ be a finite subset, ${\cal H}_2\subset (A_2)_{s.a.}$
be a finite subset and ${\cal U}\subset  J_c(K_1(A_2))$
(for some integer $N\ge 1$)  be a finite subset as required by \ref{TUnHT2}
for
$\ep/16$ (in place of $\ep$), ${\cal F}_1$ (in place of ${\cal F}$),  $\Delta$ and $A_2$ (in place of $A$).
Here we assume that $[L]|_{\cal P'}$ is well defined whenever
$L$ is a ${\cal G}_1$-$\dt_1$-multiplicative \cpc\, from $A_2.$ Moreover,
\beq
[L_1]|_{\cal P'}=[L_2]|_{\cal P'},
\eneq
 if both $L_1$ and $L_2$ are  ${\cal G}_1$-$\dt_1$-multiplicative \cpc s
from $A_2$ to a  unital \CA\, and
$\|L_1(g)-L_2(g)\|<\dt_1\rforal g\in {\cal G}_1.$

\Wlog, we may assume that
${\cal G}_1=\{z\otimes 1_{{\tilde A}}, 1_{C(\T)}\otimes a: a\in {\cal G}_{1A}\},$
${\cal H}_1=\{h'\otimes 1_{{\tilde A}}, 1_{C(\T)}\otimes h'': h'\in {\cal H}_{1T}\andeqn h''\in {\cal H}_{1A}\},$
${\cal H}_2=\{h_1\otimes 1_{{\tilde A}}, 1_{C(\T)}\otimes h_2: h_1\in {\cal H}_{2T}\andeqn h_2\in {\cal H}_{2A}\},$
where ${\cal H}_{1T}\subset C(\T)_+^{\bf 1}\setminus \{0\},$ ${\cal H}_{2T}\subset C(\T)_{s.a.},$
${\cal G}_{1,A}\subset {\tilde A},$ ${\cal H}_{1A}\subset A_+^{\bf 1}\setminus \{0\}$
and ${\cal H}_{2A}$ are finite subsets.
Furthermore, we may also assume that
elements in ${\cal H}_{1T}$ and ${\cal H}_{2T}$ are polynomials of $z$ and $z^*$ of
degree no more than $N_1$ and all coefficients with absolute values no more than $M.$
In addition, we assume that ${\cal H}_{1A}\subset {\cal H}_{2A}.$
We may assume that
$
{\cal P}'={\cal P}_1\cup \boldsymbol{\bt}({\cal P}_2)\cup\boldsymbol{\bt}([1_{\tilde A}]),
$
where ${\cal P}_1, {\cal P}_2\subset \underline{K}(A)$ are finite subsets.
We further assume that
\beq
{\rm Bott}(\phi, v(0))|_{{\cal P}_2}={\rm Bott}(\phi, v(t))|_{{\cal P}_2},
\eneq
if $\|[\phi(a),\, v(t)]\|<\dt_1\rforal a\in {\cal G}_{1A}$  and for any continuous
path of unitaries $\{v(t): t\in [0,1]\}.$

We may further assume
that,
\beq\label{BUfull-8}
{\cal U}={\cal U}_1\cup \{\overline{1\otimes z}\}\cup  {\cal U}_2,
\eneq
where ${\cal U}_1=\{\overline{1_{C(\T)}\otimes a}: a\in {\cal U}_1'\subset U({\tilde A})\}$  and
${\cal U}_1'$ is a finite subset,
${\cal U}_2\subset U(M_N(A_2))/CU(M_N(A_2))$ is a finite subset
whose elements represent a finite subset of $\boldsymbol{\bt}(K_0(A)).$
So we may assume that ${\cal U}_2\in J_c(\boldsymbol{\bt}(K_0(A))).$

We may even assume that ${\cal U}_2={\cal U}_{2f}\sqcup {\cal U}_{2t},$
where
${\cal U}_{2f}=\{J_c(g_{1,f}), J_c(g_{2,f}),....,J_c(g_{m(f),f})\}$ and
${\cal U}_{2t}=\{J_c(g_{1,t}), J_c(g_{2,t}), ...,J_c(g_{m(t),t})\},$
where ${\cal P}'\cap \boldsymbol{\bt}(K_0(A))=\{g_{i,f}, g_{j,t}: 1\le i\le m(f),\, 1\le j\le m(t)\}.$
Moreover,  $\{g_{1,f}, g_{2,f},....,g_{m(f),f}\}$ is a set of free generators of
a finitely generated  free subgroup of $\boldsymbol{\bt}(K_0(A))$ and
$\{g_{1,t}, g_{2,t}, ...,g_{m(t),t}\}$ are generators for a finite subgroup
of $\boldsymbol{\bt}(K_0(A)).$  Since $J_c$ is a \hm,
we may assume that there is an integer $k_m\ge 1$ such that $k_m J_c(g_{j,t})=0$ in $U(M_N(A_2))/CU(M_N(A_2)).$
{{Since $g_{i,f}\in \boldsymbol{\bt}(K_0(A))$,}} we may write
that
\beq
g_{i,f}=[(1\otimes (1-p_i)+z\otimes p_i))(1\otimes (1-q_i)+z^*\otimes q_i)],\,\,\,i=1,2,...,m(f).
\eneq
Write
$p_s=(a_{i,j}^{p_s})_{N\times N}$ and $q_s=(a_{i,j}^{q_s})_{N\times N}$ as matrices over ${\tilde A}.$
Let $w_l=(b_{i,j}^l)_{N\times N}$ be unitaries in $M_N({\tilde A})$ such that
$\overline{w_l}=J_c(g_{j,t}),$  $l=1,2,...,m(t).$

 We assume that $(2\dt_1, {\cal P}, {\cal G}_1)$ is a $KL$-triple for $A_2$, $(2\dt_1, {\cal P}_1, {\cal G}_{1A})$ is a $KL$-triple for $A$ (see 2.12 of \cite{GLp1}, for example).
We may also choose $\sigma_1$ and $\sigma_2$ such that
\beq\label{Bufull-9}
&&0<\sigma_1<(1/4)\min\{\gamma_1/16, \inf\{\Delta(\hat{f}): f\in {\cal H}_{1}\}\}/4M(N+1)\andeqn\\
&&\sigma_2=1-\gamma_2/16(N+1)M.
\eneq.
Choose $\dt_2>0$ and a finite subset ${\cal G}_{2A}\subset {\tilde A}$ {{ (and denote ${\cal G}_2:=\{g\otimes f: g\in {\cal G}_{2A}, f=\{1,z,z^*\}$)}} such that, for any two
unital ${\cal G}_2$-$\dt_2$-multiplicative \morp s $\Psi_1, \Psi_2: C(\T)\otimes {\tilde A}\to {\tilde C}$ (any unital \CA\, $C$),
any ${\cal G}_{2A}$-$\dt_2$-multiplicative \morp\, $\Psi_0: {\tilde A} \to {\tilde C}$ and unitary $
{{W}}\in {\tilde C}$ ($1\le i\le k$), {{if
\beq\label{Bf-10+1}
&&\|\Psi_0(g)-\Psi_1(g\otimes 1)\|<\dt_2\rforal g\in {\cal G}_{2A}\\\label{Bf-10+2}
 &&\|\Psi_1(z\otimes 1_{\tilde A})-
 {{W}}\|<\dt_2\andeqn
\|\Psi_1(g)-\Psi_2(g)\|<\dt_2\rforal g\in {\cal G}_2,
\eneq then}} (
$\underline{{{W}}}={{W}}\otimes 1_{M_N}$)
\beq\label{Bufull-10-1+}
&&\lceil (1-\Psi_0(p_i)+\Psi_0(p_i)\underline{
{{W}}})(1-\Psi_0(q_i)+\Psi_0(q_i) \underline{
{{W}}}^*\rceil\\
&&\hspace{0.3in}\approx_{{\gamma_2\over{2^{10}}}}\lceil \Psi_1(((1-p_i)+z\otimes p_i){{(}}(1-q_i)+z^*\otimes q_i)\rceil,\\
&&\|\lceil \Psi_1(x)\rceil -\lceil \Psi_2(x)\rceil\|<{{\gamma_2/2^{10}}}\rforal x\in {\cal U}_2',\\
&&\Psi_1(((1-p_i)+z\otimes p_i)(1-q_i)+z^*\otimes q_i))\\\label{Bf-1510}
&&\hspace{0.3in}\approx_{{\gamma_2\over{2^{10}}}}
\Psi_1({{((1-p_i)+z\otimes p_i)}})\Psi_1((1-q_i)+z^*\otimes q_i)),
\eneq
furthermore for $d_i^{(1)}=p_i,$ $d_i^{(2)}=q_i,$ there are projections ${\bar d}_i^{(j)}\in M_N({\tilde  C})$ and unitaries  ${\bar z}^{(j)}_i\in {\bar d}_i^{(j)}M_N({\tilde C}) {\bar d}_i^{(j)}$ such that
\beq\label{Bufull-10-1++}
&&\Psi_1(((1-d_i^{(j)})+z\otimes d_i^{(j)}))\approx_{{\gamma_2\over{2^{12}}}}(1-{\bar d}_i^{(j)})+{\bar z}^{(j)}_i\andeqn\\\label{Bf-15107+1-1}
&&{\bar d}_i^{(j)}\approx_{\gamma_2\over{2^{12}}} \Psi_1(d_i^{(j)}),\,\,{\bar z}_i^{(1)}\approx_{\gamma_2\over{2^{12}}} \Psi_1(z\otimes p_i),\andeqn
{\bar z}_i^{(2)}\approx_{\gamma_2\over{2^{12}}}  \Psi_1(q_i\otimes z^*),
\eneq
where
$1\le i\le k,$ $j=1,2${{.}}

 Let $\dt_3>0$ and let ${\cal G}_3\subset C(\T, {\tilde A})^o$ be a finite subset
required by \ref{Lderttorsion}  for $C=C(\T, {\tilde A})^o,$ $\gamma_2/2$ (in place of $\ep$) and for all unitaries
 in ${\cal U}_{2t}.$
\Wlog,  we may write
${\cal G}_3={\cal G}_{3A}\cup\{1, z, z^*\},$ where ${\cal G}_{3A}$ is a finite subset
of $A.$

Choose $\dt_A=\min\{\ep/16, \dt_1/16, \dt_2/16,\sigma_1/4, \sigma_2/4\}/8M(N+1)^3$ and
$$
{\cal G}_A={\cal F}\cup {\cal G}_{1A}\cup {\cal G}_{2A}\cup {\cal H}_{1A}\cup {\cal H}_{2A}\cup {\cal U}_1'\cup
\{a_{i,j}^{p_s}, a_{i,j}^{q_s}, b_{i,j}^l: 1\le s\le ,  1\le l\le m(t),\,1\le i,j\le N\}.
$$
Let ${\cal G}_A'\subset A$  be a finite subset such that
every element $a\in {\cal G}_A$ has the form $a=\lambda +b$ for some $\lambda\in \C$ and
$b\in {\cal G}_A'.$  Let
${\cal G}_4={\cal G}_1\cup {\cal G}_2\cup {\cal G}_3\cup {\cal H}_1\cup {\cal H}_2\cup {\cal U}_1.$

Let $\dt_4>0$ (in place of $\dt_1$) and a finite subset  ${\cal G}_5$ (in place of ${\cal G}_1$)
be as required by
%
\ref{densesp3}
for $A$ (in place of $C$), $\dt_1/4$ (in place of $\dt$), $\dt_A$ (in place of $\dt_c$), $\sigma_1,$ $\sigma_2,$  ${\cal H}_1,$ ${\cal H}_2,$  ${\cal G}_4$
(in place of ${\cal G}$), ${\cal G}_A$ (in place of ${\cal G}_c$)  and $N_1.$

By choosing even smaller $\dt_4,$
\wilog, we may assume that
${\cal G}_5=\{a\otimes f: g\in {\cal G}_{5A}\andeqn f=1,z,z^*\}$ with a large finite subset
${\cal G}_{5A}\supset {\cal G}_A.$
Let ${\cal G}_{5A}'\subset A$ be a finite subset such that
every element $g\in {\cal G}_{5A}$ has the form $g=\lambda+x$ for some $\lambda\in \C$
and $x\in {\cal G}_{5A}'.$

Choose $\sigma>0$ so it is smaller than $\min\{\sigma_1/16, \ep/16, \sigma_2/16, \dt_2/16, \dt_3/16,\dt_4/16,  \dt_A/4\}.$

Let $\dt=\sigma$ and ${\cal G}={\cal G}_{5A}'\cup {\cal G}_A.$

Now suppose that $\phi: A\to B$ is a \hm\, and $u\in CU({\tilde B})$
which satisfy the assumption (\ref{BHfull-1}),  \eqref{BHfull-1n} and  (\ref{BHfull-1n1}) for the above mentioned
$\dt,$ $\sigma,$ ${\cal G},$  ${\cal P},$ $p_i,$ and $q_i.$
There is an isomorphism $s: {{V\otimes V}}
\to {{V}}.$ Moreover,
$s\circ\imath$ is approximately unitarily equivalent to the identity map on 
{{$V$,}}
where   
{{${\imath}: V\to V\otimes V$}} is defined by $\imath(a)=a\otimes 1$ (for all $a\in 
{{V}}$).
To simplify notation, without loss of generality, we may assume that
$\phi(A)\subset B\otimes 1\subset B\otimes 
{{V}}.$
Suppose that $u\in U(B)\otimes 
{{1_V}}$ is a unitary which satisfies the assumption. As mentioned at the beginning,  we may assume that $u\in CU(B)\otimes 
{{1_V}}.$

Applying
\ref{densesp3},
we obtain $e\in (B)_+$ with $\|e\|=1$ and $h\in 
{{V_{s.a}}}$ satisfying
the conclusions of
%
\ref{densesp3}.
Note that we may assume, \wilog,
that
\beq\label{BHTL-nn9}
e\phi^{\sim}(g)=\phi^{\sim}(g)e\rforal g\in {\cal G}_{3A}\cup {\cal G}_{5A}\andeqn\\
 e\phi(g)=\phi(g)e=\phi(g)\rforal g\in
{\cal G}_{3A}'\cup {\cal G}_{5A}'.
\eneq
In particular,  for $E=\diag(\overbrace{e,e,...,e}^N)$ and $y=p_i,\, q_i,$ $i=1,2,...,m(f),$
\beq\label{HTcomme}
(\phi^{\sim}\otimes {\rm id}_{M_N})(y)E=E(\phi^{\sim}\otimes {\rm id}_{M_N})(y).
\eneq

Put $v_1=u\exp(ie\otimes h)$ and $v_2=\exp(ie\otimes h).$
Note that  ${\rm sp}(h)=[-\pi, \pi]$ and
$t_{{V}}(h)=0$ and where $t_{{V}}$ is the unique tracial state of ${{V}}.$ 
Let $L_1, L_2: C(\T)\otimes {\tilde A}\to {\widetilde{B\otimes {{V}}}}$  be given by \ref{densesp3}
such that
 \beq\label{densesp3-1}
 |\tau(L_1(f))-\tau(L_2(f))|<\sigma_1\tforal f\in {\cal H}_2,\ \tau\in T(B), \\\label{densesp3-2}
 \tau(g(v_1))\ge \sigma_2(\int g dm)\tforal g\in {\cal H}_1,\ \tau\in T(B), \andeqn
 \eneq
\vspace{-0.35in}\beq\label{densesp3-n11}
 \|L_i(c\otimes 1_{C(\T)})-\phi^{\sim}(c)\otimes 1_{{V}}\|<\dt_A\tforal c\in {\cal G}_c,\,\,\, i=1,2,\\\label{densesp3-n12}
 \|L_1(c\otimes z^j)-\phi^{\sim}(c)(u\exp({{\sqrt{-1}}}e\otimes h))^j\|<\dt_A\tforal c\in {\cal G}_c\andeqn\\\label{densesp3-n13}
 \|L_2(c\otimes z^j)-\phi(c)^{\sim} \exp(ie\otimes h)^j\|<\dt_A\tforal c\in {\cal G}_c
 \eneq
 and for all $0<|j|\le N_1,$ where $\phi^{\sim}: {\tilde A}\to {\widetilde{B\otimes {{V}}}}.$
Note by \eqref{densesp3-n11},  \eqref{densesp3-n12}  and \eqref{densesp3-n13},
we may write $L_1=\Phi_{v_1, \phi}$ and $L_2=\Phi_{v_2, \phi}.$
Let $u(t)=\exp({{\sqrt{-1}}}3t(e\otimes h))$ for $t\in [0,1/3].$
Then
\beq
\|[\phi(a),\, u(t)]\|<\dt_c\rforal a\in {\cal G}_c.
\eneq
In particular,
\vspace{-0.12in}\beq
{\rm Bott}(\phi, v_1)|_{{\cal P}_2}=0.
\eneq
Exactly the same reason, we have that
\vspace{-0.12in} \beq
{\rm Bott}(\phi, v_2)|_{{\cal P}_2}=0.
\eneq
This implies
\vspace{-0.12in} \beq
[L_1]|_{\boldsymbol{\bt}({\cal P}_2)}=[L_2]|_{\boldsymbol{\bt}({\cal P}_2)}.
\eneq
We also have
\vspace{-0.12in}\beq
[L_1]|_{{\cal P}_1}=[\phi]|_{{\cal P}_1}=[L_2]|_{{\cal P}_1}\andeqn [v_1]=[v_2]=0.
\eneq
Therefore
\vspace{-0.12in}\beq\label{BHTL-nn09}
[L_1]|_{{\cal P}'}=[L_2]|_{{\cal P}'}.
\eneq
Then, by \eqref{densesp3-2}
and the choice of $\dt_A,$
we compute (as in \eqref{densesp-1})
that
\beq\label{BHTL-nn10}
\tau(L_i(h))\ge \Delta(\hat{h})\rforal h\in {\cal H}_1,\,\,\,i=1,2.
\eneq
We  also have
\beq
{\rm dist}(L_1^{\dag}(x),L_2^{\dag}(x))<2\dt_A \tforal x\in  {\cal U}_1\cup\{\overline{z\otimes 1_{\tilde A}}\}.
\eneq
Write $V_2=\diag(\overbrace{v_2,v_2,...,v_2}^N)$ and $H=\diag(\overbrace{h,h,...,h}^N).$
As always, we will write  $\phi^{\sim}(y)$ for $\phi^{\sim}\otimes {\rm id}_{M_N}(y)$
for $y=p_i,\, q_i,$ $i=1,2,...,m(f).$
By \eqref{HTcomme},
\beq\label{BHTL-nn10+1}
\psi^{\sim}(p_i)V_2=\exp(i \psi^{\sim}(p_i)E\otimes H)\andeqn \psi^{\sim}(q_i)V_2=\exp(i\psi^{\sim}(q_i)E\otimes H),
\eneq
$i=1,2,...,m(f).$
However,
\beq
\tau(\psi(q_i)E\otimes H)=\tau(\psi(q_i)E)\tau_{{V}}(H)=0\rforal \tau\in T(B\otimes {{V}}).
\eneq
In the next few lines, ${\bf 1}=1_{M_N}.$
Therefore $$\psi^{\sim}(p_i)V_2+({\bf 1}-\psi^{\sim}(p_i)),\,\psi^{\sim}(q_i)V_2+
({\bf 1}-\psi^{\sim}(q_i))\in CU(M_N({\widetilde{B\otimes {{V}}}})),$$
$i=1,2,...,m(f).$
This implies that
\beq\label{BH-free-1}
L_2^{\dag}(x)={\bar 1}\rforal x\in {\cal U}_{2f}.
\eneq
 with $x=(({\bf 1}-p_i)+p_i\otimes z)(({\bf 1}-q_i)+q_i\otimes z^*),$ one then computes  from (\ref{Bf-1510})  and from the assumption \eqref{BHfull-1n1} 
 that
\beq\label{Bf-15107+1}
&&\hspace{-0.4in}{{\overline{\langle L_1(x)\rangle}\approx
_{\gamma_2/2^{10}} \overline{({\bar z}_i^{(1)}\otimes v_2+({\bf 1}-{\bar p}_i))
({\bar z}_i^{(2)}\otimes v_2+({\bf 1}-{\bar q}_i))}}}\\
&&\hspace{-0.3in}{{=\overline{({\bar z}_i^{(1)}+({\bf 1}-{\bar p}_i))({\bar p}_iV_2+({\bf 1}-{\bar p}_i)\otimes 1_{{V}}) ({\bar z}_i^{(2)}+({\bf 1}-{\bar q}_i))
({\bar q}_iV_2+(1-{\bar q}_i))}}}\\
&&=\overline{({\bar z}_i^{(1)}+({\bf 1}-{\bar p}_i))({\bar z}_i^{(2)}+({\bf 1}-{\bar q}_i))}\approx_{\sigma} {\bar 1}.
\eneq
where ${\bar p}_i, {\bar q}_i, {\bar z}_i^{(1)}, {\bar z}_i^{(2)}$ are as above (see the lines below (\ref{Bf-1510})), replacing $\Psi_1$ by $L_1.$  It follows that
\beq\label{BH-free}
{\rm dist}(L_1^{\dag}(x), {\bar 1})<\gamma_2/4 \tforal x\in \{\overline{1\otimes z}\}\cup {\cal U}_{2f}.
\eneq
By the choice of $\dt_3$ and ${\cal G}_4,$ and by applying  \ref{Lderttorsion},
we also have
\beq\label{BTtorsion}
{\rm dist}(\overline{\lceil L_1(w_l) \rceil}, \overline{\lceil L_2(w_l^*)\rceil} )<\gamma_2/2,\,\,\, l=1,2,...,m(t).
\eneq
Combing  \eqref{BH-free-1}, \eqref{BH-free} and \eqref{BTtorsion},  we obtain
that
\beq\label{BH-cu}
{\rm dist}(L_1^{\dag}(w), L^{\dag}_2(w))<\gamma_2\rforal w\in {\cal U}.
\eneq

By  \eqref{BHTL-nn09}, \eqref{densesp3-1}, \eqref{BHTL-nn10} and  \eqref{BH-cu}, and by applying
\ref{TUnHT2}, we obtain a  unitary  $
{{W_1}}\in {\widetilde{B\otimes {{V}}}}$ such that
\beq\label{BHTL-nn111}
\|
{{W_1}}^*L_2(f)
{{W_1}}-L_1(f)\|<\ep/16\rforal f\in {\cal F}_1.
\eneq
Therefore
\beq\label{BHTL-nn112}
&&\hspace{-0.2in}\hspace{-0.2in}\|[L_1(a),\,\,\, {{W_1}}^*v_2{{W_1}}]\|<\ep/8\andeqn \|L_1(a)-{{W_1}}^*L_1(a){{W_1}}\|<\ep/8
\rforal a\in {\cal F}\\\label{BHTL-nn113}
&& \andeqn \|v_1-{{W_1}}^*v_2{{W_1}}\|<\ep/8.
\eneq
Let $v_1^*{{W_1}}^*v_2{{W_1}}=\exp ({{\sqrt{-1}}}h_1)$ for some
$h_1\in {\tilde B}_{s.a.}$ such that $\|h_1\|\le 2\arcsin (\ep/16).$
Now define $u(t)=u\exp(i3t(e\otimes h))$ for $t\in [0,1/3],$  $u(t)=u(1/3)\exp(i3(t-1/3)h_1)$
for $t\in (1/3, 2/3]$ and $u(t)=u(2/3){{W_1}}^*\exp ({{\sqrt{-1}}}(3(t-2/3))(e\otimes h){{W_1}}$
for $t\in (2/3, 1].$  So $\{u(t): t\in [0,1]\}$ is a continuous path of
unitaries in ${\widetilde{B\otimes {{V}}}}$ such that $u(0)=u$ and $u(1)=1_{\tilde B}.$
Moreover, we estimate, by \eqref{BHfull-1},
 \eqref{BHTL-nn112} and \eqref{BHTL-nn112} that
\beq
\|[\phi(a),\,\, u(t)]\|<\ep\rforal a\in {\cal F}.
\eneq

\end{proof}



%

\begin{cor}\label{BHK00}
Let $A\in {\cal M}_0$ with continuous scale.
For any $1>\ep>0$ and any finite subset ${\cal F}\subset A,$ there exist  $\dt>0,$
a finite subset
${\cal G}\subset A$ 
satisfying the following:

Let $B=B_1\otimes 
{{V}},$
where $B_1\in {\cal M}_0$  with continuous scale which satisfies the UCT and  
${{V}}$ is UHF-algebras of infinite type.
Suppose that $\phi: A\to B$ is a  \hm.

If $u\in U({\tilde B})$  is a unitary such that
\beq\label{BHfull-1cc}
&&\|[\phi(x),\, u]\|<\dt\tforal x\in {\cal G},
\eneq
there exists a continuous path of unitaries $\{u(t): t\in [0,1]\}\subset U({\tilde B})$ such
that
\beq\label{BHTL-3}
&&u(0)=u,\,\,\, u(1)=1_{\tilde B},\\
&&\|[\phi(a),\, u(t)]\|<\ep\tforal a\in {\cal F}\tand for\,\, all\,\, t\in [0,1].
\eneq
\end{cor}

\begin{proof}
Fix a finite subset ${\cal G}'\subset A^{\bf 1}$ and $\ep'>0,$  there exist
positive elements $e', e'', e'''\in B\setminus \{0\}$ with $\|e'\|=1=\|e''\|=\|e'''\|,$
$e'e''=e''e'=e'$  such that $e'''e'=e'e'''=0${{, and}}
\beq\label{BHTL-nn0}
\|\phi(g)e'-\phi(g)\|<\ep'/2\andeqn \|e'\phi(g)-\phi(g)\|<\ep'/2\rforal g\in {\cal G}'.
\eneq
Let $\pi^{B\sim}: {\tilde B}\to \C$ be the quotient map.
\Wlog, we may assume that $\pi^{B\sim}(u)=1_\C.$
Since $U({\tilde B})=U_0({\tilde B}),$ we may
write $u=\prod_{i=1}^m\exp(i h_{j0})$ for some
$h_{j0}\in {\tilde B}_{s.a.}.$  Write $h_{j,0}=r_j+h_{j0}',$ where
$r_j\in \R$ and $h_{j0}'\in B_{s.a.}.$  Note that
$\sum_{j=1}^m r_j=2\pi k_u$ for  $k_u\in \Z.$  Therefore $u=\prod_{j=1}^m\exp(i h_{j0}').$
We may also assume, \wilog, that $u=1+x,$ where $x\in \overline{e''Be''}.$
It is easy to find an element $h_0\in \overline{e'''Be''}$ such that
$\tau(h_0)=\sum_{j=1}^m \tau(h_{j0}')$ for $\tau\in T(B).$
Let $u_0(t)=\exp(-{{\sqrt{-1}}} th_0)$ for all $t\in [0,1].$
Note that
\beq\label{BHTL-nn1}
uu_0(0)=u\andeqn uu_0(1)\in CU({\tilde B}).
\eneq
Moreover, by \eqref{BHTL-nn0},
\vspace{-0.1in}\beq
\|\phi(g)u_0(t)-u_0(t)\phi(g)\|<2\ep'\rforal g\in {\cal G}'\andeqn \rforal t\in [0,1].
\eneq
In other words, we have just reduced
the general case to the case that $u\in CU({\tilde B}).$
In other words, we may assume that, \wilog,
that
${\bar u}={\bar 1}.$

Now we will apply \ref{BHfull}.  Note, from the above, we may assume \eqref{BHfull-1n2} holds.
Since $K_0(A)=\{0\},$ \eqref{BHfull-1n1} automatically holds. Since both $A$ and $B$ are $KK$-contractible,
\eqref{BHfull-1n} also holds.
\end{proof}

\begin{lem}\label{LHTkill}
Let $A\in {\cal B}_T$  have continuous scale.
For
any finite subset
${\cal P}\subset \underline{K}(A),$
 there exists $\dt_0>0$ and a finite subset ${\cal G}_0\subset A$ satisfy
the following:
For any $\ep>0,$ any finite subset ${\cal F}\subset A$  and any \hm\,   $\phi: A\to B=B_1\otimes Q$
which maps strictly positive elements to strictly positive elements,
 where $B_1\cong B_1\otimes \zo\in {\cal D}_0$ has continuous scale,
suppose
that  $u\in U({\tilde B})$ satisfies
\beq
\|[\phi(g),\, u]\|<\dt_0 \tforal g\in {\cal G}_0.
\eneq
Then there exists another unitary $v\in U({\tilde B})$ such that
\beq\label{NewHT-2}
&&\|[\phi(g), \, v]\|<\min\{\ep, \dt_0\} \tforal g\in {\cal G}_0\cup {\cal F}\tand \\\label{NewHT-3}
&&{\rm Bott}(\phi,\, uv)|_{\cal P}=0 \andeqn [uv]=0\,\,\, {\rm in}\,\,\, K_1(B).
\eneq
\end{lem}

\begin{proof}
Define
$\Delta_1(\hat{h})=\inf\{\tau(h): \tau\in T(A)\}$ for $h\in A_+^{\bf 1}\setminus \{0\}.$
Let $\Delta=\Delta_1/2.$
Let $T: A_+^{\bf 1}\setminus \{0\}\to \R_+\setminus \{0\}\times \N$ be the map given by $\Delta$ as in
\ref{FfullDelta}.
Let  ${\cal P}$
be given.

Write $A=\overline{\cup_{n=1}^{\infty}A_n},$ where $A_n=A(W, \af_n)\oplus W_n$ as in {{section 7.}}
\Wlog, we may assume ${\cal F}\subset A_{N'}$ for some
integer $N'$ and ${\cal P}\subset [\imath']({\cal P}_{N'})$ for some finite
subset ${\cal P}_{N'}\subset \underline{K}(A_{N'}),$ where $\imath': A_{N'}\to A$ is the embedding.

Let $\dt_0>0$ and let ${\cal G}_0\subset A_{N'}$ be a finite subset satisfying the following:
${\rm Bott}(L,\, w)|_{\cal P}$ is well defined for
any ${\cal G}_0$-$\dt_0$-multiplicative \cpc\, $L: A\to C$ and
any unitary $w\in {\tilde C}$ with $\|[L(g),\, w]\|<2\dt_0$ for all $g\in {\cal G}_0.$
Moreover, if $w'$ is another unitary, we also require that
\beq
{\rm Bott}(L, ww')|_{\cal P}={\rm Bott}(L, w)|_{\cal P}+{\rm Bott}(\phi, w')|_{\cal P},
\eneq
when $\|[L(g),\, w']\|<2\dt_0$ for all $g\in {\cal G}_0.$

Let $\phi$ and $u$ be given satisfying the assumption
for the above ${\cal G}_0$ and $\dt_0.$

Now fix $\ep>0$ and a finite subset ${\cal F}\subset A.$

Let $\ep_1=\min\{\dt_0/4, \ep/16\}$
and ${\cal F}_1={\cal F}\cup {\cal G}_0.$
Let $\dt_1>0$ (in place of $\dt$),  $\gamma>0,$ $\eta>0,$ ${\cal G}_1\subset A$ (in place of ${\cal G}$) be a finite
subset, ${\cal P}_1\subset \underline{K}(A)$ (in place ${\cal P}$) be a finite subset,
${\cal U}\subset U({\tilde A})$ be a finite subset,
${\cal H}_1\subset A_+\setminus \{0\}$ be a finite subset, and ${\cal H}_2\subset A_{s.a.}$ be a finite
subset required by \ref{TUNIq} for the above $T$ (and for ${\bf T}(n,k)=n$ as $K_0({\tilde B_1})$ is weakly
unperforated). {{Let us assume that ${\cal P}_1$ contains the set $\{[u]\in K_1(A)\subset \underline{K}(A): u\in {\cal U}\}$, by enlarger ${\cal P}_1$ if necessary.}}

\Wlog, we may assume
that ${\cal P}_1\subset [\imath]({\cal P}_N)$ for some finite
subset ${\cal P}_{N}\subset \underline{K}(A_{N}),$ where $N\ge N'$ and  $\imath: A_{N}\to A$ is the embedding.
We assume that $\dt_1<\dt_0.$  \Wlog, by choosing large $N,$ we may assume
that ${\cal G}_1\cup {\cal H}_1\cup {\cal H}_2\subset (A_N)_+^{\bf 1}.$
We may also assume that ${\cal U}\subset U({\tilde A_N}).$
Write $w=\lambda_w+\af(w),$ where $\lambda_w\in \T\subset \C$ and $\af(w)\in A_N.$
As in the remark of  \ref{TUNIq}, we may assume that $[w]\not=0$ and $[w]\in {\cal P}_N$ for all $w\in {\cal U}.$
 Let $G_u$ be the subgroup generated by
$\{\overline{w}: w\in {\cal U}\}.$ We may view $G_u\subset J_c(K_1(A))$ (see the statement of \ref{UniqN1}).
Moreover, for any ${\cal G}_1$-$\dt_1$-multiplicative \cpc\, $L'$ from
$A_N$ to a non-unital \CA\, 
{{$C$}}  induces a \hm\, $\lambda': G_u\to U({\tilde C})/CU({\tilde C})$
(see 14.5 of \cite{Lnloc}).
Furthermore,  since $K_i(A_N)$ is finitely generated, $i=0,1,$ we may assume, with even smaller
$\dt_1$ and larger ${\cal G}_1,$
that
 {{$[\Phi_{u', L'}]$}}defines an element in $KL(C(\T, {\tilde A}_N), 
 {{\tilde C}}),$
if $\|[L'(g),\, 
{{u'}}]\|<\dt_1$ for all $g\in {\cal G}_1.$

Set ${\cal G}= {\cal F}_1\cup {\cal G}_1\cup \{\af(w): w\in {\cal U}\}$
and set
 a rational number
$$
0<\sigma_0<\min\{\inf\{\Delta(\hat{h}): h\in {\cal H}_1\}, \gamma/4\}.
$$

Choose $\dt=\min\{\ep_1/16, \dt_1/16, \gamma/16, \eta/16\}.$
We may write $u=1_{\tilde B}+\af(u),$ where $\af(u)\in B.$
Since $B\otimes Q\cong B,$ $K_i(B)$ is divisible ($i=0,1$). Therefore $KL(A, B)={\rm Hom}(K_*(A), K_*(B))$ and
there  is $\kappa\in KL(C(\T, {\tilde A_N}), {\tilde B})$ such that
\beq\label{LHTkill-2}
[\Phi_{u, \phi\circ \imath}]|_{{\cal P}_{N'}\cup {\boldsymbol{\bt}}({\cal P}_{N'})}=\kappa|_{{\cal P}_{N'}\cup {\boldsymbol{\bt}}({\cal P}_{N'})}\andeqn [u]=\kappa([z\otimes 1_{{\tilde A}_N}]).
\eneq
Note that  $B\cong B\otimes \zo.$  Define $\psi_{b,w}: B\otimes \zo\to
B\otimes {\cal W}$ by letting $\psi_{b,w}(b\otimes a)=b\otimes \psi_{z,w}(a)$
for all $b\in B$ and $a\in \zo,$  where $
 {{\phi_{z,w}:}}\zo\to {\cal W}$
is a \hm\,  defined in \ref{DWZmaps}.
Note also 
{{${\cal W}\otimes Q\cong {\cal W}.$}}There is a \hm\, $\psi_{\sigma,{\cal W}}: {\cal W}\to {\cal W}$ such that
$d_{t_{\cal W}}(\psi_{\sigma,{\cal W}}(e_W))=1-\sigma_0$
and
\beq\label{LHTkill-3}
t_{\cal W}(
{{\psi_{\sigma,{\cal W}}}}(a))=(1-\sigma_0)t_{\cal W}(a)\rforal a\in {\cal W}.
\eneq
Let 
{{$\phi_{w,z}$}} be as in \ref{DWZmaps}.
Note that $t_{\cal W}=t_Z\circ 
{{\phi_{w,z}}}$ and $t_Z=t_{\cal W}\circ 
{{\phi_{z,w}}},$
where $t_{\cal W}$ and $t_Z$ are tracial states of ${\cal W}$ and $\zo,$ respectively.
Let $\psi_{b,\sigma}: B\to B$ be defined by
$\psi_{b, \sigma}(b\otimes a)=
{{b\otimes \phi_{w,z}\circ \psi_{\sigma, {\cal W}}\circ \phi_{z,w}(a)}}$
for all $b\in B$ and $a\in \zo.$
 Note  that
$\psi_{b, \sigma}(B)^{\perp}\not=\{0\}.$

Let $\phi_{\sigma}=\psi_{b,\sigma}\circ \phi$ and {{$u_{\sigma}\in {\tilde B}$ satisfy}} $\af(u_\sigma)=\psi_{b,\sigma}(\af(u)).$
Then,  by \eqref{LHTkill-3}
\beq
|\tau\circ \phi(a)-\tau\circ \phi_{\sigma}(a)|\le
\sigma_0|\tau(a)|\rforal a\in A.
\eneq
In particular,
\beq\label{LHTkill-5}
\tau\circ \psi_{\sigma}(h)\ge (1-\sigma_0)\tau(\phi(h))\ge (1-\sigma_0)\Delta_1(\hat{h})\rforal h\in (A_+)^{\bf 1}\setminus \{0\}.
\eneq
Choose two mutually orthogonal non-zero positive elements
$e_1, e_2\in \psi_{b, \sigma}(B)^{\perp}.$
Note that
\beq
\sum_{i=1}^2\tau(e_i)<\sigma_0\rforal \tau\in T(B).
\eneq

Consider \CA\, $C_0=C(\T, {\tilde A_N})^o.$   By \ref{CCTW},
 $C(\T, {\tilde A})^o$ satisfies the condition in \ref{DFixA}.
It follows from \ref{TExistence} that there exists an asymptotic \cpc s
$L_n: C_0\to B\otimes M_{k(n)}$ such that
\beq
[L_n^{\sim}]|_{{{\cal P}\cup {\boldsymbol{\bt}}({\cal P}\cup \{[1_{\tilde C}]\})}}
=\kappa^{\circledast}|_{{\cal P}\cup {\boldsymbol{\bt}}({\cal P}\cup \{[1_{\tilde C}]\})},
\eneq
where $k(n)\to\infty$ and where
\hspace{-0.1in}\beq\label{01-18-2020+2}
\kappa^{\circledast}|_{\underline{K}(A_N)}=\kappa|_{\underline{K}(A_N)}\andeqn \kappa^{\circledast}|_{{\boldsymbol{\bt}}(\underline{K}({\tilde A_N}))}}=-\kappa|_{{\boldsymbol{\bt}}(\underline{K}({\tilde{A_N}))}.
\eneq
In particular, \beq\label{01-19-2020} \kappa^{\circledast}({\boldsymbol{\bt}}([1_{\tilde A_N}]))=-\kappa({\boldsymbol{\bt}}([1_{\tilde A_N}]))=-[u].\eneq
For each $n,$ there  are two  sequences of \cpc s
$\psi_{0,m}:B\otimes M_{k(n)}\to B_{0,m}\subset B\otimes M_{k(n)}$
and $\psi_{1,m}: B\otimes M_{k(n)}\to D_m\subset B\otimes M_{k(n)}$
such that
\beq
&&\lim_{m\to\infty}\|x-(\psi_{0,m}(x)\oplus  \psi_{1,m}(x))\|=0\rforal x\in B\otimes M_{k(n)},\\
&&\lim_{m\to\infty}\|\psi_{i,m}(ab)-\psi_{i,m}(a)\psi_{i,m}(b)\|=0\rforal a, b\in B\otimes M_{k(n)},\,\,\, i=0,1,\\\label{LHT-kill-13}
&&\lim_{m\to\infty}\sup\{d_\tau(e_{b,m}): \tau\in T(B)\}=0,
\eneq
where {{$D_m\in C_0^0$, $B_{0,m}\perp D_m$, and $e_{b,m}\in B_{0,m}$ is a strictly positive element of $B_{0,m}$. }}
 Since $K_i(D_m)=\{0\},$ $i=0,1,$ by choosing sufficiently large $n$ and $m,$
put $L_n'=\psi_{0,m}\circ L_n,$
we may assume that
$L_n'^{\sim}$ is ${\cal G}$-$\dt/2$-multiplicative (with embedding $\imath: C_0\to C(\T, {\tilde A})^o$) and
\beq\label{01-12-2020}
[L_n'^{\sim}\circ \imath]|_{{\cal P}\cup {\boldsymbol{\bt}}({\cal P}\cup \{[1_{\tilde A_N}])}=\kappa^{\circledast}|_{{\cal P}\cup {\boldsymbol{\bt}}({\cal P}\cup \{[1_{\tilde A_N}])}.
\eneq

Moreover,  by \eqref{LHT-kill-13}, we may assume
that $e_{b,m}\lesssim e_{0,1},$ where $e_{0,1}\in B,$ $e_{0,1}e_1=e_1e_{0,1}=e_{0,1}.$
Since $B$ has almost stable rank one, there is a unitary $w_1\in {\tilde B}$ such that
${\rm Ad}\, w_1\circ L_n'(a)\in B_{0,e}=\overline{e_1Be_1}$ for all $a\in A.$  Put $L_n''={\rm Ad}\, w_1\circ L_n'^{\sim}.$
Let $u_0\in {\widetilde{B_{0,e}}}$ such that $u_0=1_{\tilde{B_{0,e}}}+\af(u_0)$ for some
$\af(u_0)\in (B_{0,e})_{s.a.}$ and
\beq\label{01-18-2020}
\|L_n''(z\otimes 1_{{\tilde{A}_N}})-u_0\|<\dt/16.
\eneq
{{It follows from (\ref{01-12-2020}) and  (\ref{01-19-2020}) that
\beq\label{01-19-2020+}
[u_0]=\kappa^{\circledast}({\boldsymbol{\bt}}([1_{\tilde A_N}]))=-[u]\in K_1(B)
\eneq}}
Define $L: A\to B$ by (for some sufficiently large $n$ as specified above)
\beq\label{01-18-2020+}
L(a)=L''_n(a)\oplus  \psi_{b,\sigma}\circ \phi(a)\rforal a\in A.
\eneq
It is ready to check that $L$ is ${\cal G}_1$-$\dt$-multiplicative.
Let $\lambda': G_u\to U({\tilde B})/CU({\tilde B})$ be a \hm\,
induced by $L.$  Let $\lambda=\phi^{\dag}|_{\overline{{\cal U}}}-\lambda'.$
Since $\psi_{b, \sigma}\circ \phi$ factors through $B\otimes {\cal W},$
$[\psi_{b, \sigma}\circ \phi]=0.$ {{By (\ref{01-12-2020}) and (\ref{LHTkill-2}) and the fact that ${\cal P}_1$ contains the set $\{[u]\in K_1(A)\subset \underline{K}(A): u\in {\cal U}\}$, we know that
$[L]|_{\{[u], u\in G_u\}}=[\phi]|_{\{[u], u\in G_u\}}$. Consequently, }}the map $\lambda$ maps
$G_u$ into $U_0({\tilde B})/CU({\tilde B}).$
Since $U_0({\tilde B})/CU({\tilde B})$ is divisible, we may extend  $\lambda$  to a map from
$J_c(K_1(A))$ into $\Aff(T({\tilde B}))/\Z.$
Choose a non-zero  element $e_0\in B$ with $e_0e_2=e_2e_0=e_0$ such
that $d_\tau(e_0)$ is continuous on $T(B).$
Let $\lambda_T: T(\overline{e_0Be_0})\to T(A)$ be an affine continuous map defined by
$\lambda_T(t)=\tau_A$ for all $t\in T(\overline{e_0Be_0}),$ where ${{\tau_A}}$ is a fixed trace in $T(A).$
Define $\lambda_{cu}: U({\tilde A})/CU({\tilde A})\to U_0({\widetilde{e_0Be_0}})/CU({\widetilde{e_0Be_0}})$
by $\lambda_{cu}|_{J_c(K_1({\tilde A}))}=\lambda$ and
$\lambda_{cu}|_{U_0({\tilde A})/CU({\tilde A})}=\lambda_T^{\sharp},$ i.e.,
$\lambda_{cu}(f)(t)=f(\lambda_T(t))$ for all $t\in T(\overline{e_0Be_0}).$
Define $\lambda_K: \underline{K}(A)\to \underline{K}({\overline{e_0Be_0}})$
by $\lambda_K=0.$ Then $(\lambda_J, \lambda_{cu}, \lambda_T)$ is compatible.
It follows from \ref{11ExtT1} that there exists a \hm\, $\phi_{cu}: A\to \overline{e_0Be_0}$ such
that $([\phi_{cu}], \phi_{cu}^{\dag}, (\phi_{cu})_T)=(\lambda_K, \lambda_{cu}, \lambda_T).$ {{(Note that $\phi_{c,u}\bot L,$ since $e_1, e_2\in \psi_{b,\sigma}(B)^{\bot}$ and $e_1e_2=0$.)}}

Now define $\Phi: A\to B$ by $\Phi(a)=\phi_{cu}(a)\oplus L(a)$ for $a\in A.$
Then $\Phi$ is  ${\cal G}_1$-$\dt$-multiplicative,
\beq\label{LHTkill-25}
&&\tau\circ \Phi(h)\ge \Delta(\hat{h})\rforal h\in {\cal H}_1,\,\,\,\hspace{0.5in} {\rm (by\,\,\eqref {LHTkill-5})}\\
&&|\tau\circ \Phi(h)-\tau\circ \phi(h)\|<\gamma\rforal h\in {\cal H}_2,\\
&&{[}\Phi{]}|_{\cal P}=[\phi]|_{\cal P}\andeqn\\
&&\Phi^{\dag}({\bar w})=\lambda(\bar{w})+\lambda'(\bar{w})=\phi^{\dag}(\bar w)\rforal w\in {\cal U}.
\eneq
{{Let $v'=1_{\tilde B}+\af(u_0)+\psi_{b,\sigma}(\af(u))$. By (\ref{01-18-2020}), (\ref{01-18-2020+}), and $[\psi_{b,\sigma}\circ\phi]=0$, We have
\beq\label{01-18-2020+1}
{\rm Bott}(\Phi, v')|_{\cal P}={\rm Bott}(L_n''|_A, u_0)|_{\cal P}
=[L_n'']\circ \beta|_{\cal P}
\eneq}}
{{Combining with (\ref{01-12-2020}) and (\ref{01-18-2020+2}), one obtains
\beq\label{01-18-2020+3}
{\rm Bott}(\Phi, v')|_{\cal P}=\kappa^{\circledast}\circ\beta|{\cal P}=-\kappa\circ \beta|_{\cal P}.
\eneq}}
By \eqref{LHTkill-25}, $\Phi$ is also
$T$-${\cal H}_1$-full.
By applying \ref{TUNIq}, we obtain a unitary $W\in {\tilde B}$ such that
\beq\label{LHTkill-30}
\|W^*\Phi(f)W-\phi(f)\|<\ep_1\rforal f\in {\cal F}\cup {\cal G}_0.
\eneq
Let $v= W^*(1_{\tilde B}+\af(u_0)+\psi_{b,\sigma}(\af(u)))W.$
Then $v$ is a unitary. {{It follows from (\ref{01-19-2020+}) that
\beq\label{01-19-2020+1}
[v]=-[u]
\eneq}}
We have
\beq
\|[\phi(f), v]\|<\ep_1+\dt\rforal f\in {\cal F}\cup {\cal G}_0.
\eneq
{{Note that ${\rm Bott}(\phi, v)={\rm Bott}(\Phi, v')$. Recall that from (\ref{LHTkill-2}), ${\rm Bott}(\phi, u)|_{\cal P}=\kappa\circ \beta|_{\cal P}$. By  (\ref{01-18-2020+3}) and  (\ref{01-19-2020+1}), we}}  then compute that
\beq
{\rm Bott}(\phi, uv)|_{\cal P}={\rm Bott}(\phi, u)|_{\cal P}+{\rm Bott}(\phi, v)|_{\cal P}=0\andeqn [uv]=0.
\eneq

\end{proof}

\begin{rem}
Lemma \ref{LHTkill} still holds by replacing $Q$ by any UHF-algebra of infinite type
if $K_i(A)$ is finitely generated.
It should be noted that $\dt_0$  and ${\cal G}_0$ are independent
of $\ep$ and ${\cal F}.$

\end{rem}

\begin{lem}\label{LWHTW}
Let $A\in {\cal B}_T$ have continuous scale.
For any $\ep>0$ and any finite subset ${\cal F}\subset A,$ there exists $\dt>0$ and a finite subset
${\cal G}\subset A$ satisfying the following:
Suppose that $\phi: A\to B\cong B\otimes {\cal W},$ where $B\in {\cal D}_0$ with continuous scale,  is a \hm\,
which maps strictly positive elements to strictly positive elements and suppose
that there is a unitary $u\in {\tilde B}$ such that
\beq
\|[\phi(g),\, u]\|<\dt\tforal g\in {\cal G}.
\eneq
Then there exists a continuous path of unitaries $\{u(t): t\in [0,1]\}\subset {\tilde B}$
such that $u(0)=u,$  $u(1)=1_{\tilde B}$ and
\beq
\|[\phi(f),\, u(t)]\|<\ep\tforal f\in {\cal F}.
\eneq
\end{lem}

\begin{proof}
Note that, by \ref{zoselfabsorbing},  $A\cong A\otimes \zo.$ We identify $A$ with
$A\otimes \zo.$
Let $
{{\phi_{w,z}}}: {\cal W}\to \zo$
be defined in \ref{DWZmaps}.
Let $\psi_{w,a}: A\otimes {\cal W}\to A\otimes \zo$ defined by
$\psi_{w,a}(a\otimes w)=a\otimes 
{{\phi_{w,z}}}$ for all $a\in A$
and
$w\in {\cal W}.$
Put $A_1=A\otimes {\cal W}.$ Fix $\ep>0$ and a finite subset ${\cal F}\subset A.$

Note $T(A)=T(A\otimes {\cal W})$ and $\rho_{{\tilde A}}(K_0({\widetilde{A\otimes {\cal W}}}))=\Z.$
It follows from \ref{11ExtT1} that there exists a \hm\, $h_{a,w}: A\to A\otimes {\cal W}$ such that
$(h_{a,w})_T={\rm id}_{T(A)}$ and $h_{a,w}^{\dag}|_{J_c(K_1(A))}={\bar 1}$ and
$h_{a,w}^{\dag}|_{\Aff(T({\tilde A}))/\Z}={\rm id}_{\Aff(T({\tilde A}))/Z}.$

Let ${\cal F}_1=h_{a,w}({\cal F}).$
Choose ${\cal G}_w\in A\otimes {\cal W}$ and $\dt_W>0$ which are required
by \ref{BHK00}
for {{$A\otimes {\cal W}\in {\cal M}_0$,}} ${\cal F}_1$ and $\ep/16.$


Suppose that $\psi: A\to B$ is a \hm\, which maps strictly positive elements to
strictly positive elements and suppose that there is a unitary $v\in {\tilde B}$
such that
\beq
\|[\psi(g),\, v]\|<\dt_W/2\rforal g\in \psi_{w,a}({\cal G}_w)
\eneq
and suppose that $\psi^{\dag}$ maps $J_c(K_1(A))$ to ${\bar 1}.$

Consider  \hm\, $\psi': A\to B$ defined by
$\psi'=\psi\circ \psi_{w,a}\circ h_{a,w}.$ Note that $[\psi']=[\psi]$ in $KL(A, B)$
(since $B\cong B\otimes {\cal W}$)
and $\tau\circ \psi'=\tau\circ \psi$ for all $\tau\in T(B)$  and $\psi^{\dag}=(\psi')^{\dag}$ {{(Note that $(\psi')^{\dag}$ maps $J_c(K_1(A))$ to ${\bar 1}$)}}.
Therefore, by \ref{TUNIq} (and \ref{Rsec53}), there is a unitary  $V\in {\tilde B}$
such that
\beq\label{LWHTW-1}
\|V^*\psi'(g)V-\psi(g)\|<\min\{\dt_W/2, \ep/16\}\rforal g\in \psi_{w,a}({\cal G}_w)\cup {\cal F}.
\eneq

Define  $\psi_W: A\otimes {\cal W}\to B$ by $\psi_W={\rm Ad}\, V\circ \psi\circ \psi_{w,a}.$
Then
\beq
\|[\psi_W(g), \, v]\|<\dt\rforal g\in {\cal G}_W.
\eneq
 Note $A\otimes {\cal W}\in  {\cal M}_0.$ It follows from \ref{BHK00}
that there exists a continuous path of unitaries
$\{v(t): t\in [0,1]\}\subset U({\tilde B})$ with $v(0)=u$ and $u(1)=1_{\tilde B}$ such  that
\beq
\|[\psi_W(g),\, v(t)]\|<\ep/16\rforal  g\in {\cal F}_1.
\eneq
Therefore,
\beq
\|[{\rm Ad}\, V\circ \psi'(f), v(t)]\|<\ep/16\rforal f\in {\cal F}.
\eneq
It follows from this and \eqref{LWHTW-1}
that
\beq
\|[\psi(f),\, v(t)]\|<\ep/8\rforal f\in {\cal F}.
\eneq

Now we consider the general case that $\psi^{\dag}(J_c(K_1(A)))\not={\bar 1}.$
Let
\beq
&&\Delta_A(\hat{a})=\inf \{\tau(a): \tau\in T(A)\}\rforal a\in A_+^{\bf 1}\setminus \{0\}\andeqn\\
&&\hspace{-0.4in}\Delta_0(\hat{c})=\inf\{( \int fdm)\Delta_A(a): c\ge f\otimes a, \,\,\, f\in C(\T),\,\, \, a\in A\}
\eneq
for all $ c\in C(\T, {\tilde A})^o,$
where $m$ is the normalized Haar  measure on $\T.$
Put $\Delta=\Delta_0/2.$

Put $A_c=C(\T, {\tilde A})^o.$ Let ${\cal H}_1\subset ( ({\tilde A_c})_+^{\bf 1}\setminus \{0\}$ be a finite subset,
$\gamma_1>0,$ $\gamma_2>0,$ $\dt_c>0,$
${\cal G}_1\subset {\tilde A_c}$ (in place of ${\cal G}$)
and  ${\cal P}\subset \underline{K}({\tilde{A_c}}),$
${\cal H}_2\subset A_c$ and
${\cal U}\subset J_c(K_1({\tilde A}_c))$  be finite subsets  with  $[{\cal U}]\subset {\cal P}$
 be required by \ref{TUnHT2} for $\min\{\dt_W/4, \ep/16\}$ (in place of $\ep$) and $\psi_{a,w}({\cal G}_W)$ (in place
 of ${\cal F}$) and $\Delta.$
 With smaller $\dt_c>0,$ $\gamma_i,$ \wilog, we may assume
 that ${\cal H}_1=\{g\otimes 1_{\tilde A}: g\in {\cal H}_{1,T}\}\cup \{1_{C(\T)}\otimes a: a\in {\cal H}_{1,A}\},$
 and ${\cal G}_1=\{g\otimes 1_{\tilde A}: g\in {\cal G}_{1,T}\}\cup \{1_{C(\T)}\otimes a: a\in {\cal G}_{1,A}\},$
 ${\cal H}_2=\{g\otimes 1_{{\tilde A}}: g\in {\cal H}_{2,T}\}\cup \{1_{C(\T)}\otimes a: a\in {\cal H}_{2,A}\},$
 where ${\cal H}_{1,T}, {\cal H}_{2,T},$ ${\cal G}_{1,T}\subset C(\T),$ and
 ${\cal H}_{1,A}, {\cal H}_{2,A},$ ${\cal G}_{1,A}\subset {\tilde A}$ are finite subsets.

 Let ${\cal G}'={\cal G}_{1,A}\cup {\cal F}$ and $\dt'=\min\{\dt_c/2, \dt_W/2, \ep/16\}.$
 Let $0<\dt<\dt'$ and ${\cal G}\supset {\cal G}'$ be finite subset such that
 any ${\cal G}$-$\dt$-multiplicative \cpc\, $L'$ from $A$ to a
 \CA\,  $C$ and any unitary $u'\in {\tilde C}$ with property
 $\|[L'(g),\, u']\|<2\dt$ for all $g\in {\cal G}$ gives
 a ${\cal G}_1$-$\dt$-multiplicative \cpc\, from $C(\T, {\tilde A})$ to {{${\tilde C}.$}} 

 Suppose that $\phi: A\to B$ is a \hm\, which maps strictly positive elements to
 strictly positive elements and $u\in {\tilde B}$
 such that
 \beq
 \|[\phi(g),\, u]\|<\dt\rforal g\in {\cal G}.
 \eneq
 Note that $B\otimes Q\otimes Q\cong B.$ We may assume that $\phi(A)\subset B\otimes 1_Q\otimes 1_Q$
 and $u\in {\tilde B}\otimes 1_Q\otimes 1_Q.$
 Let $\{e_n\}$ be an approximate identity for $A.$
Consider $v_n=u(\exp(ie_n\otimes h)),$ where $h\in Q\otimes 1_Q$ with ${\rm sp}(h)=[-\pi, \pi]$ and $t_Q(h)=0$ and where
$t_Q$ is the tracial state of $Q.$
Let $p_k, q_{1,k}, q_{2,k}\in 1_Q\otimes Q$ be  mutually orthogonal projections with $t_{{Q}}(p_k)=1-1/k,$ 
{{$t_Q(q_{i,k})=1/2k,$}}$i=1,2,$ and $p_k\oplus q_{1,k}\oplus q_{2,k}=1_Q\otimes 1_Q,$ $k=1,2,....$
Put $B_k=B\otimes p_k,$ $B_{i,k}=B\otimes q_{i,k},$ $i=1,2,$ $k=1,2,....$
By \ref{11ExtT1},  there are \hm s $
{{\Psi_{i,k}:}} A\to B_{i,k}$ such that
$\tau(\Psi_{i,k}(a))=(1/2k)\tau(\phi(a))$ for all $a\in A$
and
\beq\label{LWHTW-9}
\Psi_{1,k}^{\dag}|_{J_c(K_1(A))}=-(1-{1\over{k}})\phi^{\dag}|_{J_c(K_1(A))}\andeqn
\Psi_{2,k}^{\dag}|_{J_c(K_1(A))}=(1-{1\over{k}})\phi^{\dag}|_{J_c(K_1(A))},
\eneq
 $k=1,2,....$
Define $\psi_{n,k}': A\to C_k:=B_k\oplus B_{1,k}$ by $\psi'_{n,k}(a)=\phi(a)\otimes p_k\oplus \Psi_{1,k}(a)$ for all $a\in A,$
and define $\psi_{n,k}: A\to B\otimes 1_Q\otimes 1_Q$ by
$\psi_{n,k}(a)=\psi_{n,k}'(a)\oplus \Psi_{2,k}(a)$ for all $a\in A,$ $k=1,2,....$
Write $v_n=\lambda+\af(v_n)$ for some $\lambda\in \T$ and $\af(v_n)\in B\otimes 1_Q\otimes 1_Q.$
Let $v_{n,k}=\lambda\cdot 1_{\tilde C_k} +\af(v_n)(p_k\oplus q_{1,k})$ and
$w_{n,k}=\lambda\cdot 1_{\tilde B}+\af(v_n)(p_k\oplus q_{1,k}).$
Choose a \cpc\, $L_{n,k}=\Phi_{w_{n,k},
{{\psi_{n,k}}}}: C(\T,{\tilde A})^o\to B\otimes Q\otimes Q$ induced
by the unitary $w_{n,k}$ and $\psi_{n,k}.$ Let $\Phi_{v_n, \phi}: C(\T, {\tilde A})^o\to B\otimes Q\otimes Q$
be the one  induced by $v_n$ and $\phi.$

Note that $U({\tilde B})/CU({\tilde B})=\Aff(T({\tilde B}))/\Z.$
By applying \ref{densesp},
for all sufficiently large $n$ and $k$ (we then fix a pair $n$ and $k$)
\beq
&&\tau(L_{n,k}(h))\ge \Delta_0(\hat{h})/2= \Delta(\hat{h})\rforal \tau\in T(B)\andeqn \rforal  h\in {\cal H}_1,\\
&&|\tau(L_{n,k}(h))-\tau(\Phi_{v_n,\phi})(h)|<\gamma_1\rforal h\in {\cal H}_2\andeqn\\
&&{\rm dist}(L_{n,k}^{\dag}({\bar w}), \Phi_{v_n, \phi}^{\dag}({\bar w}))<\gamma_2\rforal w\in {\cal U}.
\eneq
It follows from \ref{TUnHT2} that there exists a unitary $U\in {\widetilde{B\otimes Q\otimes Q}}$ such that
\beq\label{LWHTW-20}
&&\|U^*\psi_{n,k}(g)U-\phi(g)\|<\min\{\dt_W/4, \ep/16\}\rforal g\in \psi_{w,a}({\cal  G}_W)\andeqn\\\label{LWHTW-21}
&&\|U^*w_{n,k}U-v_n\|<\min\{\dt_W/4, \ep/16\}.
\eneq

Now consider ${\rm Ad}\, U\circ \psi_{n,k}': A\to D_k:=U^*C_kU$ and the unitary $U^*v_{n,k}U\in {\tilde D}_k.$
Note, by \eqref{LWHTW-9},  $({\rm Ad}\, U\circ \psi_{n,k}')^{\dag}|_{J_c(K_1(A))}={\bar 1}.$
So we reduce this case to the case that has been proved.
 Thus there is a continuous path of unitaries $\{V(t): t\in [2/3, 1]\}\subset {\tilde D}_k$
 such that $V(2/3)=U^*v_{n,k}U$ and $V(1)=1_{\tilde D_k}$ and
\beq
\|[{\rm Ad}\, U\circ \psi_{n,k}'(f),\, V(t)]\|<\ep/8\tforal f\in {\cal F}.
\eneq
Note that  $U^*w_{n,k}U=\lambda\circ 1_{\tilde B}+U^*\af(v_{n,k})U.$
Write $V(t)=\lambda(t)\cdot 1_{\tilde D_k}+\af(V(t))$ for some $\lambda(t)\in \T$ and $\af(V(t))\in D_k.$
Put $Z(t)=\lambda(t)\cdot 1_{\tilde B}+\af(V(t)).$ Then $Z(2/3)=U^*w_{n,k}U$ and $Z(1)=1_{\tilde B}.$
Since $B_{2,k}\perp C_k,$  we have that
\beq
\|[{\rm Ad}\, U\circ \psi_{n,k}(g),\, Z(t)]\|<\ep/8\rforal f\in {\cal F}.
\eneq
By \eqref{LWHTW-21}, we may write  $v_n^*U^*w_{n,k}U=\exp(i b)$ for some $b\in {\tilde B}_{s.a.}$ with $\|b\|\le 2\arcsin(\ep/32).$
Define $Z(t)=v_n\exp({{\sqrt{-1}}}(3(t-1/3) b)$ for $t\in [1/3, 2/3).$  Then $Z(1/3)=v_n.$
We also have
\beq
\|{\rm Ad}\, U\circ \psi_{n,k}(g),\, Z(t)]\|<\ep/8\tforal t\in [1/3, 1].
\eneq
It follows that
\beq
\|[\phi(g),\, Z(t)]\|<\ep/8+\ep/16\rforal t\in [1/3, 1].
\eneq
Define $Z(t)=u(\exp(3{{\sqrt{-1}}}te_n\otimes h))$ for $t\in [0, 1/3).$ Then
$Z(0)=u$ and $\{Z(t): t\in [0,1]\}$ is a continuous path of unitaries in ${\tilde B}.$
Moreover,
\beq
\|[\phi(g),\, Z(t)]\|<\ep\rforal g\in {\cal F}\andeqn t\in [0,1].
\eneq

\end{proof}

\begin{thm}\label{MTHT}
Let $A\in {\cal B}_T$ have continuous scale.
Let ${\cal P}\subset \underline{K}(A)$ be a finite subset, let
$\{p_1,p_2,...,p_k, q_1,q_2,...,q_k\}$ be  projections of $M_s({\tilde A})$
(for some integer $s\ge 1$) such that
$\{[p_1]-[q_1],[p_2]-[q_2],...,[p_k]-[q_k]\}\subset {\cal P}$ generates a free subgroup $G_{u0}$ of $K_0(A),$
let
$\sigma>0,$  $\ep_0>0$ and  ${\cal F}_0\subset A$ be a finite subset.
There exist $\dt_0>0$ and
${\cal G}_0\subset A$  such that the following hold:
For
any $\ep>0,$
any finite subset ${\cal F}\subset A,$
 any \hm\,  $\phi: A\to B=B_1\otimes Q$
 which maps strictly positive elements to strictly positive elements,
 where $B_1\cong B_1\otimes \zo\in {\cal D}_0$ has continuous scale,
and any
  unitary $u\in U({\tilde B})$ such that
\beq
\|[\phi(g),\, u]\|<\dt_0 \rforal g\in {\cal G}_0,
\eneq
there exists  a continuous path
of unitaries $\{v(t): t\in [0,1]\}\subset U({\tilde B})$ such that
\beq\label{MTHT-2}
&&\|[\phi(g), \, v(0)]\|<\ep \tforal g\in {\cal G}_0\cup {\cal F},\\\label{MTHT-2+}
&&\|{[}\phi(f), \, v(t)]\|<\ep_0\tforal f\in {\cal F}_0,
 \\\label{MTHT-3-2020}
&&{\rm Bott}(\phi,\, uv(1))|_{\cal P}=0,\,\,\, [uv(1)]=0\andeqn\\\label{MTHT-3+}
&&\hspace{-0.2in}{\rm dist}(\overline{\lceil ((1_s-\phi(p_i))+(uv(1))_{\underline{s}}\phi(p_i))(1_s-\phi(q_i))+(uv(1))_{\underline{s}}^*\phi(q_i))\rceil}, {\bar 1})<\sigma,
\eneq
where $1_s=1_{M_s}$ and $(uv(1))_{\underline{s}}=uv(1)\otimes 1_{M_s}.$
\end{thm}

\begin{proof}
Define
$\Delta_1(\hat{h})=\inf\{\tau(h): \tau\in T(A)\}$ for $h\in A_+^{\bf 1}\setminus \{0\}.$
Let $\Delta=\Delta_1/2.$
Let $T: A_+^{\bf 1}\setminus \{0\}\to \R_+\setminus \{0\}\times \N$ be the map given by $\Delta$ as in
\ref{FfullDelta}.
Let $\ep_0,$ $\sigma,$  ${\cal F}_0,$ ${\cal P}$
and $\{p_1,...,p_k, q_1, q_2,...,q_k\}\subset M_s({\tilde A})$ be given.
In what follows, if $v'$ is a unitary, $v'_{\underline{s}}=v'\otimes 1_{M_s}.$


Write $p_l=(a_{i,j}^{p_l})_{s\times s}$ and $q_l=(a_{i,j}^{q_l})_{s\times s},$ where
$a_{i,j}^{p_l}, a_{i,j}^{q_l}\in {\tilde A},$ $1\le i,j\le s,$ $1\le l\le k.$
Let ${\cal F}_p$ be a finite subset  in $A$ such that
$a_{i,j}^{p_l}, a_{i,j}^{q_l}\in \C\cdot 1+{\cal F}_p.$

In what follows, if $L': A\to C'$ is a map, we will
continue to use $L'$ for $L'^{\sim}: {\tilde A}\to {\tilde C'}$ and
$L'\otimes {\rm id}_{M_s}$ as well as $L^{'\sim}\otimes {\rm id}_{M_s}$ when
it is convenient.  Moreover, $1_s:=1_{M_s}.$

Let
$\dt_0'>0$ and let ${\cal G}_0'\subset A$ be a finite subset satisfying the following:
${\rm Bott}(L,\, w)|_{\cal P}$ is well defined for
any ${\cal G}_0'$-$\dt_0'$-multiplicative \cpc\, $L: A\to C$ and
any unitary $w\in {\tilde C}$ with $\|[L(g),\, w]\|<2\dt_0'$ for all $g\in {\cal G}_0.$
Also, if $w'$ is another unitary, we also require that
\beq
{\rm Bott}(L, ww')|_{\cal P}={\rm Bott}(L, w)|_{\cal P}+{\rm Bott}(\phi, w')|_{\cal P},
\eneq
when $\|[L(g),\, w']\|<\dt_0'$ for all $g\in {\cal G}_0'.$
Moreover, for any ${\cal G}_0'$-$\dt_0'$-multiplicative \cpc\, $L'$ from
$A$ to a non-unital \CA\, $C'$  induces a \hm\, $\lambda': G_u\to U({\tilde C})/CU({\tilde C})$
(see 14.5 of \cite{Lnloc}).  Furthermore, using 14.5 of \cite{Lnloc} again,
we assume that, for any unitary $w'\in M_s({\tilde C})$ with the property
that $\|[L'(g),\,w']\|<2\dt_0'$ for all $g\in {\cal G}_0',$  $\Phi_{w', L'}$ induces a \hm\, $\lambda_{L', w'}$ from
$G_{u0}$ to $U({\tilde C})/CU({\tilde C})$ and, for $1\le i\le k,$
\beq
\hspace{-0.2in} {\rm dist}(\overline{\lceil (1_s-L'(p_i)+w_{\underline{s}}'L'(p_i))(1_s-L'(q_i)+(w_{\underline{s}}')^*L'(q_i))\rceil}, \lambda_{L', w'}([p_i]-[q_i]))<\sigma/16,
\eneq
where  $w_{\underline{s}}'=w'\otimes 1_s.$
We may assume that $\dt_0'$ is smaller than $\dt_0$ in \ref{LHTkill} and ${\cal G}_0'$ is larger than
${\cal G}_0$  in \ref{LHTkill} for  the above ${\cal P}.$

Let $\dt_W>0$ and let ${\cal G}_W\subset A$ be a finite subset
required by \ref{LWHTW} for $\min\{\ep_0/4, \dt_0'/2\}$ (in place of $\ep$) and ${\cal F}_0\cup {\cal G}_0'.$
Put $\dt_0''=\min\{\dt_0'/4, d_W/4\}$ and ${\cal G}_0''={\cal G}_0'\cup {\cal G}_W\cup {\cal F}_0\cup {\cal F}_p.$


Let $\ep_1=\min\{\dt_0''/4, \ep_0/16, \sigma/16\}/2^{10}(s+1)^2.$
Let $\dt_1>0$ (in place of $\dt$),  $\gamma>0,$ $\eta>0,$ ${\cal G}_1\subset A$ (in place of ${\cal G}$) be a finite
subset, ${\cal P}_1\subset \underline{K}(A)$ (in place ${\cal P}$) be a finite subset,
${\cal U}\subset U({\tilde A})$ be a finite subset,
${\cal H}_1\subset A_+\setminus \{0\}$ be a finite subset, and ${\cal H}_2\subset A_{s.a.}$ be a finite
subset required by \ref{TUNIq} for $\ep_1$ (in place of $\ep$) and ${\cal G}_0''$ (in place of ${\cal F}$) the above $T$
(and ${\bf T}(n,k)=n$).

We assume that $\dt_1<\dt_0''$ and
that ${\cal G}_1\cup {\cal H}_1\cup {\cal H}_2\subset (A)_+^{\bf 1}.$
Write $w=\lambda_w+\af(w),$ where $\lambda_w\in \T\subset \C$ and $\af(w)\in A.$
As in the remark of  \ref{TUNIq}, we may assume that $[w]\not=0$ and $[w]\in {\cal P}$ for all $w\in {\cal U}.$
 Let $G_u$ be the subgroup generated by
$\{\overline{w}: w\in {\cal U}\}.$ We may view $G_u\subset J_c(K_1(A))$ (see the statement of \ref{UniqN1}).

Note that  $B\cong B\otimes \zo.$  Define $\psi_{b,W}: B\otimes \zo\to
B\otimes {\cal W}$ by letting $\psi_{b,W}(b\otimes a)=b\otimes 
{{\phi_{z,w}}}(a)
$
for all $b\in B$ and $a\in \zo,$  where $
{{\phi_{z,w}}}: \zo\to {\cal W}$
is a \hm\,  defined in \ref{DWZmaps}.  Note that, by 6.8 of \cite{eglnkk0},
$B\otimes {\cal W}$ is in ${\cal M}_0$ with continuous scale.

Set ${\cal G}_2={\cal G}_0\cup {\cal G}_1\cup \{\af(w): w\in {\cal U}\}$
and set
 a rational number
$$
0<\sigma_0<\min\{\inf\{\Delta(\hat{h}): h\in {\cal H}_1\}, \gamma/4\}.
$$
{{Let $\{e_n\}$ be an approximate identity for $A$ such that $e_{n+1}e_n=e_n$ and $(e_nAe_n)^\perp\not=\{0\}$ for all $n.$
Define $\psi_n(a)=\phi(e_nae_n)$ for all $a\in A.$ Then $\lim_{n\to\infty}\|\psi_n(a)-\phi(a)\|=0$ for all $a\in  A.$
Choose a sufficiently large $n$ such that $\psi_n|_{\cal P}=\phi|_{\cal P}.$ Therefore,}}
\wilog,  we may assume that there are $e_A, e_A'\in A_+$ with $\|e_A'\|=\|e_A\|=1$
such that, {{with $\psi_A(a)=\phi(e_Aae_A)$ }}
\beq\label{MTHT-3-0}
\hspace{-0.2in}e_Ag=ge_A=g\rforal g\in {\cal G}_2,\,\,\, e_A'e_A=e_A,  ({\overline{e_A'Ae_A'}})^{\perp}\not=\{0\}{{\andeqn \phi|_{\cal P}=\psi_A|_{\cal P}.}}
\eneq
Choose  a pair of mutually orthogonal non-zero  positive elements $e_0, e_0'\in ({\overline{e_A'Ae_A'}})^{\perp}$ such that
\beq\label{MTHT-3-1}
d_\tau(e_0+e_0')<\sigma_0\rforal \tau\in T(A).
\eneq
Choose an integer $K\ge 1$ such that \beq\label{01-14-2020}
{{1/K<\min\{\sigma_0/4, \inf\{d_\tau(e_0): \tau\in T(A)\}\}}}\eneq and
 choose $\dt_0=\min\{\ep_1/16, \dt_1/16, \gamma/16, \eta/16\}/64(s+1)^3(K+1)^2.$
 Put ${\cal G}_0={\cal G}_2\cup\{e_A, e_A', e_0, e_0'\}.$

 Now let $\phi$ and $u$ be given satisfying the assumption
for the above ${\cal G}_0$ and $\dt_0.$ Let $\ep>0$ and ${\cal F}\subset A$ be a finite subset.
We may write $u=1_{\tilde B}+\af(u),$ where $\af(u)\in B.$ Put ${\cal Q}={\cal P}\cup {\boldsymbol{\bt}}({\cal P}).$



Note also ${\cal W}\otimes Q\cong {\cal W}.$  Let $e_q\in Q$ be a projection
with $t_U(e_q)=1/K,$ where $t_Q$ is the tracial state of $Q.$
Define  $\psi_{1/K,W}: {\cal W}\to {\cal W}\otimes Q$  by $\psi_{1/K,W}(a)=a\otimes e_q$ for all $a\in {\cal W}.$
Then
\beq\label{MTHT-3}
t_{\cal W}(\psi_{1/K,W}(a))=(1/K)t_{\cal W}(a)\rforal a\in {\cal W}.
\eneq
Let $
{{\phi_{w,z}}}$ be as in \ref{DWZmaps}.
Note that $t_{\cal W}=t_Z\circ 
{{\phi_{w,z}}}$ and $t_Z=t_{\cal W}\circ 
{{\phi_{z,w}}},$
where $t_{\cal W}$ and $t_Z$ are tracial states of ${\cal W}$ and $\zo,$ respectively.
Let $\psi_{b, 1/K}: B\to B$ be defined by
$\psi_{b, 1/K}(b\otimes a)=b\otimes 
{{\phi_{w,z}}}\circ \psi_{1/K,W}\circ 
{{\phi_{z,w}}}(a)$
for all $b\in B$ and $a\in \zo.$   Let $\psi_{b,w, 1/K}:
B\to B\otimes {\cal W}\otimes e_q$ be defined by
$\psi_{b,w,1/K}(b\otimes a)=b\otimes \psi_{1/K, W}\circ 
{{\phi_{z,w}}}(a)$ for all
$b\in B$ and $a\in \zo.$

By applying \ref{LHTkill}, there is a unitary $v_1\in {\tilde B}$
such that
\beq\label{MTHT-4}
&&\|[\phi(g),\, v_1]\|<\min\{\dt_0, \ep\}\rforal g\in {\cal F}\cup  {\cal G}_0\andeqn\\\label{MTHT-4+1}
&&{\rm Bott}(\phi,\, uv_1)|_{\cal P}=0\andeqn [uv_1]=0.
\eneq
\vspace{-0.1in}Note that
\beq\label{MTHT-5}
\|[\phi(g),\, uv_1]\|<\dt_0+\min\{\dt_0, \ep\}\rforal g\in {\cal G}_0.
\eneq
We may write $uv_1=1_{\tilde B}+\af(uv_1)$ for some $\af(uv_1)\in B.$
Define $\psi': A\to B$ by $\psi'(a)=\psi_{b, 1/K}\circ \phi(a)$ for all $a\in A.$ Using {{\eqref{01-14-2020},}} 
by replacing $\psi'$ by ${\rm Ad}\, w_1\circ \psi'$ for some unitary $w_1,$ we may assume
that $\psi'(A)\subset  B_0:=\overline{e_{0,b}Be_{0,b}},$ where
$e_{0,b}=\phi(e_0).$ Let $v_2'=1_{\tilde B}+\psi_{b, 1/K}(\af(uv_1)),$
$v_2=((v_2')^*)^K$ and $v_2''=1_{\tilde B_0}+\psi_{b, 1/K}(\af(uv_1)).$
Note that $[\psi']|_{\cal P}=0,$ since it factors through $B\otimes {\cal W}.$
Moreover
\beq\label{MTHT-5+1}
{\rm Bott}(\psi',\, v_2'')|_{\cal P}=0\andeqn {\rm Bott}(\psi', (v_2'')^K)|_{\cal P}=0.
\eneq
Let $\lambda_{\phi,uv_1}: G_{u0}\to U(M_s({\tilde B}))/CU(M_s({\tilde B}))$ be the \hm\,
induced by $\phi$ and $uv_1,$ via a map $\Phi_{uv_1, \phi}.$  Then \eqref{MTHT-4+1} implies
that $\lambda_{\phi, uv_1}$ maps $G_{u0}$ to $\Aff(T({\tilde B}))/\Z$ (see also \cite{GLX}).
Let $\lambda_{\psi', v_2'}: G_{u0}\to \Aff(T({\tilde B})/\Z$ be the \hm\,
induced by $\Phi_{v_2', \psi'}.$  Since $\tau\circ \psi_{b,1/K}(b)=(1/K)\tau(b)$ for all $b\in B$
and for all $\tau\in T(B),$
it is straightforward that we may write
\beq\label{MTHT-6}
\lambda_{\psi', v_2'}([p_i]-[q_i])=
(1/K)\lambda_{\phi, uv_1}([p_i]-[q_i]),
\eneq
$i=1,2,...,k.$
It follows that, by the choice of $\dt_1$ and $\dt_2,$  since $v_2=((v_2')^*)^K,$
\beq\label{MTHT-7}
{\rm dist}(\zeta_i',\,-(\lambda_{\phi, uv_1}([p_i]-[q_i])))<\sigma/16,
\eneq
where $\zeta_i'=\overline{\lceil ((1_s-\psi'(p_i))+(v_2)_{\underline{s}}\psi'(p_i))((1-\psi'(q_i)+(v_2^*)_{\underline{s}}\psi'(q_i)))\rceil},$ $i=1,2,...,k.$
As in the proof of \ref{LHTkill}, by applying \ref{11ExtT1}, we obtain a
\hm, $\psi_{cu}: A\to \overline{e_{b,0}'Be_{b,0}'},$ where $e_{b,0}'=\phi(e_0'),$ such
that
\beq\label{MTHT-7+1}
[\psi_{cu}]=0\,\,\,{\rm in}\,\,\, KL(A, B)\andeqn \psi_{cu}^{\dag}=-(\psi')^{\dag}.
\eneq

Define
$\psi: A\to B$ by
$\psi(a)=\psi_{cu}(a)\oplus \psi'(a)\oplus \phi(e_Aae_A)$ for all $a\in A.$
Then $\psi$ is  ${\cal G}_2$-$2\dt_2$-multiplicative {{(see  the last part of \eqref{MTHT-3-0})}},
\beq\label{MTHT-25}
&&\tau\circ \psi(h)\ge \Delta(\hat{h})\rforal h\in {\cal H}_1,\\
&&|\tau\circ \psi(h)-\tau\circ \phi(h)|<\gamma\rforal h\in {\cal H}_2,\\
&&{[}\psi{]}|_{\cal P}=[\phi]|_{\cal P}\andeqn\\
&&\psi^{\dag}({\bar w})=-(\psi')^{\dag}({\bar w})+((\psi')^{\dag}({\bar w})+\phi^{\dag}(\bar{w}))=\phi^{\dag}(\bar w)\rforal w\in {\cal U}.
\eneq
 By \eqref{MTHT-25}, $\psi$ is
$T$-${\cal H}_1$-full.
By applying \ref{TUNIq} (as $K_0({\tilde B})$ is weakly unperforated), we obtain a unitary $U\in {\tilde B}$ such that
\beq\label{LHTkill-30}
\|U^*\psi(f)U-\phi(f)\|<\ep_1\rforal f\in  {\cal G}_0'.
\eneq
Let $v= v_1U^*(v_2)U.$ Then $v$ is a unitary.
We have
\beq
\|[\phi(f), v]\|<2\ep_1+(K+1)\dt_0\rforal f\in {\cal G}_0'.
\eneq
We then compute that, by \eqref{MTHT-4+1}, \eqref{LHTkill-30} and \eqref{MTHT-5+1},  and by  the fact
that $\phi(e_A)v_2=v_2\phi(e_A)=\phi(e_A),$
\beq
{\rm Bott}(\phi, uv)|_{\cal P}&=&{\rm Bott}(\phi, uv_1)|_{\cal P}+{\rm Bott}(\phi, U^*v_2U)|_{\cal P}\\
&=& 0+{\rm Bott}(\psi, v_2)|_{\cal P}\\\label{MTHT-18n1}
&=&0+ {\rm Bott}(\phi(e_A\cdot e_A), 1)+{\rm Bott}(\psi', v_2)|_{\cal P}=0.
\eneq
Put $\Psi={\rm Ad}\, U\circ \psi,$  $\psi''={\rm Ad}\, U\circ \psi'$ and $u_2=U^*v_2U.$
Put $\ep_s=s^2\ep_1$ and $\dt_2=(K+1)\dt_0.$
We have (recall $w'_{\underline{s}}=w'\otimes 1_s$)
\beq
&&\hspace{-0.5in}
(1_s-\phi(p_i){{)}}+(uv)_{\underline{s}}\phi(p_i)\\
&&\hspace{-0.2in}=
(1_s-\phi(p_i){{)}}+(uv_1u_2)_{\underline{s}}\phi(p_i)\\
&&\hspace{-0.2in}\approx_{\ep_s} (1_s-\phi(p_i){{)}}+(uv_1)_{\underline{s}}(u_2)_{\underline{s}}\Psi(p_i)\hspace{0.7in} ({\rm using}\,\, \eqref{LHTkill-30})\\
&&\hspace{-0.2in}\approx_{2s^2\dt_2} (1_s-\phi(p_i))+(uv_1)_{\underline{s}}\Psi(p_i)(u_2)_{\underline{s}}\Psi(p_i)\hspace{0.5in} \\
&&\hspace{-0.2in}\approx_{2\ep_s} {{(1_s-\phi(p_i))(1_s-\Psi(p_i)) +(uv_1)_{\underline{s}}\phi(p_i)\Psi(p_i)(u_2)_{\underline{s}}\Psi(p_i)}}\\\label{1802n}
&&\hspace{-0.2in}\approx_{2\ep_s}((1_s-\phi(p_i){{)}}+(uv_1)_{\underline{s}}\phi(p_i))((1_s-\Psi(p_i))+(u_2)_{\underline{s}}\Psi(p_i)).
\eneq
Similarly,
\beq\label{1802n2}
\hspace{-0.3in}(1_s-\phi(q_i){{)}}+(uv)_{\underline{s}}\phi(q_i)
\approx_{6\ep_s}((1_s-\phi(q_i){{)}}+(uv_1)_{\underline{s}}
\phi(q_i)){{(}}(1_s-\Psi(q_i))+
(u_2)_{\underline{s}}\Psi(q_i)).
\eneq
Put
$$
Z_i=\lceil((1_s-\Psi(p_i))+(u_2)_{\underline{s}}\Psi(p_i))((1_s-\Psi(q_i))+(u_2)^*_{\underline{s}}\Psi(q_i))\rceil.
$$
Then, since we have assumed that $\psi'(A)\subset \overline{e_{0,b}Be_{0,b}},$ one computes, by \eqref{MTHT-3-0},  that
\beq
\overline{Z_i}=\zeta_i',\,\,\,i=1,2,...,k.
\eneq
Then, in $U(M_s({\tilde B})/CU(M_s({\tilde B})),$   for $i=1,2,...,k,$  by \eqref{1802n} and \eqref{1802n2},

\beq
&&\hspace{-0.3in}\overline{\lceil ((1_s-\phi(p_i){{)}}+(uv)_{\underline{s}}\phi(p_i))
({{(1_s-\phi(q_i))}}+(uv)_{\underline{s}}^*\phi(q_i))\rceil}\\
&&\approx_{12\ep_s} \overline{\lceil ((1_s-\phi(p_i)+(uv_1)_{\underline{s}}\phi(p_i))Z_i(
{{(1_s-\phi(q_i))}}+(uv_1)^*_{\underline{s}}\phi(q_i){{)}}\rceil}\\
&&= \overline{\lceil ((1_s-\phi(p_i)+(uv_1)_{\underline{s}}\phi(p_i))
(
(1_s-\phi(q_i){{)}}+(uv_1)^*_{\underline{s}}\phi(q_i))\rceil}\overline{Z_i}\\\label{1802n3}
&&\approx
\lambda_{\phi, uv_1}([p_i]-[q_i]))\overline{Z_i}\approx_{\sigma/16}{\bar 1}.\,\,\,\hspace{0.6in} {\rm (see \eqref{MTHT-7})}
\eneq

Now back to $\psi'.$ Let $\phi_{00}: A\to B_W:=B\otimes {\cal W}\otimes e_q$ be defined by
$\phi_{00}=\psi_{b,w,1/K}\circ \phi.$
Then
\beq
\|[\phi_{00}(g),\, ((v_2'')^*)^K)]\|<2K\dt_0<\dt_1/2\rforal g\in {\cal G}_0.
\eneq
By the choice of $\dt_W$ and ${\cal G}_W$ and by applying \ref{LWHTW}, there exists
a continuous path of unitaries $\{V(t): t\in [0,1]\}$ in ${\widetilde{B\otimes {\cal W}\otimes e_q}}$
such that $V(0)=1_{\tilde B_W},$ $V(1)=(v_2''^*)^K$ and
\beq
\|[\phi_{00}(g),\, V(t)]\|<\min\{\ep_0/4, \dt_0'/2\}\rforal  g\in {\cal F}\cup {\cal G}_0'.
\eneq
Write $V(t)=\lambda(t)\cdot 1_{{\tilde B_W}}+\af(V(t))$ for some $\lambda(t)\in \T$ and $\af(V(T))\in B_W.$
Put
\beq
v(t)=v_1 U^*(\lambda(t)\cdot 1_{\tilde B}+\af(V(t)))U \rforal t\in [0,1].
\eneq
Then we have
\beq
\|[\phi(f),\, v(t))]\|<\min\{\ep_0,\dt_0''\}\rforal f\in {\cal F}_0.
\eneq
Note that $v(0)=v_1$ and $v(1)=v.$  So, \eqref{MTHT-2+} holds.
Also, by \eqref{MTHT-4},  \eqref{MTHT-2} holds and, by \eqref{MTHT-18n1}, {{\eqref{MTHT-3-2020}}} 
 holds.
Moreover, by the choice of $\ep_1$ and by \eqref{1802n3},  \eqref{MTHT-3+} also holds.

\end{proof}

\begin{cor}\label{Chomotopy}
Let $A\in {\cal B}_T$  have  continuous scale.
For any $1>\ep_0>0$ and any finite subset ${\cal F}_0\subset A,$ there exist  $\dt>0$ and
a finite subset ${\cal G}\subset A$  satisfying the following:

For any $\ep>0$ and any finite subset ${\cal F}\subset A$ and any
\hm\,  $\phi: A\to B\otimes Q$  which maps   strictly
positive elements to  strictly positive elements, where $B\cong B\otimes \zo\in {\cal D}_0$ has continuous scale.
If $u\in U({\widetilde{B\otimes Q}})$  is a unitary such that
\beq\label{BHfull-100}
&&\|[\phi(x),\, u]\|<\dt\tforal x\in {\cal G},
\eneq
there exists a unitary $v\in {\widetilde{B\otimes Q}}$ such that
\beq
\|[\phi(f),\, v]\|<\ep\tforal f\in {\cal F},
\eneq
and there exists a continuous path of unitaries $\{u(t): t\in [0,1]\}\subset U_0({\widetilde{B\otimes Q}})$ such
that
\beq\label{BHTL-3}
&&u(0)=uv,\,\,\, u(1)=1\\
&&\|[\phi(a),\, u(t)]\|<\ep_0\tforal a\in {\cal F}_0\tand for\,\, all\,\, t\in [0,1].
\eneq

\end{cor}

\begin{proof}
This is a combination of  \ref{MTHT} and \ref{BHfull}.
Let $\ep_0>0$ and ${\cal F}_0$ be given.
Let $\dt_1>0,$ $\sigma>0$,
${\cal G}_1\subset A$ be a finite subset, let  $\{p_1,p_2,...,p_k, q_1,q_2,...,q_k\}$ be  projections of $M_N({\tilde A})$
(for some integer $N\ge 1$) such that
$\{[p_1]-[q_1],[p_2]-[q_2],...,[p_k]-[q_k]\}$ generates a free subgroup $G_u$ of $K_0(A),$
and ${\cal P}\subset \underline{K}(A)$ be finite subset required by \ref{BHfull}.

Let $\dt_0>0$ and ${\cal G}_0$ be required by \ref{MTHT} for
$\min\{\dt_1, \ep_0\}$ (in place of $\ep_0$), $\sigma$ and ${\cal G}_1\cup {\cal F}_0$ (in place of ${\cal F}_0$)
and ${\cal P}$ and $G_u.$

Now suppose that $\phi$ and $u$ satisfy the assumption for this pair of $\dt_0$ and ${\cal G}_0.$
Let $\ep>0$ and ${\cal F}\subset A$ be given.
Then, by applying \ref{MTHT}, there is a unitary $v\in {\tilde B_1}=B\otimes Q$
and a continuous path of unitaries $\{v(t): t\in [0, 1/2]\})\}\subset {\tilde B_1}$  such that
$v(0)=v,$
\beq
&&\hspace{-0.2in}\|[\phi(f), \, v]\|<\ep\rforal f\in {\cal F},\\
&&\hspace{-0.2in}\|[\phi(g),\, v(t)]\|<\ep_0\rforal g\in {\cal F}_0\\
&&\hspace{-0.2in}{\rm Bott}(\phi,\, uv(1/2))|_{\cal P}=\{0\},\,\, [uv(1/2)]=0\andeqn\\
&&\hspace{-0.4in}{\rm dist}(\overline{\lceil ((1_s-\phi(p_i))+(uv(1/2))_{\underline{s}}\phi(p_i))(1_s-\phi(q_i))+(uv(1/2))_{\underline{s}}^*\phi(q_i))\rceil}, {\bar 1})<\sigma,
\eneq
where $1_s=1_{M_s}$ and $(uv(1/2))_{\underline{s}}=uv(1/2)\otimes 1_{M_s}.$
Note,  since $B$ is non-unital, it is easy to see that we may assume, \wilog, that  everything mentioned above
lie in $M_N({\tilde B}_0),$ where $B_0$ is a hereditary \SCA\, of $B$ so that $B_0^{\perp}\not=\{0\}.$
By the proof of \ref{BHK00}, therefore
one may assume $uv(1/2)\in CU({\tilde B}).$
It follows from \ref{BHfull} that there is a continuous path of unitaries $\{u(t): t\in [1/2, 1]\}\subset {\tilde B_1}$
such that
$u(1/2)=uv(1/2),$ $u(1)=1_{\tilde B_1}$ and
\beq
\|[\phi(f),\, u(t)]\|<\ep_0\rforal f\in {\cal F}_0\rforal t\in [1/2,1].
\eneq
Finally, define $u(t)=uv(t)$ for $t\in [0,1/2].$
\end{proof}

\section{Finite nuclear dimension}

The following proposition follows from the definition immediately.

\begin{prop}\label{Pdtad}
Let $A$ be a  non-unital separable amenable simple \CA.
Then $A$ has tracially approximate divisible property in the sense of
{{10.1 of \cite{eglnp1}}} if and only if the
following holds:

For any $\ep>0,$ any finite subset ${\cal F}\subset A,$ any  integer $n\ge 1$ and any non-zero
elements $a_0\in A_+\setminus \{0\},$
there are  mutually orthogonal positive elements
$e_i,$  $i=0, 1,2,...,n,$ elements $w_i,$ $i=1,2,...,n,$  such that
$w_i^*w_i=e_1^2,$ $w_iw_i^*=e_i^2,$
$i=1,2,...,n,$ $e_0\lesssim a_0$ and
\beq
\|x-\sum_{i=0}^ne_ixe_i\|<\ep\tand \|w_ix-xw_i\|<\ep,\,\,\, 1\le i\le n, \tforal x\in {\cal F}.
\eneq
\end{prop}

\begin{thm}\label{TD0=fn}
Let $A$ be a non-unital separable simple \CA\, with continuous scale and with finite nuclear dimension which  satisfies the UCT.
Suppose that $T(A)\not=\emptyset$ and every tracial state of $A$ is a ${\cal W}$-trace.
Then $A\in {\cal D}_0$.
\end{thm}

\begin{proof}
Since every tracial state of $A$ is a ${\cal {\cal W}}$-trace, by
{{12.3 of \cite{eglnp1} (see 18.3 of \cite{GLp1}),}}
$K_0(A)={\rm ker}\rho_A.$
Suppose that $A$ is tracially approximately divisible.
 Then, since we assume that every tracial state  of $A$  is a ${\cal W}$ trace, by \ref{Pdivisiblehere} of this
 paper and {{6.5}} of \cite{eglnkk0}
and the proof of 18.6 of \cite{GLp1}),
$A\in {\cal D}_0.$
In particular, $A\otimes \zo$ and $A\otimes U$ are in ${\cal D}_0$ for every
UHF-algebra $U$ of infinite type.
Therefore it suffices to show that $A$  is tracially approximately divisible.

It follows from \cite{aTz} that $A\cong A\otimes {\cal Z}.$
Put $B=A\otimes \zo,$ $B_q=B\otimes Q$ and $A_q=A\otimes Q.$ Pick a pair of relatively prime supernatural numbers $\p$ and $\q$ such that $M_{\p}\otimes M_{\q}=Q.$  Let
\beq
{\cal Z}_{\p, \q}=\{f\in C([0,1], Q): f(0)\in M_\p\andeqn f(1)\in M_\q\}\andeqn\\
D\otimes {\cal Z}_{\p,\q}=\{f\in C([0,1], D\otimes Q): f(0)\in D\otimes M_\p\andeqn f(1)\in D\otimes M_\q\}
\eneq
for  any \CA\, $D.$ Note, by \cite{RW}, ${\cal Z}$ is a stationary inductive limit of ${\cal Z}_{\p, \q}$
with trace cllapsing connecting map.

{{Let $\ep>0,$  let ${\cal F}\subset A\otimes {\cal Z},$
let $a_0\in (A\otimes {\cal Z})_+\setminus \{0\},$  and let $n\ge 1$ be an integer.}} Put $\eta=\inf\{d_\tau(a_0): \tau\in T(A\otimes {\cal Z})\}.$
Since $A$ is assumed to have continuous scale, one may find a positive element
$f_e\in A\otimes {\cal Z}$  with $\|f_e\|=1$ such that
\beq\label{LFNZ-1}
\tau(f_e)>1-\eta/16(n+1)^3\rforal \tau\in T(A\otimes {\cal Z}).
\eneq
We assume that $f_e\in {\cal F}.$
\Wlog, we may also assume that ${\cal F}\subset A\otimes {\cal Z}_{\p,\q}.$
We may further  assume, \wilog, that
there is $0<1/2<d_0<1$ such that
\beq\label{1910-152-n1}
f(t)=f(1)\rforal t>d_0
\eneq
 and for all $f\in {\cal F}.$
Note $\Aff(T(B))=\Aff(T(A))$ and $U({\tilde B})/CU({\tilde B})=U({\tilde A})/CU({\tilde A}).$
There is a $KK$-equivalence $\kappa\in KL(B,A)$ which is compatible
to the identifications {\ {$\kappa_T: \Aff(T(B))\to \Aff(T(A))$ and $\kappa_{cu}: U({\tilde B})/CU({\tilde B})\to U({\tilde A})/CU({\tilde A})$ above.}}
We will consider the triple
$(\kappa, \kappa_T, \kappa_{cu}).$
Let $\phi_\p: B\otimes M_\p \to A\otimes M_\p$  and
$\phi_\q: B\otimes M_\q\to   A\otimes M_\q$  be  isomorphisms given by \ref{Misothm}
and induced by $(\kappa\otimes [{\rm id}_{M_\p}], \kappa_T, \kappa_{cu}\otimes ({\rm id}_{M_\p})_{cu}),$
and by $(\kappa\otimes [{\rm id}_{M_\q}], \kappa_T, \kappa_{cu}\otimes ({\rm id}_{M_\q})_{cu}).$
Let $\psi_\p: B\otimes M_\p\otimes M_\q=B\otimes Q \to A\otimes M_\p\otimes M_\q=A\otimes Q$
given by $\psi_\p=\phi_\p\otimes {\rm id}_{M_\q}$ and
let $\psi_\q=\phi_\q\otimes {\rm id}_{M_\p}: B\otimes Q\to A\otimes Q.$
Then
\beq
([\psi_\q], (\psi_\q)_T, \psi_\q^{\dag})=([\psi_\p], (\psi_\p)_T, \psi^{\dag}_\p).
\eneq

Let ${\cal F}_1=\{f(1): f\in {\cal F}\}$ in  $A\otimes M_\p\otimes M_\q.$
Let ${\cal G}_{1,b}=\{\psi_\q^{-1}(f): f\in {\cal F}_1\}\subset B\otimes Q.$
Fix an $\ep>0.$  Put $C_{00}=C_0((0,1])\oplus M_n(C_0((0,1])))$
and $C_g=\{(f,0), (0, f\otimes e_{i,i}), (0,f\otimes e_{1,i}): 1\le i\le n\}$ form a set of generators, where
$f\in C_0((0,1])$ is the identity function on $[0,1]$ and  $\{e_{i,j}\}_{1\le i,j\le n }$ is
a system of  matrix units for $M_n.$ It is well known that
$C_{00}$ is semi-projective.
Let $\dt_c>0$ {{satisfy}} the following:
if $L: C_{00}\to C'$ is a $C_g$-$\dt_c$-multiplicative \cpc\, for a \CA\, $C',$
there exists a \hm\, $h_c: C_{00}\to C'$ such that
\beq
\|h_c(g)-L(g)\|<\min\{\ep, \eta\}/64(n+1)^3\rforal g\in C_g.
\eneq
Let $\ep_0=\min\{\ep/(n+1)^316, \dt_c/4, \eta/(n+1)^316\}.$

Let $\dt>0$ and ${\cal G}\subset A\otimes Q$ be a finite subset required by \ref{Chomotopy} for
$\ep_0$  and ${\cal F}_1.$ \Wlog, we may assume
that ${\cal G}\subset (A\otimes Q)^{\bf 1}$ and ${\cal F}_1\subset {\cal G}.$
Let $\ep_1=\min\{\ep_0/2, \dt/4\}$ and ${\cal G}_1=\psi_q^{-1}({\cal G})\cup {\cal G}_{1,b}\subset B\otimes Q.$

It follows from
\ref{TUNIq} (see \ref{Rsec53})
that there exists a unitary $u\in {\widetilde{A\otimes Q}}$ such that
\beq\label{LFNZ-10}
\|u^*\psi_\p (g)u-\psi_\q(g)\|<\ep_1/4 \rforal g\in {\cal G}_1.
\eneq
Write $u=\lambda +\af(u)$ for some $\af(u)\in A\otimes Q.$
Choose $e_{00}, e_{01}\in (A\otimes Q)_+$ with $\|e_{00}\|=\|e_{01}\|=1$
such that $e_{00}e_{01}=e_{00}$ and $\|e_{00}x-x\|<\ep_1/16$ and
$\|x-xe_{00}\|<\ep_1/16$ for all $x\in {\cal G}_1$ and $x=\af(u).$
We also assume that  there is a non-zero $e_{00}'\in A\otimes Q$  {{such that}} $e_{00}'e_{01}=0.$
There is a unitary $u_1\in \C\cdot 1_{{\widetilde{A\otimes Q}}}+\overline{e_{00}(A\otimes Q)e_{00}}$
such that $\|u_1-u\|<\ep/8.$ Since $A\otimes Q\in {\cal D}_0,$
 by 11.5 of \cite{eglnp1}, it has stable rank one.
{{Thus there}} is a unitary $u_2\in \C\cdot 1_{{\widetilde{A\otimes Q}}}+\overline{e_{00}'(A\otimes Q)e_{00}'}$
such that $[u_2]=-[u]$ in $K_1(A).$
Put $u_3=uu_2.$ Then, since $e_{00}'e_{01}=0,$ by \eqref{LFNZ-10},
\beq\label{LFNZ-11}
\|u_3^*\psi_\p(g)u_3-\psi_\q(g)\|<\ep_1/2\rforal g\in {\cal G}_1.
\eneq
But now $u_3\in U_0({\widetilde{A\otimes Q}}).$
There is a continuous path of unitaries
$\{u(t): t\in [0, d]\}\subset U({\widetilde{A\otimes Q}})$ such that $u(0)=1$ and
$u(t)=u_3$ for all $t\in [d, 1]$ and for some $0<d_0<d<1.$
Define
\beq
\gamma(f)(t)=\begin{cases}\psi_\p^{-1}( u(t)f(t))u(t)^*)& \,\,\ t\in [0, d];\\
                                  {(1-t)\over{1-d}}\psi^{-1}_\p(u(d)f(d)u(d)^*)+{(t-d)\over{1-d}}\psi_\q^{-1}(f(1)) &
                                  \,\,\, t\in (d, 1].
                                  \end{cases}
\eneq
Note that $\gamma(f)\in B\otimes {\cal Z}_{\p,\q}.$
For $f\in {\cal F},$ let $g=\psi_\q^{-1}(f(1))=\psi_\q^{-1}(f(d)),$  by  \eqref{LFNZ-11},
\beq
\|g-\psi^{-1}_\p(u(d)\psi_\q(g)u(d)^*)\|<\ep_1/2.
\eneq
In other words, if $f\in {\cal F},$
\beq\label{LFNZ-11+}
\|\psi^{-1}_\p(u(d)f(d)u(d)^*)-\psi_\q^{-1}(f(1))\|<\ep_1/2
\eneq

Let ${\cal F}_2=\{\gamma(f): f\in {\cal F}\}\subset B\otimes {\cal Z}_{\p,\q}.$
{{Note that $B$ is a simple \CA\, and $B\cong B\otimes \zo,$ $B$ has tracially approximate divisible property
(see  \ref{TD0=D}  and \ref{Ptensor}).
Since
${\cal Z}_{\p,\q}$  is unital and $B$ has tracially approximate divisible property,}}
there
exist mutually orthogonal positive elements
$e_i,$  $i=0, 1,2,...,n,$ elements $w_i,$ $i=1,2,...,n,$ in $ B\otimes {\cal Z}_{\p,\q}$ such that
$w_i^*w_i=e_1^2,$ $w_iw_i^*=e_i^2,$ $e_0e_i=0,$ $i=1,2,...,n,$  and
\beq\label{LFNZ-12n}
\|x-\sum_{i=0}^ne_ixe_i\|<\ep_1/4,\,\,\|xw_i-w_ix\|<\ep_1/4,\,\,1\le i\le n\rforal x\in {\cal F}_2\andeqn\\
 d_\tau(e_0)\le \eta/4\rforal \tau\in T(B\otimes {\cal Z}_{\p,\q}).
\eneq
Since $f_e\in {\cal F},$ \eqref{LFNZ-12n} also implies
that
\beq\label{LFNZ-13}
\sum_{i=1}^n\tau(e_i)\ge 1-\ep_0/4-\eta/16n^2\rforal \tau\in T(B\otimes {\cal Z}_{\p,\q}).
\eneq
\Wlog, we may assume
that $e_i(t)=e_i(1)$ and $w_i(t)=w_i(1)$ for all $t\in [d_1, 1]$ for some
$d_1>d> d_0.$

Let ${\cal G}_2={\cal G}_1\cup\{e_i(1), w_i(1): 1\le i\le n\}.$
By applying \ref{TUNIq} again, we obtain another unitary $u_4\in {\widetilde{A\otimes Q}}$
such that
\beq\label{LFNZ-14}
\|u_4^*(u_3^*\psi_\p(g)u_3)u_4-\psi_\q(g)\|<\ep_1/16\rforal g\in {\cal G}_2.
\eneq
Therefore {{(see also \eqref{LFNZ-11}),}} for any $g\in {\cal G}_1,$
\vspace{-0.1in}\beq
\|[{\rm Ad}\, u_3\circ \psi_\p(g),\, u_4]\|<\ep_1.
\eneq
It follows from \ref{Chomotopy} that there exists a unitary $u_5$ and a continuous path of
unitaries  $\{v(t): t\in [d_1,r]\}$ in ${\widetilde{A\otimes Q}}$
(for some $1>r>d_1$)  with
$v(r)=u_4u_5$ and $v(d_1)=1_{{\widetilde{A\otimes Q}}}$ such that
\beq\label{LFNZ-15}
&&\|[{\rm Ad}\, u_3\circ \psi_\p(g),\, u_5]\|<\ep_1/16\rforal g\in {\cal G}_2\andeqn\\\label{LFNZ-15+1}
&&\|[{\rm Ad}\, u_3\circ \psi_\p(f),\, v(t)]\|<\ep_0\rforal f\in {\cal F}_1 {{\andeqn t\in [d_1, r]}}.
\eneq
It follows from \eqref{LFNZ-14} and \eqref{LFNZ-15} that
\beq\label{LFNZ-15+2}
\|v(r)^*(u_3^*\psi_\p(g){{u}}_3){{v}}(r)-\psi_\q(g)\|<\ep_1/8\rforal g\in {\cal G}_2.
\eneq
Now define
\beq
&&\hspace{-0.65in}b_i'=\begin{cases} u^*(t)\psi_\p(e_i(t))u(t)  & t\in [0,d_1],\\
                              v^*(t)u_3^*\psi_\p(e_i(t))u_3v(t) & t\in [d_1, r],\\
                ({(1-t)\over{1-r}}v(r)^*u_3^*\psi_\p(e_i(1))u_3v(r))+{(t-r)\over{1-r}}\psi_\q(e_i(1))    & t\in (r,1],\cr
                \end{cases}
                i=0,1,2,...,n\andeqn\\
                &&\hspace{-0.65in}z_i'= \begin{cases} u^*(t)\psi_\p(w_i(t))u(t)  & t\in [0,d_1],\\
                              v^*(t)u_3^*\psi_\p(w_i(t))u_3v(t) & t\in [d_1, r],\\
                ({(1-t)\over{1-r}}v(r)^*u_3^*\psi_\p(w_i(1))u_3v(r))+{(t-r)\over{1-r}}\psi_\q(w_i(1))    & t\in (r,1],\cr
                \end{cases} \,\,i=1,2,...,n.
         \eneq
 { {From  the definition of $e_i$ and $w_i,$  and
 \eqref{LFNZ-15+2}
  we have}}
  \beq\label{June-2018-1}
\hspace{-0.1in}\|(z'_i)^*z'_i -{{(b'_1)^2}}\|<\ep_1,~\|z'_i(z'_i)^*-{{(b'_i)^2}}\|<\ep_1,~\mbox{for all}~i\geq 1~~\mbox{and}~~ \|b'_ib'_l\|<\ep_1 ~~\mbox{for all} ~i\not=l.
\eneq

         For the next  few estimates, recall that $f(t)=f(1)$ for all $t\in [d_0,1],$
         $e_i(t)=e_i(1)$ for all $t\in [d_1, 1],$ and $u(t)=u(d)$ for all $t\in [d, d_1].$

 For $t\in [0,d_0],$  since $\gamma(f)\in {\cal F}_2,$
 by \eqref{LFNZ-12n},
 \beq
 \|f(t)-\sum_{i=0}^n b_i' (t)f(t)b_i'(t)\|<\ep_1\tforal f\in {\cal F}{.}
 \eneq
For $t\in [0,d_1],$ by the definition of $\gamma(f),$  by \eqref{LFNZ-11+}, {{the definition of $b_i',$}} and \eqref{LFNZ-12n},
we have
\beq
f(t)\approx_{\ep_1/2} u(t)^*\psi_\p(\gamma(f(t)))u(t)\approx_{\ep_1/4} \sum_{i=0}^n b_i'(t)u(t)^*\psi_\p(\gamma(f(t)))u(t)b_i'(t)\\
\approx_{\ep_1/2} \sum_{i=0}^n b_i' (t)f(t)b_i'(t)\,\,\,\rforal f\in {\cal F}.
\eneq
For $t\in [d_1,r],$ {{by \eqref{1910-152-n1},}}  \eqref{LFNZ-11},  {{\eqref{LFNZ-15+1}, and \eqref{LFNZ-12n},}} with $g=\psi_\q^{-1}(f(1)),$
\beq
&&\hspace{-0.4in}f(t)=f(1)\approx_{\ep_1} u_3^*\psi_\p(g)u_3\approx_{\ep_0} v(t)^*u_3^*\psi_\p(g)u_3v(t)\\
&&\approx_{\ep_1/2} {\rm Ad}\, u_3v(t)\circ \psi_\p(\gamma(f(t)))
\approx_{\ep_1/4} {\rm Ad}\, u_3v(t)\circ \psi_\p(\sum_{i=0}^n e_i(t)\gamma(f(t))e_i(t))\\
\vspace{-0.1in}&&\hspace{0.9in}\approx_{\ep_1} \sum_{i=0}^n b_i'(t) f(t)b_i'(t).
\eneq
On $[r,1],$  by the above,  and by \eqref{LFNZ-15+2}, as $e_i(1)\in {\cal G}_2,$
 \beq
 f(t)=f(1)\approx_{3\ep_1} \sum_{i=0}^n b_i'(r) f(r)b_i'(r)\approx_{4(n+1)\ep_1/8}
 \sum_{i=0}^n b_i'(t)f(t)b_i'(t).
 \eneq
 Coming all  {{the four estimates}} above, we have that
 \beq\label{LFNZ-13-1}
 \|f-\sum_{i=0}^n b_i' fb_i'\|<(n+1)\ep_1+\ep_0<\ep/16(n+1)^2\rforal f\in {\cal F}.
 \eneq
 We also compute
 that
 \beq\label{LFNZ-13-11}
 \|z_i'f-fz_i'\|<2\ep_1+\ep_0,\,\, \, 1\le i\le n,\rforal f\in {\cal F}.
 \eneq
 By the semi-projectivity of $C_{00}$ and \eqref{June-2018-1}, and choice of $\dt_c$ ($\ep_1<\dt_c$),
 we obtain $b_i, z_j\in A\otimes {\cal Z}_{\p,\q},$ $i=0,1,2,...,n,$ $j=1,2,...,n,$
 such that
 \beq
 \|b_i-b_i'\|<\min\{\ep,\eta\}/(64n^2)\andeqn \|z_j-z_j'\|<\min\{\ep,\eta\}/(64n^2),\\
 b_ib_l=0\,\,\,{\rm if}\,\,\, i\not=l,\,\,\, z_i^*z_i=b_1^2,\,\,\,z_iz_i^*=b_i^2,
 \eneq
 $i,l=0,1,2,...,n$ and $j=1,2,...,n.$
 By \eqref{LFNZ-13-1} and \eqref{LFNZ-13-11},
 \beq
 \|f-\sum_{i=0}^n b_ifb_i\|<\ep\andeqn \|z_if-fz_i\|<\ep,\,\,\,1\le i\le n,\rforal x\in {\cal F}.
 \eneq
 We also estimate, by \eqref{LFNZ-13},
 that
\vspace{-0.1in} \beq
 \tau(\sum_{i=1}^n b_i)>1-\eta/2\rforal \tau\in T(A\otimes {\cal Z}).
 \eneq
 It follows that $d_\tau(b_0)<\eta$ for all $\tau\in T(A\otimes {\cal Z}).$
 This implies that $b_0\lesssim a_0.$
 Therefore $A\otimes {\cal Z}$ has the tracial approximate divisible property (see \ref{Pdtad}).
\end{proof}

\begin{lem}\label{Lqtrace}
Let $A$ be a separable simple ${\cal Z}$-stable  \CA\,  with
continuous scale, {{and}} with $T(A)\not=\emptyset,$ and $QT(A)=T(A).$
Let $x\in {\rm ker}\rho_A.$
Then there exists  a \hm\,  $\psi: A\to M_4(A)$ which maps
strictly positive elements  to strictly positive elements,
$\psi_{*0}(x)=0$ and $(\tau\otimes Tr)(\psi(a))=4\tau(a)$ for all
$a\in A$ and $\tau\in T(A).$
where $Tr$ is the  standard trace on $M_4.$

\end{lem}

\begin{proof}
We first assume that $A$ is stably projectionless. By A8 of the appendix of \cite{eglnkk0},
there exists a projection
$p\in M_r({\tilde A})$  such that $[p]=[1_A]-x$ in $K_0({\tilde A})$ for some integer $r>0.$
By  A6 of the appendix of \cite{eglnkk0},
{{we may assume}} that
$p\in M_2({\tilde A}).$  Denote by $1_2\in M_2({\tilde A})$ the identity of $M_2({\tilde A}).$
Put $q=1_2-p.$ Note that $p+q=1_2.$
Write $\{e_{ij}\}_{2\times 2}$ as the matrix unit  for $M_2.$
By replacing $p$ by $Z^*pZ,$ where $Z$ is a unitary matrix with scalar entires, we may assume
that $\pi(p)=e_{11},$ where $\pi$ is the map induced by the quotient map ${\tilde A}\to \C.$
Later we will also use $\pi$ for the quotient map $M_2(M(A))\to M_2(M(A)/A).$
Note that we also have $\pi(q)=e_{22}.$

We have $\tau(p)=\tau(q)$ for all $\tau\in T(A).$
Let $A_1=pM_2(A)p$ and $A_2=qM_2(A)q.$
Let $a_p\in A_1$ and $a_q\in A_2$ be strict positive element of $A_1$ and $A_2,$ respectively.
We also assume that $0\le a_p\le 1$ and $0\le a_q\le 1.$
Then
\beq
d_\tau(a_p)=\tau(p)=\tau(q)=d_\tau(a_q)\rforal \tau\in T(A).
\eneq
By \cite{aTz}, $A$ is ${\cal Z}$-stable.  Note that we assume that $A$ is stably projectionless.
Then, {{by  Theorem 1.2 of \cite{Rz},}}  $a_p\sim a_q$ in $Cu(A).$  Also, by \cite{Rz}, $A$ almost has  stable rank one.
{{By \ref{Srk1}}} there is a partial isometry $w\in M_2(A)^{**}$ such
that $w^*a, aw\in M_2(A)$ and $ww^*a=aww^*=a$ for all $a\in A_1$  and $wb, bw^*\in M_2(A)$ for all $b\in A_2$ such that $w^*a_pw:=b_q$ is a strictly positive element
of $A_2.$

Moreover,
\beq\label{151-n1}
w^*(a_p)^{1/n}w=b_q^{1/n}\rforal n.
\eneq
Consider $W={{pwq+qw^*p}}.$ Then, for any $a\in M_2(A),$
$W^*a, aW\in M_2(A).$  In fact, we may write
$$
a=pap+paq+qap+qaq
$$
for any  $a\in M_2(A).$ Then, for any $a\in M_2(A),$
\beq{{
W^*a=qw^*pap+qw^*paq+pwqap+pwqaq\in M_2(A)\andeqn aW\in M_2(A).}}
\eneq
{{(Note that $M_2(A)$ is an ideal in $M_2(\tilde{A})$ and $p,q\in M_2(\tilde{A})$.)}}
Therefore $W\in M(M_2(A))=M_2(M(A)).$ Since $p+q=1_2,$
$a_p+b_q$ {{is}} a strictly positive element of $M_2(A).$
{{Hence $a_q^{1/n}+b_q^{1/n}\to 1_2$ in the strict topology.}}
{{We also have, by \eqref{151-n1},
\beq
W^*(a_p^{1/n}+b_q^{1/n})W&=&w^*a_p^{1/n}w+wb_q^{1/n}w^*=b_q^{1/n}+w(w^*a_q^{1/n}w)w^*\\
&=&b_q^{1/n}+(ww^*)a_p^{1/n}(ww^*)=b_q^{1/n}+a_q^{1/n}.
\eneq
}}
{{It follows that $W^*W=1_2.$  As $W^*=W,$}}
 $W$ is a {{self adjoint}} unitary in $M_2(M(A))$ and $\phi(a)=W^*aW$ for all $a\in A$ defines
an automorphism of $M_2(A).$
Note that $a_p^{1/n}$ converges strictly to the identity of $M(A_1).$ Note also
\beq\nonumber
a_p^{1/n} (a_p+b_q)^{1/2}=a_p^{1/n}a_p^{1/2}\to 1_{\tilde A_1}(a_p+b_q)^{1/2}\andeqn
(a_p+b_q)^{1/2} a_p^{1/n}\to (a_p+b_q)^{1/2}1_{\tilde A_1}
\eneq
(in norm).
It  follows that $a_p^{1/n}$ converges to a projection $p'\in M_2(M(A))$ strictly.
Exactly the same argument shows that $b_p^{1/n}$ converges  strictly to a projection $q'\in M_2(M(A)).$
Since $a_p+b_q$ is a strictly positive element of $M_2(A),$
this implies that $p'+q'=1_2.$
Since $p\in M_2(M(A))$ and
$pa_p^{1/n}=a_p^{1/n},$ $p\ge p'.$ Also $pq'=q'p=0.$  Similarly $q\ge q'$ and $qp'=p'q=0.$
Since $p+q=1_2$ and $p'+q'=1_2,$
this implies that $p=p'$ and $q=q'.$
Since $W\in M_2(M(A)),$ it follows that
\beq\label{151-n2}
W^*pW=q.
\eneq

We now show  that $W^*M_2({\tilde A})W=M_2({\tilde A}).$
Write
\beq
W=\begin{pmatrix} w_{11} & w_{12}\\
                                w_{21} & w_{22}\end{pmatrix}.
                                \eneq
Note that  $\pi(p)=e_{11}$ and $\pi(q)=e_{22}.$  {{Since $W$ is a self adjoint unitary}}, by \eqref{151-n2},
$\pi(W)^*e_{11}\pi(W)=e_{22}$, {{$\pi(W)^*e_{22}\pi(W)=e_{11}$, and $e_{11}\pi(W)=\pi(W) e_{22}$.}}
Hence
\beq
{{\pi(w_{11})=0=\pi(w_{22}).}}
\eneq
In other words,
\beq
\pi(W)=\begin{pmatrix} 0 & \pi(w_{12})\\
                              \pi(w_{21}) & 0\end{pmatrix}.
                             \eneq
Then, since $W$ is a unitary,
\beq\nonumber
&&\pi(W^*)\begin{pmatrix} 0 & 1\\
                                      1& 0\end{pmatrix} \pi(W)=\begin{pmatrix} \pi(w_{21}^*) & 0\\
                                                                                           0 & \pi(w_{12}^*)\end{pmatrix}\pi(W)=
                                                                                           \begin{pmatrix}  0 & \pi(w_{21}^*w_{21}) \\
                                                                                             \pi(w_{12}^*w_{12}) &0\end{pmatrix}=\begin{pmatrix} 0 & 1\\
                                      1& 0\end{pmatrix}.
\eneq
Since $e_{1,1},$ $e_{2,2}$ and $\begin{pmatrix} 0 & 1\\
                                      1& 0\end{pmatrix}$ generates $M_2,$ this implies that $\pi(W^*)s\pi(W)\in M_2$ for any scalar matrix $s.$ Therefore
$W^*M_2({\tilde A})W=M_2({\tilde A}).$
This extends $\phi$  from $M_2({\tilde A})$ to  $M_2({\tilde A})$ as an isomorphism.
It follows that $\phi_{*0}(2[1_{\tilde A}])=2[1_{\tilde A}].$ Put $y=\phi_{*0}([1_{\tilde A}])-[1_{\tilde A}].$
Then $2y=0.$  Also
{{\beq
\phi_{*0}(x)=\phi_{*0}([1_{\tilde{A}}]-[p])=\phi_{*0}([1_{\tilde{A}}])-[q]\\
=(\phi_{*0}([1_{\tilde{A}}])-[1_{\tilde A}]) {{+[1_{\tilde A}]}}-(2[1_{\tilde A}]-[p])\\ \label{Lqtrace-10}
=y+([p]-[1_{\tilde A}])=
y-x.
\eneq}}
Define $\psi: A\to M_4(A)$ by
\beq
\psi(a)=\diag(a,a, \phi(a), \phi(a))\rforal a\in A.
\eneq
Note that
\beq
\psi_{*0}(x)=2x+{{2(y-x)}}=2x-2x=0.
\eneq

In case that $A$ is not stably projectionless, let $e\in M_m(A)$ be a nonzero projection.
Put $B=eM_m(A)e.$ Then $B$ is a unital simple \CA\, with nonzero quasidiagonal traces.
Then the conclusion follows from 5.5 of \cite{Dmaps}.

\end{proof}

\begin{cor}\label{Z0ML}
Let $A$ be a separable simple \CA\, which is ${\cal Z}$-stable and $QT(A)=T(A)\not=\emptyset.$
For any finitely generated subgroup $G_0\subset {\rm ker}\rho_A,$ there exists an integer $m\ge 1$ and
a \hm\, $\psi: A\to M_m(A)$ which maps
strictly positive elements  to strictly positive elements,
$\psi_{*0}(x)=0$ {{for all $x\in G_0$}} and $(\tau\otimes Tr_m)(\psi(a))=m\tau(a)$ for all
$a\in A$ and $\tau\in T(A),$
where $Tr_m$ is the  standard trace on $M_m.$

\end{cor}

\begin{proof}
Let $x_1,x_2,...,x_k\in G_0$ be a set of generators of $G_0.$
We prove the corollary by induction.
By \ref{Lqtrace}, it holds for $k=1.$  Suppose
that it holds for all integers $1\le k'<k.$
Let $G_{0,1}\subset G_0$ which  is generated  by $x_1, x_2,...,x_{k-1}.$
By the inductive assumption, there exists a \hm\, $\psi_1: A\to M_{m'}(A)$
which maps strictly positive elements to strictly positive elements,
$(\psi_1)_{*0}(x)=0$ for all $x\in G_{0,1}$ and
$(\tau\otimes Tr_{m'})(\psi_1(a))=m'\tau(a)$ for all $a\in A.$

Let $y=(\psi_1)_{*0}(x_k)$ and $B=M_{m'}(A).$
Lemma \ref{Lqtrace} shows
that there exists a \hm\, $\phi: B\to M_{4}(B)$  such that $\phi$ maps strictly positive elements to strictly positive elements,
$(\phi)_{*0}(y)=0$ and
$(\tau\otimes Tr_{4})(\psi(b))=4\tau(a)$ for all $b\in A$ and $\tau\in T(B).$
Let $m=4m'.$  Define $\psi: A\to M_m(A)$ by
\beq
\psi(a)={{\phi\circ \psi_1(a)}}
\rforal a\in A.
\eneq
Then $(\psi)_{*0}(x)=0$ for all $x\in G_0.$
Lemma follows.
\end{proof}

\begin{thm}\label{Tz01}
Let $A$   be  a finite separable simple  \CA s  with finite nuclear dimension which satisfies
the UCT.  Suppose that  $T(A)\not=\emptyset$ and  $K_0(A)={\rm ker}\rho_A.$
Then $gTR(A)\le 1.$
\end{thm}

\begin{proof}
Note that, by \cite{Wnz}, that $A$ is ${\cal Z}$-stable. It follows that there exists $e\in A_+\setminus \{0\}$
such that $eAe$ has continuous scale.   \Wlog, we may assume that $A=eAe.$
It follows from \cite{TWW} that every tracial state  of $A$ is quasidiagonal.
We will prove that  every tracial state of $A$ is a ${\cal W}$-trace. Then \ref{TD0=fn} applies.
We will follow exactly the same proof of
{{7.4 of \cite{eglnkk0}.}}

As in the proof of 7.4 of \cite{eglnkk0}, it suffices to show
that every tracial state of $A\otimes Q$ is a ${\cal W}$-trace.
Therefore, from now on, in this proof, we assume that $A=A\otimes Q.$
Let $K_0(A)=\cup_{n=1}^{\infty}G_n,$ where $G_n\subset G_{n+1}$ is a sequence
of finitely generated subgroups.
Let ${\cal P}\subset \underline{K}(A)$ be a finite subset and $G_{\cal P}$ be the subgroup
 generated by ${\cal P}.$
 We may assume that ${\cal P}{{\cap K_0(A)}}\subset G_n$ for some integer $n\ge 1.$
 It follows from \ref{Z0ML} that there exists a \hm\, $\phi_n: A\to A\otimes M_m(\C)\to A\otimes Q$
 which maps strictly positive elements to strictly positive elements,
 $\phi_{*0}(x)=0$ for all $x\in G_n$ and $\tau(\phi(a))=\tau(a)$ for all $a\in A$ and
 $\tau\in T(A\otimes Q).$
 By \cite{TWW}, every tracial state of $A\otimes Q$ is quasidiagonal,
 there exists a sequence  of \cpc s ${{\psi}}_k: A\otimes Q\to Q$ such that
 \beq
 \lim_{k\to\infty}\|{{\psi}}_k(ab)-{{\psi}}_k(a){{\psi_k}}(b)\|=0\andeqn
 \lim_{k\to\infty}\tau({{\psi}}_k(a))=\tau(a)\rforal a,\, b\in A.
 \eneq
 For each $n$ choose $k(n)$ such that
 $L_n: A\otimes Q\to Q$ defined by
 $L_n(a)=\psi_{k(n)}\circ \phi_n(a)$ for all $a\in A$ has the property
 that $[L_n]|_{G_n}=0$ and
 \beq
  \lim_{n\to\infty}\|L_n(ab)-L_n(a)L_n(b)\|=0\andeqn
 \lim_{n\to\infty}{{{\rm tr}_Q}}(L_n(a))=\tau(a)\rforal a\in {{ A,}}
 \eneq
 {{where ${\rm tr}_Q$ is the unique trace on $Q$.}}
Since both $A\otimes Q$ and $Q$ are divisible, {{and $K_1(Q)=\{0\},$}} for any finite subset
${\cal P}\subset \underline{K}(A),$ $[L_n]|_{\cal P}=\{0\}$ for all sufficiently large $n.$

 {{By Lemma 7.2}} and the proof of 7.4 of \cite{eglnkk0})
  {{there}} exists a sequence of \cpc s $\Phi_n: A\to {\cal W}$ such that
 \beq
 \lim_{n\to\infty} \|\Phi_n(ab)-\Phi_n(a)\Phi_n(b)\|=0\andeqn
 \tau(a)=\lim_{n\to\infty}t_{\cal W}\circ \Phi_n(a)\rforal a\in A.
 \eneq
{{To see this, let}}  ${\cal P}\subset \underline{K}(A)$ be a finite subset.
 {{Then}} $[L_n]|_{\cal P}=0$ for all sufficiently large $n.$
 In the proof of 7.4 of \cite{eglnkk0} let us  replace $\psi$ there by  $L_n$ above.  Since $[L_n]|_{\cal P}=0$, we obtain  $[\Psi_0]|_{\cal P}=[\Psi_1]|_{\cal P}$  with  $L_n$ in place of $\psi$---namely, $\Psi_0$ is defined to be $m$ copies of $L_n$ (in place of $\psi$ there) and $\Psi_1$ is defined to be $m+1$ copies of $L_n$ (in place of $\psi$ there). The fact $[\Psi_0]|_{\cal P}=[\Psi_1]|_{\cal P}$ is used to connect $\Psi_0$ and $\Psi_1.$
Thus, by the same proof of
 {{7.4}} of \cite{eglnkk0}), {{we can construct
  $\{\Phi_n\}$ as required}}.
   {{This proves}}
  that  every tracial state of $eAe$ is a ${\cal W}$-trace. It follows from \ref{TD0=fn} that
$eAe\in {\cal D}_0.$

\end{proof}

\begin{thm}\label{TMT}
Let $A_1$  and $A_2$ be  two  separable simple \CA s  with  finite nuclear dimension  which satisfy
the UCT.  Suppose that $K_0(A_i)={\rm ker}\rho_{A_i}$
 ($i=0,1$).
Then
$A_1\cong A_2$ if and only if
\beq
(K_0(A), K_1(A), {\tilde T}(A), \Sigma_A)\cong (K_0(B), K_1(B), {\tilde T}(B), \Sigma_B).
\eneq
Moreover,
{{in case that ${\tilde{T}}(A)\not=\{0\},$}}
both $A$ and $B$ are  stably isomorphic to one of $B_T$ constructed in section 7.

\end{thm}

\begin{proof}
Since $A$ and $B$ have
finite nuclear dimension, $A$ and $B$ are both stably finite or purely infinite (which
are the case that ${\tilde T}(A)={\tilde T}(B)=\{0\}$).
By \cite{KP} and \cite{Pclass}, we may assume that $A$ and $B$ are stably finite and by \cite{BR},
${\tilde T}(A), {\tilde T}(B)\not=\{0\}.$

It follows from \cite{aTz} that both $A$ and $B$ are ${\cal Z}$-stable.
Let $e_A\in A_+$ with $\|e_A\|=1$
and $e_B\in B_+$ with $\|e_B\|=1$
such that both $A_0:=\overline{e_AAe_A}$ and ${{B_0}}:=\overline{e_BBe_B}$ have continuous scales
{{(see 5.2 of
\cite{eglnp1}).}}
It follows from
\ref{Tz01} that  both $A_0$ and $B_0$ are in ${\cal D}_0$ {{which implies $gTR(A)\le 1$ and $gTR(B)\le 1.$}}
 Then Theorem
\ref{T1main} applies.


\end{proof}

\begin{cor}\label{Czo}
Let $A$   be  a stably finite separable simple \CA s  with  finite nuclear dimension  which satisfies
the UCT. Then the following are equivalent:

(1) $A$ is isomorphic to $\zo;$

(2) $A$ has a unique tracial state, $K_0(A)={\rm ker}\rho_A=\Z$ and  $K_1(A)=\{0\}$  and

(3) $A$ is stably projectionless and has a unique tracial state, $K_0(A)=\Z$ and $K_1(A)=\{0\}.$
\end{cor}

\begin{proof}
The equivalence of (1) and (2) follows from \ref{TMT}.
It is obvious that (2) implies (3).  If $A$ is stably projectionless, then, by  A8  of \cite{eglnkk0},
$K_0(A)=\Z={\rm ker}\rho_A.$ Therefore (3) implies (2).

\end{proof}






Finally we offer the following result (as Theorem \ref{TTT3}).

\begin{thm}\label{Mclass3}
Let $A$ and $B$ be two  separable simple \CA s with finite nuclear dimension which satisfy
the UCT.
Then $A\otimes \zo\cong B\otimes \zo$ if
and only if
\beq
(K_0(A), K_1(A), {\tilde T}(A), \Sigma_{A})\cong (K_0(B), K_1(B), {\tilde T}(B), \Sigma_{B}).
\eneq
{{(We emphasis that there is no order on $K_0$-groups. Also in case that ${\tilde T}(A)=\{0\},$ we view $\Sigma_A=0.$)}}
\end{thm}

\begin{proof}
First, if $A$ is infinite,
it follows ${\tilde T}(A)=\{0\}.$
Moreover since $A$ has finite nuclear dimension, it is purely infinite.
Since ${\rm Ell}(A)\cong {\rm Ell}(B),$
${\tilde T}(B)=\{0\}.$ So $B$ is also not stably finite.
As $B$ has finite nuclear dimension, $B$ is also purely infinite.
Thus, the infinite case is covered by the classification of non-unital purely infinite simple \CA s
(see \cite{KP} and \cite{Pclass}).

We now assume both $A$ and $B$ are finite. We only need to show the ``if" part.

Put $A_1=A\otimes \zo$ and $B_1=B\otimes \zo.$
Then
we have, {{ignoring the order structure on $K_0(A)$,}}
\beq\label{Mclass3-2}
(K_0(A_1), K_1(A_1), {\tilde T}(A_1), \Sigma_{A_1})&=&
(K_0(A), K_1(A), {\tilde T}(A), \Sigma_A)\\\label{Mclass3-3}
(K_0(B_1), K_1(B_1), {\tilde T}(B_1), \Sigma_{B_1})&=&
(K_0(B), K_1(B), {\tilde T}(B), \Sigma_B).
\eneq

Let $e_A\in ({\rm Ped}(A_1))_+$ with $\|e_A\|=1$ and $e_B\in ({\rm Ped}(B_1))_+$ with $\|e_B\|=1$ such
that $A_0={\rm Ped}(A_0)$ and $B_0={\rm Ped}(B_0),$ where
$A_0:=\overline{e_A(A_1)e_A}$ and $B_0:=\overline{e_B(B_1)e_B}.$
It follows from Proposition
12.5 of \cite{eglnp1}  that all tracial states of $A_0\otimes \zo$ and $B_0\otimes \zo$
are ${\cal W}$ traces.  It follows from 6.6 of \cite{eglnkk0} (see 17.6 and the proof of  18.6 of \cite{GLp1}) that
$A_0\otimes \zo, B_0\otimes \zo \in {\cal D}_0.$
Note
$A_0$ and $B_0$ are hereditary \SCA s of $A_1\otimes \zo$ and $B_1\otimes \zo,$ respectively.
Note  also $A_1\otimes \zo\cong A_1$ and $B_1\otimes \zo\cong B_1,$  by \ref{zoselfabsorbing}.
Therefore $gTR(A_1)\le 1$ and $gTR(B_1)\le 1.$  Since $A_0, B_0\in {\cal D}_0,$  by \ref{D0kerrho},
$K_0(A_1)={\rm ker}\rho_{A_1}$ and $K_0(B_1)={\rm ker}\rho_{B_1}.$
Thus the theorem follows from \eqref{Mclass3-2}, \eqref{Mclass3-3} and Theorme \ref{T1main}.


\end{proof}

\section{Appendix}
In this appendix, we show that separable amenable \CA\, in ${\cal D}$ are ${\cal Z}$-stable.
The proof is a non-unital version of Matui and Sato's proof in \cite{MS} which
is identical to the unital case with only a few modification.  We will follow steps
of their proof as well as the notation in \cite{MS}.

\begin{lem}[cf. 2.4 of \cite{MS}]\label{LMS24}
Let $A$ be a separable simple \CA\, with continuous scale  and with $T(A)\not=\emptyset$ and let
$a\in A_+\setminus \{0\}.$
Then there exits $\af>0$ such that
\beq
\af \liminf_{n\to\infty}\inf_{\tau\in T(A)}\tau(f_n)\le \liminf_{n\to\infty}\inf_{\tau\in T(A)}\tau(f_n^{1/2}af_n^{1/2})
\eneq
for any central sequence $(f_n)_n$ of positive contractions of $A.$
\end{lem}

\begin{proof}
By
{{5.6  of \cite{eglnp1},}}
 $A$ is strongly uniformly full in $A.$
Therefore there are  $M(a), N(a)>0$  such that,  for $b\in A_+$ with $\|b\|\le 1$ and for any $\ep>0,$
there are $x_i\in A$ with $\|x_i\|\le M(a),$ $i=1,2,..., N(a)$ such that
\beq
\|\sum_{i=1}^{N(a)} x_i^*ax_i-b\|<\ep.
\eneq
Put $\af_0=M(a)^2N(a)$ and $\af={4\over{3\af_0}}.$
Let $\{f_n\}_n$ be given. We may assume
that $$\liminf_{n\to\infty}\inf_{\tau\in T(A)}\tau(f_n)=\bt>0,$$
{{otherwise there is nothing to prove.}}
Since $A$ has continuous scale, there exists
$e\in A_+$ with $\|e\|=1$ such that
\beq
\tau((1-e^{1/2})c(1-e^{1/2}))<\bt/8\rforal \tau\in T(A)
\eneq
for any $c\in A_+$ with $\|c\|=1.$
Then there are $y_i\in A$ such that $\|y_i\|\le M(a),$ $i=1,2,...,N(a)$
such that
\vspace{-0.12in}\beq
\|\sum_{i=1}^{N(a)}y_i^*ay_i-e\|<\bt/8,\,\,\, i=1,2,....
\eneq
One also has that
\vspace{-0.14in}\beq
\tau((1-e)f_n)<\bt/8,\,\,\, n\in \N.
\eneq
Then, keeping in mind that $(f_n)_n$ is a central sequence,
\beq\nonumber
&&\hspace{-0.2in}\bt=\liminf_{n\to\infty}\inf_{\tau\in T(A)}\tau(f_n)
\le \liminf_{n\to\infty}\inf_{\tau\in T(A)}\tau(ef_n)+\bt/8
\le  \liminf_{n\to\infty}\inf_{\tau\in T(A)}\sum_{i=1}^{N(a)}\tau(y_{i}^*ay_{i}f_n) +\bt/4
\\\nonumber
&&
\hspace{-0.2in}=\liminf_{n\to\infty}\inf_{\tau\in T(A)}\sum_{i=1}^{N(a)}\tau(y_i^*a^{1/2}f_na^{1/2}y_i)+\bt/4
=\liminf_{n\to\infty}\inf_{\tau\in T(A)}\sum_{i=1}^{N(a)}\tau(f_n^{1/2}a^{1/2}y_iy_i^*a^{1/2}f_n^{1/2})+\bt/4\\\nonumber
&&\le \af_0\liminf_{n\to\infty}\inf_{\tau\in T(A)}\tau(f_n^{1/2}af_n^{1/2})+\bt/4.
\eneq
Thus
\vspace{-0.12in}\beq
3\bt/4\le \af_0\liminf_{n\to\infty}\inf_{\tau\in T(A)}\tau(f_n^{1/2}af_n^{1/2}).
\eneq
\end{proof}

\begin{df}[2.1 of \cite{MS}]\label{DMS21}
Let $A$ be a separable \CA\, with $T(A)\not=\emptyset$ and
let $\phi: A\to  A$ be a  completely positive linear map. Suppose
that $T(A)$ is compact.  Recall  that $\phi$ is said to be excised in small central sequence
if for any  central sequence $(e_n)_n$ and $(f_n)_n$ of positive contractions  in $A$ satisfying
\vspace{-0.12in}\beq\label{DMS21-1}
\lim_{n\to\infty}\sup_{\tau\in T(A)}\tau(e_n)=0\andeqn
\lim_{m\to\infty}\liminf_{n\to\infty}\inf_{\tau\in T(A)}\tau(f_n^m)>0,
\eneq
there exists  $s_n\in A$  with $\|s_n\|\le \|\phi\|^{1/2}$ and $n\in \N$ such that
\vspace{-0.12in}\beq
\lim_{n\to\infty}\|s_n^*as_n-\phi(a)e_n\|=0\rforal a\in A \andeqn \lim_{n\to\infty}\|f_ns_n-s_n\|=0.
\eneq
\end{df}

\begin{lem}[2.5 of \cite{MS}]\label{LMS25}
Let $A$ be a separable simple \CA\, with $T(A)\not=\emptyset$ with continuous scale.
Suppose also that $A$ has the strict comparison for positive elements.
Let $(e_n)_n$ and $(f_n)_n$ be as \eqref{DMS21-1}.
Then for any $a\in A_+$ with $\|a\|=1,$   there exists a sequence
$(r_n)_n$  in $A$ such that
\beq
\lim_{n\to\infty}\|r_n^*f_n^{1/2} af_n^{1/2}r_n-e_n\|=0\tand\,\,\,\limsup_{n\to\infty}\|r_n\|
=\limsup_{n\to\infty}\|e_n\|^{1/2}.
\eneq
\end{lem}

\begin{proof}
The proof of this is exactly the same as that of Lemma 2.5 of \cite{MS} using
\ref{LMS24} instead of 2.4 in \cite{MS}.
\end{proof}

\begin{prop}[2.2 of \cite{MS}]\label{PMS22}
Let $A$ be a separable amenable simple \CA\, with $T(A)\not=\emptyset$ and  with continuous scale.
Suppose that $A$ has strict comparison for positive elements.
Let $\omega$ be a non-zero pure state of $A,$  $c_i, d_i\in A,$ $i=1,2,...,N.$
Then a completely positive linear map  $\phi: A\to A$ defined by
$\phi(a)=\sum_{i,j=1}^N \omega(d_i^*ad_j)c_i^*c_j$
can be excised by small central sequences.
\end{prop}

\begin{proof}

Let $\ep>0$ and let ${\cal F}\subset A$ be a finite subset.  It suffices to show that
there exists $s_n\in A,$ $n\in \N,$ such that $\|s_n\|\le \|\phi\|^{1/2}+\ep$ and
\beq\label{MS22-1}
\lim_{n\to\infty}\|s_n^*as_n-\phi(a)e_n\|<\ep\andeqn \lim_{n\to\infty}\|f_ns_n-s_n\|=0.
\eneq
Let ${\cal G}=\{d_i^*ad_j: a\in {\cal F}, 1\le i,j\le N\}$ and let $\dt=\ep/N^2.$

By Proposition 2.2 of \cite{AAP}, there is $a\in A_+$ with $\|a\|=1$ such that
$\|a(\omega(x)-x)a\|<\dt$ for all  $x\in {\cal G}.$
Let $\{e_n\}_n$ and $\{f_n\}_n$ be as in \eqref{DMS21-1}.  By 2.3 of \cite{MS}, there is a central sequence
$\{{\tilde f}_n\}_n$ of positive contractions of $A$ such that
$\{{\tilde f}_nf_n\}_n=\{f_n\}_n$ in $A_{\infty}$ and
\beq
\lim_{m\to\infty}\liminf_{n\to\infty}\inf_{\tau\in T(A)}\tau({\tilde f}_n^m)=
\lim_{m\to\infty}\liminf_{n\to\infty}\inf_{\tau\in T(A)}\tau(f_n^m).
\eneq
Applying \ref{LMS25}  to $\{e_n\}_n,$ $\{f_n\}_n,$  and $a^2,$ we obtain $r_{n}\in A,$ $n\in \N,$ satisfying
\beq\label{PMS22-10}
\lim_{n\to\infty}\|r_{n}^* {\tilde f}_n^{1/2}a^2{\tilde f}_n^{1/2}r_{n}-e_n\|=0\andeqn \limsup_{n\to\infty}\|r_{n}\|\le 1.
\eneq
Define
\vspace{-0.15in}\beq
s_n=\sum_{i=1}^N d_ia{\tilde f}_n^{1/2}r_{n}c_i,\,\,\,n=1,2,....
\eneq
The rest of the proof is exactly the same as that of proof of Proposition 2.2  in \cite{MS} with
one exception. We need to address the norm
of $s_n.$ Note that, by \eqref{MS22-1},
\beq
\|s_n^*bs_n\|\le \|\phi\|+\ep\rforal b\in A_+^{\bf 1}.
\eneq
Therefore by replacing $s_n$ by $E_ns_n$ for some $E_n\in A_+^{\bf 1}$ as subsequence of
an approximate identity of $A,$ we may assume $\|s_n\|\le \|\phi\|^{1/2}.$
\end{proof}

\begin{lem}[3.1 of \cite{MS}]\label{LMS31}
Let $A$ be a separable amenable simple non-elementary
 \CA, and let $\omega$ be a non-zero pure state of $A.$
 Then any \cpc\, $\phi: A\to A$ can be approximated
 point-wisely in norm by \cpc s $\psi$ of the from
 \beq
 \psi(a)=\sum_{l=1}^N \sum_{i,j=1}^N \omega(d_i^*ad_j)c_{l,i}^*c_{l,j}\tforal a\in A,
 \eneq
where $c_{l,i},\, d_i\in A,$ $l,i=1,2,...,N.$
\end{lem}

\begin{proof}
The proof is identical to that of 3.1 of \cite{MS}.  Unital condition
can be easily removed.
In the first place that unital condition is mentioned,  by using an approximate
identity $\{e_n\}$ of $A,$  and consider $\rho(e_n)^{-1/2} \rho(\,\cdot\,) \rho(e_n)^{-1/2}$
and $\sigma(\rho(e_n)^{1/2}\, \cdot \, \rho(e_n)^{1/2})$
for some large $n,$ we can assume that $\rho(e_n)$ is the unit of $M_N,$
by considering a hereditary \SCA\, of a full matrix algebras exactly the way as described
in that proof.  Then, since we assume that $A$ is simple and
non-elementary, $\pi(A)$ does not contain any non-zero compact operators on ${\cal H}$
in the second paragraph of that proof.  So Voiculescu  theorem applies.
The rest of proof are unchanged.
\end{proof}

\begin{lem}\label{LMS33}
Let $A\in {\cal D}$ be separable \CA\,  with continuous scale.
Then, for any integer $k\ge 1,$ there exists
an order zero c.p.c. map
$\psi:  M_k\to A_{\infty}\cap A'$ such that
\beq
\lim_{n\to\infty}\inf\{|\tau(c_n^m)-1/k|:\tau\in T(A)\}=0\tforal m\in \N,
\eneq
where $c_n=\psi(e)$ and $e\in M_k$ is a minimal
rank one projection of $M_k.$
\end{lem}

\begin{proof}
This proof can be extracted  from the proof of  {{10.4 of \cite{eglnp1}.}}
First keep in mind, by  {{9.4 of \cite{eglnp1},}}
$A$ has strict comparison for positive elements.
In the case that $A\in {\cal D}_0,$ this directly follows from
{{10.7 of \cite{eglnp1}.}}  In this case,  {{by 10.7 of \cite{eglnp1},}}
there are two sequences of \SCA s
$A_{0,n},$ $M_k(D_n)$ of $A,$ two sequences
of \cpc s $\phi_n^{(0)}: A\to A_{0,n}$  and
$\phi_n^{(1)}: A\to D_n\in {\cal C}_0^{0'}$ with $M_k(D_n)\perp A_{0,m}$ satisfy the following:
\beq\label{TDappdiv-1}
&&\lim_{n\to\infty}\|\phi_n^{(i)}(ab)-\phi_n^{(i)}(a)\phi_n^{(i)}(b)\|=0\rforal a,\, b\in A, \,\,i=0,1,\\\label{TDappdiv-1+}
&&\vspace{-0.1in}\hspace{-0.2in}\lim_{n\to\infty}\|a-(\phi_n^{(0)}(a)\oplus \diag(\overbrace{\phi_n^{(1)}(a),\phi_n^{(1)}(a),...,\phi_n^{(1)}(a)}^k)\|=0\rforal a\in A,\\\label{TDappdiv-1+2}
&&\lim_{n\to\infty} \sup_{\tau\in T(A)}d_\tau(c_n)=0,\\\label{TDappdiv-2}
&&\tau(f_{1/4}(\psi_n^{(1)}(a_0)))\ge d \rforal \tau\in T(D_n)
\eneq
 and $\phi_n^{(1)}(a_0)$ is a strictly positive element in $D_n,$
 where $c_n$ is a strictly positive element of $A_{0,n}$ and $1>d>0.$
It is easy to see (see the proof of  9.1 of \cite{eglnp1})
 that
\vspace{-0.12in}\beq\label{TDappdiv-6}
\hspace{-0.2in}\lim_{n\to\infty}\sup\{|\tau(a)-\tau\circ \diag(\overbrace{\phi_n^{(1)}(a), \phi_n^{(1)}(a),...,\phi_n^{(1)}(a)}^k)|: \tau\in T(A)\}=0\rforal a\in A.
\eneq
Let $e_{0,n}$ and $e_{1,n}$ be approximate identities for $A_{0,n}$ and $D_n,$ respectively.
Define $e_{j,l, n}=f_{1/2l}(e_{j,n}),$ $j=0,1,$  $l\in \N.$
Then  $\{e_{0,l,n}\}_l$  and $\{e_{1,l,n}\}_l$ are approximate identities  for $A_{0,n}$ and $D_n,$ respectively.
Define
${\bar e}_{1,l,n}=\diag(\overbrace{e_{1,m,n}, e_{1,m,n},...,e_{1,m,n}}^k).$
Put $E_{l,n}=e_{0,l,n}+{\bar e}_{1,l,n}.$
Then
since $T(A)$ is compact, as we assume $A$ has continuous scale,
$\lim_{l\to\infty}\sup_{\tau\in T(A)}\tau(E_{m,n})=1.$

Therefore,  by \eqref{TDappdiv-1+2},  it is easy to   choose  a subsequence
$j_n$
such that
\beq
\lim_{n\to\infty} \sup_{\tau\in T(A)}| \tau(e_{1,j_n,n}^m)-1/k|=0\rforal m\in \N,
\eneq
and by \eqref{TDappdiv-1+}, $\{e_{1,j_n,n}\}$ is a central sequence.
Note that we identify $e_{1,j_n,n}$ with\\ $\diag(e_{1,j_n,n},\overbrace{0,...,0}^{k-1})\subset M_k(D_n).$
Put $e_{1,j_n,n,i}=\diag(\overbrace{0,...,0,}^{i-1}, e_{1,j_n,n,i},0,...,0),$ $i=1,2,...,k.$
There are $w_{i,n}\in M_k(D_n)$ such that
$w_{i,n}^*w_{i,n}=e_{1,j_n,n,1}$ and $w_{i,n}w_{i,n}^*=e_{1,j_n,n,i},$ $i=2,3,...,k.$
Since $A$ is stably projectionless, the \SCA\, generated by $e_{1,j_n,n,i}$ and $w_{i,n}$
is isomorphic to $C_0(C(0,1], M_k).$ Note $\{w_{i,n}\}$ can be chosen to be central (by \eqref{TDappdiv-1} and
\eqref{TDappdiv-1+}.
Put $c_n=e_{1,j_n,n}.$ We obtain a \cpc $\psi: M_k\to A_{\infty}\cap A'.$

In the case that $A\in {\cal D},$  $M_k(D_n)$ is replaced by $D_n$ and \eqref{TDappdiv-1+} is replaced by
\beq\label{TDappdiv-1+22}
\lim_{n\to\infty}\|a-\diag(\phi_n^{(0)}(a), \diag(\phi_n^{(1)}(a))\|=0\rforal a\in A.
\eneq
But, as in the proof of {{10.4 of \cite{eglnp1},}}
the algebra $D$ in that proof is ${\cal Z}$-stable.
Therefore, in the proof of
{{10.2 of \cite{eglnp1},}} one has that  {{(as (e.10.6) there)}}
\vspace{-0.1in}\beq
\|[\phi_{n,m}(x), y]\|<\ep/16K^2\rforal x\in {\cal F}
\eneq
and $y\in \{d''^{1/2}, d'',  v'', e_j'', w_j'', j=1,2,...,K\}.$
Note that one can choose $K=nk$ and
using $n$ copies of $e_j''$ and $w_j'',$ the same argument above also produces
the \cpc map $\phi$ from $M_k.$
\end{proof}

\begin{lem}\label{LMS2to3}
Let $A$ be a separable amenable simple \CA\, in ${\cal D}$  with continuous scale.
Then every completely positive linear map $\phi: A\to A$ can be excised by small central sequences.
\end{lem}

\begin{proof}
Let $\phi: A\to A$ be a \cpc\, (so we assume $\|\phi\|=1$ \wilog). Let $\{e_n\}_n$ and $\{f_n\}_n$ be as in \ref{DMS21}.
By \ref{LMS24}, we may assume that there exists a pure state $\omega$ of $A$   and
$c_{l,i} d_i\in A,$ $l,i=1,2,...,N,$ such that
\beq
\phi(a)=\sum_{l=1}^N\sum_{i,j=1}^N\omega(d_i^*ad_j)c_{l,i}^*c_{l,j}\rforal a\in A.
\eneq
Set $\phi_l(a)=\sum_{i,j=1}^N\omega(d_i^*ad_j)c_{l,i}^*c_{l,j}$ for all $a\in A,$ $l=1,2,...,N.$
Thus $\phi=\sum_{l=1}^N\phi_l.$
Note that Lemma 3.4 of \cite{MS} holds for non-unital case, in particular, holds
for the case $A\in {\cal D}$ which can also be directly proved by repeatedly using the construction
in \ref{LMS33} in $\overline{f_nAf_n}.$  Therefore we also have a central
sequence $\{f_{l,n}\}_n,$ $l=1,2,...,N,$ of positive contractions in $A$ such that
$\{f_nf_{l,n}\}_n=\{f_{l,n}\},$ $\{f_{l,n}f_{l',n}\}_n=0,$ $l\not=l',$ $l=1,2,...,N,$ in $A_{\infty}\cap A',$
and
\beq
\lim_{m\to\infty}\limsup_{n\to\infty}\inf_{\tau\in T(A)}\tau(f_{l,n}^m)>0.
\eneq
Applying  \ref{PMS22} to $\phi_l,$  $\{e_n\}_n$ and $\{f_{l,n}\}_n,$  we obtain  a sequence
$\{s_{l,n}\}_n$ in $A^{\bf 1}$ such that
\vspace{-0.08in}\beq
\lim_{n\to\infty}\|s_{l,n}^*as_{l,n}-\phi_l(a)e_n\|=0\andeqn
\lim_{n\to\infty}\|f_ns_{l,n}-s_{l,n}\|=0.
\eneq
Put $s_n=\sum_{l=1}^Ns_{l,n}.$  One estimates that (recall that $\|s_{l,n}\|\le  1$)
\vspace{-0.1in}\beq\nonumber
\|f_ns_n-s_n\| &\le & \sum_{l=1}^N\|f_ns_{l,n}-s_{l,n}\|\\\nonumber
\vspace{-0.08in} &\le & \sum_{l=1}^N(\|f_ns_{l,n}-f_nf_{l,n}s_{l,n}\|+\|f_nf_{l,n}s_{l,n}-f_{l,n}s_{l.n}\|+\|f_{l,n}s_{l,n}-s_{l,n}\|)\\\nonumber
\vspace{-0.08in} &\le & \sum_{l=1}^N(\|f_n\|\|s_{l,n}-f_{l,n}s_{l,n}\|+\|f_nf_{l,n}-f_{l,n}\|\|s_{l,n}\|+\|f_{l,n}s_{l,n}-s_{l,n}\|)\to 0,
\eneq
as $n\to\infty.$
If $l\not=l',$ then, since $\{f_{l,n}\}_n$ is central and $\{f_{l,n}f_{l',n}\}_n=0$ in $A_{\infty},$
\beq
\lim_{n\to\infty}\|s_{l,n}^*as_{l',n}\|=\lim_{n\to\infty}\|s_{l,n}^*f_{l,n}af_{l',n}s_{l.n}\|=0.
\eneq
Therefore, for all $a\in A,$
\vspace{-0.12in}\beq
\lim_{n\to\infty}\|s_n^*as_n-\phi(a)e_n\|=\lim_{n\to\infty}\|\sum_{l=1}^Ns_{l,n}^*as_{l,n}-\phi_l(a)e_n\|=0.
\eneq
\end{proof}

\begin{df}[cf. 4.1 of \cite{MS}]\label{DMS41perp}
Let $A$ be a separable \CA\, with $T(A)\not=\emptyset$ and
with $T(A)$ compact.  We say $A$ has property (SI) if for any central sequence
$\{e_n\}_n$ and $\{f_n\}_n$ which satisfy \eqref{DMS21-1}, there exists
a central sequence $\{s_n\}_n$ in $A$  such that
\beq
\lim_{n\to\infty}\|f_ns_n-s_n\|=0\tand \{s_n^*s_n\}_n-\{e_n\}_n\in A^{\perp},
\eneq
where $A^{\perp}=\{\{b_n\}_n\in A_{\infty}: \{b_n\}_nA=A\{b_n\}_n=0\}.$
\end{df}

\begin{lem}\label{LMS4to1}
Let $A$ be a separable amenable \CA\, in ${\cal D}$ with continuous scale.
Then $A$ has (SI).
\end{lem}

\begin{proof}
Let $\{e_n\}_n$ and $\{f_n\}_n$ be as in \eqref{DMS21-1}.
Then, by \ref{LMS2to3}, ${\rm id}_A$ can be excised in small central  sequences.
Thus there is a sequence $s_n'\in A^{\bf 1}$ such that
$\lim_{n\to\infty}\|(s_n')^*a(s_n')-ae_n\|=0$ for all $a\in A$ and $\lim_{n\to\infty}\|f_ns_n-s_n\|=0.$
Fix an approximate identity $\{d_n\}$ of $A.$
By passing  to $s_{n_k}', e_{n_k}'$ and $f_{n_k},$ if necessary, we may assume
further that
\beq
\lim_{n\to\infty}\|(s_n')^*d_n(s_n')-d_ne_n\|=0\andeqn \lim_{n\to\infty}\|f_nd_n^{1/2}-d_n^{1/2}f_n\|=0.
\eneq
Define $s_n=d_n^{1/2}s_n',$ $n=1,2,....$
Then
\beq
\lim_{n\to\infty}\|s_n^*s_n-d_ne_n\|=0\andeqn
\lim_{n\to\infty}\|f_ns_n-s_n\|=\lim_{n\to\infty}\|d_n^{1/2}(f_ns_n'-s_n')\|=0.
\eneq
Moreover, for any $a\in A,$ since $\{d_n\}$ is an approximate identity for $A,$
\beq\label{LMS4to1-10}
\lim_{n\to\infty} \|a(s_n^*s_n)-ae_n)\|\le \lim_{n\to\infty}\|a(s_n')^*d_n(s_n')-ad_ne_n\|+
\lim_{n\to\infty}\|ad_ne_n-ae_n\|=0.
\eneq
It follows that $\{s_n^*s_n\}_n-\{e_n\}_n\in A^{\perp}.$
Moreover, for $a\in A,$ by \eqref{LMS4to1-10},
\beq
\lim_{n\to\infty}\|[s_n,\, a]\|^2=\lim_{n\to\infty}\|as_n^*s_na-a^*s_n^*as_n-s_n^*a^*s_na+s_n^*a^*as_n\|\\
=\lim_{n\to\infty}\|as_n^*s_na-a^*e_na\|=\lim_{n\to\infty}\|a(s_n^*s_n-e_n)a\|=0.
\eneq
Therefore $\{s_n\}_n$ is a central sequence.
\end{proof}

\begin{thm}\label{TMST}
Every  separable amenable \CA\, in ${\cal D}$  is ${\cal Z}$-stable.
\end{thm}

\begin{proof}
Let $A\in {\cal D}.$  It suffices to show
that a non-zero hereditary \SCA\, of $A$ is ${\cal Z}$-stable.
Therefore,  by  {{11.7 of \cite{eglnp1},}}
we may assume that  $A$ has continuous scale.

Fix any integer $k>1.$ By Lemma \ref{LMS33}, we obtain a central sequence
$\{c_{i,n}\}_n$ in $A,$ $i=1,2,...,k,$  such that $\{c_{i,n}c_{j,n}^*\}_n=\dt_{i,j}\{c_{1,n}^2\}_n$
in $A_{\infty}$ and
\beq\label{TMST-1}
\lim_{n\to\infty}\sup_{\tau\in T(A)}|\tau(c_{1,n}^m)-1/k|=0\rforal m\in \N.
\eneq
Thus we obtain an order zero \cpc\, $\phi: M_k\to A_{\infty}\cap A'$
such that $\phi(e)=\{c_{1,n}\}_n$ for a minimal projection $e\in M_k.$
Let $\{d_n\}$ be an approximate identity for $A.$
Then $\{d_n\}_n$ is a central sequence.
Then $\overline{\{d_n\}_n}$ is the identity of $A_{\infty}\cap A'/A^{\perp},$
where $\overline{\{d_n\}_n}$ is the image of $\{d_n\}_n$ in $A_{\infty}\cap A'/A^{\perp}.$
We may choose such $\{d_n\}$ so that
$\{d_n-\sum_{i=1}^Nc_{i,n}^*c_{i,n}\}_n\in (A_{\infty})_+.$
Note that, since $A$ has continuous scale,
$\lim_{n\to\infty} \sup_{\tau\in T(A)}\tau(d_n)=1.$
Let $\{e_n\}$ be a  central sequence of positive contraction such that
$\{e_n\}_n=\{d_n-\sum_{i=1}^kc_{i,n}^*c_{i,n}\}_n.$
As in \ref{LMS33} $\{c_{i,n}\}_n$ can be chosen so that
\beq
\limsup_{n\to\infty}\sup_{\tau\in T(A)}\tau(e_n)=0
\eneq
which can also be computed directly from \eqref{TMST-1}.
Then, we also have
\beq
\lim_{m\to\infty}\liminf_{n\to\infty}\inf_{\tau\in T(A)}\tau(c_{1,n}^m)=1/k.
\eneq
By the property (SI), we obtain a central sequence $\{s_n\}$ in $A^{\bf 1}$ such
that
\beq
\{s_n^*s_n\}_n-\{e_n\}_n\in A^{\perp}\andeqn
\lim_{n\to\infty}\{c_{1,n}s_n\}_n=\{s_n\}_n \,\,\, {\rm in}\,\,\, A_{\infty}.
\eneq
Thus we obtain an order zero \cpc\, $\Phi: M_k\to A_{\infty}\cap A'/A^{\perp}$ induced
by $\phi$ and $s=\overline{\{s_n\}_n}\in A_{\infty}\cap A'/A^{\perp}$ such that,
\beq
s^*s+\Phi(1_{M_k})=1\andeqn \Phi(e)s=s\,\,\, {\rm in}\,\,\, A_{\infty}\cap A'/A^{\perp}
\eneq
This implies that $A\otimes {\cal Z}\cong A$  as in the  proof of (iv) $\Longrightarrow$ (i)
in section 4 of \cite{MS}, see also,  for example,  Proposition 5.3 and 5.6 of \cite{aTz}.

\end{proof}

\begin{rem}
More general result related to this appendix will appear elsewhere.
\end{rem}

\providecommand{\href}[2]{#2}

\end{document}